\documentclass[11pt,letterpaper]{amsart}
\usepackage[margin=1.25in]{geometry}
\usepackage{xcolor}

\usepackage[T1]{fontenc}
\usepackage[utf8]{inputenc}
\usepackage{lmodern}
\usepackage{microtype}
\usepackage{float}
\usepackage{tikz}
\usetikzlibrary{arrows.meta,positioning,fit,calc}
\usepackage{amsmath,amssymb,amsthm,mathtools,mathrsfs,bm}
\allowdisplaybreaks[0]
\usepackage{mathtools}
\mathtoolsset{showonlyrefs=true}
\numberwithin{equation}{section}

\usepackage{enumitem}
\usepackage[hidelinks,pdfusetitle]{hyperref}

\usepackage{booktabs,tabularx,longtable,array}

\makeatletter
\def\@tocline#1#2#3#4#5#6#7{\relax
  \ifnum #1>\c@tocdepth 
  \else
    \par \addpenalty\@secpenalty\addvspace{#2}%
    \begingroup \hyphenpenalty\@M
    \@ifempty{#4}{%
      \@tempdima\csname r@tocindent\number#1\endcsname\relax
    }{%
      \@tempdima#4\relax
    }%
    \parindent\z@ \leftskip#3\relax \advance\leftskip\@tempdima\relax
    \rightskip\@pnumwidth plus4em \parfillskip-\@pnumwidth
    #5\leavevmode\hskip-\@tempdima
      \ifcase #1
       \or\or \hskip 2.5em \or \hskip 2em \else \hskip 3em \fi%
      #6\nobreak\relax
    \hfill\hbox to\@pnumwidth{\@tocpagenum{#7}}\par
    \nobreak
    \endgroup
  \fi}
\makeatother

\setlist[enumerate]{itemsep=0.25em,topsep=0.5em}
\setlist[description]{itemsep=0.35em,topsep=0.5em}

\newtheorem{theorem}{Theorem}[section]
\newtheorem{proposition}[theorem]{Proposition}
\newtheorem{lemma}[theorem]{Lemma}
\newtheorem{corollary}[theorem]{Corollary}

\theoremstyle{definition}
\newtheorem{definition}[theorem]{Definition}

\theoremstyle{remark}
\newtheorem{remark}[theorem]{Remark}

\newcommand{\N}{\mathbb N}
\newcommand{\Nzero}{\mathbb N_0}
\newcommand{\Z}{\mathbb Z}
\newcommand{\R}{\mathbb R}
\newcommand{\C}{\mathbb C}
\newcommand{\T}{\mathbb T}
\newcommand{\HDir}{\mathscr H}
\newcommand{\ZN}{\mathbb Z_N}
\newcommand{\Bohr}{\mathcal B}

\newcommand{\E}{\mathbb E}
\newcommand{\Prob}{\mathbb P}
\newcommand{\e}{\mathrm e}
\newcommand{\dd}{\,\mathrm d}
\newcommand{\ii}{\mathrm i}
\newcommand{\one}{\mathbf 1}
\newcommand{\Ree}{\operatorname{Re}}
\newcommand{\Imm}{\operatorname{Im}}
\newcommand{\Var}{\operatorname{Var}}
\newcommand{\Cov}{\operatorname{Cov}}
\newcommand{\supp}{\operatorname{supp}}
\newcommand{\dist}{\operatorname{dist}}

\newcommand{\abs}[1]{\left|#1\right|}
\newcommand{\norm}[1]{\left\lVert#1\right\rVert}
\newcommand{\ip}[2]{\left\langle#1,#2\right\rangle}

\title[The local embedding problem]
{The Local Embedding Problem for Hardy Spaces of Dirichlet Series}

\author{Bonan Chen}
\address{
School of Mathematical Sciences,
Soochow University,
Suzhou 215006, P. R. China
}
\email{bnchen@suda.edu.cn}

\author{Xiang Fang}
\address{
Department of Applied Mathematics,
National Yang Ming Chiao Tung University,
Hsinchu, Taiwan
}
\email{xfang@nycu.edu.tw}

\author{Feng Guo}
\address{
School of Mathematics, South China University of Technology, Guangzhou 510640, P. R. China
}
\email{70207994@nuaa.edu.cn}

\author{Shengzhao Hou}
\address{
School of Mathematical Sciences,
Soochow University,
Suzhou 215006, P. R. China
}
\email{shou@suda.edu.cn}

\author{Yizhou Shao}
\address{
School of Mathematical Sciences,
Soochow University,
Suzhou 215006, P. R. China
}
\email{20244207001@stu.suda.edu.cn}

\author{Qi Zhou}
\address{
School of Mathematical Sciences,
Soochow University,
Suzhou 215006, P. R. China
}
\email{zhouqi@suda.edu.cn}

\date{}

\hypersetup{
  pdftitle={The Local Embedding Problem for Hardy Spaces of Dirichlet Series},
  pdfauthor={Bonan Chen, Xiang Fang, Feng Guo, Shengzhao Hou, Yizhou Shao, Qi Zhou}
}

\keywords{Hardy spaces of Dirichlet series, local embedding problem,
Gaussian replacement, finite-cyclic square functions,
log-correlated Gaussian fields, branching random walk}

\subjclass[2020]{Primary 30B50; Secondary 30H10, 42B25, 60G15}

\begin{document}

\begin{abstract}
We solve the local embedding problem for Hardy spaces of Dirichlet
series.  For every \(2<p<\infty\), the global \(\HDir^p\)-norm
controls the \(L^p\)-norm of a Dirichlet polynomial on every unit
interval of the line \(\Ree s=1/2\):
\[
 \sup_{\theta\in\R}
 \int_{\theta}^{\theta+1}
 \abs{P\!\left(\frac12+\ii t\right)}^p\dd t
 \le
 C_p\norm{P}_{\HDir^p}^{p},
\]
where \(C_p\) depends only on \(p\), not on the interval or on the
number and choice of primes occurring in \(P\).  The inequality was
previously known at \(p=2\) and, by taking powers, at every even
exponent; it had been conjectured that these were the only positive
cases above \(2\).  Together with the known counterexamples for
\(0<p<2\), our theorem shows that the embedding holds for finite
\(p\) exactly when \(p\ge2\).

The proof uses duality to control many critical-line evaluations
together.  A frequency decomposition isolates the contribution for
which the relevant frequencies can nearly coincide.  We estimate it
by replacing the leading terms in the logarithms of the Euler factors
with Gaussian variables, using the same variable for each prime at
all evaluation points; the accumulated errors remain bounded.
Comparison with a critical branching random walk and estimates for a
finite-cyclic square function then give a bound independent of the
number of prime variables.  For \(1\le p<\infty\), known
equivalences also give the sharp range for the conformally invariant
half-plane embedding, the reverse transfer of local Carleson-measure
estimates, and boundedness of all characteristic-zero
Gordon--Hedenmalm composition operators.
\end{abstract}

\maketitle

\vspace{1.5cm}

\tableofcontents

\newpage

\section{Introduction and main results}
\label{sec:introduction}

\subsection{Hardy spaces of Dirichlet series and the local embedding problem}
\label{subsec:local-embedding-problem}
\label{subsec:earlier-work}

Unique prime factorization gives ordinary Dirichlet series an  
infinite-dimensional holomorphic model.   If $P$ is a Dirichlet polynomial given by
\[
 P(s)=\sum_n a_n n^{-s}
 \qquad\text{and}\qquad
 n=\prod_r r^{\nu_r(n)},
\]
then its Bohr lift is the analytic polynomial
\[
 \Bohr P(z)=\sum_n a_n z^{\nu(n)}
\]
where \(z^{\nu(n)}=\prod_r z_r^{\nu_r(n)}\).  Only finitely many
primes occur in \(P\), so \(\Bohr P\) depends on finitely many
coordinates of \(\T^\infty\).  For every \(0<p<\infty\), one has
\[
 \norm{P}_{\HDir^p}^{p}
 =
 \int_{\T^\infty}\abs{\Bohr P(z)}^p\dd m(z)
 =
 \lim_{T\to\infty}\frac1{2T}
 \int_{-T}^{T}\abs{P(\ii t)}^p\dd t,
\]
where \(m\) is normalized Haar measure on \(\T^\infty\).
Thus the \(\HDir^p\)-norm is the Hardy norm of the Bohr lift; see
\cite{Bayart02, HLS97}.

Write
\[
 \C_{1/2}=\{s\in\C:\Ree s>1/2\}.
\]
For \(\sigma>1/2\), evaluation at \(\sigma+\ii t\) is bounded on
\(\HDir^p\) \cite{Bayart02}.  Its norm is
\(\zeta(2\sigma)^{1/p}\), which tends to infinity as
\(\sigma\downarrow1/2\) \cite[(4.1)]{BayartQueffelecSeip16}.
Thus the interior pointwise bound gives no estimate on the critical line.
This does not rule out control after integration in \(t\): an integral
estimate can hold even when individual point evaluations are unbounded.

Under the Bohr correspondence,
\[
 P\!\left(\frac12+\ii t\right)
 =
 \Bohr P\bigl((r^{-1/2-\ii t})_r\bigr).
\]
Consequently, restriction to $\Ree s=1/2$ amounts to
evaluating the Bohr lift along the curve
\[
 t\longmapsto (r^{-1/2-\ii t})_r
\]
in the infinite polydisc.
The $\HDir^p$-norm averages over all prime coordinates,
whereas the critical-line trace follows a 
one-dimensional curve.

The local embedding problem asks whether the following
uniform estimate holds:
\begin{equation}
 \sup_{\theta\in\R}
 \int_{\theta}^{\theta+1}
 \abs{P\!\left(\frac12+\ii t\right)}^p\dd t
 \le
 C_p\norm{P}_{\HDir^p}^{p},
 \label{eq:intro-local-embedding}
\end{equation}
for every Dirichlet polynomial \(P\), with \(C_p\) depending only on \(p\).
In particular, the same constant must work for every translated
interval and every choice of primes occurring in \(P\), however many
there are. 

The modern Hilbert-space theory of Hardy spaces of Dirichlet
series was initiated by Hedenmalm, Lindqvist, and Seip in a
seminal work \cite{HLS97}. Bayart introduced and developed
the general scale \(\HDir^p\), \(0<p<\infty\), and first
considered the local embedding question \cite{Bayart02}.
Saksman and Seip formulated the problem explicitly as
Question~2 in \cite{SaksmanSeip09}; it was later recorded
as Problem~2.1 in their list of open problems for Dirichlet series
\cite[Problem~2.1]{SaksmanSeip16}. 
Brevig, Ortega-Cerd{\`a}, and Seip later described the
local embedding problem as ``perhaps the most important
open problem'' in the study of $\HDir^p$
\cite[Section~4]{BrevigOrtegaCerdaSeip21}.

\newpage

The local embedding problem has several equivalent formulations.
Bayart and Brevig showed that it is equivalent to the
conformally invariant half-plane embedding and to boundedness
of every characteristic-zero Gordon--Hedenmalm composition
operator; in fact, one particular operator suffices
\cite[Section~2.1 and Theorem~3]{BayartBrevig19}.
Olsen and Saksman proved that every Carleson measure for
$\HDir^p$ is a Carleson measure for the classical half-plane
Hardy space \cite[Theorem~3]{OlsenSaksman12}.
They also showed that the converse for local Carleson measures
is equivalent to the local embedding problem, and that it
suffices to test weighted atomic measures
\cite[Theorem~4]{OlsenSaksman12}.
For $2\le p<\infty$, Saksman and Seip proved equivalence
with the $p$-Carlson identity holding at every starting point
on the critical torus
\cite[Sections~3--4]{SaksmanSeip09}.
The precise statements and their exponent ranges are given in
Section~\ref{subsec:main-results} and
Proposition~\ref{prop:classical-equivalence-map}.

The local embedding at \(p=2\) is established in
\cite[Theorem~4.11]{HLS97}; see also the note added in
proof there concerning Montgomery's form of Hilbert's
inequality. Applying this quadratic estimate to \(P^k\)
gives the embedding at every even exponent \(p=2k\).

For $0<p<2$, the local embedding fails. Indeed, let
\(
P_X(s)=\sum_{n\le X}n^{-s}.
\)
Uniformly for $0\le t\le1$,
\[
P_X\left(\frac12+\ii t\right)
=\frac{X^{1/2-\ii t}}{1/2-\ii t}+O(1),
\]
and hence
\[
\int_0^1
\abs{P_X\left(\frac12+\ii t\right)}^p\dd t
\asymp_p X^{p/2}.
\]
On the other hand, Harper's low-moment estimate
\cite{Harper20} gives
\[
\norm{P_X}_{\HDir^p}^{p}
\asymp_p
\frac{X^{p/2}}{(\log\log X)^{p/4}}.
\]
The ratio of the local integral to the global
$p$-th moment therefore tends to infinity as
$X\to\infty$. 
For $p\ge2$, this example gives no obstruction, since
\[
\sup_{t\in\R}\abs{P_X\left(\frac12+\ii t\right)}
\lesssim \sqrt X
\sim \norm{P_X}_{\HDir^2}
\le \norm{P_X}_{\HDir^p}.
\]

The known even-exponent estimates do not yield the
remaining cases by H\"older's inequality. Indeed,
an estimate at $p_1$ gives a bound in terms of
$\norm{P}_{\HDir^{p_1}}$, which cannot be controlled
by $\norm{P}_{\HDir^{p_0}}$ uniformly in $P$
when $p_0<p_1$.

Nor does standard interpolation settle the problem.
For $2<p<4$, interpolation of the known estimates
at $p=2$ and $p=4$ gives
\[
T_I:[\HDir^2,\HDir^4]_\theta\longrightarrow L^p(I),
\qquad
\frac1p=\frac{1-\theta}{2}+\frac{\theta}{4},
\]
where $I\subset\R$ is any interval of length one and
$T_I P(t)=P(1/2+\ii t)$ for $t\in I$.
To obtain the desired estimate, we would need
\[
\norm{P}_{[\HDir^2,\HDir^4]_\theta}
\lesssim_p \norm{P}_{\HDir^p}.
\]
However, the classical interpolation results for
finite-dimensional Hardy spaces do not extend in
this form to $\T^\infty$
\cite{BayartMastylo19}; see also
\cite[Section~4]{BrevigOrtegaCerdaSeip21}.

The non-even exponents $p>2$ thus remained open.
Queff\'{e}lec and Queff\'{e}lec conjectured that
the local embedding holds only for even integers
among finite exponents
\cite[p.~274]{QueffelecQueffelec21}; see also
\cite{BrevigOrtegaCerdaSeip21}.
Contrary to this conjecture, we prove that the local
embedding holds for every $2<p<\infty$.

\subsection{Main theorem, interpretation, and equivalent formulations}
\label{subsec:main-results}
\label{subsec:intro-consequences}

 Our main result is the following. 

\begin{theorem}
\label{thm:local-embedding}
For every \(2<p<\infty\), there exists a constant \(C_p<\infty\) such that
\[
 \sup_{\theta\in\R}
 \int_{\theta}^{\theta+1}
 \abs{P\!\left(\frac12+\ii t\right)}^p\dd t
 \le
 C_p\norm{P}_{\HDir^p}^{p}
\]
for every Dirichlet polynomial \(P\). 
\end{theorem}

\begin{corollary}
\label{cor:sharp-finite-range}
For \(0<p<\infty\), the local embedding inequality
\eqref{eq:intro-local-embedding} holds for every Dirichlet polynomial if and
only if
\(
 p\ge2.
\)
\end{corollary}

In particular, our theorem answers Question~3 of
Saksman and Seip \cite{SaksmanSeip09} affirmatively:
for every $2<q<p<4$,
\[
\left(
\int_0^1\abs{P\!\left(\frac12+\ii t\right)}^q\dd t
\right)^{1/q}
\lesssim_p \norm P_{\HDir^p}.
\]

The local embedding problem also has a number of equivalent reformulations involving Hardy spaces on
the half-plane, Carleson measures, and composition operators.  These
equivalences were established in earlier work; our theorem now determines
when the corresponding properties hold.

Let \(H_{\mathrm i}^p(\C_{1/2})\) denote the conformally invariant Hardy
space with boundary weight \((1+t^2)^{-1}\), let \(H^p(\C_{1/2})\) denote
the classical half-plane Hardy space with its unweighted boundary norm,
and let \(\mathcal G_0\) be the characteristic-zero Gordon--Hedenmalm class.
$$
\begin{aligned}
 \mathrm{LEP}_p
 &\Longleftrightarrow
 \left[\HDir^p\hookrightarrow H_{\mathrm i}^p(\C_{1/2})\right]
 \\
 &\Longleftrightarrow
 \left[
 \begin{array}{c}
  \text{every local }H^p(\C_{1/2})\text{-Carleson measure}\\[-1mm]
  \text{is a Carleson measure for }\HDir^p
 \end{array}
 \right]
 \\
 &\Longleftrightarrow
 \left[
 \begin{array}{c}
  \text{the Olsen--Saksman atomic testing condition holds}
 \end{array}
 \right]
 \\
 &\Longleftrightarrow
 \left[
 \begin{array}{c}
  \mathcal C_\varphi\colon\HDir^p\to\HDir^p\text{ is bounded for every }\varphi\in\mathcal G_0
 \end{array}
 \right]
 \\
 &\Longleftrightarrow
 \left[\begin{array}{c}
  \mathcal C_\psi\colon\HDir^p\to\HDir^p\text{ is bounded} \end{array}
 \right]
 \\
 &\Longleftrightarrow
 \left[
 \begin{array}{c}
  \text{the \(p\)-Carlson identity holds for every starting}\\[-1mm]
  \text{point on the critical torus}\quad (2\le p<\infty)
 \end{array}
 \right]
 .
\end{aligned}
$$
Here \(\mathcal C_\varphi f=f\circ\varphi\), and the single operator
appearing above has the symbol
$$
 \psi(s)=\frac12+\frac{1-2^{-s}}{1+2^{-s}}.
$$
For \(1\le p<\infty\), the conditions through boundedness of
\(\mathcal C_\psi\) are equivalent.  The Carlson condition is included
only for \(2\le p<\infty\).

Bayart and Brevig proved the half-plane and characteristic-zero
composition-operator equivalences
\cite[Section~2.1 and Theorem~3]{BayartBrevig19}.  In particular,
boundedness of this one operator is equivalent to boundedness of every
\(\mathcal C_\varphi\) with \(\varphi\in\mathcal G_0\).

For the Carleson-measure condition, the question is whether every measure
supported in a bounded part of the half-plane that gives an \(L^p\) estimate
for classical \(H^p\) functions also gives one for \(\HDir^p\).
Olsen and Saksman proved this equivalence and showed that it suffices to
test measures consisting of point masses,
$$
 \mu_S
 =
 \sum_n(2\sigma_n-1)\,\delta_{\sigma_n+\ii t_n},
 \qquad \sigma_n>\frac12,
$$
subject to the corresponding classical half-plane Carleson condition
\cite[Theorem~4]{OlsenSaksman12}.  The mass \(2\sigma_n-1\) is twice the
distance of the point from the critical line.

Combining the characteristic-zero equivalence with Bayart's
positive-characteristic theory gives the full Gordon--Hedenmalm boundedness
characterization.  
For $2\le p<\infty$, Saksman and Seip also proved the
equivalence with the $p$-Carlson identity
\cite[Section~3, equation~(17), and Section~4,
equation~(20)]{SaksmanSeip09}.
The time-average identity is required to hold for every
starting point on the critical torus, not merely for
almost every such point.

Proposition~\ref{prop:classical-equivalence-map} records these earlier
equivalences.  Together with the known failure below \(2\), our theorem
shows that the half-plane, Carleson-measure, and composition-operator
conditions hold, for \(1\le p<\infty\), exactly when \(p\ge2\).  The
Carlson condition is stated only for \(2\le p<\infty\) and holds
throughout that range.

The local embedding problem also admits a sampling
characterization. Here the sample points lie on the
critical line, unlike those in the atomic testing
condition of Olsen and Saksman
\cite[Theorem~4]{OlsenSaksman12}.
For an integer $X\ge3$, let
\[
P(s)=\sum_{1\le n<X}a_n n^{-s},
\qquad L_X=\log X.
\]
The function $t\mapsto P(1/2+\ii t)$ has frequencies
in $[-L_X,0]$. For $1<p<\infty$ and fixed $0<c\le1$,
let $\mathcal T$ be a finite set in a unit interval with
$|\tau-\tau'|\ge c/L_X$ whenever $\tau\ne\tau'$.
The critical atomic sampling property
$\mathrm{CAS}_{p,c}$ is the estimate
\[
\frac1{L_X}
\sum_{\tau\in\mathcal T}
\abs{P\!\left(\frac12+\ii\tau\right)}^p
\lesssim_{p,c}
\norm P_{\HDir^p}^{p},
\]
uniformly in $X$, $\mathcal T$, and $P$.
Theorem~\ref{thm:critical-sampling-bridge} proves
\[
\mathrm{LEP}_p
\quad\Longleftrightarrow\quad
\mathrm{CAS}_{p,c}
\quad\Longleftrightarrow\quad
\mathrm{DCP}_{p,c},
\]
where $\mathrm{DCP}_{p,c}$ is the corresponding dual
inequality for linear combinations of point evaluations;
see Definition~\ref{def:dual-critical-packet}. 
The dual formulation also motivates our proof:
we estimate linear combinations of critical-line
evaluations on equally spaced grids, uniformly
over their shifts.

\subsection{Proof structure for Theorem~\ref{thm:local-embedding} and main ideas}
\label{subsec:proof-overview}

Our starting point was to study the classical extremal problem for a single critical-line
evaluation: how large can a point evaluation
$P(1/2+\ii t)$ be relative to $\norm P_{\HDir^p}$? The extremizer leads us immediately to the study of Euler products; see \eqref{eq:euler-source-Lqfunction}.  But the target of the local embedding problem is an integral and we need to understand how large values are distributed
across an interval.
A natural next step is to study the finite linear combinations of point evaluations. This leads us to discover a new reformulation of LEP via the CAS  property; see Theorem~\ref{thm:critical-sampling-bridge}. This identifies the correct object although not the precise route.

Two familiar patterns illustrate the key point of CAS. The values of a general function may remain
large over a broad set, or be concentrated in high, narrow
peaks. 
For a Dirichlet polynomial $P$ with  $X\ge3$, we divide a unit interval $I$ into consecutive intervals of length $L_X^{-1}$, except possibly the last, and choose a point $\tau_j$ of maximum modulus for $P$ in each. Then
\[
 \int_I \abs{P\!\left(\frac12+\ii t\right)}^p\dd t
 \le
 \frac{1}{L_X}\sum_j
 \abs{P\!\left(\frac12+\ii\tau_j\right)}^p.
\]
A broad plateau contributes through many neighboring intervals;
separated peaks contribute through large maxima in the intervals
containing them. Both are bounded by the same sum on the right of the above inequality.

Clearly, the above upper bound holds for any continuous function. We may split the selected
points into two $L_X^{-1}$-separated families, as required in
CAS. For Dirichlet
polynomials, Theorem~\ref{thm:critical-sampling-bridge}  gives a simple yet decisive converse inequality.
Thus one uniform CAS estimate would control both patterns: plateau and peaks. This
equivalence identifies the estimate we need to control. Indeed, it leads us to isolate finite linear combinations of point evaluations as the focus of our study. 

Our proof follows this viewpoint but uses only a natural sampling. We use $N$ equally spaced points with a  common
shift; integrating over one grid spacing recovers the local integral
exactly; see \eqref{eq:grid-average-tiling}.   By $\ell^p$--$\ell^q$ duality,
the norm of the sampled values in \eqref{eq:grid-average-tiling} equals a supremum over
linear combinations of evaluations. 
Since these functionals have normalized
Euler-product representatives, we seek a
joint bound for these products.
This leads us to seek an estimate for a projected vector.
%

\medskip
\noindent\emph{(1) From critical-line sampling to the projected vector.}
We fix $2<p<\infty$, $q=p/(p-1)\in(1,2)$ for the conjugate exponent, $N=2^H$, $y=\e^N$, and $t_j=t_0+j/N$ for $0\le j<N$.
We set
\[
M_y=\prod_{r\le y}(1-r^{-1})^{-1}\asymp \log y=N,
\]
where $r$ runs over primes. Our goal is to prove the joint
sampling estimate
\begin{equation}
\sum_{j=0}^{N-1}
\abs{P\!\left(\frac12+\ii t_j\right)}^p
\le C_p M_y\norm P_{\HDir^p}^p,
\label{eq:intro-joint-sampling}
\end{equation}
uniformly in $N$ and $t_0$, for every Dirichlet polynomial $P$
involving only primes $r\le y$; see
Proposition~\ref{prop:discrete-local-sampling}.
Indeed, integrating over $t_0=\theta+\tau$,
$0\le\tau\le 1/N$, recovers the local integral exactly
by \eqref{eq:grid-average-tiling}.

On the finite torus $\T^{\pi(y)}$ with coordinates
$(\zeta_r)_{r\le y}$, the single-point extremizer
leads to the following $L^q$ function in terms of Euler product:
\begin{equation}
\begin{split}
g_{y,t}&=M_y^{-1/p}f_y(t),\\
f_y(t)&=M_y^{-(1-2/p)}
\prod_{r\le y}
(1-r^{-1/2}\e^{\ii t\log r}\zeta_r)^{-1}
\prod_{r\le y}
(1-r^{-1/2}\e^{-\ii t\log r}
\overline{\zeta_r})^{-(1-2/p)}.
\end{split}
\label{eq:euler-source-Lqfunction}
\end{equation}
Its normalization and analytic Fourier coefficients give
\begin{equation}
\norm{g_{y,t}}_{L^q}=1,
\qquad
\ip{g_{y,t}}{\Bohr P}
=
M_y^{-1/p}
\overline{P\!\left(\frac12+\ii t\right)},
\label{eq:intro-source-pairing}
\end{equation}
where $\ip{G}{F}=\int G\overline F\,\dd m$.

For $c=(c_j)_{j<N}$, set
$G_c=\sum_{j<N}c_jg_{y,t_j}$.
Since $\Bohr P$ is analytic, only the class of $G_c$
in the Hardy quotient $L^q/\mathcal N_p$ matters,
where $\mathcal N_p$ is the annihilator of $H^p$.
In particular, Fourier monomials with a negative exponent
in at least one prime coordinate vanish in this quotient.
By \eqref{eq:intro-source-pairing} and duality,
the joint sampling estimate \eqref{eq:intro-joint-sampling} follows if
\begin{equation}
\norm{[G_c]}_{L^q/\mathcal N_p}
\le C_p\norm c_{\ell^q(\ZN)}.
\label{eq:intro-quotient-bound}
\end{equation}

To simplify $G_c$, expand $f_y(t)$ in increasing prime order:
\begin{equation}
f_y(t)
=
1+\sum_{r\le y}f_{r^-}(t)
\left(\eta_{1/r}(\e^{\ii t\log r}\zeta_r)-1\right),
\label{eq:expand-fy}
\end{equation}
where $f_{r^-}(t)$ is the Euler product restricted
to primes $\rho<r$ (see \eqref{eq:old-prime-source}),
and $\eta_x$ is defined in \eqref{def:eta-function}.
We further expand $\eta_{1/r}(\e^{\ii t\log r}\zeta_r)$ in Fourier series to obtain
\[
\begin{aligned}
\eta_{1/r}(\e^{\ii t\log r}\zeta_r)-1
={}&
\sum_{k<0}\widehat{\eta_{1/r}}(k)
\e^{\ii kt\log r}\zeta_r^k+r^{-1/2}\e^{\ii t\log r}\zeta_r
+H_r^+(t),
\end{aligned}
\]
where
\[
H_r^+(t)
=
\sum_{k\ge2}r^{-k/2}\e^{\ii kt\log r}\zeta_r^k.
\]
The negative powers of $\zeta_r$ vanish in the Hardy
quotient, since $f_{r^-}(t)$ does not depend on $\zeta_r$.
We set
\[
\begin{aligned}
H_0(c)
&=M_y^{-1/p}\sum_{j<N}c_j,\\
A_r(c)
&=M_y^{-1/p}\sum_{j<N}
c_j\e^{\ii t_j\log r}f_{r^-}(t_j),\\
\mathcal Q_y(c)
&=M_y^{-1/p}\sum_{r\le y}\sum_{j<N}
c_jf_{r^-}(t_j)H_r^+(t_j).
\end{aligned}
\]
Substituting the expansion into $G_c$ gives
\begin{equation}
[G_c]
=
\left[
H_0(c)+
\sum_{r\le y}r^{-1/2}\zeta_r A_r(c)
+\mathcal Q_y(c)
\right]
\quad\text{in }L^q/\mathcal N_p.
\label{eq:intro-unprojected-decomposition}
\end{equation}
The constant term and the higher positive harmonics satisfy
\[
|H_0(c)|+\norm{\mathcal Q_y(c)}_q
\lesssim_p\norm c_{\ell^q(\ZN)}
\]
by Proposition~\ref{prop:quotient-representation}.

We further simplify the first-harmonic sum in the Hardy
quotient. If a Fourier monomial
$\zeta^\nu=\prod_{\rho<r}\zeta_\rho^{\nu_\rho}$
in $A_r(c)$ satisfies
$
\sum_{\rho<r}\nu_\rho\log\rho<0,
$
then $\nu_\rho<0$ for some $\rho<r$.
Hence $\zeta_r\zeta^\nu\in\mathcal N_p$, so this term
can be discarded. We therefore define the
\emph{projected vector}
\begin{equation}
    (\mathcal T_yc)_r
=
r^{-1/2}P_{[0,\infty)}^{r^-}A_r(c),
\qquad r\le y,
 \label{Eq:projected-vector}
\end{equation}
where $P_{[0,\infty)}^{r^-}$ retains Fourier monomials
with $\sum_{\rho<r}\nu_\rho\log\rho\geq0$.
The projection removes terms that analytic polynomials cannot see,
while the prime-by-prime structure allows us to control the remaining
leading contributions together through a martingale square function.
Consequently,
\[
\left[
\sum_{r\le y}r^{-1/2}\zeta_r A_r(c)
\right]
=
\left[
\sum_{r\le y}\zeta_r(\mathcal T_yc)_r
\right]
\quad\text{in }L^q/\mathcal N_p.
\]

Note that $\zeta_r(\mathcal T_yc)_r$ are martingale
differences in increasing prime order.
By the Burkholder--Gundy square-function inequality,
\[
\norm{\sum_{r\le y}\zeta_r(\mathcal T_yc)_r}_{L^q}
\lesssim_q
\norm{\mathcal T_yc}_{L^q(\ell^2)}.
\]
Therefore, the remaining task is to prove
\begin{equation}
\norm{\mathcal T_yc}_{L^q(\ell^2)}
\le C_p\norm c_{\ell^q(\ZN)},
\label{eq:intro-corrected-column}
\end{equation}
which is Theorem~\ref{thm:corrected-column}.

\medskip
\noindent\emph{(2) Separating the frequencies.}
To prove \eqref{eq:intro-corrected-column}, we separate the contributions into different
frequency regimes by comparing the discrete Fourier frequencies
of $c$ with the time frequencies of $\e^{\ii t\log r}f_{r^-}(t)$, that is,   with the frequencies appearing when
$\e^{\ii t\log r}f_{r^-}(t)$ is expanded into
exponential terms in the variable $t$.
Now, for each prime $r\le y$,
\[
(\mathcal T_yc)_r
=
r^{-1/2}M_y^{-1/p}P_{[0,\infty)}^{r^-}
\left(
\sum_{j<N}c_j\e^{\ii t_j\log r}f_{r^-}(t_j)
\right).
\]
Expand $c$
in its discrete Fourier series:
\[
c_j=\sum_{\kappa\in\Lambda_N}\widehat c(\kappa)
\e^{-\ii j\xi_\kappa/N},
\qquad
\xi_\kappa=2\pi\kappa,
\quad
\Lambda_N=\{-N/2,\ldots,N/2-1\},
\]
where
$\widehat c(\kappa)=N^{-1}\sum_{j<N}
c_j\e^{\ii j\xi_\kappa/N}$.
The Fourier expansion of $f_{r^-}(t)$ takes the form
\[
f_{r^-}(t)
=
\sum_\nu b_{r,\nu}\e^{\ii t\lambda(\nu)}\zeta^\nu,
\qquad
\lambda(\nu)=\sum_{\rho<r}\nu_\rho\log\rho,
\]
where $\zeta^\nu=\prod_{\rho<r}\zeta_\rho^{\nu_\rho}$,
the sum is over integer multi-indices $\nu$, and
$b_{r,\nu}$ are independent of $t$.
The projection defining $\mathcal T_yc$
retains only terms with $\lambda(\nu)\ge0$.
For each retained Fourier monomial, we write
$\lambda=\lambda(\nu)\ge0$ and
$\omega=\lambda+\log r$.
A key term in its contribution to $(\mathcal T_yc)_r$ is
\[
\begin{aligned}
\sum_{j<N}c_j\e^{\ii t_j\omega}
&=\e^{\ii t_0\omega}
\sum_{\kappa\in\Lambda_N}
\widehat c(\kappa)S_N(\omega-\xi_\kappa),
\end{aligned}
\]
where
\[
S_N(\Delta)=\sum_{j=0}^{N-1}\e^{\ii j\Delta/N},
\qquad
\Delta=\omega-\xi_\kappa
=\lambda+\log r-\xi_\kappa.
\]
In particular, $S_N(0)=N$ and
$|S_N(\Delta)|\asymp N$ for $|\Delta|\le1$.

Since the Fourier expansion of $f_{r^-}$
contains arbitrarily large frequencies $\lambda$,
the difference $\Delta$ may lie near a nonzero
multiple of $2\pi N$, where the periodic sum
$S_N$ is again of order $N$.
We therefore separate the contribution from
frequencies $\lambda>K_0:=\pi N/2$, denoted by
$\mathcal H_yc$.
Proposition~\ref{prop:high-source-tail} gives
\begin{equation}
\label{eq:intro-high-source-bound}
\norm{\mathcal H_yc}_{L^q(\ell^2)}
\lesssim_p\norm c_{\ell^q(\ZN)}.
\end{equation}

For the remaining frequencies $0\le\lambda\le K_0$,
using $\log r\le N$ and
$-\pi N\le\xi_\kappa<\pi N$, we have
\[
-\pi<\frac{\Delta}{N}
\le\frac{3\pi}{2}+1<2\pi.
\]
Thus $1-\e^{\ii \Delta/N}$ can vanish only when
$\Delta=0$. For $\Delta\ne0$, 
\[
S_N(\Delta)
=
\frac{1-\e^{\ii \Delta}}{1-\e^{\ii \Delta/N}}
=
\frac{1-\e^{\ii \omega}}{1-\e^{\ii \Delta/N}}.
\]
The numerator $1-\e^{\ii \omega}$ is independent
of $\kappa$ and defines a uniformly bounded
operator on $L^q$.
We therefore focus on the denominator. For
$\Delta\ne0$, write
\[
\frac1{1-\e^{\ii \Delta/N}}
=
\frac{\ii N}{\Delta}+\mathfrak h(\Delta/N),
\]
where
\[
\mathfrak h(z)=\frac1{1-\e^{\ii z}}-\frac{\ii}{z},
\qquad \mathfrak h(0)=\frac12.
\]
The function $\mathfrak h$ is smooth on the
remaining frequency range, so its contribution
can be estimated by Fourier inversion.
The Cauchy term $1/\Delta$ will be treated
on regions where $\Delta$ has a fixed sign
and is separated from zero.

To identify the frequency gaps where simpler estimates apply,
and to isolate the nearly matching frequencies that require
a different argument, we group the primes according to
the size of $\log r$:
\[
\mathcal I_h
=
\{r\text{ prime}:w/2<\log r\le w\},
\qquad w=2^h,\quad 0\le h\le H.
\]
We also split $c=\Pi_-c+\Pi_+c$, where $\Pi_-$
retains the Fourier frequencies $\xi_\kappa\le0$
and $\Pi_+$ retains those with $\xi_\kappa>0$.

For $\xi_\kappa>0$, we use smooth Fourier cutoffs
to separate a small-frequency part supported in
$0<\xi_\kappa<w/8$ and dyadic pieces supported in
$v/2<\xi_\kappa<2v$, where $v\in2^{\mathbb Z}$.
We then classify these terms according to
$\Delta=\lambda+\log r-\xi_\kappa$:
either $\Delta$ has a fixed sign and is separated
from zero at scale $w$ or $v$, or no such gap is
available.
Together with the high-frequency term and the
nonpositive coefficient frequencies, this gives
the five terms in the following table.
Except for the first row, $0\le\lambda\le K_0$.

\begin{center}
\small
\renewcommand{\arraystretch}{1.3}
\begin{tabularx}{\linewidth}{@{}
>{\raggedright\arraybackslash}p{0.30\linewidth}
>{\raggedright\arraybackslash}X
>{\raggedright\arraybackslash}p{0.21\linewidth}@{}}
\toprule
Term & Frequency range & Gap for $\Delta$ \\
\midrule
$\mathcal H_yc$
& $\lambda>K_0$
& High-frequency tail
\\
$\mathcal T_y^{\mathrm{lo}}(\Pi_-c)$
& $\xi_\kappa\le0$
& $\Delta>w/2$
\\
$\mathcal L_y^{\mathrm{lo}}(\Pi_+c)$
& $0<\xi_\kappa<w/8$
& $\Delta>3w/8$
\\
$\mathcal O_y(\Pi_+c)$
& $v/2<\xi_\kappa<2v$,
  $\omega<v/4$ or $\omega>4v$
& $\Delta<-v/4$ or $\Delta>2v$
\\
$\mathcal R_y(\Pi_+c)$
& $v/2<\xi_\kappa<2v$,
  $v/4\le\omega\le4v$
& Resonant term 
\\
\bottomrule
\end{tabularx}
\end{center}

The precise definitions of these operators are given
in Section~\ref{subsec:high-source-tail}
and Section~\ref{sec:corrected-column}.
Proposition~\ref{prop:exact-frequency-partition}
gives the exact decomposition
\begin{equation}
\begin{aligned}
\mathcal T_yc
={}&\mathcal H_yc
+\mathcal T_y^{\mathrm{lo}}(\Pi_-c)
+\mathcal L_y^{\mathrm{lo}}(\Pi_+c)+\mathcal O_y(\Pi_+c)
+\mathcal R_y(\Pi_+c).
\end{aligned}
\label{eq:intro-frequency-partition}
\end{equation}
The five terms fall into three groups: the high-frequency tail, three terms with well-separated frequencies, and the final resonant term where frequencies can nearly match. The first four admit simpler estimates, leaving the last resonant term for the main argument.
Indeed, the high-frequency term $\mathcal H_yc$ has already
been controlled in \eqref{eq:intro-high-source-bound}.
The resonant term $\mathcal R_y(\Pi_+c)$ will be
treated later. 
For each of the three remaining nonresonant terms,
we estimate the Cauchy contribution $\ii N/\Delta$
using the fixed-sign frequency gaps in the table,
and treat the smooth remainder
$\mathfrak h(\Delta/N)$ separately.

\emph{(i) Cauchy contribution.} To estimate the Cauchy term $1/\Delta$
in the three nonresonant regions,
we use Laplace representations to separate
the Euler-product frequency
$\omega=\lambda+\log r$
from the coefficient frequency $\xi_\kappa$.
We first estimate the Euler products jointly
over each prime band.

For each prime band $\mathcal I_h$
we work with the common Euler product
$f_{\e^{2^h}}$, rather than the varying products
$f_{r^-}$. Set
\[
F_w=P_{[0,K_0]}^{\mathrm{full}}f_{\e^{2^h}}(t_0),
\]
where $P_{[0,K_0]}^{\mathrm{full}}$ retains
the Fourier monomials $\zeta^\nu$ satisfying
\[
0\le\sum_{\rho\le\e^{2^h}}\nu_\rho\log\rho\le K_0.
\]
Unlike $P_{[0,\infty)}^{r^-}$, this projection
uses all prime variables in $f_{\e^{2^h}}$.
Note that
\[
P_{[0,K_0]}^{\mathrm{full}}f_{\e^{2^h}}(t_j)(\zeta)
=
F_w\bigl((\e^{\ii j\log\rho/N}\zeta_\rho)_{\rho\le\e^w}\bigr).
\]
The uniform $L^q$ bound for this projection,
the normalization
$\norm{f_{\e^{2^h}}(t_0)}_q=M_{\e^{2^h}}^{1/p}$,
and Mertens' theorem, together with
$\sum_{r\in\mathcal I_h}r^{-1}\lesssim1$, give
\begin{equation}
\left\|
M_y^{-1/p}
(r^{-1/2}F_w)_{r\in\mathcal I_h}
\right\|_{L^q(\ell^2)}
\lesssim_p
\left(\frac wN\right)^{1/p}.
\label{eq:intro-common-product-bound}
\end{equation}
For $r\in\mathcal I_h$, conditional expectation gives
\begin{equation}\label{intro-eq:conditional-expectation}
\E_{r^-}
\left[
P_{[0,K_0]}^{\mathrm{full}}f_{\e^{2^h}}(t_j)
\right]
=
P_{[0,K_0]}^{r^-}f_{r^-}(t_j),
\end{equation}
where
$\E_{r^-}=\E(\,\cdot\mid
\sigma(\zeta_\rho:\rho<r))$.
Lemma~\ref{lem:frozen-low-source-restoration}
then allows us to estimate the corresponding
expressions with $f_{\e^{2^h}}$ and transfer the
resulting $L^q(\ell^2)$ bounds to those with
$f_{r^-}$.

Let $b=\Pi_-c$. To estimate the contribution
of $\ii N/\Delta$ to $\mathcal T_y^{\mathrm{lo}}b$,
note that
\[
\frac1\Delta
=
\int_0^\infty
\e^{-u\lambda}\e^{-u\log r}
\e^{-u|\xi_\kappa|}\,\dd u.
\]
The factor $\e^{-u\lambda}$ acts on the
Euler-product frequencies.
Since
\[
\operatorname{Var}\left(
\e^{-u\lambda}\mathbf1_{[0,K_0]}(\lambda)
\right)\le2,
\qquad u\ge0,
\]
the multiplier estimate
\eqref{eq:transferred-damped-low-bound}
shows that \eqref{eq:intro-common-product-bound}
remains valid after multiplying the
Euler-product frequencies by $\e^{-u\lambda}$,
with a constant independent of $u$.
For $r\in\mathcal I_h$, the second factor satisfies
$\e^{-u\log r}\le\e^{-uw/2}$.
The third factor gives
\[
\sum_{\kappa\in\Lambda_N}
\widehat b(\kappa)\e^{-u|\xi_\kappa|}
=
(\mathsf P_{Nu}^N*b)(0),
\]
where $\mathsf P_\rho^N$ is the cyclic Poisson
kernel defined in \eqref{eq:cyclic-poisson-kernel}.

Combining these estimates with the
vector-valued conditional-expectation bound
of Lemma~\ref{lem:frozen-low-source-restoration},
the contribution of $\ii N/\Delta$ on
$\mathcal I_h$ has $L^q(\ell^2)$-norm at most
\[
\begin{aligned}
&C_pN\left(\frac wN\right)^{1/p}
\int_0^\infty
\e^{-uw/2}
\bigl|(\mathsf P_{Nu}^N*b)(0)\bigr|\,\dd u=C_p\widetilde B_{N/w}(b),
\end{aligned}
\]
where
\[
\widetilde B_m(b)
=
m^{-1/p}
\int_0^\infty \e^{-\rho/(2m)}
\bigl|(\mathsf P_\rho^N*b)(0)\bigr|\,\dd\rho.
\]
Proposition~\ref{prop:boundary-traces} gives
\[
\sum_{h=0}^H
\widetilde B_{N/2^h}(b)^q
\lesssim_p
\norm b_{\ell^q(\ZN)}^q
\lesssim_p
\norm c_{\ell^q(\ZN)}^q.
\]

For $d=\Pi_+c$, the small positive-frequency
term $\mathcal L_y^{\mathrm{lo}}d$ has
$\Delta>3w/8$.
Using the same argument with a smooth
coefficient-frequency cutoff in place of the
Poisson kernel, we obtain the prime-band
estimate \eqref{eq:positive-low-trace-bound}.
Proposition~\ref{prop:boundary-traces}
then bounds the contribution of $\ii N/\Delta$
to $\mathcal L_y^{\mathrm{lo}}d$ in
$L^q(\ell^2)$ by
$C_p\norm d_{\ell^q(\ZN)}
\lesssim_p\norm c_{\ell^q(\ZN)}$.

The term $\mathcal O_yd$ also involves grouping
the Fourier coefficients of $d$ into dyadic
frequency intervals.
We call these groups dyadic
coefficient-frequency shells.
On 
$v/2<\xi_\kappa<2v$, we have
\[
\Delta<-v/4 \quad\text{if }\omega<v/4,
\qquad
\Delta>2v \quad\text{if }\omega>4v.
\]
Using \eqref{eq:intro-common-product-bound}, \eqref{intro-eq:conditional-expectation},
Lemma~\ref{lem:cauchy-coefficient-kernels},
and Proposition~\ref{prop:row-interval-projections},
the contribution of $\ii N/\Delta$ from a
prime band $\mathcal I_h$ of scale $w$ and a coefficient
shell of scale $v$ has $L^q(\ell^2)$-norm
at most
\[
\begin{aligned}
&C_pN\left(\frac wN\right)^{1/p}
\frac vN\frac1v
\sum_{j<N}
\left(1+\frac{v\rho_j}{N}\right)^{-1}|d_j|=
C_p\left(\frac wv\right)^{1/p}B_{N/v}(d),
\end{aligned}
\]
where
\[
B_m(d)
=
m^{-1/p}\sum_{j<N}
\left(1+\frac{\rho_j}{m}\right)^{-1}|d_j|,
\qquad
\rho_j=\operatorname{dist}(j,N\mathbb Z).
\]

We now sum over the prime bands $\mathcal I_h$
and the dyadic coefficient-frequency scales
$v\in\mathcal V_N$, where $\mathcal V_N$ is
the finite set defined in \eqref{eq:finite-shell-set}.
Write $w=2^h$ and $v=2^k$.
The triangle inequality over coefficient shells
and $q<2$ for disjoint prime bands allow us
to bound the $q$-th power of the total
$L^q(\ell^2)$-norm by the sum below.
Since $\sum_{n\ge-5}2^{-n/p}<\infty$,
Young's inequality and
Proposition~\ref{prop:boundary-traces} give
\[
\begin{aligned}
\sum_{h=0}^{H}
\left(
\sum_{\substack{k\ge h-5\\
2^k\in\mathcal V_N}}
2^{(h-k)/p}B_{N/2^k}(d)
\right)^q\lesssim_p
\sum_{v\in\mathcal V_N}B_{N/v}(d)^q
\lesssim_p\norm d_{\ell^q(\ZN)}^q \lesssim_p\norm c_{\ell^q(\ZN)}^q.
\end{aligned}
\]

\emph{(ii) Smooth remainder.} For the smooth remainder $\mathfrak h(\Delta/N)$, we 
choose a smooth compactly supported extension
$\widetilde{\mathfrak h}$ agreeing with $\mathfrak h$
on the retained frequency range.
Fourier inversion gives
\[
\mathfrak h\left(\frac{\omega-\xi_\kappa}{N}\right)
=
\int_{\R}
\widehat{\widetilde{\mathfrak h}}(\tau)
\e^{\ii \tau\omega/N}
\e^{-\ii \tau\xi_\kappa/N}\,\dd\tau.
\]
The factor $\e^{\ii \tau\omega/N}$ acts by
rotations of the prime coordinates $\zeta_\rho$,
up to a scalar of modulus one in each component,
and hence preserves the $L^q(\ell^2)$-norm.
For the coefficient cutoffs in the three
nonresonant terms, Lemma~\ref{lem:cyclic-multiplier}
gives an $\ell^q$ multiplier bound
$C_q(1+|\tau|)$ for
$\e^{-\ii \tau\xi_\kappa/N}$.
Since $\widetilde{\mathfrak h}$ is smooth
and compactly supported,
\[
\int_{\R}
(1+|\tau|)
|\widehat{\widetilde{\mathfrak h}}(\tau)|
\,\dd\tau<\infty.
\]
Then using \eqref{eq:intro-common-product-bound}
and Lemma~\ref{lem:regular-cauchy-remainder},
the smooth contributions to
$\mathcal T_y^{\mathrm{lo}}b$ and
$\mathcal L_y^{\mathrm{lo}}d$
on each prime band of scale $w$
are bounded in $L^q(\ell^2)$ by
\[
C_p\left(\frac wN\right)^{1/p}
\norm b_{\ell^q(\ZN)}
\quad\text{and}\quad
C_p\left(\frac wN\right)^{1/p}
\norm d_{\ell^q(\ZN)},
\]
respectively.
These bounds are summable over prime bands
since $q<2$ and
\[
\sum_{h=0}^H
\left(\frac{2^h}{N}\right)^{q/p}
\lesssim_p1.
\]

For $\mathcal O_yd$, the same argument gives
$C_p(w/N)^{1/p}\norm d_{\ell^q(\ZN)}$
for each  coefficient shell $v/2<\xi_\kappa<2v$.
By noting that $w/32<v\le4N$, there are 
$O(1+H-h)$ such shells at each prime band of scale $w=2^h$.
Their contributions are also summable because
\[
\sum_{h=0}^{H}
\left(\frac{2^h}{N}\right)^{q/p}
(1+H-h)^q
\lesssim_p1.
\]

Together with the preceding Cauchy estimates,
this proves the uniform bounds for
$\mathcal T_y^{\mathrm{lo}}(\Pi_-c)$,
$\mathcal L_y^{\mathrm{lo}}(\Pi_+c)$,
and $\mathcal O_y(\Pi_+c)$
in Proposition~\ref{prop:nonresonant-columns}.

Combining the high-frequency and nonresonant
estimates gives
\[
\norm{\mathcal T_yc}_{L^q(\ell^2)}
\le
C_p\norm c_{\ell^q(\ZN)}
+
\norm{\mathcal R_y(\Pi_+c)}_{L^q(\ell^2)}.
\]

\medskip
\noindent\emph{(3) Estimating the resonant term.}
The remaining task is to prove
\begin{equation}
 \norm{\mathcal R_y\Pi_+c}_{L^q(\ell^2)}
 \lesssim_p\norm c_{\ell^q}.
 \label{eq:resonant-term-estimate}
\end{equation}
The key novelty here is to preserve the correlations between resonant peaks and use critical-tree decay to control their combined contribution.
We explain the argument in four steps.

\emph{(i) Reduction to Euler-product moments.} When the frequencies can nearly match, we can no longer rely
on their separation, so we turn instead to estimates for
the moments of the underlying Euler products.
We consider $\norm{\mathcal R_y d}_{L^q(\ell^2)}$,
where \(d=\Pi_+c\) first.
For $r\in\mathcal I_h$, we have $\log r\asymp 2^h$.
Decompose $d$ into the positive Fourier-frequency pieces $d_v$ of
\eqref{eq:coefficient-shell}, and write $v=2^{h+\ell}$.
With
\(
 z_h^{(\ell)}=d_{2^{h+\ell}},
\)
we therefore have
\[
 \mathcal R_yd=\sum_{\ell\ge-2}\mathcal R_y^{(\ell)}d.
\]
Here, $(\mathcal R_y^{(\ell)}d)_r$ is defined in \eqref{eq:offset-resonant-row}:
\[
 (\mathcal R_y^{(\ell)}d)_r
 =M_y^{-1/p}r^{-1/2}P_{J_r^{(\ell)}}^{r^-}
 \sum_{j<N}z_h^{(\ell)}(j)\e^{\ii t_j\log r}f_{r^-}(t_j),
 \qquad r\in\mathcal I_h,
\]
and $J_r^{(\ell)}$ is the interval in
\eqref{eq:resonant-row-interval}.
We therefore seek bounds for $\norm{\mathcal R_y^{(\ell)}d}_{L^q(\ell^2)}$ that are summable over $\ell$, uniformly in $N$.
Fix $\ell\ge-2$, Proposition~\ref{prop:frame-estimate} bounds the $L^q(\ell^2)$-norm of $\mathcal R_y^{(\ell)}d$ by $C_p\mathfrak P_y(z^{(\ell)})^{1/q}$:
\[
 \norm{\mathcal R_y^{(\ell)}d}_{L^q(\ell^2)}
 \lesssim_p\mathfrak P_y(z^{(\ell)})^{1/q},
\]
where $(z^{(\ell)})=(z_h^{(\ell)}(j))_{0\le h\le H,\ j\in \ZN}$ and $\mathfrak P_y$ is defined in \eqref{eq:occupation-functional} and \eqref{eq:def-Py}:
\[
 \mathfrak P_y(z^{(\ell)})
 =
 \E\left[
 \left(
 M_y^{-2/p}
 \sum_{h=0}^{H}\frac{N}{2^h}
 \sum_{j<N}|z_h^{(\ell)}(j)|^2
 |f_{\e^{2^h}}(t_j)|^2
 \right)^s
 \right].
\]
Here, $s=q/2\in(1/2,1)$.
This reduction uses the projection and conditional-expectation bounds in \eqref{eq:row-interval-projection} and \eqref{eq:vector-conditional-expectation}.

Fix $\ell$ and write $z=z^{(\ell)}$.
Applying subadditivity to every term in $\E[(\cdot)^s]$ may give a much
larger bound. Therefore, we instead seek a combination of some summands before using the subadditivity.
At level $h$, the dyadic partition $\mathcal D_h$ in
\eqref{eq:nonwrapping-dyadic-grid} divides the sample indices $\ZN$ into blocks which have $m_h:=N/2^h$ indices.  Hence,
\[
 \mathfrak P_y(z)
 =\E\left[
 \left(
 \sum_{h=0}^H\sum_{Q\in\mathcal D_h}
 M_y^{-2/p}\frac{N}{2^h}\sum_{j\in Q}
 |z_h(j)|^2|f_{\e^{2^h}}(t_j)|^2
 \right)^s\right].
\]
We shall divide $\bigcup_{h=0}^H\mathcal D_h$ into different groups before applying subadditivity.

\emph{(ii) Local moments and the binary tree.}
Now the main difficulty is that nearby values of an Euler product
are strongly correlated, so estimating them separately would
lose the decay needed to sum over scales.
A binary tree naturally shows up and captures these correlations:
smaller primes contribute fluctuations shared by nearby points,
while larger primes gradually distinguish them.
This motivates estimating their combined contribution
on short intervals. 

Before choosing their groups, we now consider the local contribution of a dyadic interval. Take a dyadic interval
$I\in\bigcup_{h=0}^H\mathcal D_h$,
let $h_I$ be its level, and write
\[
 I=\{j_0,\ldots,j_0+m_{h_I}-1\}.
\]
For each $h\ge h_I$, to study the combined contribution of the level-$h$ blocks contained in $I$ to $\mathfrak P_y(z)$, define
\[
 V_I(u)=\sum_{h=h_I}^H
 \frac{N}{2^h}M_y^{-2/p}
 \sum_{j\in I}|z_h(j)|^2|f_{\e^{2^h}}(t_j+u)|^2,
 \qquad 0\le u\le 2^{-h_I}.
\]
For each fixed \(u\), the rotations \(\zeta_r\mapsto\e^{\ii u\log r}\zeta_r\) leave the joint distribution of the variables \(\zeta_r\) unchanged, so 
$\E[V_I(u)^s]=\E[V_I(0)^s]$.
We may therefore average $\E[V_I(u)^s]$ over $0\le u\le 2^{-h_I}$ to express $\E[V_I(0)^s]$.
By Jensen's inequality and the subadditivity,
\begin{equation}
 \E[V_I(0)^s]
 \le
 \sum_{h=h_I}^H
 \left(\sum_{j\in I}|z_h(j)|^2\right)^s
 \E\left[
 \left(
 \frac{N M_y^{-2/p}}{2^{h-h_I}}
 \int_{J_I^+}|f_{\e^{2^h}}(t)|^2\dd t
 \right)^s
 \right],
 \label{eq:local-contribution}
\end{equation}
where \(J_I^+:=[t_{j_0},t_{j_0}+2^{1-h_I}]\).
The coefficients $z_h$ now enter only through their squared sums
$\sum_{j\in I}|z_h(j)|^2$, while the random part is a local Euler-product integral $\int_{J_I^+}|f_{\e^{2^h}}(t)|^2\dd t$.

We first consider the random part. More generally, we estimate
\[
 \E\left[
  \left(\int_J |f_{\e^{2^h}}(t)|^2\dd t\right)^s
 \right],
\]
where \(w=2^h\) and \(J\subset\R\) is an interval of dyadic
length satisfying
\(
 0<|J|\le1,
\)
\(
 w|J|\ge1.
\)
We approximate the integral over $J$ by Riemann sums
$\sum_{i=1}^M\lambda_i|f_{\e^{2^h}}(\tau_i)|^2$, where $\sum_i\lambda_i=|J|$,  and $\tau_i$ are arbitrary sample points in $J$.

Write $\alpha=1-2/p$ and
$\beta=1+\alpha=1/s$.
We expand the logarithm of $|f_{\e^{2^h}}|^2$,
\[
\log|f_{\e^{2^h}}(\tau_i)|^2
=
\sum_{r\le\e^{2^h}}H_i^{(r)},
\]
where, for each prime $r$, the term $H_i^{(r)}$
depends on the random variable $\zeta_r$;
see \eqref{eq:one-prime-circle-increment}
and \eqref{eq:euler-localized-one-prime-logarithm}.
We estimate
\[
\E\left[\left(\sum_{i=1}^M\lambda_i|f_{\e^{2^h}}(\tau_i)|^2\right)^s\right]
\]
by replacing, for each prime \(r\), the vector \(H^{(r)}=(H^{(r)}_i)_{i=1}^M\) with the Gaussian vector \(G^{(r)}=(G^{(r)}_i)_{i=1}^M\), where
\[
G_i^{(r)}=
 -\frac{2\alpha}{r}
 +
 \frac{2\beta}{\sqrt r}\,
\Ree(g_r\e^{\ii \tau_i\log r}),
\]
where $g_r$ is a standard complex Gaussian
with $\E|g_r|^2=1$.
The covariance of the first random term 
of $H^{(r)}_i$ and $H^{(r)}_j$ is same as that $G^{(r)}_i$ and $G^{(r)}_j$; see
\eqref{eq:one-prime-covariance-matching}.
Proposition~\ref{prop:one-prime-replacement} shows that replacing on one prime $r$ changes at most a multiplicative factor $1+C_pr^{-2}$, uniformly in $M$, $\tau_i$ and $\lambda_i$.
The replacement can be repeated over all primes, with a bounded total factor:
\[
\E\Big[(\sum_{i=1}^{M}\lambda_i|f_{\e^{2^h}}(\tau_i)|^2)^s\Big]
\lesssim_p
\E\Big[(\sum_{i=1}^{M}\lambda_i\exp(\sum_{r\le e^{2^h}}G^{(r)}_i))^s\Big]
\]
see Corollary~
\ref{cor:sequential-shared-prime-replacement}.
Put $w=2^h$. Passing to the limit in the Riemann sums, we obtain
\[
\lim_{\max_i\lambda_i\to0}
\E\left[
\left(
\sum_{i=1}^{M}\lambda_i
\exp\left(\sum_{r\le\e^w}G_i^{(r)}\right)
\right)^s
\right]
=
\E\left[
\left(
\int_J
\e^{\sqrt2\beta Z_w(t)-2\alpha V_w}
\dd t
\right)^s
\right],
\]
where $Z_w$ is the centered Gaussian field defined in \eqref{eq:gaussian-prime-field}, and $V_w$ is the variance of $Z_w(t)$.
Its covariance is
\[
 \Cov(Z_w(t),Z_w(t'))
 =
 \sum_{r\le\e^w}
 \frac{\cos((t-t')\log r)}{r}.
\]
After the deterministic normalization in \eqref{eq:gaussian-nomalization}, it remains to estimate
$\E\mathfrak G_{N,w,J}^{\,s}$, where
$\mathfrak G_{N,w,J}$ is defined in
\eqref{eq:localized-gaussian-functional}:
\begin{equation*}
 \mathfrak G_{N,w,J}
 =
 2^{-\alpha n}\frac{1}{|J|}
 \int_J
 \e^{\sqrt2\beta Z_w(t)-\beta V_w}
 \dd t,
\end{equation*}
and $n=\log_2(w|J|)$.

Let \(t_J=\inf J\), and write \(t=t_J+|J|x\), \(0\le x\le1\).
We split \(Z_w(t_J+|J|x)\) into two
independent Gaussian fields
$Z^{\mathrm c}(x)$ and $Z^{\mathrm f}(x)$,
according to whether
$\log r\le |J|^{-1}$ or
$|J|^{-1}<\log r\le w$;
see \eqref{eq:coarse-prime-field}
and \eqref{eq:fine-prime-field}.

For the first field $Z^{\mathrm c}(x)$, its covariance satisfies
\[
 0\le V_{\mathrm c}
 -\Cov(Z^{\mathrm c}(x),Z^{\mathrm c}(x'))
 \le C,
 \qquad \forall\ x,x'\in[0,1],
\]
where \( V_{\mathrm c}:=\Var(Z^{\mathrm c}(x))=\sum_{\log r\le|J|^{-1}}1/r\); 
see \eqref{eq:coarse-covariance-flatness}.
We can therefore compare $Z^{\mathrm c}(x)$ with a centered Gaussian variable $G_{\mathrm c}$ of variance $V_{\mathrm c}$; see \eqref{eq:coarse-nonwick-comparison}.
Thus the contribution of $Z^{\mathrm c}$ can be removed from $\E[(\mathfrak G_{N,w,J})^s]$ at a cost depending only on $p$;
see \eqref{eq:coarse-field-removed}.

For $Z^{\mathrm c}$, a single Gaussian variable shared by all $x\in [0,1]$ was sufficient.
However, for the remaining field $Z^{\mathrm f}$, the covariance depends on the distance between the two points:
\[
 \Cov(Z^{\mathrm f}(x),Z^{\mathrm f}(x'))
 =
 \log\min\{2^n,|x-x'|^{-1}\}+O(1);
\]
see \eqref{eq:fine-prime-covariance}.
We therefore seek a Gaussian field that is constant on short intervals but retains the correlations between different intervals.

To construct such a field, we successively bisect $[0,1]$ until the intervals have length $2^{-n}$.
The dyadic intervals at each level $k=0,\ldots,n$ correspond to the vertices at depth $k$ of a binary tree.
For each $x\in[0,1]$, let $v(x)$ denote the leaf corresponding to the level-$n$ dyadic interval containing $x$.
Attach an independent standard real Gaussian $\gamma_e$ to each edge $e$ and define the random walk on the binary tree as
\[
 T_v=\sqrt{2\log2}\sum_{e\preceq v}\gamma_e,
\]
as in \eqref{eq:critical-tree-field}.
The resulting field $T_{v(x)}$ takes the same random value at every point $x$ within each level-$n$ dyadic interval.
Its values on different level-$n$ dyadic intervals remain correlated through the Gaussian increments shared by their paths in the tree.
If the paths to $v=v(x)$ and $v'=v(x')$
share exactly $k$ edges, then
$\Cov(T_v,T_{v'})=2k\log2$.
The corresponding points satisfy
$|x-x'|\le2^{-k}$, and hence
\[
 \Cov(\sqrt2Z^{\mathrm f}(x),
      \sqrt2Z^{\mathrm f}(x'))
 \ge \Cov(T_v,T_{v'})-C;
\]
see \eqref{eq:fine-tree-covariance-lower-bound}.
The covariance estimate above, together with Lemma~\ref{lem:concave-gaussian-comparison}, leads to the comparison of fractional moments in \eqref{eq:fine-field-to-tree-limit}:
\begin{align*}
 \E
 \left[\left(
 2^{-\alpha n}
 \int_0^1
 \e^{\sqrt2\beta Z^{\mathrm f}(x)-\beta V_{\mathrm f}}
 \dd x
 \right)^s\right]
\le
 C_\beta
 \E
 \left[\left(
 \sum_{\abs v=n}
 2^{-(1+\alpha)n}
 \e^{\beta T_v-\beta n\log2}
 \right)^{1/\beta}\right].
\end{align*}
Writing
\[
 Y_n(\beta)=
 \left(
 \sum_{|v|=n}
 \e^{\beta(T_v-\Var(T_v))}
 \right)^{1/\beta},
\]
as in \eqref{eq:critical-tree-functionals},
and combining with the estimate for
$Z^{\mathrm c}$ and the bound for $\E Y_n(\beta)$ in Theorem~\ref{thm:critical-tree}, we obtain
\[
 \E\mathfrak G_{N,w,J}^{\,s}
 \lesssim_p
 \E\left[
 \left(
 2^{-\alpha n}
 \int_0^1
 \e^{\sqrt2\beta Z^{\mathrm f}(x)-\beta V_{\mathrm f}}
 \dd x
 \right)^s
 \right]
 \lesssim_p \E Y_n(\beta)
 \lesssim_p
 \frac{\log^{5/2}(n+2)}{\sqrt{n+1}}.
\]

\emph{(iii) Grouping and coefficient estimates.} The above Part (ii) allows us to control the combined contribution of many scales on one short interval. But the resonant term involves many such intervals. Estimating them separately would lose the gain we have just obtained.
This part solves this problem by grouping dyadic blocks according to the local size of the coefficients, measured by a conical square function. So we can turn the difficult Euler-product moment estimate into a square-function estimate for the coefficients, which can then be controlled by standard harmonic analysis. This controls the resonant terms with bounded frequency offsets, leaving Part (iv) to gain the extra decay needed to sum the remaining resonant terms with large offsets, that is, terms for which the coefficient-frequency  is much larger than   $\log r$.

Indeed, the bound of $\E\mathfrak G_{N,w,J}^{\,s}$, 
together with \eqref{eq:local-contribution}
and H\"older's inequality, gives
\begin{equation}
 \E[V_I(0)^s]
 \lesssim_p
 |I|^{1-s}
 \left(
 \sum_{h=h_I}^H
 \sum_{j\in I}|z_h(j)|^2
 \right)^s.
 \label{eq:local-eatimate}
\end{equation}
Recall that our goal is to bound the weighted moment $\mathfrak P_y(z)$.
The preceding estimate \eqref{eq:local-eatimate} controls the contributions from different levels within a single dyadic interval $I$ to $\mathfrak P_y(z)$.
To bound $\mathfrak P_y(z)$, we now need to group $\bigcup_{h=0}^H\mathcal D_h$ so that these local estimates can be summed.

The form of the bound of $\E[V_I(0)^s]$ suggests using $S_{\mathrm{cone}}z(x)^2$, which is defined by
\[
 S_{\mathrm{cone}}z(x)^2
 =
 \sum_{h=0}^H
 \frac{E_{Q_h(x)}}{|Q_h(x)|},
\]
where $Q_h(x)$ is the dyadic block
in $\mathcal D_h$ containing $x$ and
$E_Q=\sum_{j\in Q}|z_h(j)|^2$
for $Q\in\mathcal D_h$;
see \eqref{eq:cone-square-function}.
Using the level sets of $S_{\mathrm{cone}}z$ and a stopping-time argument,
we assign each dyadic block $Q$ with $E_Q>0$ to exactly one group $\mathcal A_{k,I}$.
Intuitively, $\mathcal A_{k,I}$ consists of the blocks selected at the
square-function $S_{\mathrm{cone}}z(x)$ level $2^k$ under a common maximal dyadic interval $I$;
see Section~\ref{subsec:tent-allocation}.
For each group $\mathcal A_{k,I}$, let $U_h$ be the union of its level-$h$ blocks; see \eqref{eq:tent-level-sets-and-energies}.
The contribution of the group $\mathcal A_{k,I}$ to $\mathfrak P_y(z)$ is
\[
 V_{k,I}(0)=\sum_{h=h_I}^H \frac{N}{2^h}M_y^{-2/p}
 \sum_{j\in U_h}|z_h(j)|^2|f_{\e^{2^h}}(t_j)|^2;
\]
see \eqref{eq:tent-shifted-functional}.
Since $q=2s$, subadditivity, the local estimate \eqref{eq:local-eatimate}, and Lemma~\ref{lem:tent-allocation} allow us to sum the resulting bounds over the groups \(\mathcal A_{k,I}\), giving
\[
 \mathfrak P_y(z)
 \le\sum_{k,I}\E[V_{k,I}(0)^s]
 \lesssim_p
 \norm{S_{\mathrm{cone}}z}_{\ell^q(\ZN)}^q.
\]
This is Proposition~\ref{prop:common-shift-occupation}.

Theorem~\ref{thm:cyclic-conical-square-function} and Proposition~\ref{prop:coefficient-square-function}, together with the estimates in (i) and Proposition~\ref{prop:common-shift-occupation}, give
\[
 \norm{\mathcal R_y^{(\ell)}d}_{L^q(\ell^2)}
 \lesssim_p\mathfrak P_y(z^{(\ell)})^{1/q}
 \lesssim_p
 \norm{S_{\mathrm{cone}}z}_{\ell^q(\ZN)}
 \lesssim_p
 (1+2^\ell)^{\Gamma_q}\norm d_{\ell^q},
\]
where $\Gamma_q$ is defined in \eqref{eq:gamma-q}. 
This suffices for $-2\le\ell\le3$.
For larger $\ell$, however, the factor $(1+2^\ell)^{\Gamma_q}$ cannot be summed uniformly in $N$.
We therefore need an additional estimate for $\mathcal R_y^{(\ell)}d$ when $\ell\ge4$.

\emph{(iv) Deformation for large frequency ratios.} The above Part (iii)  works well for
bounded frequency offsets, that is, when the coefficient
frequencies and $\log r$ are roughly the same size.
For large frequency offsets, a resonant contribution is possible only when the Euler product has a correspondingly large frequency. This suggests using the familiar idea behind Cauchy's coefficient estimates: control a function on a slightly larger domain to obtain decay for its high-frequency terms. We therefore deform the Euler product to give greater weight to higher frequencies. Although this enlarges these terms, Gaussian comparison shows that the moments of the deformed product remain controlled, so the binary-tree and grouping estimates still apply. When we return to the original Euler product, the spectral projection estimates allow us to gain exponential decay in the offset. This decay compensates for the losses in the earlier square-function estimates, allowing us to sum over all large offsets and complete the resonant-term estimate.

To be specific, for $\ell\ge4$, we estimate $\mathcal R_y^{(\ell)}d$ using the deformed Euler products $f_h^{\vartheta_0}$, where $\vartheta_0=1/32$; see \eqref{eq:deformed-euler-source}:
\begin{align*}
 f_h^\vartheta(t)
 ={}
 M_{\e^{2^{h}}}^{-\alpha}
 \prod_{r\le\e^{2^{h}}}
 \left(
 1-r^{-1/2}\e^{\vartheta\log r/2^{h}}
   \e^{\ii t\log r}\zeta_r
 \right)^{-1}
 \prod_{r\le\e^{2^{h}}}
 \left(
 1-r^{-1/2}\e^{-\vartheta\log r/2^{h}}
   \e^{-\ii t\log r}\overline{\zeta_r}
 \right)^{-\alpha},
\end{align*}
and $f_h^0=f_{\e^{2^h}}$.
The coefficient of $f_h^0$ at Euler frequency $\lambda$ is recovered from the deformed coefficient by multiplication with
$\e^{-\vartheta_0\lambda/2^h}$;
see \eqref{eq:deformation-restores-source}.
Since the interval $J_r^{(\ell)}$ in $\mathcal R_y^{(\ell)}d$ retains only
frequencies $\lambda\ge2^{h+\ell}/8$
for $r\in\mathcal I_h$, we have
\[
 \e^{-\vartheta_0\lambda/2^h}
 \le \e^{-\vartheta_02^\ell/8}.
\]
Proposition~\ref{prop:deformed-frame} therefore gives
\[
 \norm{\mathcal R_y^{(\ell)}d}_{L^q(\ell^2)}
 \lesssim_p
 \e^{-\vartheta_02^\ell/8}
 \mathfrak P_y^{\vartheta_0}
 (z^{(\ell)})^{1/q},
\]
where $\mathfrak P_y^{\vartheta_0}$ is obtained
from $\mathfrak P_y$ by replacing $f_{\e^{2^h}}$
with $f_h^{\vartheta_0}$, as in
\eqref{eq:occupation-functional}.

We are therefore left to estimate
$\mathfrak P_y^{\vartheta_0}(z^{(\ell)})$.
We follow the same argument as in (ii) and (iii),
with $f_{\e^{2^h}}$ replaced by $f_h^{\vartheta_0}$.
For each prime $r\le\e^w$, the coefficients in $H^{(r)}_i$ are different.
Consequently, after Gaussian replacement,
the centered Gaussian vector becomes $Z_w^{\vartheta_0}=(Z_w^{\vartheta_0}(\tau_i))_{i=1}^M$;
see \eqref{eq:gaussian-deformation-field}.
Proposition~\ref{prop:gaussian-deformation}
therefore gives
\[
\E\left[
\left(
\sum_{i=1}^M\lambda_i
\e^{\sqrt{2}\beta Z_w^{\vartheta_0}(\tau_i)-2\alpha V_w}
\right)^s
\right]
\le
\e^{C_p\vartheta_0}
\E\left[
\left(
\sum_{i=1}^M\lambda_i
\e^{\sqrt2\beta Z_w(\tau_i)-2\alpha V_w}
\right)^s
\right].
\]
Thus the Gaussian moment for the deformed Euler product is controlled by that for the undeformed product.

The grouping argument in (iii) therefore
applies to $\mathfrak P_y^{\vartheta_0}$ as well.
Combining these bounds gives
\[
 \norm{\mathcal R_y^{(\ell)}d}_{L^q(\ell^2)}
 \lesssim_p
 \e^{-\vartheta_02^\ell/8}
 2^{\Gamma_q\ell}\norm d_{\ell^q},
 \qquad \ell\ge4.
\]
Since
\(
 \sum_{\ell\ge4}\e^{-\vartheta_02^\ell/8}
 2^{\Gamma_q\ell}<\infty,
\)
this proves
\[
 \norm{\mathcal R_yd}_{L^q(\ell^2)}
 \lesssim_p\norm d_{\ell^q},
\]
as asserted in Proposition~\ref{prop:resonant-column}.
Finally, combining the preceding estimate with the boundedness of $\Pi_+$ in \eqref{eq:positive-negative-projection-bounds}, we obtain \eqref{eq:resonant-term-estimate}.
Combining the resonant estimate with the other four terms completes
\eqref{eq:intro-corrected-column}.

\medskip

Figure~\ref{fig:proof-architecture} below records the dependencies between
needed estimates.

\begin{figure}[H]
\centering
\begin{tikzpicture}[
  scale=0.86,
  transform shape,
  >={Stealth[length=2mm]},
  flow/.style={
    ->,
    line width=0.55pt
  },
  box/.style={
    draw,
    rounded corners=1.5pt,
    align=center,
    inner xsep=4pt,
    inner ysep=3.5pt,
    text width=3.2cm,
    minimum height=0.82cm,
    font=\footnotesize
  },
  widebox/.style={
    box,
    text width=4cm
  },
  keybox/.style={
    box,
    text width=5cm,
    minimum height=1.02cm,
    very thick
  },
  finalbox/.style={
    box,
    text width=5cm,
    very thick,
    double,
    double distance=0.5pt
  },
  group/.style={
    draw,
    dashed,
    rounded corners=2pt,
    inner sep=6pt
  },
  grouplabel/.style={
    font=\footnotesize\bfseries,
    fill=white,
    inner xsep=2.5pt,
    inner ysep=0.8pt
  }
]


\node[widebox] (s1) at (-0.3,0.25)
  {Frequency decomposition\\[-1pt]
   {\scriptsize Prop.~\ref{prop:exact-frequency-partition}}};

\node[box] (s2) at (-2.1,-1.55)
  {Nonresonant and\\
   high-frequency terms};

\node[box] (s3) at (1.5,-1.55)
  {Resonant term\\[-1pt]
   {\scriptsize Prop.~\ref{prop:resonant-column}}};

\node[keybox] (s4) at (-0.3,-3.55)
  {\textbf{Projected-vector estimate}\\
  {\scriptsize Thm.~\ref{thm:corrected-column}}};

\coordinate (sgheadL) at (-3.65,1.30);
\coordinate (sgheadR) at (3.05,1.30);

\node[
  group,
  fit=(s1)(s2)(s3)(s4)(sgheadL)(sgheadR)
] (sg) {};

\node[
  grouplabel,
  anchor=south west
] at ($(sg.north west)+(1.40,0.05)$)
  {Frequency estimates (Sec.~\ref{sec:corrected-column})};

\draw[flow] (s1.south) -- (s2.north);
\draw[flow] (s1.south) -- (s3.north);
\draw[flow] (s2.south) -- (s4.north);
\draw[flow] (s3.south) -- (s4.north);


\node[box] (h1) at (-6.00,1.10)
  {Cyclic square-function\\
   and frame estimates\\[-1pt]
   {\scriptsize Thm.~\ref{thm:cyclic-conical-square-function}}};

\node[box] (h2) at (-6.00,-0.80)
  {Multiplier and\\
   spectral-tail estimates};

\coordinate (hgheadL) at (-7.60,1.82);
\coordinate (hgheadR) at (-4.40,1.82);

\node[
  group,
  fit=(h1)(h2)(hgheadL)(hgheadR)
] (hg) {};

\node[
  grouplabel,
  anchor=south west
] at ($(hg.north west)+(-0.20,0.05)$)
  {Harmonic and sampling estimates (Secs.~\ref{sec:cyclic-analysis}--\ref{sec:frame-and-tails})
  };

\draw[flow]
  (h1.east)
  -- (-3.30,1.10)
  -- (-3.30,-0.45)
  -- (1.5,-0.45)
  -- (s3.north);

\draw[flow] (h2.east) -- (s2.west);


\node[box] (p1) at (5.70,2.60)
  {Euler-to-Gaussian comparison\\[-1pt]
   {\scriptsize Prop.~\ref{prop:euler-to-gaussian-integral}}};

\node[box] (p2) at (5.70,1.00)
  {Critical branching-\\
   random-walk estimate\\[-1pt]
   {\scriptsize Thm.~\ref{thm:critical-tree}}};

\node[box] (p3) at (5.70,-0.60)
  {Gaussian moment\\
   on a short interval\\[-1pt]
   {\scriptsize Prop.~\ref{prop:localized-gaussian-occupation}}};

\node[box] (p4) at (5.70,-2.2)
  {Euler-product moments\\
   on short intervals\\[-1pt]
  {\scriptsize Prop.~\ref{prop:localized-euler-moment}}};

\node[box] (p5) at (5.70,-3.80)
  {Moment bound for\\
   a common grid shift\\[-1pt]
   {\scriptsize Prop.~\ref{prop:common-shift-occupation}}};

\coordinate (pgheadL) at (4.30,2.73);
\coordinate (pgheadR) at (7.50,2.73);

\node[
  group,
  fit=(p1)(p2)(p3)(p4)(p5)(pgheadL)(pgheadR)
] (pg) {};

\node[
  grouplabel,
  anchor=south west
] at ($(pg.north west)+(-0.20,0.05)$)
  {Probabilistic estimates (Secs.~\ref{sec:replacement}--\ref{sec:euler-occupation})};

\draw[flow] (p2.south) -- (p3.north);
\draw[flow] (p3.south) -- (p4.north);

\draw[flow]
  (p1.east)
  -- ++(0.35,0)
  |- (p4.east);

\draw[flow] (p4.south) -- (p5.north);

\draw[flow]
  (p5.west)
  -- (3.68,-3.80)
  -- (3.68,-1.55)
  -- (s3.east);


\node[widebox] (q1) at (-0.3,-6.05)
  {Hardy quotient and\\
   martingale bound\\[-1pt]
   {\scriptsize Prop.~\ref{prop:quotient-representation}}};

\node[widebox] (q2) at (-0.3,-7.55)
  {Discrete local sampling\\[-1pt]
   {\scriptsize Prop.~\ref{prop:discrete-local-sampling}}};

\coordinate (qgheadL) at (-2.00,-5.12);
\coordinate (qgheadR) at (1.40,-5.12);

\node[
  group,
  fit=(q1)(q2)(qgheadL)(qgheadR)
] (qg) {};

\node[
  grouplabel,
  anchor=south west
] at ($(qg.north west)+(-0.05,0.05)$)
  {Return to the trace (Sec.~\ref{sec:main-proof}) };

\draw[flow] (s4.south) -- (q1.north);
\draw[flow] (q1.south) -- (q2.north);


\node[keybox] (f1) at (-0.3,-9.10)
  {\textbf{Local embedding theorem}\\[-1pt]
   {\scriptsize Thm.~\ref{thm:local-embedding}}};

\draw[flow] (q2.south) -- (f1.north);

\end{tikzpicture}

\caption{Proof structure for Theorem~\ref{thm:local-embedding}.}
\label{fig:proof-architecture}
\end{figure}
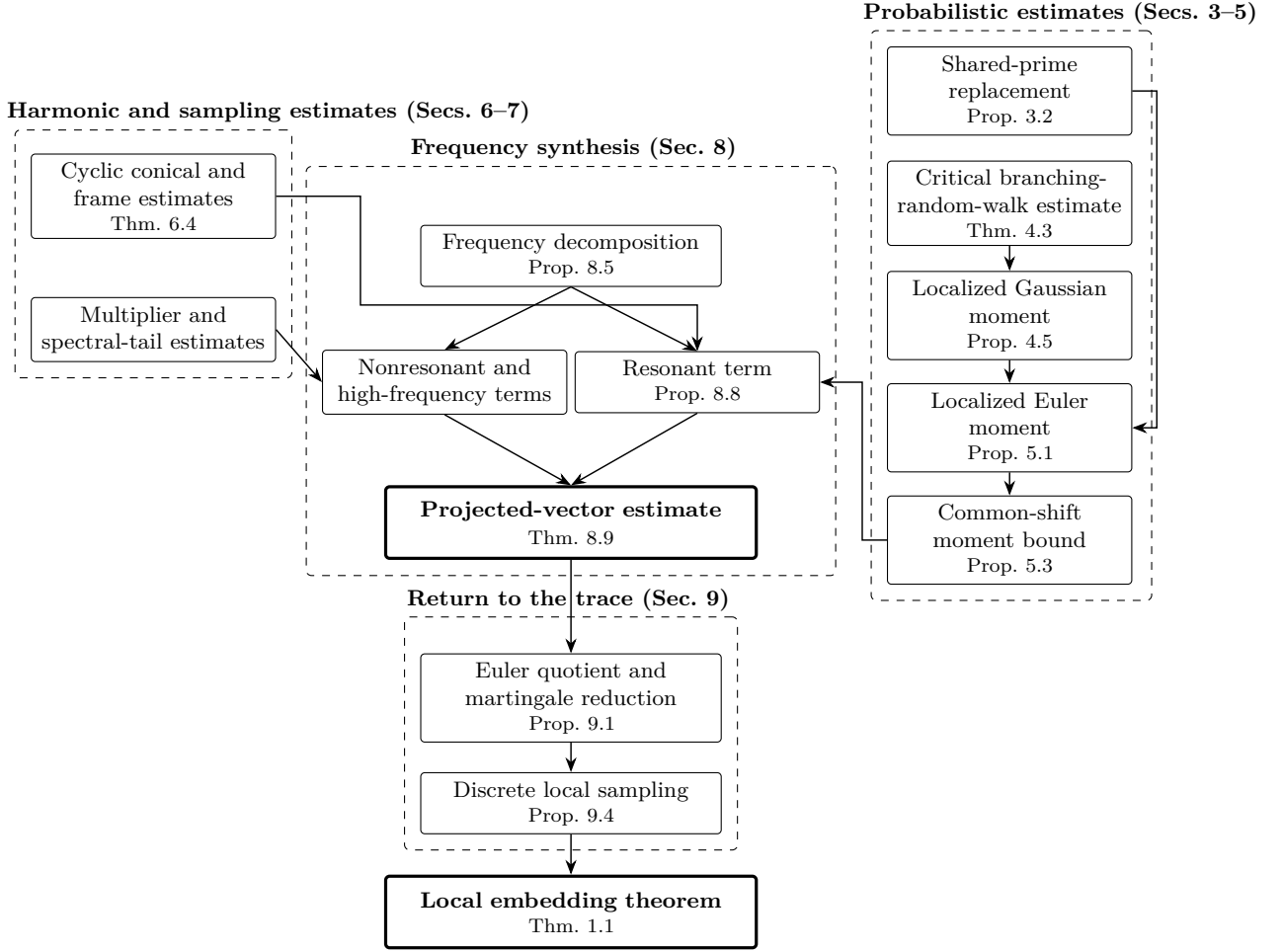

\medskip
\noindent\textbf{Why $p=2$ is the critical exponent.}
The preceding argument requires $p>2$.
Indeed, 
the passage from \eqref{eq:local-contribution}
to \eqref{eq:local-eatimate} uses the local
$s$th-moment estimate for the Euler-product
integral, obtained from the critical-tree estimate,
followed by H\"older's inequality to sum the
contributions over all levels $h\ge h_I$.
This requires a sequence $(u_n)$ satisfying
\[
 \E Y_n(\beta)\lesssim_p u_n,
 \qquad
 \sum_{n=0}^{\infty}u_n^{1/(1-s)}<\infty,
\]
where $s=1/\beta<1$.
Theorem~\ref{thm:critical-tree} provides
such a bound when $p>2$.

At $p=2$, however, we have $s=\beta=1$,
and $Y_n(1)$ is a mean-one martingale:
\[
 \E Y_n(1)=1
 \qquad\text{for every }n;
\]
see Remark~\ref{rem:critical-tree-beta-one}.
Thus the tree estimate provides no decay
across levels, and the preceding method
of combining the local estimates no longer applies.

The case $p=2$ is covered by the classical
Hilbert-space argument, while the embedding
fails for $0<p<2$, as discussed in Section \ref{subsec:local-embedding-problem}.
Hence $p=2$ is the sharp threshold for the
embedding, although the present method
applies only when $p>2$.

\subsection{Organization of the paper}
\label{subsec:organization}

Section~\ref{sec:preliminaries} introduces the notation and finite
models.  Sections~\ref{sec:replacement}--\ref{sec:euler-occupation}
compare Euler-product and Gaussian moments, establish the critical-tree
bound, and return to Euler-product estimates.
Sections~\ref{sec:cyclic-analysis}--\ref{sec:corrected-column} develop
the cyclic square functions, sampling bounds, and frequency decomposition
leading to Theorem~\ref{thm:corrected-column}.  Section~\ref{sec:main-proof}
uses this estimate to prove Theorem~\ref{thm:local-embedding}, and
Section~\ref{sec:consequences} presents the equivalent formulations
and the finite-band sampling characterization.

\section{Preliminaries and finite models}
\label{sec:preliminaries}

We first set up a finite prime model in which Euler products, frequency projections, and sampling can be studied together.

\subsection{Hardy spaces and exponent conventions}
\label{subsec:hardy-dirichlet-spaces}

All tori below carry normalized Haar measure.  We use the standard
Hardy-space conventions for Dirichlet series; see
\cite{Bayart02, HLS97}.

Let
\[
 P(s)=\sum_{n\in\mathcal F}a_n n^{-s}
\]
be a Dirichlet polynomial, where \(\mathcal F\subset\N\) is finite.  If
\[
 n=\prod_r r^{\nu_r(n)},
\]
write
\[
 \nu(n)=(\nu_r(n))_r\in\Nzero^{(\N)}
\]
for the associated finitely supported prime multi-index.

\begin{definition}[Hardy spaces of Dirichlet series]
\label{def:bohr-lift}
The Bohr lift of \(P\) is the analytic polynomial
\[
 \Bohr P(z)=\sum_{n\in\mathcal F}a_n z^{\nu(n)}.
\]
If \(y\ge2\) is large enough to contain every prime divisor of every \(n\in\mathcal F\), we
regard \(\Bohr P\) as a polynomial on \(\T^{\pi(y)}\) and define, for
\(0<p<\infty\),
\begin{equation}
 \norm{P}_{\HDir^p}^{p}
 =
 \int_{\T^{\pi(y)}}\abs{\Bohr P(z)}^p\dd m(z).
 \label{eq:hardy-dirichlet-norm}
\end{equation}
For \(p\ge1\), the \(p\)-th root of
\eqref{eq:hardy-dirichlet-norm} is a norm.  For \(0<p<1\), it is a
quasi-norm, and the corresponding completion is quasi-Banach.  The
completion of the Dirichlet polynomials in the appropriate norm or
quasi-norm is denoted by \(\HDir^p\).
\end{definition}

\begin{samepage}
\begin{definition}[The local embedding property]
\label{def:lep-p}
For \(0<p<\infty\), we write \(\mathrm{LEP}_p\) for the assertion that
there is a constant \(C_p<\infty\) such that
\[
 \sup_{\theta\in\R}
 \int_{\theta}^{\theta+1}
 \abs{P\!\left(\frac12+\ii t\right)}^p\dd t
 \le
 C_p\norm{P}_{\HDir^p}^{p}
\]
for every Dirichlet polynomial \(P\).
\end{definition}
\end{samepage}

For the remainder of the proof of
Theorem~\ref{thm:local-embedding}, we fix \(p>2\) and introduce the
associated exponent parameters:
\begin{equation}
 p>2,\qquad
 q=\frac{p}{p-1},\qquad
 s=\frac q2,\qquad
 \alpha=1-\frac2p,\qquad
 \delta=\frac2p,\quad
 a=q-1=\frac qp,\qquad
 \beta=1+\alpha.
 \label{eq:parameter-definitions}
\end{equation}
These parameters satisfy
\begin{equation}
 q\beta=2,\qquad
 s=\beta^{-1},\qquad
 \alpha s=1-s,\qquad
 \delta s=a,\qquad
 2s=q,\qquad
 \beta s=1.
 \label{eq:parameter-identities}
\end{equation}
In particular,
\begin{equation}
 1<q<2,\qquad
 \frac12<s<1,\qquad
 0<\alpha,\delta,a<1,\qquad
 1<\beta<2.
 \label{eq:parameter-ranges}
\end{equation}
Here \(q\) is the conjugate exponent to \(p\), \(s\) is used for
fractional moments, \(\alpha\) normalizes the Euler products, and
\(\beta\) is the critical-tree parameter.  Unless otherwise stated,
constants may depend on \(p\).  They may change from line to line.

\subsection{Finite prime tori and prime-frequency operators}
\label{subsec:finite-prime-tori}

Fix for the moment a finite prime cutoff \(y\ge2\), and let
\[
 \Omega_y=\T^{\pi(y)}
\]
be equipped with a product normalized Haar measure.  Its coordinate functions are
denoted by
\[
 (\zeta_r)_{r\le y},
\]
where \(r\) always denotes a prime. 

For
\[
 \nu=(\nu_r)_{r\le y}\in\Z^{\pi(y)},
 \qquad
 \zeta^\nu=\prod_{r\le y}\zeta_r^{\nu_r},
\]
our Fourier and pairing conventions are
\begin{equation}
 \widehat F(\nu)
 =
 \int_{\Omega_y}
 F(\zeta)\overline{\zeta^\nu}\dd m(\zeta),
 \qquad
 \ip{F}{G}
 =
 \int_{\Omega_y}
 F(\zeta)\overline{G(\zeta)}\dd m(\zeta).
 \label{eq:torus-fourier-pairing}
\end{equation}

For \(u\in\R\), define the Kronecker flow
\[
 U_uF((\zeta_r)_{r\le y})
 =
 F((\e^{\ii u\log r}\zeta_r)_{r\le y}).
\]
The operators \((U_u)_{u\in\R}\) form a strongly continuous group of
measure-preserving isometries on every \(L^v(\Omega_y)\),
\(1\le v<\infty\).  Although there is one coordinate for each prime,
the flow has only one real parameter.  Each Fourier character $\zeta^\nu$ therefore has a
single frequency:
\[
 U_u\zeta^\nu
 =
 \e^{\ii u\sum_{r\le y}\nu_r\log r}\zeta^\nu.
\]
We therefore define the action on any homogeneous term before the cutoff $y$, or the full prime-frequency generator, by
\begin{equation}
 D_{\mathrm{full}}\zeta^\nu
 =
 \left(\sum_{r\le y}\nu_r\log r\right)\zeta^\nu.
 \label{eq:full-prime-generator}
\end{equation}

For a prime \(r\le y\), we also need to retain only the variables
corresponding to smaller primes.  Let
\[
 \mathscr F_{r^-}
 =
 \sigma(\zeta_\rho:\rho<r),
 \qquad
 \E_{r^-}F=\E(F\mid\mathscr F_{r^-}),
\]
and define the generator associated with primes below \(r\) by
\begin{equation}
 D_{r^-}\zeta^\nu
 =
 \left(\sum_{\rho<r}\nu_\rho\log\rho\right)\zeta^\nu.
 \label{eq:old-prime-generator}
\end{equation}

For a Borel set \(E\subset\R\), define the corresponding frequency
projections on trigonometric polynomials by
\[
 P_E^{\mathrm{full}}
 =
 \one_E(D_{\mathrm{full}}),
 \qquad
 P_E^{r^-}
 =
 \one_E(D_{r^-}).
\]
Equivalently, on a Fourier character \(\zeta^\nu\),
\[
 P_E^{\mathrm{full}}\zeta^\nu
 =
 \one_E\!\left(
 \sum_{\rho\le y}\nu_\rho\log\rho
 \right)\zeta^\nu,
 \qquad
 P_E^{r^-}\zeta^\nu
 =
 \one_E\!\left(
 \sum_{\rho<r}\nu_\rho\log\rho
 \right)\zeta^\nu.
\]
For intervals
and half-lines, we also use notation such as
\[
 P_{D\ge0}^{r^-}
 =
 \one_{[0,\infty)}(D_{r^-}).
\]

These are initially operators on trigonometric polynomials; we do not
assert \(L^q\)-boundedness for arbitrary Borel sets.  The projections
needed later will be estimated in their respective sections.

For every bounded Borel function \(m\colon\R\to\C\) and every
trigonometric polynomial \(F\), one has the characterwise identity
\begin{equation}
 \E_{r^-}m(D_{\mathrm{full}})F
 =
 m(D_{r^-})\E_{r^-}F.
 \label{eq:conditional-spectral-commutation}
\end{equation}
Indeed, conditional expectation removes every character involving a
prime \(\rho\ge r\).  On each remaining character, the two generators
have the same eigenvalue.  Absolute convergence extends the identity to
functions with absolutely summable Fourier coefficients, including the
Euler products introduced below at each fixed finite prime cutoff; see \eqref{eq:euler-source}.

For a finite index set \(\mathcal R\) and a measurable vector
\(F=(F_r)_{r\in\mathcal R}\), we write
\[
 \norm{F}_{L^q(\ell^2(\mathcal R))}
 =
 \left[
 \int_{\Omega_y}
 \left(
 \sum_{r\in\mathcal R}\abs{F_r(\zeta)}^2
 \right)^{q/2}
 \dd m(\zeta)
 \right]^{1/q}.
\]

\subsection{Euler products}
\label{subsec:euler-sources}

We now introduce the Euler products used in the dual argument, based on \eqref{eq:intro-source-pairing} and \eqref{eq:grid-average-tiling}.
Their \(L^q\)-norms can be computed explicitly.  Moreover, averaging
over the larger prime variables by conditional expectation gives the corresponding shorter product; see \eqref{eq:euler-source-conditioning}.

For \(2\le b\le y\), define
\begin{equation}
 M_b
 =
 \prod_{r\le b}(1-r^{-1})^{-1},
 \qquad
 S_b(t)
 =
 \prod_{r\le b}
 (1-r^{-1/2}\e^{\ii t\log r}\zeta_r)^{-1}.
 \label{eq:mertens-and-sb}
\end{equation}
The normalized Euler product is
\begin{align}
 f_b(t)
 ={}
 M_b^{-\alpha}
 \prod_{r\le b}
 (1-r^{-1/2}\e^{\ii t\log r}\zeta_r)^{-1}
 \prod_{r\le b}
 (1-r^{-1/2}\e^{-\ii t\log r}\overline{\zeta_r})^{-\alpha}.
 \label{eq:euler-source}
\end{align}

For \(\abs z<1\), we use the binomial expansion
\begin{equation}
 (1-z)^{-\alpha}
 =
\sum_{k=0}^{\infty}\frac{(\alpha)_k}{k!}z^k,
 \label{eq:binomial-branch}
\end{equation}
where \((\alpha)_k\) denotes the Pochhammer symbol.  Since
\(r^{-1/2}<1\), this defines every fractional factor in
\eqref{eq:euler-source} absolutely convergently and agrees with the
principal analytic branch.

Applying the Kronecker flow is the same as shifting the time parameter:
\begin{equation}
 U_uf_b(t)=f_b(t+u),
 \qquad t,u\in\R.
 \label{eq:euler-source-flow}
\end{equation}

The identity \(q(1+\alpha)=2\) is the reason for the normalization
in \eqref{eq:euler-source}.  Indeed, by
\eqref{eq:parameter-identities},
\[
 \abs{f_b(t)}^q
 =
 M_b^{-\alpha q}\abs{S_b(t)}^{q(1+\alpha)}
 =
 M_b^{-\alpha q}\abs{S_b(t)}^2.
\]
Since
\[
 \E\abs{S_b(t)}^2
 =
 \prod_{r\le b}
 \sum_{k=0}^{\infty}r^{-k}
 =
 M_b,
\]
we obtain
\begin{equation}
\norm{f_b(t)}_{L^q(\Omega_y)}
 =
 M_b^{1/p}.
 \label{eq:euler-source-normalization}
\end{equation}
Both quantities are independent of \(t\).

For a prime \(r\le y\), we use the notation
\[
 M_{r^-}
 =
 \prod_{\rho<r}(1-\rho^{-1})^{-1}
\]
and denote by \(f_{r^-}(t)\) the corresponding Euler product with all
products restricted to primes \(\rho<r\); explicitly,
\begin{align}
 f_{r^-}(t)
 ={}
 M_{r^-}^{-\alpha}
 \prod_{\rho<r}
 (1-\rho^{-1/2}\e^{\ii t\log\rho}\zeta_\rho)^{-1}
 \prod_{\rho<r}
 (1-\rho^{-1/2}\e^{-\ii t\log\rho}
       \overline{\zeta_\rho})^{-\alpha}.
 \label{eq:old-prime-source}
\end{align}

For a single prime factor, set \(x=\rho^{-1}\).  Using
\eqref{eq:binomial-branch} and Fourier orthogonality on \(\T\),
\begin{align*}
 \int_{\T}
 (1-x)^\alpha
 (1-\sqrt{x}z)^{-1}
 (1-\sqrt{x}\bar z)^{-\alpha}
 \dd m(z)
 =
 1.
\end{align*}
Thus each normalized one-prime factor has mean one.
Integrating successively over the prime coordinates
\(\zeta_\rho\) with \(r\le\rho\le b\), we obtain
\begin{equation}
 \E_{r^-}f_b(t)=f_{r^-}(t),
 \qquad r\le b.
 \label{eq:euler-source-conditioning}
\end{equation}
Thus the Euler products are compatible with conditional expectation
with respect to the variables indexed by primes below \(r\).

\subsection{The cyclic model and deformed Euler products}
\label{subsec:cyclic-model}

We now introduce the cyclic model which organizes the sample points, while a slight deformation of the Euler products will later provide the decay needed to control large frequencies.
For the finite cyclic model used below, fix
\begin{equation}
 H\in\N,\qquad H\ge1,\qquad
 N=2^H,\qquad
 y=\e^N,\qquad
 t_j=t_0+\frac jN,
 \quad 0\le j<N,
 \label{eq:finite-grid-parameters}
\end{equation}
where \(t_0\in\R\) is arbitrary. So there are \(N\) sample points, spaced
\(1/N\) apart.  We choose \(y=\e^N\) so that Mertens' theorem gives
\(M_y\asymp N\), matching the Euler-product normalization to the
number of samples.

We identify
\[
 \ZN=\{0,\ldots,N-1\}
\]
with the cyclic group under addition modulo \(N\), and all
\(\ell^u(\ZN)\)-norms use counting measure.  Represent the cyclic dual by
\[
 \Lambda_N=\{-N/2,\ldots,N/2-1\}.
\]
For \(\kappa\in\Lambda_N\), put
\[
 \theta_\kappa=\frac{2\pi\kappa}{N},
 \qquad
 \xi_\kappa=N\theta_\kappa=2\pi\kappa.
\]
Our discrete Fourier conventions are
\begin{equation}
 \widehat c(\kappa)
 =
 \frac1N\sum_{j=0}^{N-1}c_j\e^{\ii j\theta_\kappa},
 \qquad
 c_j
 =
 \sum_{\kappa\in\Lambda_N}
 \widehat c(\kappa)\e^{-\ii j\theta_\kappa}.
 \label{eq:cyclic-dft}
\end{equation}
Thus
\[
 (\mathcal F_N^{-1}g)(j)
 =
 \sum_{\kappa\in\Lambda_N}
 g(\kappa)\e^{-\ii j\theta_\kappa}.
\]

For \(\sigma\colon\Lambda_N\to\C\), let
\[
 K_\sigma(j)
 =
 \frac1N
 \sum_{\kappa\in\Lambda_N}
 \sigma(\kappa)\e^{-\ii j\theta_\kappa}.
\]
Then, with \(*\) denoting counting convolution on \(\ZN\),
\[
 K_\sigma*c
 =
 \mathcal F_N^{-1}(\sigma\widehat c).
\]

For an integer representative \(j\), we define the cyclic distance to be
\[
 \rho_N(j)
 =
 \dist(j,N\Z)
 =
 \min_{k\in\Z}\abs{j+kN}.
\]
When \(N\) is fixed, we abbreviate this by
\(
 \rho_j=\rho_N(j).
\)

For a fixed level \(0\le h\le H\), put
\begin{equation}
 w_h=2^h,
 \qquad
 m_h=N2^{-h}.
 \label{eq:dyadic-scale-parameters}
\end{equation}
At level \(h\) there are \(w_h\) blocks, each of length \(m_h\), so
\(w_hm_h=N\).  Let \(\mathcal D_h\) be the partition of
\(\ZN\) into the consecutive blocks
\begin{equation}
 \mathcal D_h
 =
 \left\{
 \{km_h,\ldots,(k+1)m_h-1\}:
 0\le k<2^h
 \right\}.
 \label{eq:nonwrapping-dyadic-grid}
\end{equation}
For \(x\in\ZN\), let \(Q_h(x)\) be the unique member of
\(\mathcal D_h\) containing \(x\).  Unlike the cyclic distance
\(\rho_N\), these blocks use the order \(0,\ldots,N-1\) and do not
wrap around the endpoints.

For the frequency estimates we fix an even smooth cutoff near zero,
\[
 \chi\in C_c^\infty((-1/8,1/8)),
 \qquad
 0\le\chi\le1,
 \qquad
 \chi=1\ \text{on }[-1/16,1/16],
\]
and a smooth cutoff for positive frequencies $\xi_\kappa$,
\[
 \phi\in C_c^\infty((1/2,2)),
 \qquad
 0\le\phi\le1,
\]
such that
\begin{equation}
 \sum_{v\in2^\Z}\phi(\lambda/v)=1,
 \qquad \lambda>0.
 \label{eq:dyadic-frequency-partition}
\end{equation}

At level \(h\), we also use a deformation of
\(f_{\e^{2^{h}}}\): it changes the sizes of the analytic and conjugate factors in opposite directions, while keeping the same prime variables.
Precisely, for \(0\le h\le H\) and \(0\le\vartheta\le1/32\), we define
\begin{align}
 f_h^\vartheta(t)
 ={}
 M_{\e^{2^{h}}}^{-\alpha}
 \prod_{r\le\e^{2^{h}}}
 \left(
 1-r^{-1/2}\e^{\vartheta\log r/2^{h}}
   \e^{\ii t\log r}\zeta_r
 \right)^{-1}
 \prod_{r\le\e^{2^{h}}}
 \left(
 1-r^{-1/2}\e^{-\vartheta\log r/2^{h}}
   \e^{-\ii t\log r}\overline{\zeta_r}
 \right)^{-\alpha}.
 \label{eq:deformed-euler-source}
\end{align}
The factors are well defined by the branch fixed in
\eqref{eq:binomial-branch}, since for \(r\le\e^{2^{h}}\),
\[
 r^{-1/2}\e^{\vartheta\log r/2^{h}}
 \le
 2^{-1/2}\e^{1/32}
 <1.
\]
At \(\vartheta=0\), the deformation reduces to the Euler product \eqref{eq:euler-source}; that is,
\(
 f_h^0(t)=f_{\e^{2^{h}}}(t).
\)
The family \((f_h^\vartheta)_{0\le h\le H}\) satisfies the same
time-translation identity as the undeformed products:
\begin{equation}
 U_uf_h^\vartheta(t)=f_h^\vartheta(t+u).
 \label{eq:deformed-source-flow}
\end{equation}

\section{Gaussian comparison for Euler products}
\label{sec:replacement}

The key in this section is to replace each prime's contribution by a Gaussian one
while preserving the main correlations between different sample points.
We first consider fractional moments of integrals of the deformed Euler products $f_h^\vartheta$:
\begin{equation}
 \E\left[
  \left(
   \int_J |f_h^\vartheta(t)|^2\dd t
  \right)^s
 \right],
 \qquad 0\le\vartheta\le1/32,\ 0\le h\le H.
 \label{eq:Euler-product-integrals-s-moment}
\end{equation}
These bounds will be used in Section~\ref{sec:euler-occupation}.
The purpose of this section is to compare the fractional moments
of integrals of the deformed Euler products $f_h^\vartheta$ with Gaussian moments; see Proposition \ref{prop:euler-to-gaussian-integral}.

We approximate the integral by finite Riemann sums
\(
 \sum_{i=1}^M
 \lambda_i|f_h^\vartheta(\tau_i)|^2,
\)
where $\tau_i$ are sample points and $\lambda_i\ge0$ are the corresponding weights.
To estimate fractional moments of the Riemann sums, we first replace the Euler-product terms by Gaussian terms in Section \ref{subsec:one-prime-replacement} and then compare the resulting Gaussian moments with those for \(\vartheta=0\) in Section \ref{subsec:gaussian-comparison-deformation}.

\subsection{Euler-to-Gaussian comparison}
\label{subsec:one-prime-replacement}

Now we replace one prime at a time by a Gaussian variable shared
across all sample points, preserving their main correlations
while keeping the errors summable over all primes.
Let \(M\in\N\), \(\lambda_i\ge0\), and we define
\begin{equation}
 \Phi(u)
 =
 \left(
 \sum_{i=1}^{M}\lambda_i\e^{u_i}
 \right)^s,
 \qquad
 u=(u_1,\ldots,u_M)\in\R^M.
 \label{eq:positive-sum-functional}
\end{equation}
Recall that \(s\) is defined in
\eqref{eq:parameter-definitions}, and \(0<s<1\) is fixed throughout the proof.
We record the following derivative and shift bounds for $\Phi$.

\begin{lemma}[Dimension-free derivatives]
\label{lem:positive-sum-derivatives}
For every $u\in\R^M$, integer $1\le k\le4$,
and $v_1,\ldots,v_k\in\R^M$, one has
\begin{equation}
 \abs{
 D^k\Phi(u)[v_1,\ldots,v_k]
 }
 \le
 C_{k,p}\Phi(u)
 \prod_{\nu=1}^{k}
 \norm{v_\nu}_{\ell^\infty},
 \label{eq:positive-sum-derivative-bound}
\end{equation}
where the constant is independent of
\(M\), the weights \(\lambda_i\), and \(u\).
Moreover, for every \(h\in\R^M\),
\begin{equation}
 \e^{-s\norm h_\infty}\Phi(u)
 \le
 \Phi(u+h)
 \le
 \e^{s\norm h_\infty}\Phi(u).
 \label{eq:positive-sum-shift-bound}
\end{equation}
\end{lemma}

\begin{proof}
The shift bound
\eqref{eq:positive-sum-shift-bound}
follows immediately from
\eqref{eq:positive-sum-functional}.
Indeed,
\[
 \e^{-\|h\|_\infty}
 \left(\sum_{i=1}^M \lambda_i\e^{u_i}
 \right)
 \le
 \sum_{i=1}^M
 \lambda_i\e^{u_i+h_i}
 \le
 \e^{\|h\|_\infty}
 \left(\sum_{i=1}^M \lambda_i\e^{u_i}
 \right).
\]
Raising these inequalities to the power \(s\) gives
\eqref{eq:positive-sum-shift-bound}.
We now prove \eqref{eq:positive-sum-derivative-bound}.
Put
\[
 W=\sum_{i=1}^{M}\lambda_i\e^{u_i}.
\]
If \(W=0\), then every \(\lambda_i=0\).
In this case \(\Phi\) is identically zero,
so all assertions are immediate.
Suppose that \(W>0\), and set
\[
 p_i=\frac{\lambda_i\e^{u_i}}{W},
 \qquad 1\le i\le M.
\]
Then, we have
\[
 p_i\ge0,
 \qquad
 \sum_{i=1}^{M}p_i=1.
\]
Therefore, for \(t=(t_1,\ldots,t_k)\), one has
\[
 \Phi\left(
 u+\sum_{\nu=1}^{k}t_\nu v_\nu
 \right)
 =
 \Phi(u)
 \left[
 \sum_{i=1}^{M}
 p_i
 \exp\left(
 \sum_{\nu=1}^{k}t_\nu v_{\nu,i}
 \right)
 \right]^s.
\]
We now consider the differentiation. 
For the first two orders, differentiation at \(t=0\) gives
\begin{align*}
 D\Phi(u)[v]
 &=s\Phi(u)\sum_{i=1}^M p_iv_i,
 \\
 D^2\Phi(u)[v_1,v_2]
 &=s\Phi(u)\sum_{i=1}^M p_iv_{1,i}v_{2,i}
 +s(s-1)\Phi(u)
 \left(\sum_{i=1}^M p_iv_{1,i}\right)
 \left(\sum_{i=1}^M p_iv_{2,i}\right).
\end{align*}
More generally, for \(k\ge1\), let \(\mathfrak P_k\) denote the set of
partitions of \(\{1,\ldots,k\}\), and write
\[
 (s)^{\downarrow}_m
 =
 s(s-1)\cdots(s-m+1).
\]
Differentiation at \(t=0\) gives
\begin{equation}
 D^k\Phi(u)[v_1,\ldots,v_k]
 =
 \Phi(u)
 \sum_{\pi\in\mathfrak P_k}
 (s)^{\downarrow}_{\abs\pi}
 \prod_{B\in\pi}
 \left(
 \sum_{i=1}^{M}
 p_i\prod_{\nu\in B}v_{\nu,i}
 \right).
 \label{eq:positive-sum-partition-formula}
\end{equation}
For every block \(B\), we have
\[
 \abs{
 \sum_{i=1}^{M}
 p_i\prod_{\nu\in B}v_{\nu,i}
 }
 \le
 \prod_{\nu\in B}\norm{v_\nu}_\infty.
\]
The number of partitions depends only on \(k\), and each factor is bounded independently of \(M\).
Combining this bound with \eqref{eq:positive-sum-partition-formula}, we prove
\eqref{eq:positive-sum-derivative-bound}.
\end{proof}

Fix a prime \(r\) and
\(0\le\vartheta\le1/32\).
For each \(1\le i\le M\), we let
\[
 a_i\in[0,1],
 \qquad
 \varphi_i\in\R,
\]
and put
\begin{equation}
 c_{i,k}
 =
 \e^{k\vartheta a_i}
 +
 \alpha\e^{-k\vartheta a_i},
 \qquad k\ge1.
 \label{eq:one-prime-cik}
\end{equation}
For each prime \(r\), let \(\zeta_r=\e^{\ii \Theta_r}\)
be Haar distributed on \(\T\), and let \(g_r\)
be a standard complex Gaussian variable, so that
\[
 \E g_r=\E g_r^2=0,
 \qquad
 \E\abs{g_r}^2=1.
\]
Assume that the pairs \((\zeta_r,g_r)\) are
independent across primes \(r\), and that \(\zeta_r\)
and \(g_r\) are independent for each prime.
We define the random vectors
\(H^{(r)},G^{(r)}\in\R^M\) by
\begin{align}
 H_i^{(r)}
 &=
 2\alpha\log(1-r^{-1})
 +
 2\sum_{k=1}^{\infty}
 \frac{r^{-k/2}}{k}
 c_{i,k}
 \cos k(\Theta_r+\varphi_i),
 \label{eq:one-prime-circle-increment}
 \\
 G_i^{(r)}
 &=
 -\frac{2\alpha}{r}
 +
 \frac{2}{\sqrt r}\,
 c_{i,1}
 \Ree(g_r\e^{\ii \varphi_i}),
 \label{eq:one-prime-gaussian-increment}
\end{align}
for every $i=1,\cdots, M$.

Taking the logarithm of \(\abs{f_h^\vartheta(\tau_i)}^2\)
gives a sum over primes \(r\le\e^w\), where \(w=2^h\). The contribution of each prime \(r\)
is \(H_i^{(r)}\) with
\[
 a_i=\frac{\log r}{w},\qquad
 \varphi_i=\tau_i\log r.
\]
Indeed, the logarithmic expansion
\[
 \log\abs{1-\rho\e^{\ii\psi}}^{-2}
 =2\sum_{k=1}^{\infty}\frac{\rho^k}{k}\cos(k\psi),
 \qquad 0\le\rho<1,\quad \psi\in\R,
\]
shows that the analytic and conjugate factors of
\(f_h^\vartheta\) give the two terms in \(c_{i,k}\),
respectively.
The normalization factor \((1-r^{-1})^\alpha\)
contributes \(2\alpha\log(1-r^{-1})\) to the logarithm
of the squared modulus.
Thus
\begin{equation}
 \log\abs{f_h^\vartheta(\tau_i)}^2
 =\sum_{r\le\e^w}H_i^{(r)},\qquad i=1,\ldots,M.
 \label{eq:euler-localized-one-prime-logarithm}
\end{equation}

The vector \(G^{(r)}\) replaces the first random term of $H^{(r)}$ by a Gaussian term with the same covariance at the sample points $\tau_i$.
We will compare the fractional moment of
\(
 \sum_{i=1}^M\lambda_i\abs{f_h^\vartheta(\tau_i)}^2
\)
with its Gaussian counterpart by replacing \(H^{(r)}\)
with \(G^{(r)}\).
The proposition below controls the other random terms and the approximation of the constant term.

\begin{proposition}[One-prime replacement]
\label{prop:one-prime-replacement}
For a fixed prime \(r\), the functional \(\Phi\) in
\eqref{eq:positive-sum-functional} satisfies
\begin{equation}
 \abs{
 \E\Phi(u+H^{(r)})
 -
 \E\Phi(u+G^{(r)})
 }
 \le
 C_pr^{-2}\Phi(u),\qquad u\in\R^M.
 \label{eq:one-prime-additive-replacement}
\end{equation}
Moreover,
\begin{equation}
 \frac{1}{1+C_pr^{-2}}
 \E\Phi(u+G^{(r)})
 \le
 \E\Phi(u+H^{(r)})
 \le
 (1+C_pr^{-2})
 \E\Phi(u+G^{(r)}).
 \label{eq:one-prime-relative-replacement}
\end{equation}
The constant \(C_p\) depends only on \(p\).
\end{proposition}

\begin{proof}
We write
\[
 H^{(r)}=\mu+L+R,
 \qquad
 G^{(r)}=\widetilde\mu+\Gamma,
\]
where
\begin{align*}
 \mu_i
 &=2\alpha\log(1-r^{-1}),
 &
 \widetilde\mu_i
 &=-2\alpha r^{-1},
 \\
 L_i
 &=2r^{-1/2}c_{i,1}
 \cos(\Theta_r+\varphi_i),
 &
 \Gamma_i
 &=2r^{-1/2}c_{i,1}
 \Ree(g_r\e^{\ii \varphi_i}),
 \\
 R_i
 &=2\sum_{k=2}^{\infty}
 \frac{r^{-k/2}}{k}
 c_{i,k}
 \cos k(\Theta_r+\varphi_i).
\end{align*}
The terms in the decomposition have the following roles:

\begin{center}
\small
\renewcommand{\arraystretch}{1.3}
\begin{tabular}{@{}cc@{}}
\toprule
Term & Role \\
\midrule
$\mu$ & Deterministic term of $H^{(r)}$ \\
$\widetilde\mu$ & Deterministic term of $G^{(r)}$ \\
$L$ & The $k=1$ term in the random part of $H^{(r)}$ \\
$R$ & Sum of the terms with $k\ge2$ in $H^{(r)}$ \\
$\Gamma$ & Centered Gaussian replacement of $L$ \\
\bottomrule
\end{tabular}
\end{center}
We use supremum norms so that the derivative bound
\eqref{eq:positive-sum-derivative-bound} introduces no factor
depending on the number of coordinates.

Uniformly in all the parameters above,
\begin{equation}
 \norm\mu_\infty+
 \norm{\widetilde\mu}_\infty
 \le C_pr^{-1},
 \qquad
 \norm{\mu-\widetilde\mu}_\infty
 \le C_pr^{-2}.
 \label{eq:one-prime-drift-bounds}
\end{equation}
These estimates follow from
\[
 \log(1-r^{-1})=-r^{-1}+O(r^{-2}),
 \qquad r\ge2.
\]
Moreover, we have
\begin{equation}
 \norm L_\infty\le C_pr^{-1/2},
 \qquad
 \norm R_\infty\le C_pr^{-1}.
 \label{eq:one-prime-harmonic-bounds}
\end{equation}
Indeed,
\[
 r^{-1/2}\e^{\vartheta a_i}
 \le 2^{-1/2}\e^{1/32}<1,
\]
which implies that
\[
 \sum_{k=2}^{\infty}\frac{r^{-k/2}}{k}c_{i,k}
 \le
 (1+\alpha)\sum_{k=2}^{\infty}
 \left(r^{-1/2}\e^{\vartheta a_i}\right)^k
 \le C_pr^{-1}.
\]
The bound for \(L\) follows from the uniform bound on \(c_{i,1}\).
The same bound also gives
\[
 \norm\Gamma_\infty\le C_pr^{-1/2}\abs{g_r}.
\]
Circle orthogonality and the Gaussian covariance identity give
\begin{equation}
 \E L_i=\E R_i=\E\Gamma_i=0,
 \qquad
 \E(L_iR_j)=0,
 \label{eq:one-prime-first-moments}
\end{equation}
\begin{equation}
 \E(L_iL_jL_k)
 =
 \E(\Gamma_i\Gamma_j\Gamma_k)
 =
 0,
 \label{eq:one-prime-third-moments}
\end{equation}
and
\begin{equation}
 \E(L_iL_j)
 =
 2r^{-1}c_{i,1}c_{j,1}
 \cos(\varphi_i-\varphi_j)
 =
 \E(\Gamma_i\Gamma_j).
 \label{eq:one-prime-covariance-matching}
\end{equation}
For a random vector \(V\), Taylor's formula at the prime-independent point
\(u\) gives
\begin{equation}
 \E\Phi(u+V)
 =
 \Phi(u)
 +
 \sum_{m=1}^{3}
 \frac1{m!}
 \E\bigl[
 D^m\Phi(u)
 [V,\ldots,V]
 \bigr]
+
 \E\mathcal R_4(V),
 \label{eq:one-prime-taylor-expansion}
\end{equation}
where
\begin{equation}
 \mathcal R_4(V)
 =
 \frac1{3!}
 \int_0^1
 (1-t)^3
 D^4\Phi(u+tV)
 [V,V,V,V]\dd t.
 \label{eq:one-prime-fourth-remainder}
\end{equation}
We compare \eqref{eq:one-prime-taylor-expansion}
for \(V=H^{(r)}\) and \(V=G^{(r)}\).  The comparison through degree three can be organized as follows.

At first order, we have
\begin{equation*}
\E D\Phi(u)[H^{(r)}]
 =D\Phi(u)[\mu],
 \qquad
 \E D\Phi(u)[G^{(r)}]
 =D\Phi(u)[\widetilde\mu].
\end{equation*} 
Using \eqref{eq:positive-sum-derivative-bound} and
\eqref{eq:one-prime-drift-bounds}, we obtain
\begin{align*}
\left|
 \E D\Phi(u)[H^{(r)}]
 -
 \E D\Phi(u)[G^{(r)}]
 \right|
 =
 \left|D\Phi(u)[\mu-\widetilde\mu]\right|
 \le C_pr^{-2}\Phi(u).
\end{align*}

At second order, \eqref{eq:one-prime-covariance-matching} gives
\[
 \E D^2\Phi(u)[L,L]
 =\sum_{i,j=1}^M\partial_{ij}\Phi(u)\E(L_iL_j)
=\sum_{i,j=1}^M\partial_{ij}\Phi(u)\E(\Gamma_i\Gamma_j)
=\E D^2\Phi(u)[\Gamma,\Gamma].
\]
The mixed \(LR\) moment vanishes by
\eqref{eq:one-prime-first-moments}.
The terms containing one deterministic vector and one centered vector
also have zero expectation.  Hence
\begin{align*}
\E D^2\Phi(u)[H^{(r)},H^{(r)}]
 &=D^2\Phi(u)[\mu,\mu]
   +\E D^2\Phi(u)[L,L]
   +\E D^2\Phi(u)[R,R],
 \\
\E D^2\Phi(u)[G^{(r)},G^{(r)}]
 &=D^2\Phi(u)[\widetilde\mu,\widetilde\mu]
   +\E D^2\Phi(u)[\Gamma,\Gamma].
\end{align*}
Therefore, we have
\begin{align*}
 \left|\E D^2\Phi(u)[H^{(r)},H^{(r)}]
 -\E D^2\Phi(u)[G^{(r)},G^{(r)}]\right|
 &\le
 C_p\Phi(u)
 \left(
 \|\mu\|_\infty^2+
 \|\widetilde\mu\|_\infty^2+
 \E\|R\|_\infty^2
 \right)\\
&\le C_pr^{-2}\Phi(u).
\end{align*}

At third order, the pure cubic terms in \(L\) and \(\Gamma\)
vanish by \eqref{eq:one-prime-third-moments}.
Using also the zero first moments and the vanishing \(LR\) moments in \eqref{eq:one-prime-first-moments},
symmetry and multilinearity give
\begin{align*}
 \E D^3\Phi(u)[H^{(r)},H^{(r)},H^{(r)}]
 &=D^3\Phi(u)[\mu,\mu,\mu]
   +3\E D^3\Phi(u)[\mu,L,L]\\
 &\quad+3\E D^3\Phi(u)[\mu,R,R]
   +3\E D^3\Phi(u)[L,L,R]\\
 &\quad+3\E D^3\Phi(u)[L,R,R]
   +\E D^3\Phi(u)[R,R,R],\\
 \E D^3\Phi(u)[G^{(r)},G^{(r)},G^{(r)}]
 &=D^3\Phi(u)[\widetilde\mu,\widetilde\mu,\widetilde\mu]
   +3\E D^3\Phi(u)[\widetilde\mu,\Gamma,\Gamma].
\end{align*}
We estimate the following mixed terms:
\begin{align*}
 \left|\E D^3\Phi(u)[\mu,L,L]\right|
 &\le
 C_p\Phi(u)
 \|\mu\|_\infty\,\E\|L\|_\infty^2
 \le C_pr^{-2}\Phi(u),
 \\
 \left|\E D^3\Phi(u)[\widetilde\mu,\Gamma,\Gamma]\right|
 &\le
 C_p\Phi(u)
 \|\widetilde\mu\|_\infty\,\E\|\Gamma\|_\infty^2
 \le C_pr^{-2}\Phi(u),
 \\
 \left|\E D^3\Phi(u)[L,L,R]\right|
 &\le
 C_p\Phi(u)
 \E\!\left(\|L\|_\infty^2\|R\|_\infty\right)
 \le C_pr^{-2}\Phi(u),
 \\
 \left|\E D^3\Phi(u)[L,R,R]\right|
 &\le
 C_p\Phi(u)
 \E\!\left(\|L\|_\infty\|R\|_\infty^2\right)
 \le C_pr^{-5/2}\Phi(u).
\end{align*}
Hence,
\begin{align*}
\left|\E D^3\Phi(u)[H^{(r)},H^{(r)},H^{(r)}]
-\E D^3\Phi(u)[G^{(r)},G^{(r)},G^{(r)}]\right|
\le C_pr^{-2}\Phi(u).
\end{align*}

It remains to estimate the fourth-order remainders.  From
\eqref{eq:one-prime-drift-bounds} and
\eqref{eq:one-prime-harmonic-bounds}, we obtain that
\[
 \norm{H^{(r)}}_\infty
 \le C_pr^{-1/2}.
\]
By \eqref{eq:positive-sum-derivative-bound} with \(k=4\)
and \eqref{eq:positive-sum-shift-bound},
\[
 \abs{\E\mathcal R_4(H^{(r)})}
\le C_p\Phi(u)\,
 \E\!\left[
 \norm{H^{(r)}}_\infty^4
 \e^{s\norm{H^{(r)}}_\infty}
 \right]
\le C_pr^{-2}\Phi(u).
\]
On the Gaussian side,
\[
 \norm{G^{(r)}}_\infty
 \le
 C_p\left(
 r^{-1}+r^{-1/2}\abs{g_r}
 \right).
\]
Since \(r\ge2\),
\[
 \left(r^{-1}+r^{-1/2}\abs{g_r}\right)^4
 \le 8r^{-2}\left(1+\abs{g_r}^4\right).
\]
Therefore,
\[
 \begin{aligned}
 \abs{\E\mathcal R_4(G^{(r)})}
 &\le C_p\Phi(u)\,
 \E\!\left[
 \left(r^{-1}+r^{-1/2}\abs{g_r}\right)^4
 \e^{C_p(r^{-1}+r^{-1/2}\abs{g_r})}
 \right]\\
 &\le C_pr^{-2}\Phi(u)\,
 \E\!\left[
 \left(1+\abs{g_r}^4\right)\e^{C_p\abs{g_r}}
 \right]\\
 &\le C_pr^{-2}\Phi(u).
 \end{aligned}
\]
The last expectation is finite by the Gaussian tail of \(\abs{g_r}\).
The same \(g_r\) controls all coordinates, so this estimate introduces
no dependence on \(M\).
This proves \eqref{eq:one-prime-additive-replacement}.

For the relative estimate, the bound on \(H^{(r)}\) and \(r\ge2\)
give \(\norm{H^{(r)}}_\infty\le C_p\).
By \eqref{eq:positive-sum-shift-bound}, we have
\[
 \E\Phi(u+H^{(r)})
 \ge \e^{-C_p}\Phi(u).
\]
On the event \(\{\abs{g_r}\le1\}\), which has fixed positive
probability, we also have \(\norm{G^{(r)}}_\infty\le C_p\).
Thus
\[
 \E\Phi(u+G^{(r)})
 \ge
 \Prob\{\abs{g_r}\le1\}\e^{-C_p}\Phi(u).
\]
Combining these lower bounds with
\eqref{eq:one-prime-additive-replacement}, and increasing \(C_p\)
if necessary, proves \eqref{eq:one-prime-relative-replacement}.
\end{proof}

\begin{corollary}[Sequential prime replacement]
\label{cor:sequential-shared-prime-replacement}
For \(w\ge1\), let \(H^{(r)}\) and \(G^{(r)}\) be defined by
\eqref{eq:one-prime-circle-increment}
and \eqref{eq:one-prime-gaussian-increment}
for every prime \(r\le\e^w\).
Then, we have
\begin{equation}
 \E\Phi\left(
 \sum_{r\le\e^w}H^{(r)}
 \right)
 \asymp_p
 \E\Phi\left(
 \sum_{r\le\e^w}G^{(r)}
 \right),
 \label{eq:sequential-shared-prime-replacement}
\end{equation}
where the comparison constant depends only on \(p\).
\end{corollary}

\begin{proof}
We list the primes \(r\le\e^w\) as
\(r_1,\ldots,r_K\). 
For \(0\le j\le K\), write
\[
 A_j
 =
 \E\Phi\left(
 \sum_{\ell\le j}G^{(r_\ell)}
 +
 \sum_{\ell>j}H^{(r_\ell)}
 \right).
\]
Fix \(1\le j\le K\).
Condition on all pairs \((\zeta_{r_\ell},g_{r_\ell})\)
with \(\ell\ne j\).
The contribution of the other primes is then a fixed vector
\[
 u=
 \sum_{\ell<j}G^{(r_\ell)}
 +
 \sum_{\ell>j}H^{(r_\ell)}
 \in\R^M.
\]
Proposition~\ref{prop:one-prime-replacement}
applies uniformly to this vector and gives
\[
 \frac{1}{1+C_pr_j^{-2}}A_j
 \le A_{j-1}
 \le(1+C_pr_j^{-2})A_j
\]
after taking expectations over the
other prime variables.
Iterating this comparison gives
\[
 \left(\prod_{r\le\e^w}(1+C_pr^{-2})\right)^{-1}A_K
 \le A_0
 \le
 \left(\prod_{r\le\e^w}(1+C_pr^{-2})\right)A_K.
\]
Since \(\sum_r r^{-2}<\infty\), we have
\begin{equation}
 \prod_{r\le\e^w}(1+C_pr^{-2})
 \le
 \exp\left(C_p\sum_r r^{-2}\right)
 <\infty.
 \label{eq:sequential-prime-product}
\end{equation}
This bound is independent of \(w\), so
\eqref{eq:sequential-shared-prime-replacement} follows.
\end{proof}

\subsection{Gaussian comparison under deformation}
\label{subsec:gaussian-comparison-deformation}

The preceding argument compares the Euler-product
moments with those of the Gaussian sum
\(
 \E\Phi\left(
 \sum_{r\le\e^w}G^{(r)}
 \right).
\)
To study this sum, we use the parameters
\(
 a_i=\log r/w,
\)
\(
 \varphi_i=\tau_i\log r
\)
for each prime $r\le\e^w$, where $w=2^h$.

Next, we define
\begin{equation}
 c_\vartheta(x)
 =
 \e^{\vartheta x}
 +
 \alpha\e^{-\vartheta x},
 \qquad 0\le x\le1,
 \label{eq:gaussian-deformation-c}
\end{equation}
and the Gaussian field
\begin{equation}
 Z_w^\vartheta(t)
 =
 \frac{\sqrt2}{\beta}
 \sum_{r\le\e^w}
 \frac{c_\vartheta(\log r/w)}{\sqrt r}
 \Ree(g_r\e^{\ii t\log r}).
 \label{eq:gaussian-deformation-field}
\end{equation}
Then, for all $1\le i\le M$, we have
\[
 \sum_{r\le\e^w}G_i^{(r)}
 =
 -2\alpha V_w
 +
 \sqrt2\beta Z_w^\vartheta(\tau_i),
 \qquad
 V_w=\sum_{r\le\e^w}\frac1r.
\]
Hence,
\[
\E\Phi\left(
 \sum_{r\le\e^w}G^{(r)}
 \right)=
\E\left[
 \left(
 \sum_{i=1}^{M}
 \lambda_i
 \e^{-2\alpha V_w+
 \sqrt2\beta Z_w^\vartheta(\tau_i)}
 \right)^s
 \right].
\]
We will show that this fractional moment on the right hand side is comparable to the one with $\vartheta=0$,
uniformly in $w$, the sample times $\tau_i$, and the nonnegative weights $\lambda_i$; see Proposition \ref{prop:gaussian-deformation}.
To prove this, we compare the covariance matrices of the corresponding Gaussian vectors.
We first recall the following form of Kahane's convexity inequality; see \cite[Lemma~1]{Kahane1985} and
\cite[Theorem~2.1]{RhodesVargas2014}.

\begin{lemma}[Concave covariance comparison]
\label{lem:concave-gaussian-comparison}
Let \(X=(X_i)_{i=1}^{M}\) and \(Y=(Y_i)_{i=1}^{M}\) be centered real Gaussian
vectors with covariance matrices \(K_X\) and \(K_Y\).  Suppose that
\begin{equation}
 K_X(i,j)\le K_Y(i,j),
 \qquad 1\le i,j\le M.
 \label{eq:entrywise-covariance-order}
\end{equation}
Then, for every finite family
\(\lambda_i\ge0\),
\begin{equation}
 \E
 \left(
 \sum_{i=1}^{M}
 \lambda_i
 \e^{Y_i-K_Y(i,i)/2}
 \right)^s
\le
 \E
 \left(
 \sum_{i=1}^{M}
 \lambda_i
 \e^{X_i-K_X(i,i)/2}
 \right)^s.
 \label{eq:concave-covariance-comparison}
\end{equation}
The result remains valid when either covariance matrix is singular.
\end{lemma}

For brevity, denote $Z_w^\vartheta=(Z_w^\vartheta(\tau_i))_{i=1}^M$.
\begin{proposition}[Gaussian comparison under deformation]
\label{prop:gaussian-deformation}
Let $w\ge1$, $M\ge1$, and
$0\le\vartheta\le1/32$.
Let $\tau_1,\ldots,\tau_M\in\R$,
let $\lambda_i\ge0$, and let
$d=(d_1,\ldots,d_M)\in\R^M$.
Then, we have
\begin{equation}
\e^{-C_p\vartheta}
\E\Phi\left(
 d+\sqrt2\beta
 Z_w^0
\right)
\le
\E\Phi\left(
 d+\sqrt2\beta
Z_w^\vartheta
\right)
\le
\e^{C_p\vartheta}
\E\Phi\left(
 d+\sqrt2\beta
 Z_w^0
\right).
\label{eq:gaussian-deformation-comparison}
\end{equation}
The constant $C_p$ depends only on $p$.
\end{proposition}

\begin{proof}
Let
\[
 K_\vartheta(i,j)
 =
 \Cov\left(
 \sqrt2\beta Z_w^\vartheta(\tau_i),
 \sqrt2\beta Z_w^\vartheta(\tau_j)
 \right).
\]
Since
\[
 \E\left[
 \Ree(g_r\e^{\ii u})\Ree(g_r\e^{\ii v})
 \right]
 =
 \frac12\cos(u-v),
\]
we have
\begin{equation}
 K_\vartheta(i,j)
 =
 2\sum_{r\le\e^w}
 \frac1r
 c_\vartheta\!\left(\frac{\log r}{w}\right)^2
 \cos\bigl((\tau_i-\tau_j)\log r\bigr).
 \label{eq:gaussian-deformation-covariance}
\end{equation}
For $0\le u\le1$ and
$0\le\vartheta\le1/32$, by
\eqref{eq:gaussian-deformation-c}, we have
\begin{equation}
 \abs{
 c_\vartheta(u)^2-c_0(u)^2
 }
 \le C_p\vartheta u.
 \label{eq:gaussian-deformation-local-covariance}
\end{equation}
Using the standard prime estimate
\[
 \sum_{r\le X}\frac{\log r}{r}
 \le C\log X,
 \qquad X\ge2,
\]
see \cite[Theorem~2.7(b), p.~50]
{MontgomeryVaughan06}, we obtain
\begin{equation}
 \abs{K_\vartheta(i,j)-K_0(i,j)}
 \le
 C_p\vartheta
 \sum_{r\le\e^w}
 \frac{\log r}{wr}
 \le
 C_p\vartheta.
 \label{eq:gaussian-deformation-covariance-difference}
\end{equation}

Let $B\sim\mathcal N(0,C_p\vartheta)$ be independent of both Gaussian fields $Z_w^{0}(t)$ and $Z_w^{\vartheta}(t)$.
Then, for every $1\le i,j\le M$, one has
\begin{equation}
\begin{aligned}
\Cov\left(
 \sqrt2\beta Z_w^0(\tau_i)+B,
 \sqrt2\beta Z_w^0(\tau_j)+B
\right)
&=K_0(i,j)+C_p\vartheta
\ge K_\vartheta(i,j),\\
\Cov\left(
 \sqrt2\beta Z_w^\vartheta(\tau_i)+B,
 \sqrt2\beta Z_w^\vartheta(\tau_j)+B
\right)
&=K_\vartheta(i,j)+C_p\vartheta
\ge K_0(i,j).
\end{aligned}
\label{eq:gaussian-deformation-two-orders}
\end{equation}
Since the diagonal entries of each covariance matrix are
constant, we apply Lemma~\ref{lem:concave-gaussian-comparison}
with the nonnegative weights \(\lambda_i\e^{d_i}\)
to the first covariance inequality in
\eqref{eq:gaussian-deformation-two-orders} and obtain
\begin{equation*}
\E\Phi\left(
 d+\sqrt2\beta Z_w^0
 +\left(B-\frac{K_0(1,1)+C_p\vartheta}{2}\right)\mathbf1
\right)
\le
\E\Phi\left(
 d+\sqrt2\beta Z_w^\vartheta
 -\frac{K_\vartheta(1,1)}{2}\mathbf1
\right).
\end{equation*}
Note that for every $u\in\R^M$ and $t\in\R$,
\[
 \Phi(u+t\mathbf1)=\e^{st}\Phi(u).
\]
The Gaussian moment-generating formula gives
\[
 \E\e^{sB-sC_p\vartheta/2}
 =\exp\left(\frac{s^2-s}{2}C_p\vartheta\right)
 =\e^{-s(1-s)C_p\vartheta/2}.
\]
Since \(B\) is independent of \(Z_w^0\), it follows that
\[
 \E\Phi\left(
 d+\sqrt2\beta Z_w^0
 +\left(B-\frac{K_0(1,1)+C_p\vartheta}{2}\right)\mathbf1
\right)
=\e^{-s(1-s)C_p\vartheta/2}
 \E\Phi\left(
 d+\sqrt2\beta Z_w^0
 -\frac{K_0(1,1)}{2}\mathbf1
 \right).
\]
Using the second inequality in
\eqref{eq:gaussian-deformation-two-orders}
and the same argument above, we obtain
\begin{equation}
\begin{aligned}
\e^{-s(1-s)C_p\vartheta/2}
\E\Phi\left(
 d+\sqrt2\beta Z_w^0
 -\frac{K_0(1,1)}{2}\mathbf1
\right)
&\le
\E\Phi\left(
 d+\sqrt2\beta Z_w^\vartheta
 -\frac{K_\vartheta(1,1)}{2}\mathbf1
\right)\\
&\le
\e^{s(1-s)C_p\vartheta/2}
\E\Phi\left(
 d+\sqrt2\beta Z_w^0
 -\frac{K_0(1,1)}{2}\mathbf1
\right).
\end{aligned}
\label{eq:gaussian-deformation-fixed-weights}
\end{equation}

Finally, by
\eqref{eq:gaussian-deformation-covariance-difference}, one has
\[
 \abs{K_\vartheta(1,1)-K_0(1,1)}
 \le C_p\vartheta.
\]
By homogeneity and
\eqref{eq:gaussian-deformation-fixed-weights},
\begin{align*}
 \E\Phi\left(d+\sqrt2\beta Z_w^\vartheta\right)
 &=
 \e^{sK_\vartheta(1,1)/2}
 \E\Phi\left(
 d+\sqrt2\beta Z_w^\vartheta
 -\frac{K_\vartheta(1,1)}2\mathbf1
 \right)\\
 &\quad\le
 \exp\left(
 \frac{sK_\vartheta(1,1)}2
 +\frac{s(1-s)}2C_p\vartheta
 \right)
 \E\Phi\left(
 d+\sqrt2\beta Z_w^0
 -\frac{K_0(1,1)}2\mathbf1
 \right)\\
 &\quad=
 \exp\left(
 \frac{s}{2}\bigl(K_\vartheta(1,1)-K_0(1,1)\bigr)
 +\frac{s(1-s)}2C_p\vartheta
 \right)
 \E\Phi\left(d+\sqrt2\beta Z_w^0\right)\\
 &\quad\le
 \e^{C_p\vartheta}
 \E\Phi\left(d+\sqrt2\beta Z_w^0\right).
\end{align*}
The last inequality uses
\(\abs{K_\vartheta(1,1)-K_0(1,1)}\le C_p\vartheta\),
after increasing \(C_p\) if necessary.
Using the other inequality in
\eqref{eq:gaussian-deformation-fixed-weights}
in the same way gives
\[
 \e^{-C_p\vartheta}
 \E\Phi\left(d+\sqrt2\beta Z_w^0\right)
 \le
 \E\Phi\left(d+\sqrt2\beta Z_w^\vartheta\right).
\]
This proves \eqref{eq:gaussian-deformation-comparison}.
\end{proof}

We now return to the Euler-product integrals \eqref{eq:Euler-product-integrals-s-moment} considered at the beginning of this section.
The preceding comparisons give the following result.

\begin{proposition}[Euler-to-Gaussian comparison for integrals]
\label{prop:euler-to-gaussian-integral}
Let $h\in\{0,\ldots,H\}$, put $w=2^h$, and let
$0\le\vartheta\le1/32$.
Let $J\subset\R$ be a bounded interval.
Then
\begin{equation}
\E\left[
 \left(
 \int_J |f_h^\vartheta(t)|^2\dd t
 \right)^s
\right]
\asymp_p
\E\left[
 \left(
 \int_J
 \e^{-2\alpha V_w+\sqrt2\beta Z_w^0(t)}
 \dd t
 \right)^s
\right].
\label{eq:euler-to-gaussian-integral-comparison}
\end{equation}
The comparison constants depend only on $p$.
\end{proposition}

\begin{proof}
Choose Riemann sums approximating the integrals over $J$. 
By Corollary~\ref{cor:sequential-shared-prime-replacement} and Proposition~\ref{prop:gaussian-deformation}, we obtain
\[
\E\left[
 \left(
 \sum_{i=1}^{M}
 \lambda_i|f_h^\vartheta(\tau_i)|^2
 \right)^s
\right]
\asymp_p
\E\left[
 \left(
 \sum_{i=1}^{M}
 \lambda_i
 \e^{-2\alpha V_w+\sqrt2\beta Z_w^0(\tau_i)}
 \right)^s
\right].
\]
The comparison constants are independent
of the Riemann partition.

For fixed $w$, both integrands are continuous. 
The Euler-product integrand is bounded, while the Gaussian integrand has an integrable supremum on $J$.
We may therefore pass to the limit in both expectations, which proves the result.
\end{proof}

\section{The critical tree and the localized Gaussian field}
\label{sec:critical-gaussian}

In this section we shall see that a binary tree captures the shared fluctuations of nearby points and provides the decay needed to sum over scales.
Indeed, the purpose of this section is to estimate the fractional moment of the Gaussian integral appearing in Proposition~\ref{prop:euler-to-gaussian-integral}:
\begin{equation}
 \E\left[
  \left(
   \int_J
   \e^{-2\alpha V_w+\sqrt2\beta Z_w^0(t)}
   \dd t
  \right)^s
 \right].
 \label{eq:def-fractional-moment-Gaussian-integral}
\end{equation}
To estimate \eqref{eq:def-fractional-moment-Gaussian-integral}, we compare the Gaussian field with a binary branching random walk. 
We first define the tree and record its basic properties in Section \ref{subsec:critical-tree-functional}. 
We then bound the Gaussian integral moment by a moment of the tree weights in Section \ref{subsec:localized-gaussian-occupation}. 
Finally, we estimate this tree moment as the depth increases in Section \ref{subsec:critical-tree-estimate}.

Throughout this section, we retain the parameters fixed in
\eqref{eq:parameter-definitions}. In particular,
\[
 \alpha=1-\frac2p,
 \qquad
 \beta=\alpha+1,
 \qquad
 s=\frac q2=\frac1\beta.
\]

\subsection{The binary tree and basic estimates}
\label{subsec:critical-tree-functional}

Let \(\mathscr T\) be the rooted binary tree.  For a vertex \(v\), we write \(\abs v\) for its generation, and let \(v_k\) be its ancestor at generation
\(k\), \(0\le k\le\abs v\).  For an edge \(e\) and a vertex \(v\), we write \(e\preceq v\) if \(e\) lies on the unique path from the root to \(v\).
We attach to every edge \(e\) an independent real
standard Gaussian \(\gamma_e\), and set
\begin{equation}
 \sigma=\sqrt{2\log2},
 \qquad
 T_v=\sigma\sum_{e\preceq v}\gamma_e,
 \qquad
 X_v=T_v-\Var(T_v)=T_v-2\abs v\log2.
 \label{eq:critical-tree-field}
\end{equation}
Let \(\mathscr F_n\) be the sigma-field generated by the Gaussian variables on the edges up to generation \(n\).
For \(n\ge0\), we define
\begin{equation}
 W_n=\sum_{\abs v=n}\e^{X_v},
 \qquad
 Y_n(\beta)
 =
 \left(
 \sum_{\abs v=n}\e^{\beta X_v}
 \right)^{1/\beta}.
 \label{eq:critical-tree-functionals}
\end{equation}
Here \(W_n\) sums the weights at generation \(n\), while \(Y_n(\beta)\) is their \(\ell^\beta\)-norm. 
The power \(\beta>1\) gives greater relative weight to the larger values of \(\e^{X_v}\).
If \(v'\) is a child of \(v\) and \(e=(v,v')\) is the connecting edge, then
\[
 \E\!\left(\e^{X_{v'}}\mid\mathscr F_{\abs v}\right)
 =
 \e^{X_v}\E\e^{\sigma\gamma_e-\sigma^2}
 =
 \frac12\e^{X_v}.
\]
Thus, \((W_n)_{n\ge0}\) is a nonnegative martingale with \(W_0=1\), and hence \(\E W_n=1\) for every \(n\ge0\).

We will use two basic results about the tree.
The first is the many-to-one formula, which expresses an expected weighted sum over leaves in terms of Brownian motion evaluated at integer times; see \cite[Theorem~1.1, p.~5]{Shi15}.

\begin{lemma}[Many-to-one formula]
\label{lem:critical-tree-many-to-one}
Let \(n\ge0\), and let
\(F\colon\R^n\to[0,\infty]\) be measurable.  Then, we have
\begin{equation}
 \E
 \sum_{\abs v=n}
 \e^{X_v}
 F(X_{v_1},\ldots,X_{v_n})
 =
 \E F(\sigma B_1,\ldots,\sigma B_n),
 \label{eq:critical-tree-many-to-one}
\end{equation}
where \(B\) is a standard real Brownian motion.
\end{lemma}

The next two bounds control paths that become too high and exceptionally large values of \(Y_n(\beta)\), respectively.

\begin{lemma}[Tree maximum and moment bounds]
\label{lem:critical-tree-elementary-bounds}
For every \(n\ge0\) and \(h\ge0\), one has
\begin{equation}
 \Prob\left\{
 \max_{\substack{v\in\mathscr T\\ \abs v\le n}}X_v>h
 \right\}
 \le\e^{-h}.
 \label{eq:critical-tree-maximum}
\end{equation}
Moreover, for every
\(
 0\le\eta\le\beta-1,
\)
\begin{equation}
 \E Y_n(\beta)^{1+\eta}
 \le
 \exp\left(\frac{\sigma^2\eta^2n}{2}\right).
 \label{eq:critical-tree-supermoment}
\end{equation}
\end{lemma}

\begin{proof}
Note that the process \((W_k)_{k\ge0}\) in
\eqref{eq:critical-tree-functionals} is a nonnegative mean-one martingale.
If \(X_v>h\) for some \(\abs v\le n\), then
\[
 W_{\abs v}\ge\e^{X_v}>\e^h.
\]
Thus
\[
\Bigg\{
 \max_{\substack{v\in\mathscr T\\ \abs v\le n}}X_v>h
 \Bigg\}
 \subseteq
 \left\{
 \max_{0\le k\le n}W_k>\e^h
 \right\}.
\]
Doob's maximal inequality gives
\[
 \Prob\left\{
 \max_{0\le k\le n}W_k>\e^h
 \right\}
 \le \e^{-h},
\]
which proves
\eqref{eq:critical-tree-maximum}.

Let \(0\le\eta\le\beta-1\).  Then, we have
\[
 Y_n(\beta)^{1+\eta}
 \le
 \sum_{\abs v=n}\e^{(1+\eta)X_v}.
\]
Using
\(X_v\sim\mathcal N(-\sigma^2n,\sigma^2n)\), together with
\(2^n=\e^{\sigma^2n/2}\), proves
\eqref{eq:critical-tree-supermoment}.
Explicitly,
\[
 \E\sum_{\abs v=n}\e^{(1+\eta)X_v}
 =
 2^n\exp\left(
 -(1+\eta)\sigma^2n
 +\frac{(1+\eta)^2\sigma^2n}{2}
 \right)
 =\exp\left(\frac{\sigma^2\eta^2n}{2}\right).
\]
\end{proof}

\subsection{Gaussian-to-tree comparison}
\label{subsec:localized-gaussian-occupation}

We use the binary tree to bound the fractional moment in
\eqref{eq:def-fractional-moment-Gaussian-integral}.

For the undeformed Gaussian field defined in
\eqref{eq:gaussian-deformation-field}, we write
\begin{equation}
 Z_w(t):=Z_w^0(t),
 \qquad
 V_w=\Var(Z_w(t))
 =\sum_{r\le\e^w}\frac1r.
 \label{eq:gaussian-prime-field}
\end{equation}
Independence of the prime variables gives
\begin{equation}
 \Cov\bigl(Z_w(t),Z_w(t')\bigr)
 =
 \sum_{r\le\e^w}
 \frac{\cos((t-t')\log r)}r.
 \label{eq:gaussian-prime-covariance}
\end{equation}

Let \(w=2^h\), where \(h\in\mathbb N_0\) and \(h\le H\),
and let \(J\subset\R\) be an interval of dyadic length
satisfying
\[
 0<|J|\le1,
 \qquad
 w|J|\ge1.
\]
We set
\[
 n=\log_2(w|J|)\in\mathbb N_0.
\]
Dividing \(J\) into intervals of length \(1/w\) gives
\(2^n\) intervals, which we associate with the leaves
of a depth-\(n\) binary tree.

We now define
\begin{equation}
 \mathfrak G_{N,w,J}
 =
 2^{-\alpha n}\frac1{|J|}
 \int_J
 \e^{\sqrt2\beta Z_w(t)-\beta V_w}\dd t.
 \label{eq:localized-gaussian-functional}
\end{equation}
The term \(-\beta V_w\) is chosen so that \(\E[\e^{s(\sqrt2\beta Z_w(t)-\beta V_w)}]=1\) for every \(t\), since \(s\beta=1\).
The functional \(\mathfrak G_{N,w,J}\) is related to the Gaussian integral in
\eqref{eq:def-fractional-moment-Gaussian-integral} by
\begin{equation}
 \int_J
 \e^{\sqrt2\beta Z_w(t)-2\alpha V_w}\dd t
 =
 |J|\,2^{\alpha n}
 \e^{(\beta-2\alpha)V_w}
 \mathfrak G_{N,w,J}.
 \label{eq:gaussian-nomalization}
\end{equation}
Then, it is sufficient to study the fractional moment of $\mathfrak G_{N,w,J}$. The next proposition bounds this moment by $\E Y_n(\beta)$.

\begin{proposition}[Gaussian-to-tree comparison]
\label{prop:gaussian-to-tree-comparison}
For every \(p>2\), there exists a constant \(C_p<\infty\),
depending only on \(p\), such that
\begin{equation}
 \E\mathfrak G_{N,w,J}^{\,s}
 \le C_p\,\E Y_n(\beta).
 \label{eq:gaussian-to-tree-comparison}
\end{equation}
\end{proposition}

\begin{proof}
We let \(t_J\) be the left endpoint of \(J\), and write
\[
 t=t_J+|J|x,
 \qquad
 0\le x\le1.
\]
Since \(\dd t=|J|\dd x\), the change of variables
cancels the factor \(1/|J|\) in
\(\mathfrak G_{N,w,J}\).
Thus the normalized average over \(J\) becomes an integral
over \([0,1]\).

We split the random field \(Z_w(t_J+|J|x)\) into two parts,
the coarse field \(Z^{\mathrm c}(x)\) and the fine field
\(Z^{\mathrm f}(x)\):
\[
 Z_w(t_J+|J|x)
 =
 Z^{\mathrm c}(x)+Z^{\mathrm f}(x),
\]
where
\begin{align}
 Z^{\mathrm c}(x)
 &=
 \sqrt2
 \sum_{\log r\le |J|^{-1}}
 r^{-1/2}
 \Ree(g_r\e^{\ii (t_J+|J|x)\log r}),
 \label{eq:coarse-prime-field}
 \\
 Z^{\mathrm f}(x)
 &=
 \sqrt2
 \sum_{|J|^{-1}<\log r\le w}
 r^{-1/2}
 \Ree(g_r\e^{\ii (t_J+|J|x)\log r}).
 \label{eq:fine-prime-field}
\end{align}
The two fields are independent because they use disjoint
sets of prime variables.
Let
\[
 V_{\mathrm c}=\Var(Z^{\mathrm c}(x)),
 \qquad
 V_{\mathrm f}=\Var(Z^{\mathrm f}(x)).
\]
These variances are independent of \(x\), and
\[
 V_w=V_{\mathrm c}+V_{\mathrm f}.
\]
To describe how the coarse field \(Z^{\mathrm c}(x)\) and the fine field \(Z^{\mathrm f}(x)\) vary with \(x\), write
\[
 \e^{\ii(t_J+|J|x)\log r}
 =
 \e^{\ii t_J\log r}\e^{\ii x|J|\log r}.
\]
Thus, the absolute difference between the angles at \(x\) and \(x'\) is
\[
 |J|\,|x-x'|\log r.
\]
If \(\log r\le|J|^{-1}\), then
\[
 \bigl||J|(x-x')\log r\bigr|
 \le |x-x'|\le1.
\]
Hence each prime term in the coarse field \(Z^{\mathrm c}(x)\) has an angle variation of at most one over the whole interval \([0,1]\).
For each prime term in the fine field \(Z^{\mathrm f}(x)\), we instead have
\[
 1<|J|\log r\le w|J|=2^n,
\]
so its angle can change more rapidly with \(x\).
However, whenever \(|x-x'|\le2^{-n}\),
\[
 \bigl||J|(x-x')\log r\bigr|
 \le 2^n|x-x'|\le1.
\]
Thus, on each dyadic interval of length \(2^{-n}\), every prime term in the fine field \(Z^{\mathrm f}(x)\) also has an angle variation of at most one.
By \eqref{eq:gaussian-prime-covariance}, each prime \(r\) contributes
\[
 \frac{\cos(|J|(x-x')\log r)}r
\]
to the covariance 
$\Cov(Z_w(t_J+|J|x),Z_w(t_J+|J|x'))$, 
whereas its contribution to the variance is \(1/r\).
As the following estimate shows, the covariance of the coarse field at any two points then differs from \(V_{\mathrm c}\) by at most an absolute constant.
This is the reason for splitting the prime sum at \(\log r=|J|^{-1}\).

For the coarse field \(Z^{\mathrm c}(x)\),
\eqref{eq:gaussian-prime-covariance} gives, for all \(x,x'\in[0,1]\),
\[
 0
 \le
 V_{\mathrm c}
 -
 \Cov\bigl(Z^{\mathrm c}(x),Z^{\mathrm c}(x')\bigr)
\le
 \frac{|J|^2|x-x'|^2}{2}
 \sum_{\log r\le |J|^{-1}}
 \frac{(\log r)^2}{r}.
\]
By Mertens' estimate, there exists an absolute constant \(C_0\) such that
\begin{equation}
 0
 \le
 V_{\mathrm c}
 -
 \Cov\bigl(Z^{\mathrm c}(x),Z^{\mathrm c}(x')\bigr)
 \le \frac{C_0}{2},
 \qquad x,x'\in[0,1].
 \label{eq:coarse-covariance-flatness}
\end{equation}
We approximate the integral in
\eqref{eq:localized-gaussian-functional} by Riemann sums.
Let \(\mathcal P\) be a partition of \([0,1]\), given by
\[
 0=a_0<a_1<\cdots<a_{N_{\mathcal P}}=1.
\]
For each \(1\le i\le N_{\mathcal P}\), we choose any point
\(x_i\in[a_{i-1},a_i]\), and set
\[
 \lambda_i=a_i-a_{i-1},\qquad
 |\mathcal P|=\max_i\lambda_i.
\]
Let
\[
 G_{\mathrm c}\sim\mathcal N(0,V_{\mathrm c}),
 \qquad
 B_{\mathrm c}\sim\mathcal N(0,C_0),
\]
be independent of each other and of all prime variables.
By \eqref{eq:coarse-covariance-flatness}, we have
\[
 \Cov\bigl(
  \sqrt2Z^{\mathrm c}(x_i)+B_{\mathrm c},
  \sqrt2Z^{\mathrm c}(x_j)+B_{\mathrm c}
 \bigr)
 \ge
 \Cov(\sqrt2G_{\mathrm c},\sqrt2G_{\mathrm c}).
\]
Condition on $Z^{\mathrm f}$, set
\begin{equation}
 \rho_i
 =
 2^{-\alpha n}\lambda_i
 \e^{\sqrt2\beta Z^{\mathrm f}(x_i)-\beta V_{\mathrm f}},
 \label{eq:coarse-conditional-weights}
\end{equation}
and write $\E_{\mathrm c}$ for expectation over $Z^{\mathrm c}$, $G_{\mathrm c}$ and $B_{\mathrm c}$.
Lemma~\ref{lem:concave-gaussian-comparison} therefore gives
\begin{equation}
\E_{\mathrm c}
 \left[
  \left(
   \e^{\beta B_{\mathrm c}-\beta^2C_0/2}
   \sum_i
   \rho_i
   \e^{\sqrt2\beta Z^{\mathrm c}(x_i)-\beta V_{\mathrm c}}
  \right)^s
 \right]
\le
 \E_{\mathrm c}
 \left[
  \left(
   \sum_i
   \rho_i
   \e^{\sqrt2\beta G_{\mathrm c}-\beta V_{\mathrm c}}
  \right)^s
 \right].
 \label{eq:coarse-nonwick-comparison}
\end{equation}
More explicitly, the variance-normalized comparison is
\[
\E_{\mathrm c}\left[
 \left(
 \sum_i\rho_i
 \e^{\beta(\sqrt2Z^{\mathrm c}(x_i)+B_{\mathrm c})
 -\beta^2V_{\mathrm c}-\beta^2C_0/2}
 \right)^s\right]
\le
 \E_{\mathrm c}\left[
 \left(
 \sum_i\rho_i
 \e^{\sqrt2\beta G_{\mathrm c}-\beta^2V_{\mathrm c}}
 \right)^s\right].
\]
Multiplying this inequality by
\(\e^{s(\beta^2-\beta)V_{\mathrm c}}\)
gives \eqref{eq:coarse-nonwick-comparison}.
Since \(s\beta=1\), we have
\[
 \E_{\mathrm c}
 \e^{s(\sqrt2\beta G_{\mathrm c}-\beta V_{\mathrm c})}
 =
 \e^{-V_{\mathrm c}}
 \E_{\mathrm c}\e^{\sqrt2G_{\mathrm c}}
 =1,
\qquad
 \E_{\mathrm c}
 \e^{s(\beta B_{\mathrm c}-\beta^2C_0/2)}
 =
 \e^{-(\beta-1)C_0/2}.
\]
Then,
\begin{equation}
 \E_{\mathrm c}
 \left[
  \left(
   \sum_i
   \rho_i
   \e^{\sqrt2\beta Z^{\mathrm c}(x_i)-\beta V_{\mathrm c}}
  \right)^s
 \right]
 \le
 C_\beta
 \left(\sum_i\rho_i\right)^s.
 \label{eq:coarse-field-comparison}
\end{equation}
Substituting \eqref{eq:coarse-conditional-weights} and then taking expectation over the fine field $Z^{\mathrm f}$, we obtain
\begin{equation}
\E
\left[\left(
 2^{-\alpha n}
 \sum_i\lambda_i
 \e^{
 \sqrt2\beta(Z^{\mathrm c}(x_i)+Z^{\mathrm f}(x_i))
 -
 \beta(V_{\mathrm c}+V_{\mathrm f})
 }
\right)^s\right]
\le
 C_\beta
 \E
\left[\left(
 2^{-\alpha n}
 \sum_i\lambda_i
 \e^{\sqrt2\beta Z^{\mathrm f}(x_i)-\beta V_{\mathrm f}}
 \right)^s\right].
 \label{eq:coarse-field-comparison-riemann}
\end{equation}
Let \(|\mathcal P|\) tend to zero. 
For this limit, keep \(w\) and \(J\) fixed.
Both Gaussian integrands are continuous because their
fields are finite sums of continuous functions.
Since \(\lambda_i\ge0\) and \(\sum_i\lambda_i=1\),
the \(s\)-th powers of the Riemann sums on both sides
are bounded by
\[
 \exp\left(
 C_\beta\sum_{r\le\e^w}\frac{|g_r|}{\sqrt r}
 +C_\beta V_w
 \right).
\]
This random variable has finite expectation, since the
sum contains finitely many Gaussian variables.
Dominated convergence therefore applies to both sides.
So, 
\begin{equation}
 \E
\left[\left(
 2^{-\alpha n}
 \int_0^1
 \e^{
 \sqrt2\beta(Z^{\mathrm c}(x)+Z^{\mathrm f}(x))
 -
 \beta(V_{\mathrm c}+V_{\mathrm f})
 }
 \dd x
 \right)^s\right]
\le
 C_\beta
 \E
 \left[\left(
 2^{-\alpha n}
 \int_0^1
 \e^{\sqrt2\beta Z^{\mathrm f}(x)-\beta V_{\mathrm f}}
 \dd x
 \right)^s\right].
 \label{eq:coarse-field-removed}
\end{equation}
That is, the coarse field may be removed at a cost depending only on \(\beta\).

We next compare the fine field $Z^{\mathrm f}(x)$ with the binary tree.  
By \eqref{eq:gaussian-prime-covariance},
\begin{equation}
 \Cov(Z^{\mathrm f}(x),Z^{\mathrm f}(x'))
 =
 \log\min\{2^n,|x-x'|^{-1}\}+O(1),
 \label{eq:fine-prime-covariance}
\end{equation}
uniformly for \(x,x'\in[0,1]\), with the convention
\(|x-x'|^{-1}=+\infty\) when \(x=x'\).
For \(x=x'\), Mertens' theorem gives
\[
\Var(Z^{\mathrm f}(x))=\sum_{|J|^{-1}<\log r\le w}\frac1r=\log (w|J|)+O(1)
=n\log 2+O(1).
\]
For \(x\ne x'\), we write
\[
 \begin{aligned}
 \Cov(Z^{\mathrm f}(x),Z^{\mathrm f}(x'))
=
 \sum_{\substack{|J|^{-1}<\log r\le w\\
                  |J|\,|x-x'|\log r\le1}}
 \frac{\cos(|J|\,|x-x'|\log r)}r
+
 \sum_{\substack{|J|^{-1}<\log r\le w\\
                  |J|\,|x-x'|\log r>1}}
 \frac{\cos(|J|\,|x-x'|\log r)}r.
 \end{aligned}
\]
In the first range, replacing the cosine by \(1\) introduces
an error of at most
\[
 \frac{|J|^2|x-x'|^2}{2}
 \sum_{\log r\le(|J|\,|x-x'|)^{-1}}
 \frac{(\log r)^2}{r}
 =O(1),
\]
where we used \(1-\cos u\le u^2/2\) and the Mertens bound
\(\sum_{\log r\le u}(\log r)^2/r=O(u^2)\) for \(u\ge1\).
Hence Mertens' theorem gives
\[
 \sum_{|J|^{-1}<\log r\le
 \min\{w,(|J|\,|x-x'|)^{-1}\}}
 \frac{\cos(|J|\,|x-x'|\log r)}r
 =
 \log\min\{2^n,|x-x'|^{-1}\}+O(1).
\]
The second range is empty unless
\(w|J|\,|x-x'|>1\).
When \(w|J|\,|x-x'|>1\), partial summation and the quantitative form of Mertens' theorem \cite[Section~6.2.1, Exercise~3(b), p.~182]
{MontgomeryVaughan06}
give
\[
 \sum_{(|J|\,|x-x'|)^{-1}<\log r\le w}
 \frac{\cos(|J|\,|x-x'|\log r)}r
 =
 \int_1^{w|J|\,|x-x'|}\frac{\cos u}{u}\,\dd u+O(1)
 =O(1).
\]
Combining the two ranges proves \eqref{eq:fine-prime-covariance} uniformly in \(w,J,x,x'\).

We then divide \([0,1]\) into the \(2^n\) half-open dyadic intervals of length \(2^{-n}\), assigning the endpoint \(1\) to the last interval.  
If the dyadic intervals containing \(x\) and \(x'\) have a common binary prefix of length \(k\), then
\[
 \abs{x-x'}\le2^{-k}.
\]
Since \(0\le k\le n\), this implies
\[
 \min\{2^n,|x-x'|^{-1}\}\ge2^k.
\]
Therefore, by \eqref{eq:fine-prime-covariance}, we have
\begin{equation}
 \Cov(\sqrt2Z^{\mathrm f}(x),\sqrt2Z^{\mathrm f}(x'))
 \ge
 2k\log2-C,
 \label{eq:fine-tree-covariance-lower-bound}
\end{equation}
while
\begin{equation}
 V_{\mathrm f}
 =
 n\log2+O(1).
 \label{eq:fine-variance-normalization}
\end{equation}
For each sample point \(x_i\), let \(v(i)\) be the corresponding leaf of the
depth-\(n\) binary tree.  
For vertices \(v,v'\), let \(v\wedge v'\) denote their last
common ancestor, whose depth \(\abs{v\wedge v'}\) is the
number of edges shared by their paths from the root.
Since
\[
 \Cov(T_v,T_{v'})
 =
 2\log2\,\abs{v\wedge v'},
\]
there is an absolute constant \(C_1\) such that, for an independent
\(B_{\mathrm f}\sim\mathcal N(0,C_1)\),
\begin{equation}
 \Cov(
 \sqrt2Z^{\mathrm f}(x_i)+B_{\mathrm f},
 \sqrt2Z^{\mathrm f}(x_j)+B_{\mathrm f}
 )
 \ge
 \Cov(T_{v(i)},T_{v(j)})
 \label{eq:fine-to-tree-covariance-order}
\end{equation}
for all sample points by \eqref{eq:fine-tree-covariance-lower-bound} and
\eqref{eq:fine-variance-normalization}.
Lemma~\ref{lem:concave-gaussian-comparison} yields
\begin{equation}
\E
\left[\left(
 2^{-\alpha n}
 \sum_i\lambda_i
 \e^{
 \beta(\sqrt2Z^{\mathrm f}(x_i)+B_{\mathrm f})
 -
 \beta^2V_{\mathrm f}
 -
 \beta^2C_1/2
 }
\right)^s\right]
\le
 \E
\left[\left(
 2^{-\alpha n}
 \sum_i\lambda_i
 \e^{\beta T_{v(i)}-\beta^2n\log2}
\right)^s\right].
 \label{eq:fine-to-tree-wick-comparison}
\end{equation}
Multiplying both sides by
\(\e^{s(\beta^2-\beta)V_{\mathrm f}}\), one has
\begin{equation}
\E
\left[\left(
 2^{-\alpha n}
 \sum_i\lambda_i
 \e^{
 \sqrt2\beta Z^{\mathrm f}(x_i)
 +
 \beta B_{\mathrm f}
 -
 \beta V_{\mathrm f}
 -
 \beta^2C_1/2
 }
\right)^s\right]
\le
 C_\beta
 \E
\left[\left(
 2^{-\alpha n}
 \sum_i\lambda_i
 \e^{\beta T_{v(i)}-\beta n\log2}
\right)^s\right].
 \label{eq:fine-to-tree-buffered-comparison}
\end{equation}
Here the deterministic terms on the tree side satisfy
\[
-\beta^2n\log2+(\beta^2-\beta)V_{\mathrm f}
=-\beta n\log2
 +(\beta^2-\beta)(V_{\mathrm f}-n\log2).
\]
Thus \(C_\beta\) absorbs the factor
\[
 \exp\left(
 s(\beta^2-\beta)(V_{\mathrm f}-n\log2)
 \right),
\]
which is uniformly bounded because
\(V_{\mathrm f}-n\log2=O(1)\).
Since \(B_{\mathrm f}\) is independent of the fine field and \(s\beta=1\), we obtain the following bound for every
Riemann sum:
\begin{equation}
\E
\left[\left(
 2^{-\alpha n}
 \sum_i\lambda_i
 \e^{\sqrt2\beta Z^{\mathrm f}(x_i)-\beta V_{\mathrm f}}
\right)^s\right]
\le
 C_\beta
 \E
\left[\left(
 2^{-\alpha n}
 \sum_i\lambda_i
 \e^{\beta T_{v(i)}-\beta n\log2}
\right)^s\right].
 \label{eq:fine-field-to-tree-riemann}
\end{equation}
Here the Gaussian factor removed from the left-hand
side has expectation
\[
 \E\e^{s\beta B_{\mathrm f}-s\beta^2C_1/2}
=
 \exp\left(
 \frac{(s\beta)^2-s\beta^2}{2}C_1
 \right)
=
 \e^{-(\beta-1)C_1/2}.
\]
For each leaf \(v\), let \(I_v\) be its dyadic interval
and put
\[
 a_v(\mathcal P)
 =
 \sum_{i:x_i\in I_v}\lambda_i.
\]
Only partition intervals meeting an endpoint of \(I_v\)
can affect the difference between this weight and the
length of \(I_v\).
Consequently,
\[
 |a_v(\mathcal P)-2^{-n}|
 \le2|\mathcal P|.
\]
Thus, for the fixed depth-\(n\) tree,
\[
2^{-\alpha n}
 \sum_i\lambda_i
 \e^{\beta T_{v(i)}-\beta n\log2}
=
 2^{-\alpha n}
 \sum_{|v|=n}a_v(\mathcal P)
 \e^{\beta T_v-\beta n\log2}
\longrightarrow
 \sum_{|v|=n}
 2^{-(1+\alpha)n}
 \e^{\beta T_v-\beta n\log2}
\]
almost surely.
Since the weights sum to one, the \(s\)-th power is
bounded by
\[
 \left(
 2^{-\alpha n}
 \max_{|v|=n}
 \e^{\beta T_v-\beta n\log2}
 \right)^s.
\]
This random variable is integrable for the fixed finite
tree.
On the fine-field side, continuity and the Gaussian
domination used above apply again.
Letting \(|\mathcal P|\) tend to zero and applying dominated convergence on both sides, we obtain
\begin{align}
 \E
 \left[\left(
 2^{-\alpha n}
 \int_0^1
 \e^{\sqrt2\beta Z^{\mathrm f}(x)-\beta V_{\mathrm f}}
 \dd x
 \right)^s\right]
\le
 C_\beta
 \E
 \left[\left(
 \sum_{\abs v=n}
 2^{-(1+\alpha)n}
 \e^{\beta T_v-\beta n\log2}
 \right)^{1/\beta}\right].
 \label{eq:fine-field-to-tree-limit}
\end{align}
Since \(1+\alpha=\beta\), the definition of $X_v$ in
\eqref{eq:critical-tree-field} gives
\begin{equation*}
 \sum_{\abs v=n}
 2^{-(1+\alpha)n}
 \e^{\beta T_v-\beta n\log2}
 =
 \sum_{\abs v=n}\e^{\beta X_v}.
\end{equation*}
Combining
\eqref{eq:coarse-field-removed} and
\eqref{eq:fine-field-to-tree-limit}
proves \eqref{eq:gaussian-to-tree-comparison}.
\end{proof}

\subsection{Tree moment and localized estimates}
\label{subsec:critical-tree-estimate}

By Proposition \ref{prop:gaussian-to-tree-comparison} above, it remains to bound \(\E Y_n(\beta)\).
We estimate this expectation by separating small, intermediate, and large values of \(Y_n(\beta)\).

\begin{theorem}[Critical tree moment estimate]
\label{thm:critical-tree}
For every fixed \(1<\beta<2\), there exists a constant
\(C_\beta<\infty\) such that
\begin{equation}
 \E Y_n(\beta)
 \le
 C_\beta
 \frac{\log^{5/2}(n+2)}{\sqrt{n+1}},
 \qquad n\ge0.
 \label{eq:critical-tree-estimate}
\end{equation}
\end{theorem}
Although there are \(2^n\) leaves, the expected value of \(Y_n(\beta)\) tends to zero as \(n\to\infty\) for each fixed
\(1<\beta<2\).

\begin{remark}[The endpoint \(\beta=1\)]
\label{rem:critical-tree-beta-one}
The restriction \(\beta>1\) in
Theorem~\ref{thm:critical-tree} is intrinsic to the estimate.  Indeed,
at \(\beta=1\),
\[
 Y_n(1)
 =
 \sum_{\abs v=n}\e^{X_v}
 =
 W_n,
 \qquad
 \E Y_n(1)=1,
\]
whereas
\[
 \frac{\log^{5/2}(n+2)}{\sqrt{n+1}}
 \longrightarrow0,
 \qquad n\to\infty.
\]
Thus the estimate \eqref{eq:critical-tree-estimate} fails at
\(\beta=1\).

Moreover, its constants cannot remain uniformly bounded as
\(\beta\downarrow1\).  For each fixed \(n\),
\[
 Y_n(\beta)\longrightarrow Y_n(1)=W_n,
 \qquad \beta\downarrow1,
\]
and \(Y_n(\beta)\le W_n\).  Hence, by dominated convergence,
\[
 \E Y_n(\beta)\longrightarrow1.
\]
A bound in \eqref{eq:critical-tree-estimate} with \(C_\beta\) uniformly
bounded as \(\beta\downarrow1\) would therefore imply
\[
 1
 \lesssim
 \frac{\log^{5/2}(n+2)}{\sqrt{n+1}}
\]
for every \(n\), a contradiction.

Under the parameter relation
\(\beta=1+\alpha\), the limit \(\beta\downarrow1\) corresponds to
\(p\downarrow2\).
Thus the breakdown of the critical-tree estimate at \(\beta=1\)
corresponds exactly to the endpoint \(p=2\), showing that this tree estimate is intrinsically a \(p>2\) argument.
\end{remark}

\begin{proof}[Proof of Theorem \ref{thm:critical-tree}]
The assertion is immediate for \(n=0\), since \(Y_0(\beta)=1\).
Assume henceforth that \(n\ge1\), and put
\begin{equation}
 R=\log(n+2).
 \label{eq:critical-tree-tail-parameters}
\end{equation}
We define
\begin{equation}
 j_0=\lfloor-2R\rfloor,
 \qquad
 j_1=\lceil \sigma\sqrt{20nR}\rceil.
 \label{eq:critical-tree-bin-parameters}
\end{equation}
For every integer \(j_0\le j<j_1\), put
\begin{equation}
 h_j=j+5R,
 \qquad
 v_j=j-4R/(\beta-1),
 \qquad
 j_+=\max\{j,0\}.
 \label{eq:critical-tree-barriers}
\end{equation}
Since \(j\ge j_0\), one has \(h_j>0\).
The index \(j\) specifies the level \(\e^j\) of \(Y_n(\beta)\). 
For each such level, \(h_j\) controls the values along a path, while \(v_j\) separates the terminal values into two parts.
The choices above have the following purposes.
\begin{center}
\begingroup
\small
\setlength{\tabcolsep}{5pt}
\renewcommand{\arraystretch}{1.5}
\begin{tabular}{
 @{}
 p{0.10\textwidth}
 p{0.28\textwidth}
 p{0.53\textwidth}
 @{}
}
\toprule
Symbol & What it controls & Why it is chosen this way \\
\midrule
\(R\)
&
Size of the logarithmic offsets.
&
Since \(R=\log(n+2)\), the factors
\(\e^{-5R}\) and \(\e^{-4R}\) become
\((n+2)^{-5}\) and \((n+2)^{-4}\).
\\
\(j_0\)
&
The smallest level treated in the sum.
&
The contribution from \(Y_n(\beta)<\e^{j_0}\)
is at most \(\e^{j_0}\le(n+2)^{-2}\).
\\
\(j_1\)
&
The point where the far-tail estimate begins.
&
Above \(\e^{j_1}\), we use the moment estimate
to control the remaining tail.
\\
\(h_j\)
&
The upper barrier for the values along each path.
&
Choosing \(h_j=j+5R\) gives the
barrier-crossing bound
\(\e^{-h_j}=\e^{-j}(n+2)^{-5}\).
\\
\(v_j\)
&
The cutoff between low and central terminal values.
&
Choosing \(v_j=j-4R/(\beta-1)\) gives
\(\e^{-\beta j+(\beta-1)v_j}
=\e^{-j}(n+2)^{-4}\)
in the estimate for the low contribution.
\\
\bottomrule
\end{tabular}
\endgroup
\end{center}
These choices also give
\[
 h_j-v_j
 =
 \left(5+\frac4{\beta-1}\right)R.
\]
Thus the central terminal values lie in an interval whose length is a constant depending only on \(\beta\) times \(\log(n+2)\).

Let \(\mathcal B_j\) be the event that an ancestral value exceeds \(h_j\):
\[
 \mathcal B_j
 =
 \left\{
 \max_{\substack{v\in\mathscr T\\\abs v\le n}}
 X_v>h_j
 \right\}.
\]
By Lemma~\ref{lem:critical-tree-elementary-bounds}, we have
\begin{equation}
 \Prob(\mathcal B_j)
 \le
 \e^{-h_j}
 =
 \e^{-j}(n+2)^{-5}.
 \label{eq:critical-tree-barrier-probability}
\end{equation}
At generation \(n\), define the low and central partition functions
\begin{equation*}
 W_{\mathrm{low},j}
 =
 \sum_{\substack{\abs v=n\\X_v\le v_j}}
 \e^{\beta X_v},
 \qquad
 W_{\mathrm{cen},j}
 =
 \sum_{\substack{\abs v=n\\
                   v_j<X_v\le h_j\\
                   X_{v_k}\le h_j,\ 1\le k\le n}}
 \e^{\beta X_v}.
\end{equation*}
On \(\mathcal B_j^c\), every leaf belongs to one of these two classes.
Hence, 
\begin{equation*}
\{Y_n(\beta)\ge\e^j\}
\subseteq
\mathcal B_j\cup\left\{W_{\mathrm{low},j}\ge\frac12\e^{\beta j}\right\}
\cup\left\{W_{\mathrm{cen},j}\ge\frac12\e^{\beta j}\right\}.
\end{equation*}
For the low part $W_{\mathrm{low},j}$, one has
\begin{equation}
 \E W_{\mathrm{low},j}
\le
 \e^{(\beta-1)v_j}
 \E\sum_{\abs v=n}\e^{X_v}
=
 \e^{(\beta-1)v_j},
 \label{eq:critical-tree-low-first-moment}
\end{equation}
because \(\E W_n=1\).  Therefore Markov's inequality gives
\begin{equation}
 \Prob\left\{
 W_{\mathrm{low},j}\ge\frac12\e^{\beta j}
 \right\}
\le
 2\e^{-j}(n+2)^{-4}.
 \label{eq:critical-tree-low-tail}
\end{equation}
The choice of \(v_j\) gives
\[
 -\beta j+(\beta-1)v_j=-j-4R,
\]
which accounts for the factor \((n+2)^{-4}\) in the probability
bound for \(W_{\mathrm{low},j}\).
For the central part $W_{\mathrm{cen},j}$, since \(1<\beta <2\),
\[
 W_{\mathrm{cen},j}^{1/\beta}
 \le
 \sum_{\substack{\abs v=n\\
                   v_j<X_v\le h_j\\
                   X_{v_k}\le h_j,\ 1\le k\le n}}
 \e^{X_v}.
\]
On the event
\(\{W_{\mathrm{cen},j}\ge\tfrac12\e^{\beta j}\}\), this implies
\[
 \sum_{\substack{\abs v=n\\
 v_j<X_v\le h_j\\
 X_{v_k}\le h_j,\ 1\le k\le n}}
 \e^{X_v}
 \ge 2^{-1/\beta}\e^j.
\]
Markov's inequality and Lemma~\ref{lem:critical-tree-many-to-one} therefore give
\[
 \begin{aligned}
 \Prob\left\{
 W_{\mathrm{cen},j}\ge\tfrac12\e^{\beta j}
 \right\}
 &\le
 2^{1/\beta}\e^{-j}
 \E\sum_{\abs v=n}\e^{X_v}
 \one_{\{
 X_{v_k}\le h_j,\ 1\le k\le n;\
 v_j<X_v\le h_j
 \}}\\
 &=
 2^{1/\beta}\e^{-j}
 \Prob\left\{
 \sigma B_k\le h_j,\ 1\le k\le n;\
 v_j<\sigma B_n\le h_j
 \right\}.
 \end{aligned}
\]
Consequently, we have
\begin{equation}
 \Prob\left\{
 W_{\mathrm{cen},j}\ge\frac12\e^{\beta j}
 \right\}
\le
 2\e^{-j}\pi_{n,j},
 \label{eq:critical-tree-central-tail}
\end{equation}
where 
\begin{equation}
 \pi_{n,j}
 :=
 \Prob\left\{
 \sigma B_k\le h_j,\ 1\le k\le n,\
 \sigma B_n\in[v_j,h_j]
 \right\}.
 \label{eq:critical-tree-ballot-probability}
\end{equation}
Combining
\eqref{eq:critical-tree-barrier-probability},
\eqref{eq:critical-tree-low-tail}, and
\eqref{eq:critical-tree-central-tail}, one has
\begin{equation}
 \e^{j+1}
 \Prob\{Y_n(\beta)\ge\e^j\}
 \le
 C\left\{
 (n+2)^{-4}+\pi_{n,j}
 \right\}.
 \label{eq:critical-tree-bin-tail}
\end{equation}
By a standard bound for centered random walks due to
A\"{\i}d\'ekon and Shi
\cite[Lemma~A.1]{AidekonShi10}, with 
\(
 S_k=-\sigma B_k,
 \)
\(k=0,1,\ldots,n,
\)
\(a=h_j\) and \(b=-v_j\), we obtain
\begin{align}
 \pi_{n,j}
\le
 \frac{C}{n^{3/2}}
 \bigl[(h_j+1)\wedge\sqrt n\bigr]
 \bigl[(h_j-v_j+1)^2\wedge n\bigr]
\le
 C_\beta
 \frac{R^2}{n^{3/2}}
 \bigl[(j_++R)\wedge\sqrt n\bigr].
 \label{eq:critical-tree-ballot-bound}
\end{align}
Here \(a\wedge b=\min\{a,b\}\), and the last
inequality follows from
\[
 h_j=j+5R,
 \qquad
 h_j-v_j=
 \left(5+\frac4{\beta-1}\right)R.
\]
We now sum the tail probabilities. Since
\(Y_n(\beta)\ge0\), division into exponential
intervals gives
\begin{equation}
 \E Y_n(\beta)
 \le
 \e^{j_0}
 +
 \sum_{j=j_0}^{j_1-1}
 \e^{j+1}\Prob\{Y_n(\beta)\ge\e^j\}
+
 \E\left[
 Y_n(\beta);
 Y_n(\beta)\ge\e^{j_1}
 \right]
 \label{eq:critical-tree-tail-decomposition}
\end{equation}
Summing \eqref{eq:critical-tree-bin-tail} over
\(j_0\le j<j_1\) and using
\eqref{eq:critical-tree-ballot-bound} together with
\[
 (j_++R)\wedge\sqrt n\le\sqrt n,
\]
we obtain
\begin{equation*}
\sum_{j=j_0}^{j_1-1}
 \e^{j+1}
 \Prob\{Y_n(\beta)\ge\e^j\}
\le
 C_\beta(j_1-j_0)(n+2)^{-4}
 +
 C_\beta\frac{R^2}{n}
 (j_1-j_0).
\end{equation*}
Since
\[
 j_1-j_0\le C(\sqrt{nR}+R),
 \qquad R=\log(n+2),
\]
we obtain
\begin{equation}
 \sum_{j=j_0}^{j_1-1}
 \e^{j+1}
 \Prob\{Y_n(\beta)\ge\e^j\}
 \le
 C_\beta\frac{R^{5/2}}{\sqrt n}.
 \label{eq:critical-tree-middle-sum}
\end{equation}

It remains to control the far tail $\E\left[
 Y_n(\beta);
 Y_n(\beta)\ge\e^{j_1}
 \right]$. We set
\[
 \eta_n=\frac{\sqrt{20nR}}{\sigma n}.
\]
This value of \(\eta_n\) minimizes
\[
 -\eta\sigma\sqrt{20nR}
 +\frac{\sigma^2\eta^2n}{2},
 \qquad \eta\ge0.
\]
Since
\[
 \eta_n=\frac{\sqrt{20R}}{\sigma\sqrt n}
 \longrightarrow0,
\]
there exists an integer \(n_0(\beta)\) such that
\begin{equation}
 \eta_n\le\beta-1,
 \qquad n\ge n_0(\beta).
 \label{eq:critical-tree-eta-range}
\end{equation}
For these values of \(n\), we may apply
\eqref{eq:critical-tree-supermoment}. Since
\(j_1\ge\sigma\sqrt{20nR}\), we obtain
\begin{align}
 \E\left[
 Y_n(\beta);
 Y_n(\beta)\ge\e^{j_1}
 \right]
 &\le
 \e^{-\eta_nj_1}
 \E Y_n(\beta)^{1+\eta_n}
 \notag\\
 &\le
 \exp\left(
 -\eta_n\sigma\sqrt{20nR}
 +\frac{\sigma^2\eta_n^2n}{2}
 \right)
=
 (n+2)^{-10}.
 \label{eq:critical-tree-far-tail}
\end{align}
Combining
\eqref{eq:critical-tree-tail-decomposition},
\eqref{eq:critical-tree-middle-sum}, and
\eqref{eq:critical-tree-far-tail}, 
we prove \eqref{eq:critical-tree-estimate}
for all \(n\ge n_0(\beta)\).

For \(1\le n<n_0(\beta)\), we have
\[
 \E Y_n(\beta)\le\E W_n=1.
\]
The first inequality follows from the fact that the
\(\ell^\beta\)-norm of the weights \(\e^{X_v}\)
is at most their \(\ell^1\)-norm.
Increasing \(C_\beta\) covers these finitely many values of \(n\).
This completes the proof.
\end{proof}

We can now return to the fractional moment of the Gaussian integral in \eqref{eq:def-fractional-moment-Gaussian-integral}.

\begin{proposition}[Localized Gaussian moment]
\label{prop:localized-gaussian-occupation}
Let $J\subset\R$ be an interval of dyadic length satisfying
\[
 0<|J|\le1,
 \qquad
 w|J|\ge1.
\]
For every $p>2$, there exists a constant
$C_p<\infty$, depending only on $p$, such that
\begin{equation}
 \E\left[
  \left(
   \int_J
   \e^{-2\alpha V_w+\sqrt2\beta Z_w(t)}
   \dd t
  \right)^s
 \right]
 \le
 C_p w^s |J|
 \frac{\log^{5/2}(\log_2(w|J|)+2)}{\sqrt{\log_2(w|J|)+1}}.
 \label{eq:localized-gaussian-occupation}
\end{equation}
\end{proposition}

\begin{proof}
By Proposition \ref{prop:gaussian-to-tree-comparison} and Theorem~\ref{thm:critical-tree}, we have
\[
 \E\mathfrak G_{N,w,J}^{\,s}
 \le
 C_p\frac{\log^{5/2}(n+2)}{\sqrt{n+1}}.
\]
By \eqref{eq:gaussian-nomalization},
it remains to estimate the factor
\[
 \left(
 |J|\,2^{\alpha n}
 \e^{(\beta-2\alpha)V_w}
 \right)^s.
\]
Since $2^n=w|J|$,
$\beta-2\alpha=\delta$,
$\alpha+\delta=1$, and $\beta s=1$,
Mertens' theorem gives
\[
 \left(
 |J|\,2^{\alpha n}\e^{\delta V_w}
 \right)^s
 \asymp_p
 w^s|J|.
\]
The comparison between \(\e^{V_w}\) and \(w\) follows from
Mertens' product estimate and the bounded difference
\[
 \log M_{\e^w}-V_w
 =
 \sum_{r\le\e^w}
 \left(-\log(1-r^{-1})-\frac1r\right)
 =O(1).
\]
Indeed,
\[
 -\log(1-r^{-1})-\frac1r=O(r^{-2}),
\]
and \(\sum_r r^{-2}<\infty\).
Consequently,
\[
 \e^{V_w}\asymp M_{\e^w}\asymp w,
 \qquad w\ge1,
\]
with constants independent of \(w\).
Combining these estimates proves
\eqref{eq:localized-gaussian-occupation}.
\end{proof}

Combining the Gaussian moment bound above with the Euler-to-Gaussian comparison from Section~\ref{sec:replacement}, we obtain the following Euler-product moments on short intervals.

\begin{proposition}[Localized Euler-product moment]
\label{prop:localized-euler-moment}
Let $h\in\{0,\ldots,H\}$, $w=2^h$, and $0\le\vartheta\le1/32$.
Let $J\subset\R$ be an interval of dyadic length satisfying
\[
 0<|J|\le1,\qquad w|J|\ge1.
\]
Then, one has
\begin{equation}
 \E\left[
  \left(
   \int_J |f_h^\vartheta(t)|^2\dd t
  \right)^s
 \right]
 \le
 C_p w^s|J|
 \frac{\log^{5/2}(\log_2(w|J|)+2)}{\sqrt{\log_2(w|J|)+1}}.
 \label{eq:localized-euler-moment}
\end{equation}
The constant $C_p$ depends only on $p$.
\end{proposition}

\begin{proof}
The result follows directly from
Proposition~\ref{prop:euler-to-gaussian-integral}
and Proposition~\ref{prop:localized-gaussian-occupation}.
\end{proof}

\section{Euler moments and a dyadic square-function estimate}
\label{sec:euler-occupation}

The key idea in this section is to group contributions across scales before estimating them,
so that their combined Euler-product moments can be controlled
by a square function of the coefficients.
Let \(z=(z_h(j))_{0\le h\le H,\,0\le j<N}\) be a fixed family of
complex coefficients. For \(0\le\vartheta\le1/32\), put
\begin{equation}
 \mathfrak P_y^\vartheta(z)
 =\E\!\left[\left(
   \sum_{h=0}^{H}\frac{N}{2^h}M_y^{-\delta}
   \sum_{j=0}^{N-1}\abs{z_h(j)}^2\abs{f_h^\vartheta(t_j)}^2
 \right)^s\right].
 \label{eq:occupation-functional}
\end{equation}
The squared Euler-product moduli $\abs{f_h^\vartheta(t_j)}^2$ act as random weights on the coefficient $|z_h(j)|^2$.
We take the $s$-th moment after summing over all levels $h$ and sample times $t_j$.

Recall that \(Q_h(x)\) is the unique level-\(h\) dyadic block in $\mathcal{D}_h$ containing \(x\). Define the conical square function $S_{\mathrm{cone}}z$ as
\begin{equation}
 S_{\mathrm{cone}}z(x)^2
 =\sum_{h=0}^{H}\frac{1}{\abs{Q_h(x)}}
   \sum_{j\in Q_h(x)}\abs{z_h(j)}^2,
 \qquad x\in\ZN.
 \label{eq:cone-square-function}
\end{equation}
At each level we average the squared coefficients over the block
containing \(x\), and then add these averages over the levels.

The purpose of this section is to bound the fractional moment of a weighted sum of squared Euler-product moduli $\mathfrak P_y^\vartheta(z)$ in terms of a dyadic square function of the coefficients; see Proposition \ref{prop:common-shift-occupation}.
This estimate will be used in Proposition~\ref{prop:resonant-column} to control the resonant terms.

\begin{proposition}[Euler-product moment bound]
\label{prop:common-shift-occupation}
For every fixed coefficient family \(z\) and every \(0\le\vartheta\le1/32\), we have
\begin{equation}
 \mathfrak P_y^\vartheta(z)\le C_p\sum_{x\in\ZN}S_{\mathrm{cone}}z(x)^q.
 \label{eq:common-shift-occupation}
\end{equation}
The constant depends only on \(p\).
\end{proposition}

We first group the dyadic blocks in
$\bigcup_{h=0}^H\mathcal D_h$ by the size of $S_{\mathrm{cone}}z$ in Section~\ref{subsec:tent-allocation}.
We then use Proposition~\ref{prop:localized-euler-moment}
to bound each group and sum the resulting estimates in
Section~\ref{subsec:common-shift-estimate}.

\subsection{Grouping the dyadic blocks}
\label{subsec:tent-allocation}

Now we shall group the dyadic blocks according to the size of the coefficient
square function, so that each group can be controlled and the
local estimates can be summed across scales.
For \(Q\in\mathcal D_h\) and \(k\in\Z\), put
\begin{equation}
 E_Q=\sum_{j\in Q}\abs{z_h(j)}^2,
 \label{eq:tent-cell-energy}
\end{equation}
\begin{equation}
 \Omega_k=\{x\in\ZN:S_{\mathrm{cone}}z(x)>2^k\}.
 \label{eq:cone-level-sets}
\end{equation}
Every cell \(Q\) with \(E_Q>0\) is assigned as follows.  First choose the largest integer \(k\) such that
\[
 \abs{Q\cap\Omega_k}>\frac12\abs Q.
\]
Then choose the maximal dyadic ancestor \(I\supseteq Q\) satisfying
\[
 \abs{I\cap\Omega_k}>\frac12\abs I.
\]
We call \(I\) the top interval.
Put all blocks with the same \(k\) and the same top interval \(I\) in one group \(\mathcal A_{k,I}\).
We also refer to these groups as tents.
The level $k$ will control the energy in this collection, while maximality makes the selected intervals $I$ disjoint for a fixed level.  
These are the two facts needed to sum the local estimates.

\begin{lemma}[Deterministic tent decomposition]
\label{lem:tent-allocation}
The groups \(\mathcal A_{k,I}\) partition the blocks with \(E_Q>0\).
For each group, one has
\begin{equation}
 \sum_{Q\in\mathcal A_{k,I}}E_Q\le 8\cdot 2^{2k}\abs{I}.
 \label{eq:tent-energy-bound}
\end{equation}
For fixed \(k\), the distinct top intervals are pairwise disjoint, and
\begin{equation}
 \sum_{k,I}2^{kq}\abs{I}\le C_q\sum_{x\in\ZN}S_{\mathrm{cone}}z(x)^q,
 \label{eq:tent-layer-cake}
\end{equation}
where the sum on the left is over the nonempty groups.
\end{lemma}

\begin{proof}
If \(E_Q>0\), then \(S_{\mathrm{cone}}z(x)^2\ge E_Q/\abs{Q}>0\) on
\(Q\). Since the grid is finite, the largest admissible \(k\) exists.
The dyadic intervals containing \(Q\) form a finite chain, and \(Q\)
itself is admissible, so the top is unique. This proves the partition
claim.

For \(Q\in\mathcal A_{k,I}\), by maximality of the level \(k\),
\[
 \abs{Q\cap\Omega_{k+1}}
 \le
 \frac12\abs Q.
\]
Thus, we have
\[
 \abs{Q\setminus\Omega_{k+1}}
 \ge
 \frac12\abs Q,
\]
	which implies
\begin{equation}
 E_Q
 \le
 2
 \sum_{x\in Q\setminus\Omega_{k+1}}
 \frac{E_Q}{\abs Q}.
 \label{eq:tent-energy-spread}
\end{equation}
For a fixed \(x\), at most one block at each level contains \(x\).
Thus, for \(x\in I\setminus\Omega_{k+1}\),
\[
 \sum_{\substack{Q\in\mathcal A_{k,I}\\x\in Q}}
 \frac{E_Q}{\abs Q}
 =
\sum_{\substack{0\le h\le H\\Q_h(x)\in\mathcal A_{k,I}}}
\frac{E_{Q_h(x)}}{|Q_h(x)|}
 \le
 \sum_{h=0}^{H}\frac{E_{Q_h(x)}}{\abs{Q_h(x)}}
=S_{\mathrm{cone}}z(x)^2
 \le
 2^{2k+2}.
\]
Each term on the left is one of the terms defining
\(S_{\mathrm{cone}}z(x)^2\).
The second inequality above follows from
\(x\notin\Omega_{k+1}\).
Summing \eqref{eq:tent-energy-spread} over the group $\mathcal A_{k,I}$ therefore gives
\[
 \sum_{Q\in\mathcal A_{k,I}}E_Q
 \le
 2\sum_{x\in I\setminus\Omega_{k+1}}
 S_{\mathrm{cone}}z(x)^2
 \le
 8\cdot2^{2k}\abs I,
\]
which is \eqref{eq:tent-energy-bound}.
The factor \(2\) comes from
\(\abs{Q\setminus\Omega_{k+1}}\ge\abs Q/2\),
and the factor \(4\) comes from
\(2^{2k+2}=4\cdot2^{2k}\).
Thus the constant \(8=2\cdot4\) is absolute.

For fixed \(k\), two distinct tops cannot overlap: otherwise one would be a larger admissible interval containing the other.  
Since each such interval satisfies
\[
 \abs{I\cap\Omega_k}>\frac12\abs I,
\]  
we have
\[
 \sum_I\abs I
 \le
 2\abs{\Omega_k}.
\]
Therefore,
\[
 \sum_{k,I}2^{kq}\abs I
 \le
 2\sum_{k\in\Z}2^{kq}\abs{\Omega_k}
 \le
 C_q\sum_{x\in\ZN}S_{\mathrm{cone}}z(x)^q.
\]
For the last inequality above, all summands are nonnegative,
so we may change the order of summation:
\[
 \sum_{k\in\Z}2^{kq}\abs{\Omega_k}
 =
 \sum_{x\in\ZN}
 \sum_{\{k:2^k<S_{\mathrm{cone}}z(x)\}}2^{kq}.
\]
If \(S_{\mathrm{cone}}z(x)=0\), the inner sum is zero.
Otherwise, let \(k_0\) be the largest integer such that
\(2^{k_0}<S_{\mathrm{cone}}z(x)\).
Then
\[
 \sum_{\{k:2^k<S_{\mathrm{cone}}z(x)\}}2^{kq}
=
 \sum_{\nu=0}^{\infty}2^{(k_0-\nu)q}
=
 \frac{2^{k_0q}}{1-2^{-q}}
\le
 \frac{S_{\mathrm{cone}}z(x)^q}{1-2^{-q}}.
\]
Thus one may take \(C_q=2/(1-2^{-q})\) in this estimate,
independently of \(N\), \(H\), and \(z\).
\end{proof}

\subsection{Proof of the moment bound}
\label{subsec:common-shift-estimate}

Now the idea is that 
averaging over a common shift lets us control each group's
sampled contributions by local Euler-product moments,
before summing over all groups.

\begin{proof}[Proof of Proposition~\ref{prop:common-shift-occupation}]
For a group \(\mathcal A_{k,I}\), we write
\begin{equation}
 I=\{j_0,\ldots,j_0+m-1\}\in\mathcal D_{h_I},
 \qquad m=\abs{I}=N2^{-h_I}.
 \label{eq:euler-localization-scales}
\end{equation}
For \(h_I\le h\le H\), let
\begin{equation}
 U_h=\bigcup_{Q\in\mathcal A_{k,I}\cap\mathcal D_h}Q,
 \qquad E_h=\sum_{j\in U_h}\abs{z_h(j)}^2.
 \label{eq:tent-level-sets-and-energies}
\end{equation}
If no cell of depth \(h\) belongs to \(\mathcal A_{k,I}\), set \(U_h=\varnothing\) and \(E_h=0\).
For \(0\le u\le 2^{-h_I}\), we put
\begin{equation}
 V_{k,I}(u)=\sum_{h=h_I}^{H}\frac{N}{2^h}M_y^{-\delta}
   \sum_{j\in U_h}\abs{z_h(j)}^2\abs{f_h^\vartheta(t_j+u)}^2.
 \label{eq:tent-shifted-functional}
\end{equation}
At \(u=0\) these are the contributions of the groups to the sum inside
\eqref{eq:occupation-functional}. Since \(0<s<1\), subadditivity gives
\begin{equation}
 \mathfrak P_y^\vartheta(z)
 =\E\!\left[\left(\sum_{k,I}V_{k,I}(0)\right)^s\right]
 \le\sum_{k,I}\E[V_{k,I}(0)^s].
 \label{eq:occupation-separate-tents}
\end{equation}
It suffices to consider $\E[V_{k,I}(0)^s]$.

Fix a group \(\mathcal A_{k,I}\) for the following estimate.
Its sets \(U_h\) and energies \(E_h\) depend only on the
fixed coefficients \(z\), so they remain unchanged when
the prime variables are rotated.
By \eqref{eq:deformed-source-flow}, for every fixed \(u\), the single rotation
\[
 (\zeta_r)_{r\le y}
 \longmapsto
 (\e^{\ii u\log r}\zeta_r)_{r\le y}
\]
changes every \(f_h^\vartheta(t_j)\) into
\(f_h^\vartheta(t_j+u)\).
Since \(U_u\) preserves product Haar measure, we obtain that the random variable \(V_{k,I}(u)\) has the same distribution as \(V_{k,I}(0)\), for every fixed \(u\).
Consequently, 
\begin{equation}
 \E[V_{k,I}(u)^s]=\E[V_{k,I}(0)^s].
 \label{eq:tent-common-flow-invariance}
\end{equation}
Averaging over $u\in[0,2^{-h_I}]$ gives
\[
\E[V_{k,I}(0)^s]
=2^{h_I}\int_0^{2^{-h_I}}\E[V_{k,I}(u)^s]\dd u
=\E\!\left[2^{h_I}\int_0^{2^{-h_I}}V_{k,I}(u)^s\dd u\right].
\]
Since 
\(
 2^{h_I}\int_0^{2^{-h_I}}1\dd u=1,
\)
concavity of \(x\mapsto x^s\) gives
\[
 2^{h_I}\int_0^{2^{-h_I}}V_{k,I}(u)^s\dd u
 \le
 \left(
 2^{h_I}\int_0^{2^{-h_I}}V_{k,I}(u)\dd u
 \right)^s.
\]
Taking expectation over the prime variables gives
\begin{equation}
 \E[V_{k,I}(0)^s]
\le\E\!\left[\left(2^{h_I}\int_0^{2^{-h_I}}V_{k,I}(u)\dd u\right)^s\right].
 \label{eq:tent-shift-concavity}
\end{equation}
For every \(j\in I\), we have \(0\le t_j-t_{j_0}\le(m-1)/N<2^{-h_I}\).
Recall that $t_j=t_0+j/N$, for all $0\le j<N$.
Thus,
\begin{equation}
 [t_j,t_j+2^{-h_I}]\subseteq J_I^+,
 \label{eq:euler-localization-intervals}
\end{equation}
where
\(J_I^+=[t_{j_0},t_{j_0}+2^{1-h_I}]\).
For each fixed \(h\), the interval inclusion above gives
\[
 \sum_{j\in U_h}|z_h(j)|^2
 \int_0^{2^{-h_I}}
 |f_h^\vartheta(t_j+u)|^2\dd u
 =
 \sum_{j\in U_h}|z_h(j)|^2
 \int_{t_j}^{t_j+2^{-h_I}}
 |f_h^\vartheta(t)|^2\dd t
 \le
 E_h\int_{J_I^+}|f_h^\vartheta(t)|^2\dd t.
\]
Thus the same interval \(J_I^+\) controls the contribution from every \(j\in U_h\).
Integrating the terms of
\eqref{eq:tent-shifted-functional} over the shift $u$ and using the inclusion above,
we obtain
\begin{equation}
2^{h_I}\int_0^{2^{-h_I}}V_{k,I}(u)\dd u
\le\sum_{h=h_I}^{H}E_h\frac{N}{2^h}M_y^{-\delta}2^{h_I}
\int_{J_I^+}\abs{f_h^\vartheta(t)}^2\dd t.
\label{eq:tent-localized-reduction}
\end{equation}

We now estimate the integral on the right for a fixed \(h\). We claim that
\begin{equation}
 \E\!\left[\left(
     2^{h_I}\frac{N}{2^h}M_y^{-\delta}m^{-\alpha}
     \int_{J_I^+}\abs{f_h^\vartheta(t)}^2\dd t
 \right)^s\right]
 \le C_p\frac{\log^{5/2}(h-h_I+2)}{\sqrt{h-h_I+1}}.
 \label{eq:euler-localized-occupation}
\end{equation}
Split \(J_I^+\) at its midpoint into two intervals of length \(2^{-h_I}\).
Let \(J\) be either half. With \(w=2^h\), we have
\(0<\abs{J}\le1\) and \(w\abs{J}=2^{h-h_I}\ge1\).
Proposition \ref{prop:localized-euler-moment} therefore gives
\[
 \E\!\left[\left(\int_J\abs{f_h^\vartheta(t)}^2\dd t\right)^s\right]
 \le C_p\,2^{hs-h_I}\frac{\log^{5/2}(h-h_I+2)}{\sqrt{h-h_I+1}}.
\]
The bound is uniform in the position of \(J\). 
Note that
\[
 \left(2^{h_I}\frac{N}{2^h}M_y^{-\delta}m^{-\alpha}\right)^s 2^{hs-h_I}
 =\left(\frac{N}{M_y}\right)^{\delta s}\asymp_p1.
\]
Here \(\alpha+\delta=1\), \(\alpha s=1-s\), and \(M_y\asymp N\) for \(y=\e^N\).
Applying the estimate to both halves and using \((a+b)^s\le a^s+b^s\) proves \eqref{eq:euler-localized-occupation}.

By \eqref{eq:tent-shift-concavity}, \eqref{eq:tent-localized-reduction},
and subadditivity across the levels, we have
\begin{align*}
 \E[V_{k,I}(0)^s]
 &\le \sum_{h=h_I}^{H} E_h^s\,
    \E\!\left[\left(2^{h_I}\frac{N}{2^h}M_y^{-\delta}
         \int_{J_I^+}\abs{f_h^\vartheta(t)}^2\dd t\right)^s\right]\\
 &\le C_p m^{1-s}\sum_{h=h_I}^{H}
    E_h^s\frac{\log^{5/2}(h-h_I+2)}{\sqrt{h-h_I+1}}.
\end{align*}
In fact, since \(m\) and \(E_h\) are deterministic, we have
\begin{align*}
 \E\!\left[\left(
 E_h\,2^{h_I}\frac{N}{2^h}M_y^{-\delta}
 \int_{J_I^+}\abs{f_h^\vartheta(t)}^2\dd t
 \right)^s\right]
 &=
 (m^\alpha E_h)^s
 \E\!\left[\left(
 2^{h_I}\frac{N}{2^h}M_y^{-\delta}m^{-\alpha}
 \int_{J_I^+}\abs{f_h^\vartheta(t)}^2\dd t
 \right)^s\right]\\
 &=
 m^{1-s}E_h^s
 \E\!\left[\left(
 2^{h_I}\frac{N}{2^h}M_y^{-\delta}m^{-\alpha}
 \int_{J_I^+}\abs{f_h^\vartheta(t)}^2\dd t
 \right)^s\right]\\
 &\le
 C_p m^{1-s}E_h^s
 \frac{\log^{5/2}(h-h_I+2)}{\sqrt{h-h_I+1}}.
\end{align*}
The first equality uses \(m^\alpha m^{-\alpha}=1\),
the second uses \(\alpha s=1-s\), and the last inequality
is \eqref{eq:euler-localized-occupation}.
Then, H\"older's inequality with exponents \(1/s\) and \(1/(1-s)\) gives
\begin{align*}
 \sum_{h=h_I}^{H} E_h^s\frac{\log^{5/2}(h-h_I+2)}{\sqrt{h-h_I+1}}
 &\le \left(\sum_{h=h_I}^{H} E_h\right)^s
  \left(\sum_{n=0}^{\infty}
   \frac{\log^{5/(2-2s)}(n+2)}{(n+1)^{1/(2-2s)}}\right)^{1-s}.
\end{align*}
The last series converges because \(s>1/2\), so \(1/(2-2s)>1\).
H\"older's inequality here applies only to the deterministic energy powers \(E_h^s\) and the decay factors.
The Euler-product moment at each level has already been estimated, so no independence between levels is used.
The infinite series depends only on \(s\), hence only on \(p\), and bounds every finite depth sum uniformly.
Therefore,
\begin{equation}
 \E[V_{k,I}(0)^s]\le C_p m^{1-s}\left(\sum_{h=h_I}^{H}E_h\right)^s.
 \label{eq:tent-holder-estimate}
\end{equation}
The blocks in \(U_h\) are disjoint at a fixed level. Thus
\[
\sum_{h=h_I}^{H}E_h=\sum_{Q\in\mathcal A_{k,I}}E_Q\le8\cdot2^{2k}m
\]
by \eqref{eq:tent-energy-bound}. Substituting in \eqref{eq:tent-holder-estimate}
and using \(2s=q\), we obtain
\begin{equation}
 \E[V_{k,I}(0)^s]\le C_p\,2^{2ks}m=C_p\,2^{kq}\abs{I}.
 \label{eq:tent-final-bound}
\end{equation}
The powers simplify as
\[
 m^{1-s}(2^{2k}m)^s
 =
 m^{1-s+s}2^{2ks}
 =
 m\,2^{kq}.
\]
Summing \eqref{eq:tent-final-bound} over the groups and applying
\eqref{eq:occupation-separate-tents} and \eqref{eq:tent-layer-cake} gives
\[
 \mathfrak P_y^\vartheta(z)\le C_p\sum_{k,I}2^{kq}\abs{I}
     \le C_p\sum_{x\in\ZN}S_{\mathrm{cone}}z(x)^q.
\]
\end{proof}

\section{Finite cyclic harmonic analysis}
\label{sec:cyclic-analysis}

In this section we establish estimates for coefficient vectors on \(\ZN\),
with constants independent of \(N\).
Section~\ref{subsec:cyclic-multiplier} gives multiplier bounds
for frequency cutoffs. The kernel estimates for smooth frequency
cutoffs in Section~\ref{subsec:cyclic-kernels} are used in
Section~\ref{subsec:cyclic-conical} to compare conical and
pointwise square functions.
Section~\ref{subsec:boundary-traces} bounds weighted coefficient
sums and Poisson averages over dyadic scales for the nonresonant
terms. Section~\ref{subsec:coefficient-square-functions} bounds
the \(\ell^q\)-norm of the square function of the dyadic coefficient
pieces by the original coefficient norm.

These estimates will be used in the frequency decomposition of
Section~\ref{sec:corrected-column}.
We write the frequency coordinates as
\[
 \kappa\in\Lambda_N,
 \qquad
 \theta_\kappa=\frac{2\pi\kappa}{N},
 \qquad
 \xi_\kappa=2\pi\kappa.
\]
Our discrete Fourier transform includes the factor \(N^{-1}\),
while its inverse is unnormalized and convolution is taken with
counting measure. With these conventions, the factor \(N^{-1}\)
in the definition of \(K_\sigma\) ensures that convolution by
\(K_\sigma\) has Fourier multiplier \(\sigma\).

The following table shows how the estimates in this section
are obtained.

\begin{center}
\small
\renewcommand{\arraystretch}{1.15}
\begin{tabularx}{\linewidth}{@{}
  >{\raggedright\arraybackslash}X
  >{\raggedright\arraybackslash}X@{}}
\toprule
Tools & Estimates obtained \\
\midrule
Decay and \(\ell^2\) bounds for smooth cutoff kernels
& Comparison of the conical and pointwise square functions \\
Summation over distance rings and estimates for the cyclic
Poisson kernel
& Bounds for the weighted coefficient sums and Poisson averages
  in Section~\ref{subsec:boundary-traces}, summarized in
  \eqref{eq:boundary-trace-package} \\
The cyclic multiplier bound in Lemma~\ref{lem:cyclic-multiplier}
and Khintchine's inequality
& The coefficient square-function bound
  \eqref{eq:coefficient-square-function} \\
\bottomrule
\end{tabularx}
\end{center}

The constants may depend on the exponents and smooth cutoff functions
specified in each estimate, but are independent of \(N\) and \(H\).

\subsection{A multiplier estimate on the finite cyclic group}
\label{subsec:cyclic-multiplier}

We next establish a multiplier bound in terms of the size
and dyadic variation of its symbol, uniformly in \(N\).
With the Fourier convention of
Section~\ref{subsec:cyclic-model}, we consider a symbol
\[
m\colon\Lambda_N\longrightarrow\C.
\]
We treat the Nyquist frequency \(\kappa=-N/2\) separately
in the proof below.
We next extend the remaining values of \(m\) by zero to \(\Z\):
\begin{equation}
 \widetilde m(k)
 =
 \begin{cases}
 m(k),&-N/2+1\le k\le N/2-1,\\
 0,&\text{otherwise},
 \end{cases}
 \qquad k\in\Z,
 \label{eq:nyquist-free-zero-extension}
\end{equation}
and put
\[
 \Delta\widetilde m(k)
 =
 \widetilde m(k+1)-\widetilde m(k).
\]
We set
\begin{equation}
 \mathfrak M_N(m)
 \coloneqq
 \norm m_{\ell^\infty(\Lambda_N)}
+
 \sup_{\nu\ge0}
 \left\{
 \sum_{k=2^\nu}^{2^{\nu+1}-1}
 \abs{\Delta\widetilde m(k)}
 +
 \sum_{k=-2^{\nu+1}}^{-2^\nu-1}
 \abs{\Delta\widetilde m(k)}
 \right\}.
 \label{eq:cyclic-marcinkiewicz-norm}
\end{equation}
The quantity \(\mathfrak M_N(m)\) controls the symbol and its variation
on dyadic frequency blocks.

\medskip 
We next bound the operator norm of the Fourier multiplier
with symbol \(m\) on \(\ell^u(\ZN)\) in terms of
\(\mathfrak M_N(m)\).

\begin{lemma}
[Cyclic Marcinkiewicz multiplier estimate]
\label{lem:cyclic-multiplier}
Let \(1<u<\infty\). Then
\begin{equation}
 \norm{
 \mathcal F_N^{-1}(m\widehat c)
 }_{\ell^u(\ZN)}
 \le
 C_u\mathfrak M_N(m)
 \norm c_{\ell^u(\ZN)}
 \label{eq:cyclic-multiplier-bound}
\end{equation}
for every \(c\in\C^{\ZN}\). The constant \(C_u\) depends only on $u$.
\end{lemma}

\begin{proof}
We first separate the contribution at \(\kappa=-N/2\),
using the projection
\begin{equation}
 (P_{\mathrm{Ny}}c)_j
 =
 (-1)^j
 \frac1N
 \sum_{\ell=0}^{N-1}(-1)^\ell c_\ell.
 \label{eq:nyquist-projection}
\end{equation}
By H\"older's inequality, we have 
\[
 \abs{
 \frac1N
 \sum_{\ell=0}^{N-1}(-1)^\ell c_\ell
 }
 \le
 N^{-1/u}\norm c_{\ell^u(\ZN)}.
\]
Consequently,
\begin{equation}
 \norm{P_{\mathrm{Ny}}c}_{\ell^u(\ZN)}
 \le
 \norm c_{\ell^u(\ZN)}.
 \label{eq:nyquist-projection-contraction}
\end{equation}

We next set  \(c^\circ=(I-P_{\mathrm{Ny}})c\). Then
\[
 \widehat{c^\circ}(-N/2)=0,
 \qquad
 \norm{c^\circ}_u\le2\norm c_u.
\]
We associate with \(c^\circ\) the trigonometric polynomial
\begin{equation}
 F_{c^\circ}(x)
 =
 \sum_{\kappa=-N/2+1}^{N/2-1}
 \widehat c(\kappa)\e^{\mathrm{i} \kappa x}.
 \label{eq:cyclic-to-circle-polynomial}
\end{equation}
By \eqref{eq:cyclic-dft}, we have 
\(
 F_{c^\circ}(-2\pi j/N)=c^\circ_j.
\)
The polynomial in \eqref{eq:cyclic-to-circle-polynomial}
has degree at most \(n=N/2-1\), and \(N>2n\). The normalized Marcinkiewicz--Zygmund sampling inequality
therefore gives
\begin{equation}
 C_u^{-1}\norm F_{L^u(\T)}
 \le
 \left(
 \frac1N
 \sum_{j=0}^{N-1}
 \abs{F(2\pi j/N)}^u
 \right)^{1/u}
 \le
 C_u\norm F_{L^u(\T)}
 \label{eq:marcinkiewicz-zygmund-sampling}
\end{equation}
for every complex trigonometric polynomial \(F\) of degree at most
\(n\) and every integer \(N>2n\), with \(C_u\) depending only on
\(u\); see \cite[Chapter~I, Theorems~1--2, p.~132;
Section~10, p.~146]{MarcinkiewiczZygmund37}.
For \(n=0\), the comparison is immediate.
The negative sampling grid is a permutation of the positive one.

We next construct the corresponding multiplier on the circle. For
\(n\ge1\), we put
\[
 a_n^+=\widetilde m(n),
 \qquad
 a_n^-=\widetilde m(-n),
 \qquad
 a_0^+=a_0^-=0.
\]
Then
\[
 \sup_n\abs{a_n^\pm}\le\norm m_\infty,
\]
and
\begin{align}
 \sum_{n=2^\nu}^{2^{\nu+1}-1}
 \abs{a_{n+1}^+-a_n^+}
 &=
 \sum_{k=2^\nu}^{2^{\nu+1}-1}
 \abs{\Delta\widetilde m(k)},
 \label{eq:positive-one-sided-variation}
 \\
 \sum_{n=2^\nu}^{2^{\nu+1}-1}
 \abs{a_{n+1}^--a_n^-}
 &=
 \sum_{k=-2^{\nu+1}}^{-2^\nu-1}
 \abs{\Delta\widetilde m(k)}.
 \label{eq:negative-one-sided-variation}
\end{align}
The initial differences satisfy
\[
 \abs{a_1^\pm-a_0^\pm}\le\norm m_\infty.
\]
For \(H\ge2\), the two outer jumps are
\[
 \Delta\widetilde m(N/2-1)=-m(N/2-1),
 \qquad
 \Delta\widetilde m(-N/2)=m(-N/2+1).
\]
They occur in the positive block
\(2^{H-2}\le k\le2^{H-1}-1\) and the negative block
\(-2^{H-1}\le k\le-2^{H-2}-1\), respectively.
We handle the zero frequency by the projection onto constants;
the Nyquist frequency has already been separated above.
For \(H=1\), only the zero frequency remains after that separation.

Moreover, by the definition in
\eqref{eq:cyclic-marcinkiewicz-norm},
we have, for every \(\nu\ge0\),
\begin{equation}
 \sum_{n=2^\nu}^{2^{\nu+1}}
 \abs{a_{n+1}^\pm-a_n^\pm}
 \le
 2\mathfrak M_N(m).
\end{equation}
We separate the terms with \(2^\nu\le n<2^{\nu+1}\)
from the last term, at \(n=2^{\nu+1}\).
The first group is controlled by one dyadic variation sum,
while the last difference is controlled by the adjacent block.
For the negative sequence, this last difference is
\[
 \widetilde m(-2^{\nu+1}-1)
 -\widetilde m(-2^{\nu+1}),
\]
whose absolute value is
\(\abs{\Delta\widetilde m(-2^{\nu+1}-1)}\).
The index \(-2^{\nu+1}-1\) lies in the negative dyadic block
\(-2^{\nu+2}\le k\le-2^{\nu+1}-1\).
We therefore use two adjacent dyadic variation sums,
each bounded by \(\mathfrak M_N(m)\), giving the factor \(2\)
in the estimate above.

We apply the one-variable Marcinkiewicz
multiplier theorem \cite[Theorem~1, p.~79]{Marcinkiewicz39}
(see also
\cite[Vol.~II, Chapter~XV, \S2, p.~232]{Zygmund59})
with the real and imaginary parts of \((a_n^\pm)_{n\ge1}\).
Combining these estimates gives Fourier multipliers \(A_\pm\)
on \(L^u(\T)\),  defined by
\[
 \widehat{A_\pm F}(k)
 =
 a_{|k|}^\pm\widehat F(k),
 \qquad k\in\Z,
\]
with
\begin{equation}
 \norm{A_\pm}_{L^u(\T)\to L^u(\T)}
 \le
 C_u\mathfrak M_N(m).
 \label{eq:one-sided-multiplier-bounds}
\end{equation}

We write \(P_0\) and \(P_+\) for the projections retaining
only the Fourier coefficients at \(k=0\) and \(k>0\),
respectively, and set \((\mathcal RF)(x)=F(-x)\).
The periodic conjugate-function theorem of M.~Riesz implies
that \(P_+\) is bounded on \(L^u(\T)\); see
\cite[Theorem~II, p.~225; Theorem~III,
pp.~226--227]{Riesz28}; see also
\cite[Chapter~4, \S4.1, Theorem~4.1, p.~53]{Duren70}.
We then write the multiplier with symbol \(\widetilde m\) as
\begin{equation}
 T_{\widetilde m}
 =
 \widetilde m(0)P_0
 +
 A_+P_+
 +
 \mathcal RA_-P_+\mathcal R.
 \label{eq:positive-negative-zero-decomposition}
\end{equation}

We verify the decomposition separately at positive, negative,
and zero frequencies.
For \(k>0\), only \(A_+P_+\) contributes, multiplying the
Fourier coefficient at \(k\) by \(a_k^+=\widetilde m(k)\).
For \(k<0\), we consider the term
\(\mathcal RA_-P_+\mathcal R\).
The first application of \(\mathcal R\) changes the frequency
from \(k\) to \(-k>0\), so \(P_+\) retains it.
The operator \(A_-\) then multiplies the coefficient by
\(a_{-k}^-=\widetilde m(k)\), and the second application
of \(\mathcal R\) restores the frequency \(k\).
At \(k=0\), both terms containing \(P_+\) vanish;
the remaining term \(\widetilde m(0)P_0\) multiplies the
zero Fourier coefficient by \(\widetilde m(0)\).
Hence
\begin{equation}
 \norm{T_{\widetilde m}}_{L^u(\T)\to L^u(\T)}
 \le
 C_u\mathfrak M_N(m).
 \label{eq:circle-zero-extended-multiplier}
\end{equation}
For \(0\le j<N\), we have 
\[
 \bigl(T_{\widetilde m}F_{c^\circ}\bigr)(-2\pi j/N)
 =
 \bigl(\mathcal F_N^{-1}
       (m\widehat{c^\circ})\bigr)_j.
\]
We verify the identity above by observing that both
sides equal
\[
 \sum_{\kappa=-N/2+1}^{N/2-1}
 \widetilde m(\kappa)\widehat c(\kappa)
 \e^{-\mathrm{i}j\theta_\kappa}.
\]
For every frequency in this sum, we have
\(m(\kappa)=\widetilde m(\kappa)\) and
\(\widehat{c^\circ}(\kappa)=\widehat c(\kappa)\).
The term at \(\kappa=-N/2\) vanishes because
\(\widehat{c^\circ}(-N/2)=0\).

Both \(F_{c^\circ}\) and \(T_{\widetilde m}F_{c^\circ}\)
have degree at most \(n=N/2-1\), since multiplication of
the Fourier coefficients by \(\widetilde m\) does not introduce
new frequencies. Since \(N>2n\), we can apply the sampling
inequality \eqref{eq:marcinkiewicz-zygmund-sampling}
to both functions.
Using \(F_{c^\circ}(-2\pi j/N)=c^\circ_j\), we obtain
\[
 \left(\frac1N\sum_{j=0}^{N-1}
 \abs{F_{c^\circ}(-2\pi j/N)}^u\right)^{1/u}
 =N^{-1/u}\norm{c^\circ}_{\ell^u(\ZN)}.
\]
We now apply the sampling inequality to
\(T_{\widetilde m}F_{c^\circ}\), then the multiplier estimate
\eqref{eq:circle-zero-extended-multiplier}, and finally
the sampling inequality to \(F_{c^\circ}\):
\begin{align*}
 N^{-1/u}\norm{\mathcal F_N^{-1}(m\widehat{c^\circ})}_u
 &\le C_u\norm{T_{\widetilde m}F_{c^\circ}}_{L^u(\T)}\\
 &\le C_u\mathfrak M_N(m)\norm{F_{c^\circ}}_{L^u(\T)}\\
 &\le C_u\mathfrak M_N(m)N^{-1/u}\norm{c^\circ}_u.
\end{align*}
Multiplying both sides by \(N^{1/u}\) removes the common
factor \(N^{-1/u}\) and gives the required
\(\ell^u(\ZN)\) bound for \(c^\circ\).

Combining \eqref{eq:marcinkiewicz-zygmund-sampling} with
\eqref{eq:circle-zero-extended-multiplier}, we obtain
\[
 \norm{\mathcal F_N^{-1}(m\widehat{c^\circ})}_u
 \le
 C_u\mathfrak M_N(m)\norm{c^\circ}_u
 \le
 C_u\mathfrak M_N(m)\norm c_u.
\]
Finally,
\[
 \mathcal F_N^{-1}(m\widehat c)
 =
 \mathcal F_N^{-1}(m\widehat{c^\circ})
 +
 m(-N/2)P_{\mathrm{Ny}}c.
\]
We next bound the Nyquist contribution in this decomposition:
\[
 \norm{m(-N/2)P_{\mathrm{Ny}}c}_u
 \le\norm m_\infty\norm c_u
 \le\mathfrak M_N(m)\norm c_u.
\]
The second term is controlled by
\eqref{eq:nyquist-projection-contraction}. This proves
\eqref{eq:cyclic-multiplier-bound}.
\end{proof}

The following corollary gives bounds for frequency interval
projections and smooth shell multipliers, uniformly in \(N\).
The projection bounds will be used for the positive and nonpositive
frequency decomposition in
Section~\ref{subsec:corrected-column-definition}.

\begin{corollary}
[Cyclic interval and half-line projections]
\label{cor:cyclic-interval-projections}
Every cyclic frequency interval projection, every positive or negative
half-line projection, and every smooth shell multiplier whose zero
extension has uniformly bounded dyadic variation is bounded on
\(\ell^u(\ZN)\), \(1<u<\infty\), with a bound independent of \(N\).
\end{corollary}

\begin{proof}
The zero extension of an interval indicator has a uniformly
bounded number of jumps on each positive or negative
dyadic frequency block.
If we truncate the symbol to a frequency interval, we also
include the endpoint jumps, each bounded by the supremum
norm of the symbol.
The conclusion follows from Lemma~\ref{lem:cyclic-multiplier}.
\end{proof}

\subsection{Kernels of smooth frequency cutoffs}
\label{subsec:cyclic-kernels}

We next estimate the convolution kernels of smooth frequency
cutoffs and their first discrete differences.
A cutoff at frequency scale \(b\) has a kernel localized on
the spatial scale \(N/b\).
These bounds will be used in the proof of
Theorem~\ref{thm:cyclic-conical-square-function}.

In the conical square-function estimate below, we use \(\ell^2\)
bounds for the cutoff kernels and their first discrete differences
to control sums of squares over the dyadic intervals \(Q_h(x)\).
Now we fix a real-valued even function
\[
 \psi\in C_c^\infty((-2,2)),
 \qquad
 0\le\psi\le1,
 \qquad
 \psi=1\ \text{on }[-1,1].
\]

\begin{lemma}
[Cyclic kernel estimates]
\label{lem:cyclic-kernel-estimates}
Let
\[
 1\le b\le\frac N8,
 \qquad
 \sigma_b(\kappa)=\psi(\kappa/b),
 \qquad
 \kappa\in\Lambda_N,
\]
and let
\[
 K_b(j)
 =
 \frac1N
 \sum_{\kappa\in\Lambda_N}
 \sigma_b(\kappa)\e^{-\mathrm{i} j\theta_\kappa}
\]
be the convolution kernel of the multiplier with symbol
\(\sigma_b\). With
\[
 \Delta K_b(j)=K_b(j+1)-K_b(j),
\]
one has, for every fixed integer \(M\ge2\),
\begin{align}
 \abs{K_b(j)}
 &\le
 C_M\frac bN
 \left(
 1+\frac{b\rho_j}{N}
 \right)^{-M},
 \label{eq:cyclic-kernel-size}
 \\
 \abs{\Delta K_b(j)}
 &\le
 C_M\left(\frac bN\right)^2
 \left(
 1+\frac{b\rho_j}{N}
 \right)^{-M}.
 \label{eq:cyclic-kernel-difference}
\end{align}
Moreover, we have 
\begin{align}
 \sum_{j\in\ZN}\abs{K_b(j)}^2
 &\le
 C\frac bN,
 \label{eq:cyclic-kernel-L2}
 \\
 \sum_{j\in\ZN}\abs{\Delta K_b(j)}^2
 &\le
 C\frac{b^3}{N^3}.
 \label{eq:cyclic-kernel-difference-L2}
\end{align}
All constants are independent of \(b\) and \(N\).
\end{lemma}

\begin{proof}
Since $\operatorname{supp}\psi\subset(-2,2)$ and $b\le N/8$, 
the support of $\sigma_b$ is contained in
\[
    \{|\kappa|<2b\}\subseteq\{|\kappa|<N/4\}.
\]
We extend \(\sigma_b\) by zero to \(\Z\).
For every integer $r\ge0$, smoothness of $\psi$ gives
\begin{equation}
    \sum_{\kappa\in\mathbb Z}
    |\Delta^r\sigma_b(\kappa)|
    \le C_r b^{1-r}.
    \label{eq:symbol-finite-difference-bound}
\end{equation}

For \(r\ge1\), we apply the fundamental theorem of calculus
repeatedly to obtain the integral formula
\[
 \Delta^r\sigma_b(\kappa)
 =b^{-r}\int_{[0,1]^r}
 \psi^{(r)}\!\left(
 \frac{\kappa+t_1+\cdots+t_r}{b}
 \right)\dd t_1\cdots\dd t_r.
\]
The integrand can be nonzero only if
\(-2b-r<\kappa<2b\).
Since \(b\ge1\), this interval contains at most \(C_r b\)
integers. Consequently,
\[
 \sum_{\kappa\in\Z}\abs{\Delta^r\sigma_b(\kappa)}
 \le C_r b\,b^{-r}=C_r b^{1-r}.
\]
For \(r=0\), the support and boundedness of \(\psi\) give the
same estimate. Here we take all differences on the zero-extended sequence on
\(\Z\).

We put \(\omega_j=e^{-2\pi\mathrm{i}j/N}\).
Repeated summation by parts yields, for $j\not\equiv0\pmod N$,
\[
    (1-\omega_j)^M K_b(j)
    =
    \frac1N
    \sum_{\kappa\in\mathbb Z}
    \Delta^M\sigma_b(\kappa)\omega_j^\kappa,
\]
up to a factor of modulus one. Since \(\abs{1-\omega_j}\asymp\rho_j/N\), the finite-difference
estimate \eqref{eq:symbol-finite-difference-bound} gives
\[
    |K_b(j)|
    \le
    C_M N^{M-1}b^{1-M}\rho_j^{-M}
\]
when $b\rho_j/N\ge1$. On the other hand,
\[
    |K_b(j)|
    \le
    \frac1N\sum_\kappa|\sigma_b(\kappa)|
    \le C\frac bN.
\]
Together, these two bounds give \eqref{eq:cyclic-kernel-size}.

We use the forward difference
\(\Delta c(\kappa)=c(\kappa+1)-c(\kappa)\).
For every finitely supported sequence \(c\), shifting the
summation index gives
\[
 (1-\omega_j)\sum_{\kappa\in\Z}c(\kappa)\omega_j^\kappa
 =\omega_j\sum_{\kappa\in\Z}\Delta c(\kappa)\omega_j^\kappa.
\]
We now iterate this identity to obtain
\[
 (1-\omega_j)^M K_b(j)
 =\frac{\omega_j^M}{N}
 \sum_{\kappa\in\Z}\Delta^M\sigma_b(\kappa)\omega_j^\kappa.
\]
Since \(\abs{\omega_j}=1\), taking absolute values gives,
for \(j\not\equiv0\pmod N\),
\[
 \abs{K_b(j)}
 \le\frac1N\abs{1-\omega_j}^{-M}
 \sum_{\kappa\in\Z}\abs{\Delta^M\sigma_b(\kappa)}.
\]
Using
\(\sum_{\kappa\in\Z}\abs{\Delta^M\sigma_b(\kappa)}
 \le C_M b^{1-M}\)
and
\(\abs{1-\omega_j}\asymp\rho_j/N\),
we obtain decay in the cyclic distance \(\rho_j\).
Combining this estimate with the direct bound
\(\abs{K_b(j)}\le Cb/N\) gives the stated size bound
for all \(j\), including \(j=0\).

We next estimate the first discrete difference, which we write as
\[
    \Delta K_b(j)
    =
    \frac1N\sum_\kappa
    \tau_b(\kappa)e^{-\mathrm{i}j\theta_\kappa},
    \qquad
    \tau_b(\kappa)
    \coloneqq
    (e^{-\mathrm{i}\theta_\kappa}-1)\sigma_b(\kappa).
\]
On the support of \(\psi(\kappa/b)\), we have
\(\abs{e^{-\mathrm{i}\theta_\kappa}-1}\le Cb/N\).
The discrete Leibniz rule and \eqref{eq:symbol-finite-difference-bound} therefore imply
\[
    \sum_{\kappa\in\mathbb Z}
    |\Delta^r\tau_b(\kappa)|
    \le
    C_r\frac{b^{2-r}}N.
\]
The same summation-by-parts argument, now applied to
\(\tau_b\), gives \eqref{eq:cyclic-kernel-difference}.

We now expand \(\Delta^r\tau_b\) by the discrete Leibniz
rule and check the bound \(C_r b^{2-r}/N\) term by term.
We group the terms according to the number \(k\) of differences
applied to \(\e^{-2\pi\mathrm{i}\kappa/N}-1\).
For \(k\ge1\),
\[
 \abs{\Delta^k
 \bigl(\e^{-2\pi\mathrm{i}\kappa/N}-1\bigr)}
 =\abs{\e^{-2\pi\mathrm{i}/N}-1}^k
 \le C_kN^{-k}.
\]
In a term with \(k\) differences applied to the exponential
factor, the remaining factor is a translate of
\(\Delta^{r-k}\sigma_b\).
For fixed \(r\), its support is contained in
\(\abs\kappa\le2b+r\).
When \(k=0\), the exponential factor itself is bounded on
this support by \(C_r b/N\), since \(b\ge1\).
The resulting \(\ell^1(\Z)\)-bound is
\[
 C_r\frac bN b^{1-r}=C_r\frac{b^{2-r}}N.
\]
For \(1\le k\le r\), the corresponding \(\ell^1(\Z)\)-bounds are
\[
 C_{r,k}N^{-k}b^{1-r+k}
 =C_{r,k}\frac{b^{2-r}}N
 \left(\frac bN\right)^{k-1}
 \le C_{r,k}\frac{b^{2-r}}N.
\]
Translation preserves the \(\ell^1\)-norm, so these bounds also
apply to the shifted factors in the discrete Leibniz formula.

Finally, Parseval's identity gives
\[
    \sum_{j\in\mathbb Z_N}|K_b(j)|^2
    =
    \frac1N\sum_{\kappa\in\Lambda_N}
    |\sigma_b(\kappa)|^2
    \le C\frac bN,
\]
and
\[
\begin{aligned}
    \sum_{j\in\mathbb Z_N}|\Delta K_b(j)|^2
    =
    \frac1N\sum_{\kappa\in\Lambda_N}
    |e^{-\mathrm{i}\theta_\kappa}-1|^2|\sigma_b(\kappa)|^2
    \le C\frac{b^3}{N^3}.
\end{aligned}
\]

In the Parseval estimates above, the support of \(\sigma_b\)
contains at most \(Cb\) integer frequencies, and
\(\abs{\sigma_b}\le1\).
For \(\Delta K_b\), the Fourier multiplier is
\((\e^{-\mathrm{i}\theta_\kappa}-1)\sigma_b(\kappa)\).
On the support of \(\sigma_b\), the additional factor satisfies
\(\abs{\e^{-\mathrm{i}\theta_\kappa}-1}\le Cb/N\).
We therefore obtain
\begin{align*}
 \frac1N\sum_{\kappa\in\Lambda_N}\abs{\sigma_b(\kappa)}^2
 &\le C\frac1N\,b=C\frac bN,\\
 \frac1N\sum_{\kappa\in\Lambda_N}
 \abs{\e^{-\mathrm{i}\theta_\kappa}-1}^2
 \abs{\sigma_b(\kappa)}^2
 &\le C\frac1N\,b\left(\frac bN\right)^2
 =C\frac{b^3}{N^3}.
\end{align*}
\end{proof}

\subsection{The cyclic conical square function}
\label{subsec:cyclic-conical}

The following theorem bounds the \(\ell^q\)-norm of the conical
square function by that of the pointwise square function.
This comparison provides the conical square-function bound in
the proof of Proposition~\ref{prop:resonant-column}.

We recall that \(Q_h(x)\) is the dyadic cell of length
\(m_h=N2^{-h}\) containing \(x\), and write the conical square function as
\[
 S_{\mathrm{cone}}z(x)^2
 =\sum_{h=0}^{H}\frac1{m_h}
   \sum_{j\in Q_h(x)}\abs{z_h(j)}^2.
\]
In the theorem below, \(C_*\) controls the frequency range
of each \(z_h\): only frequencies satisfying
\(\abs\kappa\le C_*2^h\) may occur.
The bound grows at most polynomially in \(1+C_*\).

\begin{theorem}
[Conical square-function estimate]
\label{thm:cyclic-conical-square-function}
Let \(C_*\ge0\), and let \(z=(z_h)_{0\le h\le H}\), with
\(z_h\in\C^{\ZN}\), satisfy
\begin{equation}
 \supp\widehat z_h
 \subseteq
 \{\kappa\in\Lambda_N:\abs\kappa\le C_*2^h\}.
 \label{eq:conical-bandwidth-assumption}
\end{equation}
For \(1<q<2\), we set
\begin{equation}
 \Gamma_q
 =
 \max\left\{
 \frac1q,
 \frac32\left(\frac2q-1\right)
 \right\}.
 \label{eq:gamma-q}
\end{equation}
Then the conical square function defined in
\eqref{eq:cone-square-function} satisfies
\begin{equation}
 \norm{S_{\mathrm{cone}}z}_{\ell^q(\ZN)}
 \le
 C_q(1+C_*)^{\Gamma_q}
 \norm{
 \left(
 \sum_{h=0}^{H}\abs{z_h}^2
 \right)^{1/2}
 }_{\ell^q(\ZN)},
 \label{eq:cyclic-conical-estimate}
\end{equation}
where \(C_q\) depends only on \(q\).
\end{theorem}

\begin{proof}
We put \(A=1+C_*\) and split
\[
 \mathcal H_{\mathrm{lo}}=\{h:m_h\ge16A\},
 \qquad
 \mathcal H_{\mathrm{top}}=\{h:m_h<16A\}.
\]
We write \(S_{\mathrm{cone}}^{\mathrm{lo}}z\) and
\(S_{\mathrm{cone}}^{\mathrm{top}}z\) for the square functions
obtained by restricting the \(h\)-sum in the definition of
\(S_{\mathrm{cone}}z\) to \(\mathcal H_{\mathrm{lo}}\) and
\(\mathcal H_{\mathrm{top}}\), respectively.

For \(h\in\mathcal H_{\mathrm{lo}}\), the symbol
\(\sigma_{A2^h}\) from
Lemma~\ref{lem:cyclic-kernel-estimates}
equals one on \(\supp\widehat z_h\),
\(C_*2^h\le A2^h\) and \(2A2^h\le N/8\).
Hence
\begin{equation}
 K_{A2^h}*z_h=z_h,
 \qquad h\in\mathcal H_{\mathrm{lo}}.
 \label{eq:conical-reproducing-identity}
\end{equation}

To estimate \(S_{\mathrm{cone}}^{\mathrm{lo}}z\), we introduce
a linear operator. We set
\[
 E=\ell^2(\mathcal H_{\mathrm{lo}}),
 \qquad
 F=\ell^2\bigl(
 \{(h,j):h\in\mathcal H_{\mathrm{lo}},\ j\in\ZN\}
 \bigr),
\]
and define
\begin{equation}
 (Tg)_{h,j}(x)
 =
 m_h^{-1/2}\one_{\{j\in Q_h(x)\}}(K_{A2^h}*g_h)(j).
 \label{eq:conical-T-definition}
\end{equation}

For each fixed \(h\) and \(j\), exactly \(m_h\) values of \(x\)
satisfy \(j\in Q_h(x)\): they are the points in the same dyadic
cell as \(j\). Thus
\[
 \sum_{x\in\ZN}m_h^{-1}\one_{\{j\in Q_h(x)\}}=1.
\]
Parseval's identity and \(0\le\sigma_{A2^h}\le1\) give
\begin{equation}
\begin{aligned}
 \norm{Tg}_{\ell^2(\ZN;F)}^2
 &=
 \sum_{h\in\mathcal H_{\mathrm{lo}}}
 \norm{K_{A2^h}*g_h}_{\ell^2(\ZN)}^2\\
 &\le
 \sum_{h\in\mathcal H_{\mathrm{lo}}}
 \norm{g_h}_{\ell^2(\ZN)}^2
 =
 \norm g_{\ell^2(\ZN;E)}^2.
\end{aligned}
\label{eq:conical-T-L2-bound}
\end{equation}
We next estimate the kernel of \(T\) and its variation in
the input variable \(y\).
If \(\mathcal H_{\mathrm{lo}}\) is empty, then
\(S_{\mathrm{cone}}^{\mathrm{lo}}z=0\).
Otherwise, we write
\[
 (Tg)(x)=\sum_{y\in\ZN}K(x,y)g(y),
\]
where
\begin{equation}
 \bigl(K(x,y)a\bigr)_{h,j}
 =
 m_h^{-1/2}\one_{\{j\in Q_h(x)\}}K_{A2^h}(j-y)a_h,
 \qquad a\in E.
 \label{eq:conical-kernel-definition}
\end{equation}
Each \(a_h\) contributes only to the coordinates indexed by
\((h,j)\) in \(F\), so
\begin{equation}
 \norm{K(x,y)}_{E\to F}^2
 =
 \sup_{h\in\mathcal H_{\mathrm{lo}}}
 \frac1{m_h}
 \sum_{j\in Q_h(x)}\abs{K_{A2^h}(j-y)}^2.
 \label{eq:conical-kernel-norm}
\end{equation}

For the size estimate, we write \(d=\rho_N(x-y)>0\).
If \(m_h\ge d/2\), the bound
\eqref{eq:cyclic-kernel-L2} yields
\[
 \frac1{m_h}
 \sum_{j\in Q_h(x)}\abs{K_{A2^h}(j-y)}^2
 \le
 \frac{CA}{d^2}.
\]
When \(m_h<d/2\), every \(j\in Q_h(x)\) satisfies
\[
 \rho_N(j-y)\ge d-\rho_N(j-x)>d/2.
\]
The size bound \eqref{eq:cyclic-kernel-size}, with \(M=2\),
therefore gives
\[
 \abs{K_{A2^h}(j-y)}
 \le
 C\frac A{m_h}
 \left(1+\frac{Ad}{m_h}\right)^{-2}
 \le
 \frac C d.
\]
Combining these two bounds in
\eqref{eq:conical-kernel-norm} and using \(A\ge1\), we obtain
\[
 \norm{K(x,y)}_{E\to F}
 \le
 C\frac{A^{1/2}}d.
\]

We next estimate the difference \(K(x,y)-K(x,y')\), with
\(\delta=\rho_N(y-y')\le d/4\).
The kernel formula \eqref{eq:conical-kernel-definition} gives
\begin{equation}
 \norm{K(x,y)-K(x,y')}_{E\to F}^2
 =
 \sup_{h\in\mathcal H_{\mathrm{lo}}}
 \frac1{m_h}
 \sum_{j\in Q_h(x)}
 \abs{K_{A2^h}(j-y)-K_{A2^h}(j-y')}^2.
 \label{eq:conical-kernel-difference-norm}
\end{equation}
For \(\delta>0\), we take a shortest cyclic path from
\(y\) to \(y'\).
Telescoping along this path and using the \(\ell^2\) triangle
inequality, we obtain from
\eqref{eq:cyclic-kernel-difference-L2}
\begin{equation}
 \sum_{j\in\ZN}
 \abs{K_{A2^h}(j-y)-K_{A2^h}(j-y')}^2
 \le
 \delta^2\sum_{j\in\ZN}\abs{\Delta K_{A2^h}(j)}^2
 \le
 C\delta^2\frac{A^3}{m_h^3}.
 \label{eq:conical-kernel-difference-L2}
\end{equation}

We connect \(y\) to \(y'\) along a shortest cyclic path, which
has \(\delta\) nearest-neighbor steps. Applying the triangle
inequality along this path gives
\[
 \norm{K_{A2^h}(\cdot-y)-K_{A2^h}(\cdot-y')}_{\ell^2(\ZN)}
 \le\delta\norm{\Delta K_{A2^h}}_{\ell^2(\ZN)}.
\]
Squaring this \(\ell^2\)-norm bound gives the factor
\(\delta^2\) in the bound for the sum of squared kernel differences.
For \(m_h\ge d/2\), the estimate above then gives
\[
 \frac1{m_h}
 \sum_{j\in Q_h(x)}
 \abs{K_{A2^h}(j-y)-K_{A2^h}(j-y')}^2
 \le
 C\delta^2\frac{A^3}{m_h^4}
 \le
 C\frac{A^3\delta^2}{d^4}.
\]
In the remaining case \(m_h<d/2\), every point on this path
has cyclic distance at least \(d/4\) from every
\(j\in Q_h(x)\).
The pointwise difference bound
\eqref{eq:cyclic-kernel-difference}, again with \(M=2\),
then gives
\[
 \abs{K_{A2^h}(j-y)-K_{A2^h}(j-y')}
 \le
 C\frac{\delta}{d^2}.
\]
Combining the bounds from the two cases in
\eqref{eq:conical-kernel-difference-norm}, we obtain
\begin{equation}
 \norm{K(x,y)-K(x,y')}_{E\to F}
 \le
 CA^{3/2}
 \frac{\rho_N(y-y')}{\rho_N(x-y)^2}.
 \label{eq:conical-kernel-regularity}
\end{equation}

For \(y,y_0\in\ZN\), we write \(\delta=\rho_N(y-y_0)\).
For \(\delta>0\), applying
\eqref{eq:conical-kernel-regularity} to
\(K(x,y_0)-K(x,y)\) and summing over \(x\) gives
\begin{equation}
 \sum_{\rho_N(x-y_0)>4\delta}
 \norm{K(x,y)-K(x,y_0)}_{E\to F}
 \le
 CA^{3/2}.
 \label{eq:conical-kernel-Hormander}
\end{equation}
For each positive integer \(n\), at most two points \(x\in\ZN\)
satisfy \(\rho_N(x-y_0)=n\).
For \(\delta>0\), we now group the terms in the preceding sum
according to this distance and obtain
\[
 CA^{3/2}\delta\sum_{n>4\delta}n^{-2}\le CA^{3/2}.
\]
Here \(\delta\ge1\), since the cyclic distance is integer-valued.
For \(\delta=0\), the difference vanishes.
The estimate above shows that \(A^{-3/2}K\) satisfies
the integral H\"ormander condition \((C_1)\) of
\cite[Definition~3.1]{BordinFernandez92}, uniformly in \(N\).
Moreover, \(A\ge1\) and \eqref{eq:conical-T-L2-bound} give
\[
 \norm{A^{-3/2}T}_{\ell^2(\ZN;E)\to\ell^2(\ZN;F)}
 \le1.
\]
Since \((\ZN,\rho_N,\#)\) has uniformly bounded doubling constants,
we apply the weak-type theorem
\cite[Theorem~3.1, pp.~156--158]{BordinFernandez92}
to \(A^{-3/2}T\), with \(\gamma=1\) and \(r_0=s_0=2\),
to obtain
\begin{equation}
 \norm T_{\ell^1(\ZN;E)\to\ell^{1,\infty}(\ZN;F)}
 \le
 CA^{3/2}.
 \label{eq:conical-T-weak-type}
\end{equation}
The constant is independent of \(N\), \(H\), and the
dimensions of \(E\) and \(F\).
To estimate \(S_{\mathrm{cone}}^{\mathrm{lo}}z\), we interpolate
\eqref{eq:conical-T-weak-type} with
\eqref{eq:conical-T-L2-bound} and use
\eqref{eq:conical-reproducing-identity} and
\eqref{eq:conical-T-definition} to obtain
\begin{equation}
 \norm{S_{\mathrm{cone}}^{\mathrm{lo}}z}_{\ell^q}
 \le
 C_qA^{\frac32(2/q-1)}
 \norm{
 \left(
 \sum_{h\in\mathcal H_{\mathrm{lo}}}\abs{z_h}^2
 \right)^{1/2}
 }_{\ell^q}.
 \label{eq:conical-low-bound}
\end{equation}

For the remaining indices \(h\in\mathcal H_{\mathrm{top}}\),
we have \(m_h<16A\), so
\[
 S_{\mathrm{cone}}^{\mathrm{top}}z(x)^2
 \le
 \sum_{\rho_N(j-x)<16A}
 \sum_{h\in\mathcal H_{\mathrm{top}}}\abs{z_h(j)}^2.
\]
Since \(q/2<1\), subadditivity and summation over \(x\) give
\begin{equation}
 \norm{S_{\mathrm{cone}}^{\mathrm{top}}z}_{\ell^q}
 \le
 CA^{1/q}
 \norm{
 \left(
 \sum_{h\in\mathcal H_{\mathrm{top}}}\abs{z_h}^2
 \right)^{1/2}
 }_{\ell^q}.
 \label{eq:conical-top-bound}
\end{equation}

We obtain the factor \(A^{1/q}\) above by counting the points
in a cyclic ball of radius \(16A\): there are at most \(CA\)
such points. To see this directly, we put
\[
 G(j)=\sum_{h\in\mathcal H_{\mathrm{top}}}\abs{z_h(j)}^2,
\]
and use \(q/2<1\) to obtain
\begin{align*}
 \sum_{x\in\ZN}
 \left(\sum_{\rho_N(j-x)<16A}G(j)\right)^{q/2}
 &\le \sum_{x\in\ZN}\sum_{\rho_N(j-x)<16A}G(j)^{q/2}\\
 &= \sum_{j\in\ZN}G(j)^{q/2}
    \,\#\{x\in\ZN:\rho_N(j-x)<16A\}\\
 &\le CA\sum_{j\in\ZN}G(j)^{q/2}.
\end{align*}
For each fixed \(j\), the cyclic ball
\(\{x\in\ZN:\rho_N(j-x)<16A\}\) contains at most \(CA\) points.
We then take \(q\)-th roots to obtain the bound for
\(S_{\mathrm{cone}}^{\mathrm{top}}z\).
Combining \eqref{eq:conical-low-bound} and
\eqref{eq:conical-top-bound} and using \(A\ge1\), we obtain
\eqref{eq:cyclic-conical-estimate} with \(\Gamma_q\) as in
\eqref{eq:gamma-q}.
\end{proof}

\subsection{Boundary traces and the periodic Poisson kernel}
\label{subsec:boundary-traces}

We next turn to two spatially weighted coefficient sums and
an integrated Poisson average.
We bound the sums of their \(q\)-th powers over dyadic scales
uniformly in \(N\).
These estimates will control the singular terms in the proof
of Proposition~\ref{prop:nonresonant-columns}.

We define the finite set of dyadic frequency scales
\begin{equation}
 \mathcal V_N
 =
 \left\{
 v\in2^\Z:
 \phi(\xi_\kappa/v)\ne0
 \text{ for some }1\le\kappa<N/2
 \right\}.
 \label{eq:finite-shell-set}
\end{equation}
Since \(0<\xi_\kappa<\pi N\) and
\(\supp\phi\subset(1/2,2)\), every \(v\in\mathcal V_N\)
satisfies \(v<2\pi N\).
As \(v\) and \(N\) are dyadic, we obtain
\begin{equation}
 v\le4N,
 \qquad
 \frac Nv\ge\frac14.
 \label{eq:actual-trace-scale-range}
\end{equation}
Thus \(m=N/v\) can be smaller than one.

We fix an integer \(M_0>2\). For \(m>0\) and
\(d\in\C^{\ZN}\), we define
\begin{align*}
 B_m(d)
 &=
 m^{-1/p}
 \sum_{j\in\ZN}
 (1+\rho_j/m)^{-1}\abs{d_j},
 \\
 D_m(d)
 &=
 m^{-1/p}
 \sum_{j\in\ZN}
 (1+\rho_j/m)^{-M_0}\abs{d_j}.
\end{align*}
Both \(B_m(d)\) and \(D_m(d)\) sum the absolute values
\(\abs{d_j}\), with weights that decrease as the cyclic distance
\(\rho_j\) from the origin increases.
The weight in \(D_m(d)\) decreases faster than that in \(B_m(d)\).

The cyclic Poisson kernel is
\begin{equation}
 \mathsf P_\rho^N(j)
 =
 \frac1N
 \sum_{\kappa\in\Lambda_N}
 \e^{-\rho\abs{\theta_\kappa}}
 \e^{-\mathrm{i}j\theta_\kappa},
 \qquad \rho\ge0.
 \label{eq:cyclic-poisson-kernel}
\end{equation}
We use the counting convolution convention to write the action
of \(\mathsf P_\rho^N\) as
\[
 \mathsf P_\rho^N*d
 =\mathcal F_N^{-1}
   \bigl(\e^{-\rho\abs{\theta_\kappa}}\widehat d(\kappa)\bigr).
\]
We therefore use convolution by \(\mathsf P_\rho^N\) to multiply
the Fourier coefficient at frequency \(\kappa\) by
\(\e^{-\rho\abs{\theta_\kappa}}\).

We define the integrated Poisson average
\[
 \widetilde B_m(d)
 =
 m^{-1/p}
 \int_0^\infty
 \e^{-\rho/(2m)}
 \abs{(\mathsf P_\rho^N*d)(0)}
 \dd\rho.
\]
To form \(\widetilde B_m(d)\), we first convolve \(d\) with
\(\mathsf P_\rho^N\) and take the absolute value at \(j=0\).
We then integrate over \(\rho\) with the exponential weight
\(\e^{-\rho/(2m)}\) and multiply by \(m^{-1/p}\).

The estimates below for these three quantities will control
the norms of the singular terms in the  Cauchy decomposition.
In the proof of Proposition~\ref{prop:nonresonant-columns},
we apply the bound for \(B_m\) to the singular part of
\(\mathcal O_y\) defined in \eqref{eq:outside-column},
the bound for \(D_m\) to the singular terms with low positive
coefficient frequencies, and the bound for \(\widetilde B_m\)
to those with nonpositive coefficient frequencies.

To bound \(\widetilde B_m(d)\), we need a pointwise bound for
\(\mathsf P_\rho^N(j)\) in terms of the cyclic distance from \(j\)
to \(0\). We next express \(\mathsf P_\rho^N(j)\) as a sum over the
indices \(j+\ell N\), \(\ell\in\Z\).
We also establish its nonnegativity, total mass one,
and the pointwise decay estimate below. 

\begin{lemma}
[Periodic Poisson kernel]
\label{lem:periodic-poisson-kernel}
For \(\rho>0\), the cyclic Poisson kernel can be written as
\begin{equation}
 \mathsf P_\rho^N(j)
 =
 \sum_{\ell\in\Z}
 \frac{
 \rho\bigl(
 1-(-1)^{j+\ell N}\e^{-\pi\rho}
 \bigr)
 }{
 \pi\bigl(
 \rho^2+(j+\ell N)^2
 \bigr)
 }.
 \label{eq:poisson-periodization}
\end{equation}
At \(\rho=0\), the formula is understood by continuous extension.
Moreover, \(\mathsf P_\rho^N\) is even, and
\begin{equation}
 \mathsf P_\rho^N(j)\ge0,
 \qquad
 \sum_{j\in\ZN}\mathsf P_\rho^N(j)=1,
 \label{eq:poisson-positivity-mass}
\end{equation}
and
\begin{equation}
 \mathsf P_\rho^N(j)
 \le
 C
 \left\{
 \frac{1+\rho}{
 (1+\rho)^2+\rho_j^2
 }
 +
 \frac1N
 \right\}.
 \label{eq:poisson-pointwise-bound}
\end{equation}
\end{lemma}

\begin{proof}
We let \(F_\rho(\theta)=\e^{-\rho\abs\theta}\) for
\(-\pi\le\theta\le\pi\), extended \(2\pi\)-periodically.
A direct integration gives
\begin{equation}
 \widehat F_\rho(n)
 =
 \frac1{2\pi}
 \int_{-\pi}^{\pi}
 \e^{-\rho\abs\theta}\e^{-\mathrm{i} n\theta}\dd\theta
 =
 \frac{
 \rho\bigl(1-(-1)^n\e^{-\pi\rho}\bigr)
 }{
 \pi(\rho^2+n^2)
 }.
 \label{eq:continuous-poisson-coefficients}
\end{equation}
For \(\rho>0\), the coefficients in
\eqref{eq:continuous-poisson-coefficients} are
\(O_\rho((1+n^2)^{-1})\), so the Fourier series of \(F_\rho\)
converges absolutely and uniformly. Substituting this series into
\eqref{eq:cyclic-poisson-kernel}, we obtain
\begin{align*}
 \mathsf P_\rho^N(j)
 &=
 \frac1N
 \sum_{\kappa\in\Lambda_N}
 F_\rho(\theta_\kappa)\e^{-\mathrm{i} j\theta_\kappa}
 \\
 &=
 \sum_{n\in\Z}\widehat F_\rho(n)
 \left(
 \frac1N
 \sum_{\kappa\in\Lambda_N}
 \e^{\mathrm{i} (n-j)\theta_\kappa}
 \right)
 \\
 &=
 \sum_{\ell\in\Z}\widehat F_\rho(j+\ell N).
\end{align*}
 This proves \eqref{eq:poisson-periodization}.
 
The kernel is even by \eqref{eq:cyclic-poisson-kernel}
and nonnegative by \eqref{eq:poisson-periodization}.
Summing its defining Fourier series over \(j\), we also obtain
\[
 \sum_{j\in\ZN}\mathsf P_\rho^N(j)
 =F_\rho(0)=1.
\]
For the pointwise estimate, we choose \(\ell_0\in\Z\) so that
\(\abs{j+\ell_0N}=\rho_j\).
The formula \eqref{eq:continuous-poisson-coefficients} gives
\[
 \widehat F_\rho(j+\ell_0N)
 \le
 C\frac{1+\rho}{(1+\rho)^2+\rho_j^2}.
\]
For \(\ell\ne\ell_0\), the numerator in
\eqref{eq:continuous-poisson-coefficients} is at most \(2\rho\), while
\[
 \abs{j+\ell N}
 \ge c(1+\abs{\ell-\ell_0})N,
 \qquad \ell\ne\ell_0.
\]
Hence
\begin{align*} \sum_{\ell\ne\ell_0}\widehat F_\rho(j+\ell N)
 \le
 C\sum_{k=1}^{\infty}
 \frac{\rho}{\rho^2+k^2N^2}
 =
 \frac CN
 \sum_{k=1}^{\infty}
 \frac{\rho/N}{(\rho/N)^2+k^2}
 \le \frac CN.
\end{align*}
Together, the two estimates above give
\eqref{eq:poisson-pointwise-bound}.
Finally, by continuity at \(\rho=0\), we have
\(\mathsf P_0^N(j)=\one_{\{j=0\}}\).
\end{proof}

We next sum the \(q\)-th powers of \(B_m(d)\), \(D_m(d)\),
and \(\widetilde B_m(d)\) over their respective dyadic scales.
The proposition below bounds their total by
\(C_p\norm d_{\ell^q(\ZN)}^q\). 

\begin{proposition}
[Boundary-trace estimates]
\label{prop:boundary-traces}
For every \(d\in\C^{\ZN}\), we have 
\begin{equation}
 \sum_{v\in\mathcal V_N}
 B_{N/v}(d)^q
 +
 \sum_{h=0}^{H}
 D_{N/2^{h}}(d)^q
 +
 \sum_{h=0}^{H}
 \widetilde B_{N/2^{h}}(d)^q
 \le
 C_p\norm d_{\ell^q(\ZN)}^q.
 \label{eq:boundary-trace-package}
\end{equation}
\end{proposition}

\begin{proof}
The table lists the scales used in the three sums and
the estimates that control them.

\begin{center}
\small
\renewcommand{\arraystretch}{1.2}
\begin{tabularx}{\linewidth}{@{}ll>{\raggedright\arraybackslash}X@{}}
\toprule
Quantity & Scale \(m\) & How we sum over the scales \\
\midrule
\(B_m(d)\)
& \(N/v\), \(v\in\mathcal V_N\)
& For \(m\ge1\), the ring estimates give a convolution in
  the scale index, which we bound by Young's inequality.
  At the possible scales \(m=1/2\) and \(m=1/4\), we
  compare \(B_m(d)\) with \(B_1(d)\). \\
\(D_m(d)\)
& \(N2^{-h}\), \(0\le h\le H\)
& The faster decay of \((1+\rho_j/m)^{-M_0}\) gives
  summable weights in the scale index. We apply Young's
  inequality to the resulting convolution. \\
\(\widetilde B_m(d)\)
& \(N2^{-h}\), \(0\le h\le H\)
& We integrate the pointwise Poisson kernel bound and
  sum the distance-dependent contribution by convolution
  in the scale index. We sum separately the contribution
  arising from the term \(N^{-1}\) in the kernel bound. \\
\bottomrule   
\end{tabularx}
\end{center}
For \(k\ge0\), we set
\[
 \mathcal R_k=\{j\in\ZN:2^k\le1+\rho_j<2^{k+1}\},
 \qquad
a_k=\left(\sum_{j\in\mathcal R_k}\abs{d_j}^q\right)^{1/q}.
\]
For \(k<0\), we set \(\mathcal R_k=\varnothing\) and \(a_k=0\).
The sets \(\mathcal R_k\), \(k\ge0\), are disjoint
and cover \(\ZN\), so
\begin{equation}
 \sum_{k\ge0}a_k^q=\norm d_q^q.
 \label{eq:ring-mass-identity}
\end{equation}
As \(\#\mathcal R_k\le C2^k\), H\"older's inequality gives
\[
 \sum_{j\in\mathcal R_k}\abs{d_j}\le C2^{k/p}a_k.
\]
For \(m=2^\lambda\), \(\lambda\in\Nzero\), and fixed
\(\gamma\ge1\), we obtain
\begin{align}
 m^{-1/p}\sum_{j\in\ZN}
 (1+\rho_j/m)^{-\gamma}\abs{d_j}
 \le C_\gamma\left(
 \sum_{k\le\lambda}2^{-(\lambda-k)/p}a_k
 +\sum_{k>\lambda}2^{-(k-\lambda)(\gamma-1/p)}a_k
 \right).
 \label{eq:boundary-ring-convolution}
\end{align}
We now derive the two weights in the estimate above by treating
\(k\le\lambda\) and \(k>\lambda\) separately, with \(m=2^\lambda\).
For \(k\le\lambda\), the factor
\((1+\rho_j/m)^{-\gamma}\) is at most one, so we obtain
\[
 m^{-1/p}\sum_{j\in\mathcal R_k}
 (1+\rho_j/m)^{-\gamma}\abs{d_j}
 \le C2^{-\lambda/p}2^{k/p}a_k
 =C2^{-(\lambda-k)/p}a_k.
\]
For \(k>\lambda\), we have
\(1+\rho_j/m\ge(1+\rho_j)/m\ge2^{k-\lambda}\), and hence
\begin{align*}
 m^{-1/p}\sum_{j\in\mathcal R_k}
 (1+\rho_j/m)^{-\gamma}\abs{d_j}
 &\le C2^{-\lambda/p}2^{-\gamma(k-\lambda)}2^{k/p}a_k\\
 &=C2^{-(k-\lambda)(\gamma-1/p)}a_k.
\end{align*}
We therefore obtain the weight \(2^{-\nu/p}\), \(\nu\ge0\),
with \(\nu=\lambda-k\) for \(k\le\lambda\), and the weight
\(2^{-\nu(\gamma-1/p)}\), \(\nu\ge1\),
with \(\nu=k-\lambda\) for \(k>\lambda\).
These are weight sequences indexed by differences of scale indices.
For \(\gamma=1\), the exponent \(\gamma-1/p\) equals
\(1/q>0\); for \(\gamma=M_0\), it equals \(M_0-1/p>0\).
Both weight sequences are therefore summable over \(\nu\),
with \(\ell^1\)-norms depending only on \(p\) and the fixed \(M_0\).

For \(\gamma=1,M_0\), Young's inequality in \(\lambda\)
and \eqref{eq:ring-mass-identity} give
\begin{equation}
 \sum_{\substack{m\in2^\Z\\m\ge1}}B_m(d)^q
 +
 \sum_{\substack{m\in2^\Z\\m\ge1}}D_m(d)^q
 \le
 C_p\norm d_q^q.
 \label{eq:B-D-all-scales}
\end{equation}
The only possible scales \(m=N/v<1\) are \(1/2\) and \(1/4\),
by \eqref{eq:actual-trace-scale-range}.
For either scale,
\[
 B_m(d)\le m^{-1/p}B_1(d)\le4^{1/p}B_1(d).
\]
We next compare the weights at scales below one with the weight
at \(m=1\). For either \(\gamma=1\) or \(\gamma=M_0\), we write
the corresponding weight as
\[
 m^{-1/p}(1+\rho_j/m)^{-\gamma}
 =\frac{m^{\gamma-1/p}}{(m+\rho_j)^\gamma}.
\]
Its ratio to the \(m=1\) weight is
\[
 m^{\gamma-1/p}
 \left(\frac{1+\rho_j}{m+\rho_j}\right)^\gamma.
\]
For \(m\in\{1/2,1/4\}\) and \(\rho_j=0\), this ratio is
\(m^{-1/p}\le4^{1/p}\).
For \(\rho_j\ge1\), the fraction is at most two and
\(m^{\gamma-1/p}\le1\), so the ratio is at most \(2^\gamma\).
Thus both weights at these two scales are bounded by a fixed
multiple of the corresponding \(m=1\) weight.
The scales in the sum involving \(D_m\) already satisfy
\(m=N2^{-h}\ge1\).
These bounds and \eqref{eq:B-D-all-scales} control the first
two sums in \eqref{eq:boundary-trace-package}.

We next estimate \(\widetilde B_m(d)\).
Positivity and evenness of the Poisson kernel give
\begin{align}
 \widetilde B_m(d)
 &\le\sum_{j\in\ZN}\mathcal K_m^{\mathrm{tr}}(j)\abs{d_j},
 \label{eq:poisson-trace-kernel-reduction}
\end{align}
where
\begin{equation}
 \mathcal K_m^{\mathrm{tr}}(j)
 =m^{-1/p}\int_0^\infty
 \e^{-\rho/(2m)}\mathsf P_\rho^N(j)\dd\rho.
 \label{eq:integrated-poisson-kernel}
\end{equation}
Here \(m=N2^{-h}\), \(0\le h\le H\), so \(1\le m\le N\).
The \(N^{-1}\)-term in \eqref{eq:poisson-pointwise-bound}
contributes
\begin{equation}
 m^{-1/p}\int_0^\infty
 \e^{-\rho/(2m)}\frac{\dd\rho}{N}
 =\frac{2m^{1/q}}N.
 \label{eq:poisson-mean-kernel}
\end{equation}
For the integrated kernel, we obtain
\begin{equation}
 \mathcal K_m^{\mathrm{tr}}(j)
 \le C\begin{cases}
 m^{-1/p}\left[1+\log\left(1+\dfrac{m}{1+\rho_j}\right)\right]
 +\dfrac{m^{1/q}}N,
 &\rho_j\le m,\\[3mm]
 \dfrac{m^{2-1/p}}{\rho_j^2}+\dfrac{m^{1/q}}N,
 &\rho_j>m.
 \end{cases}
 \label{eq:integrated-poisson-kernel-bound}
\end{equation}
To verify this bound, we write
\[
 I_m(r)=\int_0^\infty
 \e^{-\rho/(2m)}\frac{1+\rho}{(1+\rho)^2+r^2}\dd\rho.
\]
For \(0\le r\le m\), splitting at \(\rho=m\) gives
\begin{align*}
 I_m(r)
 &\le\frac12\log\frac{(1+m)^2+r^2}{1+r^2}
 +\int_1^\infty\frac{\e^{-u/2}}u\dd u\\
 &\le C\left[1+\log\left(1+\frac m{1+r}\right)\right].
\end{align*}
For \(r\ge m\), we have 
\[
 I_m(r)
 \le\frac1{r^2}\int_0^\infty
 \e^{-\rho/(2m)}(1+\rho)\dd\rho
 =\frac{2m+4m^2}{r^2}
 \le C\frac{m^2}{r^2}.
\]

  We return to \(I_m(r)\) to explain the logarithmic term in
\eqref{eq:integrated-poisson-kernel-bound}, where \(r=\rho_j\).
For \(m\ge1\) and \(0\le r\le m\), we first split the integral
at \(\rho=m\). If \(1+r<m\), we also split \([0,m]\) at
\(\rho=1+r\), giving the three intervals
\([0,1+r]\), \([1+r,m]\), and \([m,\infty)\).
If \(1+r\ge m\), the two intervals are \([0,m]\) and
\([m,\infty)\).

In either case, the first interval is
\([0,\min\{1+r,m\}]\). On this interval,
\[
 1+\rho\le2(1+r),
 \qquad
 (1+\rho)^2+r^2\ge\frac{(1+r)^2}{2}.
\]
Since \(\e^{-\rho/(2m)}\le1\), the integrand of \(I_m(r)\)
is bounded by \(4/(1+r)\). The interval has length at most
\(1+r\), so its contribution is bounded by an absolute constant.

When \(1+r<m\), we estimate the contribution of \([1+r,m]\)
to \(I_m(r)\) using \(\e^{-\rho/(2m)}\le1\) and
\((1+\rho)/((1+\rho)^2+r^2)\le1/(1+\rho)\). Thus
\begin{align*}
 \int_{1+r}^{m}
 \e^{-\rho/(2m)}
 \frac{1+\rho}{(1+\rho)^2+r^2}\dd\rho
 &\le
 \int_{1+r}^{m}
 \frac{1+\rho}{(1+\rho)^2+r^2}\dd\rho\\
 &\le \log\frac{1+m}{2+r}
 \le \log\left(1+\frac{m}{1+r}\right).
\end{align*}
We then estimate the tail by substituting \(u=\rho/m\), which gives
\[
 \int_m^\infty
 \e^{-\rho/(2m)}
 \frac{1+\rho}{(1+\rho)^2+r^2}\dd\rho
 \le\int_1^\infty\frac{\e^{-u/2}}{u}\dd u.
\]

For \(r\ge m\), we split \(I_m(r)\) at \(\rho=r\).
On \([0,r]\), the denominator \((1+\rho)^2+r^2\) is at
least \(r^2\). Bounding the integral of
\(\e^{-\rho/(2m)}(1+\rho)\) on this interval by its integral
over \([0,\infty)\), we obtain the bound \((2m+4m^2)/r^2\).
For the integral over \([r,\infty)\), we have
\begin{align*}
 \int_r^\infty
 \e^{-\rho/(2m)}
 \frac{1+\rho}{(1+\rho)^2+r^2}\dd\rho
 \le\frac{2m}{r}\e^{-r/(2m)}
 \le C\frac{m^2}{r^2},
\end{align*}
since \((r/m)\e^{-r/(2m)}\le C\).

We write \(W_m(j)\) for the first term in the applicable branch
of \eqref{eq:integrated-poisson-kernel-bound},
without the factor \(C\).
For \(m=2^\lambda\), \(0\le\lambda\le H\), we set
\[
 T_\lambda=\sum_{j\in\ZN}W_m(j)\abs{d_j}.
\]
Using \eqref{eq:poisson-trace-kernel-reduction} and
\eqref{eq:integrated-poisson-kernel-bound}, we obtain
\[
 \widetilde B_{2^\lambda}(d)
 \le CT_\lambda+C\frac{2^{\lambda/q}}N\norm d_1.
\]

We next estimate a ring with \(k=\lambda-\nu\), \(\nu\ge0\).
The inequality \(1+\rho_j\ge2^k\) gives
\[
 \log\left(1+\frac{m}{1+\rho_j}\right)
 \le\log(1+2^\nu)\le C(1+\nu).
\]
If \(\nu\ge1\), this ring lies in \(\rho_j<m\).
On the remaining ring \(k=\lambda\), we check both expressions
for \(W_m\), since the ring can meet both sides of the threshold
\(\rho_j=m\). Each expression is bounded by \(Cm^{-1/p}\):
for \(\rho_j\le m\), we have \(m/(1+\rho_j)\le1\), while for
\(\rho_j>m\) we have \(m^{2-1/p}/\rho_j^2\le m^{-1/p}\).
We now combine these bounds with the estimate for
\(\sum_{j\in\mathcal R_{\lambda-\nu}}\abs{d_j}\).
The weight multiplying \(a_{\lambda-\nu}\) is obtained by
combining the powers of \(2\):
\[
 (1+\nu)2^{-\lambda/p}2^{(\lambda-\nu)/p}
 =(1+\nu)2^{-\nu/p}.
\]
Using \eqref{eq:integrated-poisson-kernel-bound} and
H\"older's inequality, we obtain, for \(\nu\ge0\),
\begin{equation}
 \sum_{j\in\mathcal R_{\lambda-\nu}}W_m(j)\abs{d_j}
 \le C(1+\nu)2^{-\nu/p}a_{\lambda-\nu}.
 \label{eq:poisson-inner-scale-kernel}
\end{equation}

We then estimate the contribution of the ring
\(\mathcal R_{\lambda+\nu}\), where \(\nu\ge1\).
Writing \(k=\lambda+\nu\), we have
\(\rho_j\ge2^k-1\ge2^{k-1}\) and \(\rho_j\ge m\).
At \(\rho_j=m\), both expressions defining \(W_m(j)\)
are bounded by \(Cm^{-1/p}\).
Thus \(W_m(j)\le Cm^{2-1/p}/\rho_j^2\) throughout this ring.
We apply H\"older's inequality to the sum over the ring.
The weight multiplying \(a_{\lambda+\nu}\) then follows from
\[
 2^{\lambda(2-1/p)}2^{-2(\lambda+\nu)}2^{(\lambda+\nu)/p}
 =2^{-\nu(2-1/p)}.
\]
For \(\nu\ge1\), the same argument gives
\begin{equation}
 \sum_{j\in\mathcal R_{\lambda+\nu}}W_m(j)\abs{d_j}
 \le C2^{-\nu(2-1/p)}a_{\lambda+\nu}.
 \label{eq:poisson-outer-scale-kernel}
\end{equation}

We now sum the two ring contributions to obtain, for
\(0\le\lambda\le H\),
\[
 T_\lambda\le C\left(
 \sum_{\nu\ge0}(1+\nu)2^{-\nu/p}a_{\lambda-\nu}
 +\sum_{\nu\ge1}2^{-\nu(2-1/p)}a_{\lambda+\nu}
 \right).
\]
The sequence \((a_k)\) is zero outside a finite subset of \(\Z\).
The weights \((1+\nu)2^{-\nu/p}\), \(\nu\ge0\), and
\(2^{-\nu(2-1/p)}\), \(\nu\ge1\), are both summable.
In the first sum, the weight at index \(\nu\) multiplies
\(a_{\lambda-\nu}\). In the second sum, we place the weight at
index \(-\nu\), so that \(a_{\lambda-(-\nu)}=a_{\lambda+\nu}\).
This puts both sums in convolution form.
We next apply Young's inequality on \(\Z\) to these two
convolutions and restrict the result to \(0\le\lambda\le H\).
For \(m=N2^{-h}\), we have \(\lambda=H-h\).
Thus passing from \(\lambda\) to \(h\) only reverses the finite
list of values and leaves its \(\ell^q\)-norm unchanged.
Young's inequality in \(\lambda\) then gives
\[
 \sum_{\lambda=0}^{H}T_\lambda^q\le C_p\norm d_q^q.
\]
For the remaining term, H\"older's inequality gives
\[
 \frac{m^{1/q}}N\norm d_1
 \le
 \left(\frac mN\right)^{1/q}\norm d_q.
\]
H\"older's inequality in counting measure gives
\(\norm d_1\le N^{1/p}\norm d_q\).
After taking \(q\)-th powers, the factor multiplying \(\norm d_q^q\) is
\[
 \left(\frac{m^{1/q}}N N^{1/p}\right)^q=\frac mN.
\]
At \(m=N2^{-h}\), the factor \(m/N\) equals \(2^{-h}\),
whose sum over \(0\le h\le H\) is bounded independently of \(H\).
We absorb the fixed constants from the Poisson kernel bound
into \(C_p\).
For \(m=N2^{-h}\), we have
\[
 \sum_{h=0}^{H}\frac{N2^{-h}}N
 =\sum_{h=0}^{H}2^{-h}\le2.
\]
Together, these bounds control the third sum in
\eqref{eq:boundary-trace-package}.
\end{proof}

\subsection{A coefficient square function}
\label{subsec:coefficient-square-functions}

Theorem~\ref{thm:cyclic-conical-square-function} reduces the
conical square-function estimate to a bound for the pointwise
square function. We next establish this bound for the
positive-frequency dyadic pieces in terms of the original
coefficient norm. These two estimates will be used together
in the proof of Proposition~\ref{prop:resonant-column}.

We use the fixed cutoff \(\phi\in C_c^\infty((1/2,2))\) and Fourier
convention of Section~\ref{subsec:cyclic-model}. For the finite
dyadic shell set \(\mathcal V_N\) of
Section~\ref{subsec:boundary-traces}, we  define, for \(d\in\C^{\ZN}\),
\begin{equation}
 \widehat d_v(\kappa)
 =
 \one_{\{\kappa>0\}}\phi(2\pi\kappa/v)\widehat d(\kappa),
 \qquad \kappa\in\Lambda_N,\quad v\in\mathcal V_N.
 \label{eq:coefficient-shell}
\end{equation}
All \(\ell^q(\ZN)\)-norms below use counting measure.

\begin{proposition}
[Coefficient square function]
\label{prop:coefficient-square-function}
For every fixed \(1<q<2\) and every \(d\in\C^{\ZN}\),
\begin{equation}
 \norm{
  \left(\sum_{v\in\mathcal V_N}\abs{d_v}^2\right)^{1/2}
 }_{\ell^q(\ZN)}
 \le C_q\norm d_{\ell^q(\ZN)}.
 \label{eq:coefficient-square-function}
\end{equation}
The constant \(C_q\) depends only on \(q\) and the fixed cutoff
\(\phi\), and is independent of \(N\).
\end{proposition}

\begin{proof}
For \(N=2\), the set \(\mathcal V_N\) is empty and the claim is
immediate. We now assume \(N\ge4\) and set \(K=N/2-1\).

Let \((\varepsilon_v)_{v\in\mathcal V_N}\) be independent
Rademacher variables.
For fixed values \(\varepsilon_v\in\{-1,1\}\), we set
\[
 s_\varepsilon(x)
 =\sum_{v\in\mathcal V_N}\varepsilon_v\phi(2\pi x/v),
 \qquad x\in\R,
\]
and
\begin{equation}
 m_\varepsilon(\kappa)
 =\one_{\{\kappa>0\}}s_\varepsilon(\kappa),
 \qquad \kappa\in\Lambda_N.
 \label{eq:random-shell-symbol}
\end{equation}
The zero extension in
\eqref{eq:nyquist-free-zero-extension} is
\[
 \widetilde m_\varepsilon(k)
 =\one_{\{1\le k\le K\}}s_\varepsilon(k),
 \qquad k\in\Z.
\]
For \(x>0\), the condition \(\phi(2\pi x/v)\ne0\) implies
\(\pi x<v<4\pi x\).
This interval contains at most two dyadic values of \(v\); hence
\[
\sup_{x\in\R}\abs{s_\varepsilon(x)}
 \le2\norm\phi_\infty,
 \qquad
 \norm{m_\varepsilon}_\infty\le2\norm\phi_\infty.
\]

For \(a=2^\nu\), \(\nu\ge0\), only the dyadic scales
\(\pi a<v<8\pi a\) can contribute on \([a,2a]\).
There are at most three such scales. Therefore
\begin{align}
 \sum_{k=a}^{2a-1}
 \abs{s_\varepsilon(k+1)-s_\varepsilon(k)}
 \le
 \sum_{v\in\mathcal V_N}
 \sum_{k=a}^{2a-1}
 \int_k^{k+1}\frac{2\pi}{v}
 \abs{\phi'(2\pi x/v)}\dd x
 \le
 3\norm{\phi'}_{L^1(\R)}.
 \label{eq:shell-symbol-variation}
\end{align}
We now use the change of variables \(t=2\pi x/v\) to combine
the derivative integrals in the estimate above as
\begin{align*}
 \sum_{v\in\mathcal V_N}\sum_{k=a}^{2a-1}
   \int_k^{k+1}\frac{2\pi}{v}
   \abs{\phi'(2\pi x/v)}\dd x
=
   \sum_{v\in\mathcal V_N}
   \int_{2\pi a/v}^{4\pi a/v}\abs{\phi'(t)}\dd t
 \le3\norm{\phi'}_{L^1(\R)}.
\end{align*}
Only the at most three dyadic values described above contribute
to these integrals.
The zero extension therefore satisfies
\begin{align*}
 \sum_{k=a}^{2a-1}
 \abs{\widetilde m_\varepsilon(k+1)
       -\widetilde m_\varepsilon(k)}
 &\le
 \sum_{k=a}^{2a-1}
 \abs{s_\varepsilon(k+1)-s_\varepsilon(k)}
 +\one_{\{a\le K<2a\}}\abs{s_\varepsilon(K)}\\
 &\le3\norm{\phi'}_{L^1(\R)}+2\norm\phi_\infty.
\end{align*}
The term \(\one_{\{a\le K<2a\}}\abs{s_\varepsilon(K)}\)
accounts for the jump from \(s_\varepsilon(K)\) to zero at \(K+1\).
All negative dyadic variations vanish, and we have the endpoint and
Nyquist relations below
\[
 \abs{\widetilde m_\varepsilon(1)-\widetilde m_\varepsilon(0)}
 \le2\norm\phi_\infty,
 \qquad
 m_\varepsilon(-N/2)=0.
\]
The definition in \eqref{eq:cyclic-marcinkiewicz-norm} now gives
\[
 \mathfrak M_N(m_\varepsilon)
 \le4\norm\phi_\infty+3\norm{\phi'}_{L^1(\R)}.
\]
Fourier inversion and \eqref{eq:coefficient-shell} give
\[
 \mathcal F_N^{-1}(m_\varepsilon\widehat d)
 =\sum_{v\in\mathcal V_N}\varepsilon_vd_v.
\]
Lemma~\ref{lem:cyclic-multiplier}, with \(u=q\), now gives
\begin{equation}
 \norm{\sum_{v\in\mathcal V_N}\varepsilon_vd_v}_{\ell^q(\ZN)}
 \le C_q\norm d_{\ell^q(\ZN)}
 \label{eq:random-shell-multiplier}
\end{equation}
with a constant independent of the values of \(\varepsilon_v\).

For each \(j\in\ZN\), we apply the lower Khintchine inequality
to the finite complex family \((d_v(j))_{v\in\mathcal V_N}\).
We then sum the pointwise estimate over \(j\) to obtain
\begin{align*}
 \norm{
  \left(\sum_{v\in\mathcal V_N}\abs{d_v}^2\right)^{1/2}
 }_{\ell^q(\ZN)}^q
 &=\sum_{j\in\ZN}
       \left(\sum_{v\in\mathcal V_N}\abs{d_v(j)}^2\right)^{q/2}\\
 &\le3^{(2-q)/2}\E_\varepsilon
       \norm{\sum_{v\in\mathcal V_N}\varepsilon_vd_v}
             _{\ell^q(\ZN)}^q\\
 &\le C_q\norm d_{\ell^q(\ZN)}^q.
\end{align*}
Finally, the lower Khintchine inequality
\cite[Chapter~2]{KahaneRandomSeries85}, together with
\eqref{eq:random-shell-multiplier}, gives
\eqref{eq:coefficient-square-function}.
\end{proof}

\section{Sampling and spectral-tail estimates}
\label{sec:frame-and-tails}

Section~\ref{sec:cyclic-analysis} provides
Fourier estimates for coefficient vectors
on $\ZN$. We now establish spectral projection
estimates for the Euler-product frequency
operator $D_{\mathrm{full}}$ and sampling
estimates at the prime frequencies $\log r$.

Together with vector-valued conditional expectation,
these estimates will be used in the frequency
decomposition of the projected vector in
Section~\ref{sec:corrected-column}.


Recall that $N=2^H$ and $y=\e^N$.
We group the primes according to the size
of $\log r$:
\[
\mathcal I_h
=
\{r\text{ prime}:2^{h-1}<\log r\le2^h\},
\qquad 0\le h\le H.
\]
These bands partition the primes $r\le y$.
We also set $K_0=\pi N/2$ for later use.

\subsection{Spectral projections for different frequency intervals}
\label{subsec:row-interval-projections}

Several estimates below involve spectral
projections onto frequency intervals that
depend on the prime index $r$.
Scalar $L^q$ projection bounds do not by
themselves give uniform $L^q(\ell^2)$ bounds
for these componentwise projections.
The following proposition establishes the
required estimate.

\begin{proposition}[Uniform componentwise spectral projections]
\label{prop:row-interval-projections}
Let $1<q<\infty$, let $\mathcal R$ be a finite index set, and let
\[
-\infty\le A_r\le B_r\le\infty,
\qquad r\in\mathcal R.
\]
For every family
$(F_r)_{r\in\mathcal R}
\in L^q(\Omega_y;\ell^2(\mathcal R))$,
we have
\begin{equation}
\norm{
\left(
\one_{[A_r,B_r]}(D_{\mathrm{full}})F_r
\right)_{r\in\mathcal R}
}_{L^q(\ell^2(\mathcal R))}
\le
C_q
\norm{
(F_r)_{r\in\mathcal R}
}_{L^q(\ell^2(\mathcal R))}.
\label{eq:row-interval-projection}
\end{equation}
The estimate also holds for any choice of
open or closed endpoints.
The constant $C_q$ depends only on $q$.
\end{proposition}

\begin{proof}
Since every interval projection is a difference
of two half-line projections, it suffices to prove,
for arbitrary $E_r\in\R$, $r\in\mathcal R$,
\[
\norm{
\left(
\one_{[E_r,\infty)}(D_{\mathrm{full}})F_r
\right)_{r\in\mathcal R}
}_{L^q(\ell^2(\mathcal R))}
\le
C_q
\norm{(F_r)_{r\in\mathcal R}}_{L^q(\ell^2(\mathcal R))},
\]
and the analogous estimate for $(E_r,\infty)$.

Let
\[
\widetilde\Omega=\Omega_y\times\T^{\mathcal R},
\]
and let $(\omega_r)_{r\in\mathcal R}$ be the
coordinate Steinhaus variables on
$\T^{\mathcal R}$.
By the Steinhaus--Khintchine inequality
\cite[Chapter~2]{KahaneRandomSeries85},
\begin{equation}
\norm{(F_r)_{r\in\mathcal R}}_{L^q(\ell^2)}
\asymp_q
\norm{
\sum_{r\in\mathcal R}\omega_rF_r
}_{L^q(\widetilde\Omega)}.
\label{eq:row-first-chaos-scalarization}
\end{equation}
The same equivalence holds for the
projected components.

To shift each frequency cutoff $E_r$ to zero,
write $E=(E_r)_{r\in\mathcal R}$ and define
the flow
\begin{equation}
(\widetilde U_t^E G)(\zeta,\omega)
=
G\left(
(\e^{\ii t\log\rho}\zeta_\rho)_\rho,
(\e^{-\ii tE_r}\omega_r)_r
\right),
\qquad t\in\R.
\label{eq:row-endpoint-flow}
\end{equation}
This is a strongly continuous,
measure-preserving group of isometries
on $L^q(\widetilde\Omega)$.
On Fourier characters, it acts by
\[
\widetilde U_t^E(\omega_r\zeta^\nu)
=
\e^{\ii t(\lambda(\nu)-E_r)}
\omega_r\zeta^\nu,
\qquad
\lambda(\nu)=\sum_\rho\nu_\rho\log\rho.
\]
Thus the frequency cutoff at $E_r$
for $F_r$ becomes a cutoff at zero
for $\omega_rF_r$.

By the $L^q$ boundedness of the real-line
Hilbert transform
\cite[Theorem~5.1.7]{Grafakos14}
and the Coifman--Weiss transference theorem
\cite[Theorem~3.15]{CoifmanWeiss77},
the multiplier
\[
m(\lambda)
=
\frac{1+\operatorname{sgn}(\lambda)}2
=
\one_{(0,\infty)}(\lambda)
+\frac12\one_{\{0\}}(\lambda),
\qquad \operatorname{sgn}(0)=0,
\]
defines a bounded operator $Q_E$ on
$L^q(\widetilde\Omega)$, with
$\norm{Q_E}_{L^q\to L^q}\le C_q$.
On trigonometric polynomials,
\[
Q_E\left(\sum_{r\in\mathcal R}\omega_rF_r\right)
=
\sum_{r\in\mathcal R}\omega_r
m(D_{\mathrm{full}}-E_r)F_r.
\]
The boundedness of $Q_E$, together with
\eqref{eq:row-first-chaos-scalarization},
gives the corresponding bound for
$(m(D_{\mathrm{full}}-E_r)F_r)_r$.
By density, the identity extends to
arbitrary $F_r\in L^q(\Omega_y)$.

To handle frequencies satisfying \(\lambda(\nu)=E_r\),
let
\[
\mathcal A_T^E
=
\frac1{2T}\int_{-T}^{T}
\widetilde U_t^E\,\dd t,
\qquad T>0.
\]
These averages are contractions on $L^q$.
For a character $\chi$ satisfying
$\widetilde U_t^E\chi=\e^{\ii t\lambda}\chi$,
\[
\mathcal A_T^E\chi=
\begin{cases}
\dfrac{\sin(T\lambda)}{T\lambda}\chi,
&\lambda\ne0,\\[2mm]
\chi,&\lambda=0.
\end{cases}
\]
Thus their multipliers on characters converge
to $\one_{\{0\}}$.
Since trigonometric polynomials are dense
in $L^q(\widetilde\Omega)$ and the averages
are contractions, $\mathcal A_T^E$
converges strongly to the zero-frequency
projection $P_0^E$, with
\[
\norm{P_0^E}_{L^q\to L^q}\le1.
\]
For functions of the form
$\sum_{r\in\mathcal R}\omega_rF_r$, we have
\[
P_0^E
\left(\sum_{r\in\mathcal R}\omega_rF_r\right)
=
\sum_{r\in\mathcal R}\omega_r
\one_{\{E_r\}}(D_{\mathrm{full}})F_r.
\]

Hence the operators
\[
P_>^E=Q_E-\tfrac12P_0^E,
\qquad
P_\ge^E=Q_E+\tfrac12P_0^E
\]
are uniformly bounded on
$L^q(\widetilde\Omega)$.
On functions of the form
$\sum_{r\in\mathcal R}\omega_rF_r$,
they give the componentwise projections
onto $(E_r,\infty)$ and $[E_r,\infty)$,
respectively.

Applying \eqref{eq:row-first-chaos-scalarization}
to these projected components proves both
half-line estimates.
Taking differences gives the required
interval projections with any choice of
open or closed endpoints.
Infinite endpoints and empty intervals
are immediate.
This proves \eqref{eq:row-interval-projection}.
\end{proof}

\begin{corollary}[Bounded-variation spectral multipliers]
\label{cor:transferred-bv-multipliers}
Let $\phi:\R\to\C$ be a function of bounded
variation. Then
\begin{equation}
\norm{\phi(D_{\mathrm{full}})}_{L^q\to L^q}
\le
C_q\left(
|\phi(-\infty)|+\operatorname{Var}(\phi)
\right),
\label{eq:transferred-bv-functional-calculus}
\end{equation}
where $\operatorname{Var}(\phi)$ denotes the
total variation of $\phi$ on $\R$.

In particular,
\begin{equation}
\sup_{t,K\ge0}
\norm{
\e^{-tD_{\mathrm{full}}}
\one_{[0,K]}(D_{\mathrm{full}})
}_{L^q\to L^q}
\le C_q.
\label{eq:transferred-damped-low-bound}
\end{equation}
\end{corollary}

\begin{proof}
Proposition~\ref{prop:row-interval-projections},
applied to a single component, gives uniform
$L^q$ bounds for half-line and point projections.
For the right-continuous representative
$\phi_{\mathrm{rc}}$ of $\phi$, the Stieltjes
representation
\[
\phi_{\mathrm{rc}}(\lambda)
=
\phi(-\infty)
+
\int_{\R}\one_{[u,\infty)}(\lambda)
\,\dd\phi_{\mathrm{rc}}(u)
\]
and Minkowski's inequality give
\[
\norm{\phi_{\mathrm{rc}}(D_{\mathrm{full}})}_{L^q\to L^q}
\lesssim_q
|\phi(-\infty)|+\operatorname{Var}(\phi).
\]

Moreover,
\[
\sum_{u\in\R}
|\phi(u)-\phi_{\mathrm{rc}}(u)|
\le\operatorname{Var}(\phi).
\]
Since
\[
(\phi-\phi_{\mathrm{rc}})(D_{\mathrm{full}})
=
\sum_{u\in\R}
(\phi(u)-\phi_{\mathrm{rc}}(u))
\one_{\{u\}}(D_{\mathrm{full}}),
\]
the point-projection bounds give
\eqref{eq:transferred-bv-functional-calculus}.

For the second assertion, observe that
\[
\operatorname{Var}\left(
\e^{-t\lambda}\one_{[0,K]}(\lambda)
\right)\le2,
\qquad t,K\ge0.
\]
The conclusion follows from
\eqref{eq:transferred-bv-functional-calculus}.
\end{proof}



\subsection{Sampling at logarithms of primes}
\label{subsec:prime-sampling}

We now prove a prime-sampling 
lemma that will
be used to estimate  the resonant and high-frequency terms in the proofs of Propositions~\ref{prop:frame-estimate}
and~\ref{prop:high-source-tail}. In particular, we need to estimate exponential sums sampled at
the prime frequencies $\log r$, with
weights $r^{-1/2}$.
To isolate this sampling operation, for $0\le h\le H$, define
\begin{equation}
S_hu
=
\left(
r^{-1/2}
\sum_{j<N}u_j\e^{\ii t_j\log r}
\right)_{r\in\mathcal I_h}, \qquad u\in\C^{\ZN}.
\label{eq:actual-prime-sampling-map}
\end{equation}

\begin{lemma}[Prime sampling]
\label{lem:prime-sampling}
For $0\le h\le H$, the operator $S_h$ satisfies
\begin{equation}
 \norm{S_h}_{\ell^2(\ZN)\to\ell^2(\mathcal I_h)}
 \le
 C\left(\frac N{2^h}\right)^{1/2},
 \label{eq:prime-sampling-L2-operator}
\end{equation}
and
\begin{equation}
 \norm{S_h}_{\ell^q(\ZN)\to\ell^2(\mathcal I_h)}
 \le
 C_p\left(\frac N{2^h}\right)^{1/p}.
 \label{eq:prime-sampling-Lq}
\end{equation}
Here $C$ is an absolute constant, and $C_p$
depends only on $p$.
\end{lemma}

\begin{proof}

Put
\[
V(v)=\sum_{j=0}^{N-1}u_j\e^{\ii jv/N}.
\]
Then $V$ is $2\pi N$-periodic, and
\(
S_hu
=
\left(
r^{-1/2}\e^{\ii t_0\log r}V(\log r)
\right)_{r\in\mathcal I_h}.
\)
Consequently,
\[
\norm{S_hu}_{\ell^2}^2
=
\sum_{r\in\mathcal I_h}
\frac1r\abs{V(\log r)}^2.
\]

Partition $(0,N]$ into unit intervals
\(
 I_m=(m,m+1],
 0\le m\le N-1.
\)
Chebyshev's estimate
\(
 \pi(x)\le C\frac{x}{\log x}\)
implies
\begin{equation}
 \sum_{\substack{r\ \mathrm{prime}\\
                 \log r\in I_m}}\frac1r
 \le
 \e^{-m}\pi(\e^{m+1})
 \le
 \frac{C}{m+1};
 \label{eq:prime-cell-mass}
\end{equation}
see \cite[Chapter~2, Corollary~2.6]{MontgomeryVaughan06}. 
If $I_m$ meets $(2^{h-1},2^h]$, then
$m+1>2^{h-1}$, so the last bound is
$O(2^{-h})$.

For every unit interval $I_m$,
\begin{equation}
 \sup_{v\in I_m}\abs{V(v)}^2
 \le
 2\int_{I_m}\abs{V(v)}^2\dd v
 +
 \int_{I_m}\abs{V'(v)}^2\dd v.
 \label{eq:unit-interval-supremum}
\end{equation}
Indeed, for $x,y\in I_m$,
\[
 \abs{V(x)}^2-\abs{V(y)}^2
 =
 2\Ree\int_y^x
 V'(v)\overline{V(v)}\dd v.
\]
Taking absolute values, integrating in $y$,
and using
$2\abs{V'V}\le\abs{V'}^2+\abs V^2$
gives \eqref{eq:unit-interval-supremum}.

Combining \eqref{eq:prime-cell-mass} and
\eqref{eq:unit-interval-supremum}, summing
over the relevant unit intervals, and
extending the integral to $[0,2\pi N]$
gives
\begin{equation}
 \norm{S_hu}_{\ell^2}^2
 \le
 \frac C{2^h}
 \int_0^{2\pi N}
 \left(\abs{V(v)}^2+\abs{V'(v)}^2\right)
 \dd v.
 \label{eq:prime-sampling-before-parseval}
\end{equation}
By Parseval's identity,
\[
 \int_0^{2\pi N}\abs{V(v)}^2\dd v
 =
 2\pi N\sum_{j<N}\abs{u_j}^2,
\]
and
\[
 \int_0^{2\pi N}\abs{V'(v)}^2\dd v
 =
 2\pi N\sum_{j<N}
 \frac{j^2}{N^2}\abs{u_j}^2
 \le
 2\pi N\sum_{j<N}\abs{u_j}^2.
\]
This proves \eqref{eq:prime-sampling-L2-operator}.

Next, there are $O(2^h)$ unit intervals
meeting $(2^{h-1},2^h]$.
Summing \eqref{eq:prime-cell-mass}
over these intervals gives
\(
\sum_{r\in\mathcal I_h}\frac1r\le C.
\)
Consequently,
\[
 \norm{S_hu}_{\ell^2}^2
 \le
 \left(\sum_{j<N}\abs{u_j}\right)^2
 \sum_{r\in\mathcal I_h}\frac1r
 \le
 C\norm u_{\ell^1}^2.
\]
Thus $\norm{S_h}_{\ell^1\to\ell^2}\le C$.
Interpolating this bound with
\eqref{eq:prime-sampling-L2-operator}
with parameter $2/p$ gives
\eqref{eq:prime-sampling-Lq}.
\end{proof}

\subsection{Conditional expectations and common-source sampling}
\label{subsec:frozen-source-frame}

The sampling estimate in Lemma~\ref{lem:prime-sampling}
applies to a single coefficient vector $u$ for
all primes $r\in\mathcal I_h$.
In the resonant and high-frequency terms, however,
the coefficients of the sampled sums contain
$f_{r^-}(t_j)$, and hence depend on $r$.
We therefore first work with the common Euler
product $f_h^0$ on each prime band.
The identities \eqref{eq:euler-source-conditioning}
and \eqref{eq:conditional-spectral-commutation}
then allow us to recover the original terms by
conditional expectation.
To control these conditional expectations jointly
in $L^q(\ell^2)$, we use Stein's vector-valued
inequality; see
\cite[Section~4.2.d, pp.~298--299]{HytoneEtAl16}.


\begin{proposition}[Vector-valued conditional expectations]
\label{prop:vector-conditional-expectations}
Let
\[
 \mathscr F_1\subseteq\cdots\subseteq\mathscr F_m
\]
be a finite increasing filtration, and let
\[
 \mathbb E_k=\E(\,\cdot\mid\mathscr F_k).
\]
Then, for every $X_1,\ldots,X_m\in L^q$,
\begin{equation}
 \norm{
 (\mathbb E_kX_k)_{k=1}^{m}
 }_{L^q(\ell^2)}
 \le
 C_q
 \norm{
 (X_k)_{k=1}^{m}
 }_{L^q(\ell^2)}.
 \label{eq:vector-conditional-expectation}
\end{equation}
\end{proposition}

To complement the conditional-expectation estimate,
we next apply $S_h$ in Lemma~\ref{lem:prime-sampling}
with coefficients $u_j=z_h(j)g_h(t_j)$,
which do not vary with $r\in\mathcal I_h$.
Here \[
g=(g_h)_{0\le h\le H}
\]
is a family of random functions of $t$,
all defined on $\Omega_y$, and
\[
z=(z_h(j))_{\substack{0\le h\le H\\j<N}}
\]
are fixed coefficients.
We allow a general family $g_h$ so that
the estimate applies to both the undeformed
Euler products $f_h^0$ and their deformations
$f_h^\vartheta$.
We therefore define the prime-indexed array
\begin{equation}
\bigl(\mathcal C_y^g(z)\bigr)_r
=
M_y^{-1/p}r^{-1/2}
\sum_{j<N}
z_h(j)\e^{\ii t_j\log r}g_h(t_j),
\qquad r\in\mathcal I_h.
\label{eq:generic-frozen-frame-array}
\end{equation}
The squared sampling estimate
\eqref{eq:prime-sampling-L2-operator}
leads to the weighted moment
\begin{equation}
\mathfrak P_y^g(z)
=
\E\left[
\left(
\sum_{h=0}^{H}
\frac N{2^h}M_y^{-\delta}
\sum_{j<N}
\abs{z_h(j)}^2\abs{g_h(t_j)}^2
\right)^s
\right].
\label{eq:generic-source-functional}
\end{equation}
Here $\delta=2/p$ and $s=q/2$.
For $g_h=f_h^\vartheta$, this agrees with
$\mathfrak P_y^\vartheta(z)$ defined in
\eqref{eq:occupation-functional}.

\begin{lemma}[Sampling estimate for common Euler products]
\label{lem:frozen-source-frame}
For every family $g$ and coefficient array $z$
as above,
\begin{equation}
\norm{\mathcal C_y^g(z)}_{L^q(\ell^2)}
\le
C\mathfrak P_y^g(z)^{1/q}.
\label{eq:frozen-source-frame}
\end{equation}
\end{lemma}

\begin{proof}
On each prime band
$\mathcal I_h$, the array $\mathcal C_y^g(z)$
equals $M_y^{-1/p}S_hu$, where
\[
u_j=z_h(j)g_h(t_j).
\]
Applying \eqref{eq:prime-sampling-L2-operator}
and summing over the disjoint prime bands gives
\[
\sum_{r\le y}
\abs{\bigl(\mathcal C_y^g(z)\bigr)_r}^2
\le
C\sum_{h=0}^{H}
\frac N{2^h}M_y^{-\delta}
\sum_{j<N}
\abs{z_h(j)}^2\abs{g_h(t_j)}^2.
\]
Raising this inequality to the power $s=q/2$,
taking expectations, and then taking $q$-th
roots proves \eqref{eq:frozen-source-frame}.
\end{proof}

\subsection{Estimates for the resonant terms}
\label{subsec:resonant-row-arrays}

The main message in this subsection is when frequencies can nearly match, prime sampling reduces their interaction to Euler-product moments, while deformation provides extra decay for large offsets.
For $r\in\mathcal I_h$, a Fourier monomial
of $f_{r^-}(t)$ has time frequency
\[
\lambda=\sum_{\rho<r}\nu_\rho\log\rho,
\qquad \nu_\rho\in\mathbb Z.
\]
Multiplication by $\e^{\ii t\log r}$ shifts
this frequency to $\omega=\lambda+\log r$.
The coefficient shells $d_v$ defined in
\eqref{eq:coefficient-shell} have frequencies
$\xi_\kappa\in(v/2,2v)$.
We use
\[
\frac v4\le\omega\le4v
\]
as the resonant region. Outside this interval,
$\omega-\xi_\kappa\le-v/4$ or
$\omega-\xi_\kappa\ge2v$.

Writing $v=2^{h+\ell}$, we see that for low frequencies $0\le\lambda\le K_0$, if
$\ell\le-3$, then
$4v\le2^{h-1}<\log r\le\omega$,
so the resonant region is empty.

Set $d_v=0$ if $v\notin\mathcal V_N$.
For an integer $\ell\ge-2$, define
\begin{equation}
v_{h,\ell}=2^{h+\ell},
\qquad
z_h^{(\ell)}=d_{2^{h+\ell}}.
\label{eq:offset-table}
\end{equation}
Combining the resonant condition with
$0\le\lambda\le K_0$ gives the interval
$J_r^{(\ell)}$ defined below.
These intervals appear in the exact frequency
decomposition of
Proposition~\ref{prop:exact-frequency-partition}.

For $r\in\mathcal I_h$, put
\begin{equation}
 J_r^{(\ell)}
 =
 [0,K_0]
 \cap
 \left[
 \frac{v_{h,\ell}}4-\log r,\,
 4v_{h,\ell}-\log r
 \right].
 \label{eq:resonant-row-interval}
\end{equation}
An empty interval gives the zero projection.
We define the resonant array by retaining
the Euler-product frequencies in
$J_r^{(\ell)}$ for each $r\in\mathcal I_h$:
\begin{align}
 \bigl(\mathcal R_y^{(\ell)}d\bigr)_r
 ={}&
 M_y^{-1/p}r^{-1/2}
 P_{J_r^{(\ell)}}^{r^-}
 \sum_{j<N}
 z_h^{(\ell)}(j)
 \e^{\ii t_j\log r}
 f_{r^-}(t_j),
 \qquad r\in\mathcal I_h.
 \label{eq:offset-resonant-row}
\end{align}
For the moment functional
$\mathfrak P_y^\vartheta$ defined in
\eqref{eq:occupation-functional}, we write
\begin{equation}
 \mathfrak P_y(z)=\mathfrak P_y^0(z).
 \label{eq:def-Py}
\end{equation}


\begin{proposition}[Estimate for the resonant components]
\label{prop:frame-estimate}
For every $d\in\C^{\ZN}$ and every
integer $\ell\ge-2$,
\begin{equation}
 \norm{
 \mathcal R_y^{(\ell)}d
 }_{L^q(\ell^2)}
 \le
 C_p
 \mathfrak P_y(z^{(\ell)})^{1/q}.
 \label{eq:frame-estimate}
\end{equation}
The constant $C_p$ depends only on $p$.
\end{proposition}

\begin{proof}

To estimate $\mathcal R_y^{(\ell)}d$, we
introduce an auxiliary array by replacing
$f_{r^-}$ with the common Euler product
$f_h^0$ on each prime band $\mathcal I_h$
and using the full-space spectral projection $P_{J_r^{(\ell)}}^{\mathrm{full}}$.
For $r\in\mathcal I_h$, define on $\Omega_y$
\begin{equation}
X_r
=
M_y^{-1/p}r^{-1/2}
P_{J_r^{(\ell)}}^{\mathrm{full}}
\sum_{j<N}
z_h^{(\ell)}(j)
\e^{\ii t_j\log r}
f_h^0(t_j).
\label{eq:frame-full-source-row}
\end{equation}
Since $r\le\e^{2^h}$ and
$f_h^0=f_{\e^{2^h}}$, the Euler-product
conditioning identity
\eqref{eq:euler-source-conditioning} gives
\[
\E_{r^-}f_h^0(t_j)=f_{r^-}(t_j).
\]
Together with the commutation identity
\eqref{eq:conditional-spectral-commutation},
this yields
\begin{align}
\E_{r^-}X_r
&=
M_y^{-1/p}r^{-1/2}
P_{J_r^{(\ell)}}^{r^-}
\sum_{j<N}
z_h^{(\ell)}(j)
\e^{\ii t_j\log r}
f_{r^-}(t_j)
=
\bigl(\mathcal R_y^{(\ell)}d\bigr)_r.
\label{eq:frame-live-row-as-conditional}
\end{align}

Order the primes $r\le y$ as $r_1<\cdots<r_m$.
The $\sigma$-fields $\mathscr F_{r_k^-}$ are
increasing, and
$\E_{r_k^-}=\E(\,\cdot\mid\mathscr F_{r_k^-})$.
Thus Proposition~\ref{prop:vector-conditional-expectations}
gives
\[
\norm{\mathcal R_y^{(\ell)}d}_{L^q(\ell^2)}
\le
C_q\norm{(X_r)_r}_{L^q(\ell^2)}.
\]
By  definitions, we have
\[
X_r=
P_{J_r^{(\ell)}}^{\mathrm{full}}
\bigl(\mathcal C_y^{f^0}(z^{(\ell)})\bigr)_r.
\]
Thus Proposition~\ref{prop:row-interval-projections}
gives
\[
\norm{(X_r)_r}_{L^q(\ell^2)}
\le
C_q\norm{
\mathcal C_y^{f^0}(z^{(\ell)})
}_{L^q(\ell^2)}.
\]
Finally, applying Lemma~\ref{lem:frozen-source-frame}
with $g_h=f_h^0$ and using
$\mathfrak P_y^0=\mathfrak P_y$
gives \eqref{eq:frame-estimate}.
\end{proof}

The estimate \eqref{eq:frame-estimate} provides
no decay in the offset $\ell$.
For $\ell\ge4$, however, every frequency
$\lambda\in J_r^{(\ell)}$ satisfies
\(
\lambda\ge\frac{2^{h+\ell}}8.
\)
Deforming the common Euler product then yields
a factor exponentially small in $2^\ell$,
while Proposition~\ref{prop:common-shift-occupation}
controls the resulting deformed moment.
This decay compensates for the growth in the
conical square-function estimate when
summing over $\ell$ in
Section~\ref{sec:corrected-column}.

\begin{proposition}[Exponential decay of the resonant components]
\label{prop:deformed-frame}
Let \(\vartheta_0=1/32\).
For every \(d\in\C^{\ZN}\) and every integer \(\ell\ge4\),
\begin{equation}
 \norm{\mathcal R_y^{(\ell)}d}_{L^q(\ell^2)}
 \le
 C_p\e^{-\vartheta_02^\ell/8}
 \mathfrak P_y^{\vartheta_0}(z^{(\ell)})^{1/q},
 \label{eq:deformed-frame}
\end{equation}
where \(C_p\) depends only on \(p\).
\end{proposition}

\begin{proof}
Fix $\ell\ge4$. For $r\in\mathcal I_h$, write
$w=2^h$, $v=2^{h+\ell}=2^\ell w$, and set
\[
\lambda_{r,\ell}^-=\frac v4-\log r,
\qquad
\lambda_{r,\ell}^+=\min\{K_0,4v-\log r\}.
\]
Since $\log r\le w$,
\begin{equation}
\lambda_{r,\ell}^-
\ge \frac v4-w
=v\left(\frac14-2^{-\ell}\right)
\ge\frac v8.
\label{eq:deformed-frame-lower-gap}
\end{equation}
Thus $J_r^{(\ell)}
=[\lambda_{r,\ell}^-,\lambda_{r,\ell}^+]$
whenever this interval is nonempty.

To apply Proposition~\ref{prop:row-interval-projections}
while retaining the exponential decay, we use
the following identity for $A\le B$, $w>0$,
and $\vartheta>0$:
\begin{equation}
\e^{-\vartheta D_{\mathrm{full}}/w}
\one_{[A,B]}(D_{\mathrm{full}})
=
\e^{-\vartheta A/w}
\left[
\one_{[A,B]}(D_{\mathrm{full}})
-
\int_0^\infty
\vartheta\e^{-\vartheta u}
\one_{[A+uw,B]}(D_{\mathrm{full}})\dd u
\right].
\label{eq:deformed-interval-identity}
\end{equation}
The operator on the left is understood
as the bounded spectral multiplier
$\lambda\mapsto
\e^{-\vartheta\lambda/w}\one_{[A,B]}(\lambda)$.
Here an interval with its lower endpoint
exceeding its upper endpoint is understood
to be empty. Indeed, at a spectral value
$\lambda\in[A,B]$, the expression in brackets
equals $\e^{-\vartheta(\lambda-A)/w}$,
and outside $[A,B]$ both sides vanish.

By the definition of $f_h^{\vartheta_0}$,
each Fourier coefficient of frequency $\lambda$
is the corresponding coefficient of $f_h^0$
multiplied by $\e^{\vartheta_0\lambda/w}$.
Hence, coefficientwise,
\begin{equation}
\e^{-\vartheta_0D_{\mathrm{full}}/w}
f_h^{\vartheta_0}(t)
=
f_h^0(t).
\label{eq:deformation-restores-source}
\end{equation}
The absolutely convergent Euler expansions
justify this identity.

For $r\in\mathcal I_h$, define the
unprojected deformed expressions
\[
F_r^\vartheta
=
M_y^{-1/p}r^{-1/2}
\sum_{j<N}
z_h^{(\ell)}(j)\e^{\ii t_j\log r}
f_h^\vartheta(t_j).
\]
Then the auxiliary array $X_r$ from
\eqref{eq:frame-full-source-row} satisfies
\[
X_r=P_{J_r^{(\ell)}}^{\mathrm{full}}F_r^0,
\qquad
F_r^0=
\e^{-\vartheta_0D_{\mathrm{full}}/w}
F_r^{\vartheta_0}.
\]
Applying \eqref{eq:deformed-interval-identity},
we obtain
\[
X_r
=
\e^{-\vartheta_0\lambda_{r,\ell}^-/w}
\left[
P_{J_r^{(\ell)}}^{\mathrm{full}}
-
\int_0^\infty
\vartheta_0\e^{-\vartheta_0u}
P_{[\lambda_{r,\ell}^-+uw,
\lambda_{r,\ell}^+]}^{\mathrm{full}}
\dd u
\right]F_r^{\vartheta_0},
\]
with both sides taken to be zero when
$J_r^{(\ell)}$ is empty.

For each fixed $u\ge0$,
Proposition~\ref{prop:row-interval-projections}
controls the componentwise interval projections,
uniformly in their endpoints.
By \eqref{eq:deformed-frame-lower-gap},
\[
\e^{-\vartheta_0\lambda_{r,\ell}^-/w}
\le
\e^{-\vartheta_02^\ell/8}.
\]
Since the integration weight
$\vartheta_0\e^{-\vartheta_0u}\dd u$
is independent of $r$ and has total mass one,
Minkowski's inequality gives
\[
\norm{(X_r)_r}_{L^q(\ell^2)}
\le
2C_q\e^{-\vartheta_02^\ell/8}
\norm{(F_r^{\vartheta_0})_r}_{L^q(\ell^2)}.
\]
By \eqref{eq:frame-live-row-as-conditional}
and Proposition~\ref{prop:vector-conditional-expectations},
we have
\[
\norm{\mathcal R_y^{(\ell)}d}_{L^q(\ell^2)}
\le
C_q\norm{(X_r)_r}_{L^q(\ell^2)}.
\]

Finally, Lemma~\ref{lem:frozen-source-frame},
applied with $g_h=f_h^{\vartheta_0}$, gives
\[
\norm{(F_r^{\vartheta_0})_r}_{L^q(\ell^2)}
\le
C\mathfrak P_y^{\vartheta_0}
(z^{(\ell)})^{1/q}.
\]
Combining these estimates proves
\eqref{eq:deformed-frame}.
\end{proof}

\subsection{Euler-product deformation}
\label{subsec:source-deformation}

For $w=2^h$, the following proposition controls
the part of $f_h^0$ with time frequencies
above a cutoff $K$.
It gives exponential decay in $K/w$ and will
be used in Proposition~\ref{prop:high-source-tail}.

\begin{proposition}[Euler-product deformation]
\label{prop:source-deformation}
Let $w=2^h$ with $0\le h\le H$, and let
$0<\vartheta\le1/16$.
Then, for every $t\in\R$,
\begin{equation}
 \E
 \abs{
 \e^{\vartheta D_{\mathrm{full}}/w}
 f_h^0(t)
 }^q
 \le
 C_p M_{\e^w}^a.
 \label{eq:full-power-source-deformation}
\end{equation}
Moreover, for every $K\ge0$,
\begin{equation}
 \E
 \abs{
 P_{(K,\infty)}^{\mathrm{full}}f_h^0(t)
 }^q
 \le
 C_p
 \e^{-q\vartheta K/w}
 M_{\e^w}^a.
 \label{eq:source-spectral-tail}
\end{equation}
The constant $C_p$ depends only on $p$.
\end{proposition}

\begin{proof}
Fix a prime $r\le\e^w$, and put
\(
x=r^{-1},
u=\vartheta\frac{\log r}{w}.\)
Then $x\le1/2$ and $0\le u\le\vartheta\le1/16$.
By Haar invariance, the change of variables
$z\mapsto\e^{\ii t\log r}z$ removes the
$t$-dependence from the one-prime integral.
Applying $\e^{\vartheta D_{\mathrm{full}}/w}$
coefficientwise to the Euler product
\eqref{eq:euler-source}, using
\eqref{eq:full-prime-generator},
gives the following factor for the prime $r$:
\[
(1-x)^\alpha
(1-\sqrt{x}\e^uz)^{-1}
(1-\sqrt{x}\e^{-u}\bar z)^{-\alpha}.
\]

Taking the $q$-th power of the modulus
and applying \eqref{eq:binomial-branch}
and Parseval gives
\[
\mathcal M_x(u)
=
(1-x)^{\alpha q}
\sum_{n=0}^{\infty}A_n(u)^2x^n,
\]
where
\[
A_n(u)
=
\sum_{k=0}^{n}
\frac{(q/2)_k}{k!}
\frac{(\alpha q/2)_{n-k}}{(n-k)!}
\e^{(2k-n)u}.
\]
Indeed, $x^{n/2}A_n(u)$ is the $n$-th
Taylor coefficient of
\[
(1-\sqrt{x}\e^uz)^{-q/2}
(1-\sqrt{x}\e^{-u}z)^{-\alpha q/2}.
\]

Since $q(1+\alpha)/2=1$, the
Chu--Vandermonde identity gives
\[
A_n(0)
=
\frac{(q(1+\alpha)/2)_n}{n!}
=1.
\]
By positivity of the coefficients and
$|2k-n|\le n$,
\(
0<A_n(u)\le\e^{nu}A_n(0)=\e^{nu}.\)
Consequently,
\[
\begin{aligned}
\mathcal M_x(u)
\le
(1-x)^{\alpha q}
\sum_{n=0}^{\infty}(x\e^{2u})^n=
(1-x)^{-a}
\frac{1-x}{1-x\e^{2u}}.
\end{aligned}
\]
Since $x\e^{2u}\le\tfrac12\e^{1/8}<1$,
\[
\log\frac{1-x}{1-x\e^{2u}}
=
\int_0^u
\frac{2x\e^{2v}}{1-x\e^{2v}}\dd v
\le Cxu.
\]
Thus
\[
\mathcal M_x(u)
\le
(1-x)^{-a}\e^{Cxu}.
\]

By independence of the prime coordinates,
\[
\begin{aligned}
\E\abs{
\e^{\vartheta D_{\mathrm{full}}/w}
f_h^0(t)
}^q=
\prod_{r\le\e^w}
\mathcal M_{1/r}\!
\left(\vartheta\frac{\log r}{w}\right)\le
M_{\e^w}^a
\exp\left(
\frac{C\vartheta}{w}
\sum_{r\le\e^w}\frac{\log r}{r}
\right)\le C_pM_{\e^w}^a,
\end{aligned}
\]
where the last step uses
\(
\sum_{r\le\e^w}\frac{\log r}{r}
\le Cw\)
from
\cite[Chapter~2, Theorem~2.7(b)]{MontgomeryVaughan06}.
This proves \eqref{eq:full-power-source-deformation}.

For the second assertion, let $\tau>0$
and $K\ge0$. The multiplier
\(
m_{K,\tau}(\lambda)
=
\e^{-\tau\lambda}\one_{(K,\infty)}(\lambda)\)
satisfies
\(
\operatorname{Var}(m_{K,\tau})
=
2\e^{-\tau K}.\)
Hence \eqref{eq:transferred-bv-functional-calculus}
gives
\[
\norm{
m_{K,\tau}(D_{\mathrm{full}})
}_{L^q\to L^q}
\le C_q\e^{-\tau K}.
\]

Since the Euler expansions remain absolutely
convergent after deformation, setting
$\tau=\vartheta/w$ gives the coefficientwise
identity
\[
P_{(K,\infty)}^{\mathrm{full}}f_h^0(t)
=
m_{K,\tau}(D_{\mathrm{full}})
\e^{\tau D_{\mathrm{full}}}f_h^0(t).
\]
Therefore,
\[
\begin{aligned}
\norm{
P_{(K,\infty)}^{\mathrm{full}}f_h^0(t)
}_{L^q}
&\le
C_q\e^{-\vartheta K/w}
\norm{
\e^{\vartheta D_{\mathrm{full}}/w}
f_h^0(t)
}_{L^q}\\
&\le
C_p\e^{-\vartheta K/w}
M_{\e^w}^{a/q}.
\end{aligned}
\]
Raising both sides to the $q$-th power
proves \eqref{eq:source-spectral-tail}.
\end{proof}

\subsection{The high-frequency term}
\label{subsec:high-source-tail}

We can now control the contribution from Euler-product
frequencies above $K_0=\pi N/2$ by combining
Proposition~\ref{prop:source-deformation}
with Lemma~\ref{lem:prime-sampling}.

For \(c\in\C^{\ZN}\), define the high-frequency prime-indexed array by
\begin{align}
 \bigl(\mathcal H_yc\bigr)_r
 ={}&
 M_y^{-1/p}r^{-1/2}
 P_{(K_0,\infty)}^{r^-}
 \sum_{j<N}
 c_j\e^{\ii t_j\log r}
 f_{r^-}(t_j),
 \qquad
 r\in\mathcal I_h.
 \label{eq:global-high-source-array}
\end{align}

\begin{proposition}[Estimate for the high-frequency components]
\label{prop:high-source-tail}
For every $c\in\C^{\ZN}$,
\begin{equation}
 \norm{\mathcal H_yc}_{L^q(\ell^2)}
 \le
 C_p\norm c_{\ell^q(\ZN)}.
 \label{eq:global-high-source-tail}
\end{equation}
The constant $C_p$ depends only on $p$.
\end{proposition}

\begin{proof}
Fix $0\le h\le H$, and put $w=2^h$.
Replacing $f_{r^-}$ by the common
Euler product $f_{\e^w}=f_h^0$ on $\mathcal I_h$,
define 
\begin{align}
\bigl(\mathcal H_{y,w}^{\mathrm{fr}}c\bigr)_r
={}&
M_y^{-1/p}r^{-1/2}
\sum_{j<N}
c_j\e^{\ii t_j\log r}
\left(
P_{(K_0,\infty)}^{\mathrm{full}}f_h^0
\right)(t_j).
\label{eq:frozen-high-source-array}
\end{align}
Applying \eqref{eq:prime-sampling-Lq}
pointwise gives
\begin{equation}
\begin{aligned}
\norm{\mathcal H_{y,w}^{\mathrm{fr}}c}_{\ell^2(\mathcal I_h)}
\le{}&
C_pM_y^{-1/p}
\left(\frac Nw\right)^{1/p}
\left(
\sum_{j<N}|c_j|^q
\left|
P_{(K_0,\infty)}^{\mathrm{full}}
f_h^0(t_j)
\right|^q
\right)^{1/q}.
\end{aligned}
\label{eq:frozen-high-source-sampling}
\end{equation}

Taking the $q$-th power, then expectation,
and applying \eqref{eq:source-spectral-tail}
with $\vartheta=1/16$ and
$K=K_0=\pi N/2$, we obtain
\begin{equation}
\begin{aligned}
\E\norm{\mathcal H_{y,w}^{\mathrm{fr}}c}_{\ell^2}^q
\le{}&
C_p\left(\frac Nw\right)^a
\left(\frac{M_{\e^w}}{M_y}\right)^a\e^{-q\pi N/(32w)}
\norm c_{\ell^q(\ZN)}^q,
\end{aligned}
\label{eq:frozen-high-source-before-mertens}
\end{equation}
where $a=q/p$.

By Mertens' product theorem,
\(
\frac{M_{\e^w}}{M_y}
\le C\frac wN,
 1\le w\le N;\)
see \cite[Chapter~2, Theorem~2.7(e)]
{MontgomeryVaughan06}.
Consequently,
\[
\left(\frac Nw\right)^a
\left(\frac{M_{\e^w}}{M_y}\right)^a
\le C_p.
\]
Taking $q$-th roots in
\eqref{eq:frozen-high-source-before-mertens}
gives
\[
\norm{\mathcal H_{y,w}^{\mathrm{fr}}c}_
{L^q(\ell^2(\mathcal I_h))}
\le
C_p\e^{-\pi N/(32w)}
\norm c_{\ell^q(\ZN)}.
\]

Since $r\le\e^w$, the conditioning and
spectral commutation identities
\eqref{eq:euler-source-conditioning}
and \eqref{eq:conditional-spectral-commutation}
give
\[
\E_{r^-}
P_{(K_0,\infty)}^{\mathrm{full}}f_h^0(t_j)
=
P_{(K_0,\infty)}^{r^-}f_{r^-}(t_j).
\]
Consequently,
\[
(\mathcal H_yc)_r
=
\E_{r^-}
\bigl(\mathcal H_{y,w}^{\mathrm{fr}}c\bigr)_r,
\qquad r\in\mathcal I_h.
\]
Proposition~\ref{prop:vector-conditional-expectations}
therefore yields
\begin{equation}
\norm{
\bigl((\mathcal H_yc)_r\bigr)_{r\in\mathcal I_h}
}_{L^q(\ell^2)}
\le
C_p\e^{-\pi N/(32w)}
\norm c_{\ell^q(\ZN)}.
\label{eq:live-high-source-band}
\end{equation}

Finally, the prime bands are disjoint, so
the triangle inequality gives
\[
\begin{aligned}
\norm{\mathcal H_yc}_{L^q(\ell^2)}
\le
\sum_{h=0}^{H}
\norm{
\bigl((\mathcal H_yc)_r\bigr)_{r\in\mathcal I_h}
}_{L^q(\ell^2)}\le
C_p\norm c_{\ell^q(\ZN)}
\sum_{h=0}^{H}
\e^{-\pi 2^{H-h}/32}\le
C_p\norm c_{\ell^q(\ZN)}.
\end{aligned}
\]
This proves \eqref{eq:global-high-source-tail}.
\end{proof}

\section{Frequency decomposition and the projected vector}
\label{sec:corrected-column}

To prove the projected-vector bound needed for the main theorem, we
combine the Euler-product estimates of
Sections~\ref{sec:euler-occupation} and~\ref{sec:frame-and-tails}
with the cyclic Fourier estimates of Section~\ref{sec:cyclic-analysis}.
The key is to compare the coefficient frequencies with Euler-product
frequencies shifted by \(\log r\).  When their difference has a fixed
sign and stays away from zero, the finite geometric-sum identity Lemma~\ref{lem:finite-cauchy-decomposition} and
multiplier estimates Lemma~\ref{lem:frozen-low-source-restoration} apply.  Where the frequencies can meet, we use
the common-shift moment estimate Proposition~\ref{prop:common-shift-occupation} and a conical square function \eqref{eq:cone-square-function} instead.

\subsection{The projected coefficients and the resonant term}
\label{subsec:corrected-column-definition}

Recall the discrete Fourier convention
\eqref{eq:cyclic-dft}.  For \(c\in\C^{\ZN}\), define the complementary
cyclic Fourier projections \(\Pi_-\) and \(\Pi_+\) by
\[
 \widehat{\Pi_-c}(\kappa)
 =
 \one_{\{\xi_\kappa\le0\}}\widehat c(\kappa),
 \qquad
 \widehat{\Pi_+c}(\kappa)
 =
 \one_{\{\xi_\kappa>0\}}\widehat c(\kappa),
 \qquad
 \kappa\in\Lambda_N.
\]

Recall that
\[
 \Lambda_N=\{-N/2,\ldots,N/2-1\},
 \qquad
 \xi_\kappa=2\pi\kappa.
\]

The two multiplier symbols are indicators of cyclic frequency
intervals.  Corollary~\ref{cor:cyclic-interval-projections} therefore
gives
\begin{equation}
 \norm{\Pi_-}_{\ell^q(\ZN)\to\ell^q(\ZN)}
 +
 \norm{\Pi_+}_{\ell^q(\ZN)\to\ell^q(\ZN)}
 \le C_q,
 \label{eq:positive-negative-projection-bounds}
\end{equation}
where \(C_q\) is independent of \(N\).

For \(c\in\C^{\ZN}\), define the projected vector
\(\mathcal T_yc\), indexed by the primes \(r\le y\), by
\begin{equation}
 (\mathcal T_yc)_r
 =
 r^{-1/2}M_y^{-1/p}
 P_{[0,\infty)}^{r^-}
 \sum_{j<N}
 c_j\e^{\ii t_j\log r}f_{r^-}(t_j),
 \qquad r\le y.
 \label{eq:corrected-column}
\end{equation}
Here \(f_{r^-}\) involves only primes smaller than \(r\), and
\(P_{[0,\infty)}^{r^-}\) retains nonnegative total logarithmic
frequency in those variables.
The factor $r^{-1/2}$ comes from the
first positive harmonic of the Euler factor at the prime $r$.
We must bound these contributions together, uniformly
over the primes.

For \(d\in\C^{\ZN}\), let
\(\mathcal R_y^{(\ell)}d\), \(\ell\ge-2\), denote the offset resonant
arrays defined in \eqref{eq:offset-resonant-row}, and set
\begin{equation}
 \mathcal R_yd
 =
 \sum_{\ell\ge-2}\mathcal R_y^{(\ell)}d.
 \label{eq:resonant-column-sum}
\end{equation}
The offset \(\ell\) compares a dyadic coefficient-frequency scale
with the scale of the primes in that component.  Only finitely many
terms in \eqref{eq:resonant-column-sum} are nonzero for fixed \(N\),
since \(d_v=0\) outside \(\mathcal V_N\).  We treat bounded and large positive offsets separately.

\subsection{Finite Cauchy calculus}
\label{subsec:finite-cauchy-calculus}

Now we shall show that expanding the sampled coefficient vector into cyclic Fourier modes
produces a finite geometric sum.  Its denominator depends on
\(\lambda+\log r-\xi_\kappa\), where \(\lambda\) is an Euler-product
frequency and \(\xi_\kappa\) is a coefficient frequency.

For each prime \(r\le y\), define
\begin{equation}
 \mathscr D_r=D_{\mathrm{full}}+(\log r)I.
 \label{eq:shifted-row-generator}
\end{equation}
On a vector \(F=(F_r)_r\), write
\((\mathscr D F)_r=\mathscr D_rF_r\).  Then the operators
\[
 \e^{\ii t\mathscr D}F
 =
 \left(
 \e^{\ii t\log r}\e^{\ii tD_{\mathrm{full}}}F_r
 \right)_r
\]
form an isometric group.

\begin{lemma}[Finite Cauchy decomposition]
\label{lem:finite-cauchy-decomposition} 
Let \(F\) be a trigonometric polynomial whose \(D_{\mathrm{full}}\)-spectrum is contained in \([0,K_0]\), and let \(c\in\C^{\ZN}\). Then 

\begin{equation} \sum_{j=0}^{N-1}c_j\e^{\ii j\mathscr D_r/N}F = \sum_{\kappa\in\Lambda_N} \widehat c(\kappa) \mathfrak D_N \left( \frac{\mathscr D_r-\xi_\kappa I}{N} \right)F, \label{eq:finite-cauchy-dirichlet-kernel} \end{equation}
where 
\[ \mathfrak D_N(z)=\sum_{j=0}^{N-1}\e^{\ii jz}. \] 
On every joint spectral piece with \(D_{\mathrm{full}}\)-frequency
\(\lambda\) on which \(\lambda+\log r-\xi_\kappa\ne0\), \begin{equation} \mathfrak D_N \left( \frac{\mathscr D_r-\xi_\kappa I}{N} \right) = \left( I-\e^{\ii \mathscr D_r} \right) \left( I-\e^{\ii (\mathscr D_r-\xi_\kappa I)/N} \right)^{-1}. \label{eq:finite-cauchy-common-numerator} \end{equation} 
The numerator in \eqref{eq:finite-cauchy-common-numerator} is independent of \(\kappa\) and has operator norm at most two. Moreover, if \(\lambda\in[0,K_0]\) is a
\(D_{\mathrm{full}}\)-spectral value, then on the complete low-frequency range
\[
 0<\log r\le N,
 \qquad
 -\pi N\le\xi_\kappa<\pi N,
\]
one has 
\begin{equation} -\pi < \frac{\lambda+\log r-\xi_\kappa}{N} \le \frac{3\pi}{2}+1 < 2\pi. \label{eq:global-alias-free-interval} \end{equation} 
Thus zero is the only possible singularity of the denominator on the actual range. Let 
\begin{equation} \mathfrak h(z) = (1-\e^{\ii z})^{-1}-\frac{\ii}{z}, \qquad \mathfrak h(0)=\frac12. 
\label{eq:regular-cauchy-function} 
\end{equation}
\noeqref{eq:regular-cauchy-function}
Then \(\mathfrak h\) is smooth on a neighborhood of the interval in \eqref{eq:global-alias-free-interval}, and, on every zero-free spectral piece, \begin{equation} \left( I-\e^{\ii (\mathscr D_r-\xi_\kappa I)/N} \right)^{-1} = \ii N(\mathscr D_r-\xi_\kappa I)^{-1} + \mathfrak h \left( \frac{\mathscr D_r-\xi_\kappa I}{N} \right). \label{eq:singular-regular-cauchy-split} \end{equation} 
\end{lemma} 

\begin{proof} By the inverse DFT, we have \[ c_j = \sum_{\kappa\in\Lambda_N} \widehat c(\kappa)\e^{-\ii j\xi_\kappa/N}. \] Substitution and finite summation give \[ \sum_jc_j\e^{\ii j\mathscr D_r/N}F = \sum_{\kappa\in\Lambda_N} \widehat c(\kappa) \sum_{j=0}^{N-1} \e^{\ii j(\mathscr D_r-\xi_\kappa I)/N}F, \] which is \eqref{eq:finite-cauchy-dirichlet-kernel}. 
 All multipliers are interpreted on the finite Fourier support of \(F\).
Let \(\lambda\) denote the \(D_{\mathrm{full}}\)-frequency of a Fourier
character occurring in \(F\).  If
\(\lambda+\log r-\xi_\kappa\ne0\), the range estimate
\eqref{eq:global-alias-free-interval}, verified below, ensures that
\[
 1-\e^{\ii (\lambda+\log r-\xi_\kappa)/N}\ne0.
\]
The finite geometric identity therefore gives, on these characters,
\[
 \mathfrak D_N
 \left(
 \frac{\mathscr D_r-\xi_\kappa I}{N}
 \right)
 =
 \left(
 I-\e^{\ii (\mathscr D_r-\xi_\kappa I)}
 \right)
 \left(
 I-\e^{\ii (\mathscr D_r-\xi_\kappa I)/N}
 \right)^{-1}.
\]
Since \[ \e^{-\ii \xi_\kappa} = \e^{-\ii 2\pi\kappa} = 1, \] the numerator is \[ I-\e^{\ii \mathscr D_r}. \] The operator \(\e^{\ii \mathscr D_r}\) is an isometry on \(L^q(\Omega_y)\),
and hence the numerator has norm at most two. The lower endpoint in \eqref{eq:global-alias-free-interval} follows from 
\[
 \lambda\ge0,
 \qquad
 \log r>0,
 \qquad
 \xi_\kappa<\pi N.
\]
The upper endpoint follows from 
\[
 \lambda\le K_0=\frac{\pi N}{2},
 \qquad
 \log r\le\log y=N,
 \qquad
 \xi_\kappa\ge-\pi N.
\]
Thus the actual range avoids every nonzero multiple of \(2\pi\). The expansion in \eqref{eq:singular-regular-cauchy-split} follows from the removable singularity of \(\mathfrak h\) at zero.
\end{proof}

\subsection{Kernel estimates and conditional expectations}
\label{subsec:cauchy-frozen-estimates}

When the frequency difference has a fixed sign, its reciprocal
can be written as a one-sided Laplace integral.  The sign determines
which exponential is decaying.  The resulting coefficient multipliers
are controlled by the following kernels on the cyclic group.

For \(t\ge0\) and \(w>0\), define
\begin{equation}
 \sigma_{t,w}(\kappa)
 =
 \e^{t\xi_\kappa}\chi(\xi_\kappa/w),
 \qquad
 J_{t,w}(j)
 =
 \frac1N\sum_{\kappa\in\Lambda_N}
 \sigma_{t,w}(\kappa)e^{-\mathrm{i}j\theta_\kappa}.
 \label{eq:positive-low-kernel}
\end{equation}
For \(w,v>0\), put
\begin{equation}
 a_{w,v}(\xi)
 =
 \phi(\xi/v)\bigl(1-\chi(\xi/w)\bigr),
 \label{eq:outside-shell-symbol}
\end{equation}
and define
\begin{align}
 J_{t,w,v}^{+}(j)
 &=
 \frac1N
 \mathcal F_N^{-1}
 \left(
 a_{w,v}(\xi_\kappa)\e^{t\xi_\kappa}
 \right)(j),
 \label{eq:outside-plus-kernel}
 \\
 J_{t,w,v}^{-}(j)
 &=
 \frac1N
 \mathcal F_N^{-1}
 \left(
 a_{w,v}(\xi_\kappa)\e^{-t\xi_\kappa}
 \right)(j).
 \label{eq:outside-minus-kernel}
\end{align}

\begin{lemma}[Estimates for Cauchy coefficient kernels]
\label{lem:cauchy-coefficient-kernels}
Let \(t\ge0\), \(1\le w\le N\), and \(v>0\). For every fixed integer \(M_0>2\),
\begin{equation}
 \abs{J_{t,w}(j)}
 \le
 C_{M_0}
 \frac wN
 (1+wt)^{M_0}
 \e^{wt/8}
 \left(
 1+\frac wN\rho_j
 \right)^{-M_0}.
 \label{eq:positive-low-kernel-bound}
\end{equation}
Moreover,
\begin{align}
 \abs{J_{t,w,v}^{+}(j)}
 &\le
 C
 \frac vN
 (1+vt)\e^{2vt}
 \left(
 1+\frac vN\rho_j
 \right)^{-1},
 \label{eq:outside-plus-kernel-bound}
 \\
 \abs{J_{t,w,v}^{-}(j)}
 &\le
 C
 \frac vN
 (1+vt)\e^{-vt/2}
 \left(
 1+\frac vN\rho_j
 \right)^{-1}.
 \label{eq:outside-minus-kernel-bound}
\end{align}

\end{lemma}

\begin{proof}
For any finitely supported sequence $b$ on $\Lambda_N$,
let $\widetilde b$ denote its zero extension to $\mathbb Z$.
Repeated finite Abel summation gives, for $m\ge1$ and
$\rho_j>0$,
\[
 \left|
 \frac1N(\mathcal F_N^{-1}b)(j)
 \right|
 \le
 \frac{C_mN^{m-1}}{\rho_j^m}
 \sum_{\kappa\in\mathbb Z}
 |\Delta^m\widetilde b(\kappa)|.
\]
For $\rho_j>0$, set $z=\e^{-2\pi\ii j/N}$. Since
$\widetilde b$ has finite support, we can repeatedly shift
the summation index to obtain
\[
 \sum_{k\in\Z}\Delta^m\widetilde b(k)z^k
=(z^{-1}-1)^m\sum_{k\in\Z}\widetilde b(k)z^k,
 \qquad m\ge1.
\]
Since $0<\rho_j\le N/2$, we also have
\[
 |z^{-1}-1|
 =2\sin\!\left(\frac{\pi\rho_j}{N}\right)
 \ge\frac{4\rho_j}{N}.
\]
Together with the trivial bound and the estimate at $j=0$,
the case $m=1$ yields
\begin{equation}
 \left|
 \frac1N(\mathcal F_N^{-1}b)(j)
 \right|
 \le
 \min\left\{
 \frac1N\sum_\kappa|b(\kappa)|,\,
 \frac{C}{1+\rho_j}
 \sum_\kappa|\Delta\widetilde b(\kappa)|
 \right\}.
 \label{eq:finite-Abel-one-difference}
\end{equation}

For \eqref{eq:positive-low-kernel}, set
\[
 f(x)=\e^{2\pi tx}\chi(2\pi x/w).
\]
The support of $\chi$ gives
\[
 \supp f\subset\left(-\frac{w}{16\pi},\frac{w}{16\pi}\right)
 \subset\left(-\frac N2,\frac N2\right).
\]
Thus \(f(k)=\sigma_{t,w}(k)\) for \(k\in\Lambda_N\)
and \(f(k)=0\) for \(k\in\Z\setminus\Lambda_N\). More explicitly, Leibniz's rule reads
\[
 f^{(m)}(x)
 =\e^{2\pi tx}\sum_{\ell=0}^{m}\binom m\ell
  (2\pi t)^\ell\left(\frac{2\pi}{w}\right)^{m-\ell}
  \chi^{(m-\ell)}(2\pi x/w).
\]
All summands are supported where $|2\pi x|<w/8$.
Since $|\operatorname{supp}f|\lesssim w$,
the Leibniz rule gives, for $0\le m\le M_0$,
\[
 \|f^{(m)}\|_\infty
 \le C_{M_0}w^{-m}(1+wt)^m\e^{wt/8}.
\]
Using the integral representation
\[
 \Delta^m f(\kappa)
 =
 \int_{[0,1]^m}
 f^{(m)}(\kappa+s_1+\cdots+s_m)
 \,ds_1\cdots ds_m,
\]
If $\Delta^m f(k)\ne0$, then
$[k,k+m]\cap\supp f\ne\varnothing$. Hence, for fixed
$0\le m\le M_0$, the number of possible integers $k$ obeys
\[
 \#\{k\in\Z:[k,k+m]\cap\supp f\ne\varnothing\}
 \le C_{M_0}(w+m+1)\le C_{M_0}w,
\]
using $w\ge1$. 
we obtain
\begin{equation}
 \sum_{\kappa\in\mathbb Z}
 |\Delta^m f(\kappa)|
 \le
 C_{M_0}w^{1-m}(1+wt)^m\e^{wt/8}.
 \label{eq:positive-low-symbol-differences}
\end{equation}
The trivial bound and the $m=M_0$ Abel estimate
therefore give \eqref{eq:positive-low-kernel-bound}.
Indeed, the two estimates before their interpolation are
\[
 \begin{aligned}
 |J_{t,w}(j)|
 &\le C\frac wN\e^{wt/8},\\
 |J_{t,w}(j)|
 &\le C_{M_0}\frac wN(1+wt)^{M_0}\e^{wt/8}
 \left(\frac{N}{w\rho_j}\right)^{M_0}
 \quad(\rho_j>0).
 \end{aligned}
\]
Writing $s=w\rho_j/N$ and using
$\min\{1,s^{-M_0}\}\le 2^{M_0}(1+s)^{-M_0}$
(for $s>0$) gives the claimed decay; the trivial bound
also covers $\rho_j=0$.

For \eqref{eq:outside-plus-kernel} and
\eqref{eq:outside-minus-kernel}, the cutoff
in \eqref{eq:outside-shell-symbol} is supported where
\begin{equation}
 \frac v2<\xi_\kappa<2v.
 \label{eq:outside-shell-support}
\end{equation}
Moreover,
\[
 \|a_{w,v}\|_\infty\le1,
 \qquad
 \operatorname{Var}(a_{w,v})\le C.
\]
To verify the variation bound, write \(a=a_{w,v}\).
The product rule gives
\[
 a'(\xi)
 =\frac1v\phi'(\xi/v)\bigl(1-\chi(\xi/w)\bigr)
  -\frac1w\phi(\xi/v)\chi'(\xi/w).
\]
Integrating the absolute values of the two terms, with
\(s=\xi/v\) in the first and \(s=\xi/w\) in the second, gives
\[
 \int_{\R}|a'(\xi)|\dd\xi
 \le \|\phi'\|_{L^1(\R)}
      +\|\phi\|_\infty\|\chi'\|_{L^1(\R)}
 \le C.
\]
The support condition \(\supp a\subset(v/2,2v)\), together
with \(\|a\|_\infty\le1\), also gives
\(\int_{\R}|a(\xi)|\dd\xi\le 3v/2\).
The support condition and bounded variation imply
\[
 \sum_\kappa
 |a_{w,v}(\xi_\kappa)\e^{\pm t\xi_\kappa}|
 \le CvE_\pm,
\]
and
\[
 \sum_\kappa
 \left|
 \Delta\left[
 a_{w,v}(\xi_\kappa)\e^{\pm t\xi_\kappa}
 \right]
 \right|
 \le C(1+vt)E_\pm,
\]
where the sequences are extended by zero and
\[
 E_+=\e^{2vt},
 \qquad
 E_-=\e^{-vt/2}.
\]
Let $g_\pm(\xi)=a_{w,v}(\xi)\e^{\pm t\xi}$ and $K=N/2-1$.
On $\supp a\subset(v/2,2v)$, we have $\e^{\pm t\xi}\le E_\pm$.
Since $\|a\|_\infty\le1$, this gives $\|g_\pm\|_\infty\le E_\pm$.
The product rule then gives
\[
 \int_{\R}|g_\pm'(\xi)|\dd\xi
 \le E_\pm\int_{\R}\bigl(|a'(\xi)|+t|a(\xi)|\bigr)\dd\xi
 \le C(1+vt)E_\pm.
\]
For $b_\pm(k)=g_\pm(2\pi k)$ on $\Lambda_N$ and its
zero extension $\widetilde b_\pm$, the fact that
$g_\pm(0)=0$ yields
\[
 \begin{aligned}
 \sum_{k\in\Z}|\Delta\widetilde b_\pm(k)|
 &=\sum_{k=0}^{K-1}|g_\pm(2\pi(k+1))-g_\pm(2\pi k)|
   +|g_\pm(2\pi K)|\\
 &\le\int_{\R}|g_\pm'(\xi)|\dd\xi+\|g_\pm\|_\infty\\
 &\le C(1+vt)E_\pm.
 \end{aligned}
\]
The final absolute-value term includes the possible
Nyquist-endpoint jump. Also, the number of nonzero sampled
frequencies is $O(v)$ (and is zero when $v\le\pi$), so
$\sum_k|b_\pm(k)|\le CvE_\pm$.

Applying \eqref{eq:finite-Abel-one-difference}
to the inverse Fourier transforms in
\eqref{eq:outside-plus-kernel} and
\eqref{eq:outside-minus-kernel} gives
\[
 |J_{t,w,v}^\pm(j)|
 \le
 C(1+vt)E_\pm
 \min\left\{
 \frac vN,\frac1{1+\rho_j}
 \right\}
 \le
 C\frac vN(1+vt)E_\pm
 \left(1+\frac{v\rho_j}{N}\right)^{-1}.
\]
These are precisely
\eqref{eq:outside-plus-kernel-bound} and
\eqref{eq:outside-minus-kernel-bound}.
\end{proof}

For primes in one dyadic band, we first work with a single Euler
product truncated at the upper end of that band.  This lets us
estimate all the components together, before conditional expectation
recovers the smaller-prime expression \(X_r^{\mathrm{live}}\) in
\eqref{eq:generic-live-row}.  The normalization contributes the factor
\((w/N)^{1/p}\) below.

For \(0\le h\le H\), we write
 $w=2^{h}$ and define
\begin{equation}
 F_w
 =
 P_{[0,K_0]}^{\mathrm{full}}f_{e^{w}}(t_0).
 \label{eq:frozen-low-source}
\end{equation}

\begin{lemma}[Low-frequency estimates and conditional recovery]
\label{lem:frozen-low-source-restoration}
For every \(0\le h\le H\),
\begin{equation}
 \norm{
 M_y^{-1/p}
 (r^{-1/2}F_w)_{r\in\mathcal I_h}
 }_{L^q(\ell^2)}
 \le
 C_p\left(\frac wN\right)^{1/p}.
 \label{eq:frozen-low-source-scale}
\end{equation}
More generally, let \(u_{r,j}\) be deterministic coefficients and let \(m_r\) be
spectral multiplier symbols for which the following expressions are defined.
For \(r\in\mathcal I_h\), put
\begin{align}
 X_r^{\mathrm{fr}}
 &=
 M_y^{-1/p}r^{-1/2}
 m_r(D_{\mathrm{full}})
 \sum_{j<N}u_{r,j}f_{e^{w}}(t_j),
 \label{eq:generic-frozen-row}
 \\
 X_r^{\mathrm{live}}
 &=
 M_y^{-1/p}r^{-1/2}
 m_r(D_{r^-})
 \sum_{j<N}u_{r,j}f_{r^-}(t_j).
 \label{eq:generic-live-row}
\end{align}
Then
\begin{equation}
 X_r^{\mathrm{live}}
 =
 \E_{r^-}X_r^{\mathrm{fr}},
 \label{eq:frozen-live-conditional-identity}
\end{equation}
and
\begin{equation}
 \norm{
 (X_r^{\mathrm{live}})_r
 }_{L^q(\ell^2)}
 \le
 C_p
 \norm{
 (X_r^{\mathrm{fr}})_r
 }_{L^q(\ell^2)}.
 \label{eq:frozen-live-restoration}
\end{equation}
\end{lemma}

\begin{proof}
Proposition~\ref{prop:row-interval-projections}, applied to a single
spectral interval, gives
\[
 \norm{F_w}_q
 \le
 C_q\norm{f_{e^{w}}(t_0)}_q
 =
 C_qM_{e^{w}}^{1/p}
\]
by \eqref{eq:euler-source-normalization}.  Since
\(
 \sum_{r\in\mathcal I_h}\frac1r\le C,
\)
we obtain
\[
 \norm{
 M_y^{-1/p}
 (r^{-1/2}F_w)_{r\in\mathcal I_h}
 }_{L^q(\ell^2)}
 \le
 C_p
 \left(
 \frac{M_{e^{w}}}{M_{\e^N}}
 \right)^{1/p}.
\]
Mertens' product theorem gives
\[
 \frac{M_{e^{w}}}{M_{\e^N}}
 \le
 C\frac wN;
\]
see \cite[Chapter~2, Theorem~2.7(e)]{MontgomeryVaughan06}.
This proves \eqref{eq:frozen-low-source-scale}.

The identity
\eqref{eq:frozen-live-conditional-identity} follows from
\eqref{eq:conditional-spectral-commutation} and
\eqref{eq:euler-source-conditioning}.  Applying
Proposition~\ref{prop:vector-conditional-expectations} proves
\eqref{eq:frozen-live-restoration}.
\end{proof}

The finite geometric-sum identity~\eqref{eq:finite-cauchy-common-numerator} also has a regular term.  Once
its singular part \(\mathrm{i}/z\) has been subtracted, the remaining function
\(\mathfrak h\) is smooth on the corresponding frequency interval.  It can
therefore be bounded by Fourier inversion and the same frequency
estimate.

Choose once and for all
\[
 \widetilde{\mathfrak h}\in C_c^\infty(\R)
\]
which agrees with \(\mathfrak h\) on a neighborhood of the interval in
\eqref{eq:global-alias-free-interval}.  We use the Fourier convention
\[
 \widetilde{\mathfrak h}(z)
 =
 \int_{\R}
 \widehat{\widetilde{\mathfrak h}}(\tau)
 \e^{\ii \tau z}\dd\tau.
\]

\begin{lemma}[Estimate for the regular Cauchy remainder]
\label{lem:regular-cauchy-remainder}
Fix a prime scale \(w=2^{h}\), and let
\(\sigma\colon\Lambda_N\to\C\) be a coefficient symbol satisfying
\begin{equation}
 \norm{
 \mathcal F_N^{-1}
 \left(
 \sigma(\kappa)\e^{-\ii \tau\xi_\kappa/N}
 \widehat d(\kappa)
 \right)
 }_{\ell^q(\ZN)}
 \le
 A(1+\abs\tau)\norm d_q
 \label{eq:regular-coefficient-hypothesis}
\end{equation}
for every \(\tau\in\R\).  Let \(\mathcal P=(\mathcal P_r)_{r\in\mathcal
I_h}\) be a componentwise spectral projection, or a sum of a fixed number of such
projections, uniformly bounded on \(L^q(\ell^2)\).

Define
\begin{align}
 (\mathcal E_{\sigma,w}d)_r
 =
 M_y^{-1/p}r^{-1/2}
 \mathcal P_r
 \left(
 I-\e^{\ii \mathscr D_r}
 \right)
\sum_{\kappa\in\Lambda_N}
 \sigma(\kappa)\widehat d(\kappa)
 \widetilde{\mathfrak h}
 \left(
 \frac{\mathscr D_r-\xi_\kappa I}{N}
 \right)F_w.
 \label{eq:regular-cauchy-row}
\end{align}\noeqref{eq:regular-cauchy-row} 
Then
\begin{equation}
 \norm{
 \mathcal E_{\sigma,w}d
 }_{L^q(\ell^2)}
 \le
 C_pA
 \left(\frac wN\right)^{1/p}
 \norm d_q.
 \label{eq:regular-cauchy-bound}
\end{equation}
\end{lemma}

\begin{proof}
Fourier inversion gives the joint functional-calculus identity
\begin{align}
 \widetilde{\mathfrak h}
 \left(
 \frac{\mathscr D_r-\xi_\kappa I}{N}
 \right)
 =
 \int_{\R}
 \widehat{\widetilde{\mathfrak h}}(\tau)
 \e^{\ii \tau\mathscr D_r/N}
 \e^{-\ii \tau\xi_\kappa/N}
 \dd\tau.
 \label{eq:joint-regular-functional-calculus}
\end{align}\noeqref{eq:joint-regular-functional-calculus}
For each \(\tau\), the operator
\[
 (G_r)_r
 \longmapsto
 \left(
 \e^{\ii \tau\mathscr D_r/N}G_r
 \right)_r
\]
is an isometry.  The coefficient factor at the cyclic index \(0\) is
\[
 \mathcal F_N^{-1}
 \left(
 \sigma(\kappa)\e^{-\ii \tau\xi_\kappa/N}
 \widehat d(\kappa)
 \right)(0),
\]
whose modulus is bounded by the complete \(\ell^q\)-norm in
\eqref{eq:regular-coefficient-hypothesis}.  The common numerator has norm
at most two, and \(\mathcal P\) is uniformly bounded.  
Hence, Minkowski's inequality and
Lemma~\ref{lem:frozen-low-source-restoration} prove \eqref{eq:regular-cauchy-bound}, since \(\widetilde{\mathfrak h}\) is smooth and compactly supported.
\end{proof}

\begin{remark}[Coefficient symbols in the nonresonant estimates]
We verify that the hypothesis
\eqref{eq:regular-coefficient-hypothesis}
holds uniformly for the three coefficient symbols
appearing in the nonresonant estimates:
\[
\begin{aligned}
\sigma^{(1)}(\kappa)
&=\one_{\{\xi_\kappa\le0\}},\\
\sigma^{(2)}(\kappa)
&=\one_{\{\xi_\kappa>0\}}\chi(\xi_\kappa/w),\\
\sigma^{(3)}(\kappa)
&=a_{w,v}(\xi_\kappa)
 =\phi(\xi_\kappa/v)
  \bigl(1-\chi(\xi_\kappa/w)\bigr).
\end{aligned}
\]

We apply Lemma~\ref{lem:cyclic-multiplier}.
Let $\widetilde\sigma$ be the zero
extension defined in
\eqref{eq:nyquist-free-zero-extension}.
We first show that all three symbols satisfy
\[
\|\sigma\|_\infty+
\sum_{k\in\mathbb Z}
|\Delta\widetilde\sigma(k)|
\le C,
\]
uniformly in $N,w,v$.

Indeed, $\sigma^{(1)}$ is constant on the nonpositive
frequencies and has exactly two jumps after
zero extension, so its total variation is $2$.
That is,
\[
 \widetilde\sigma^{(1)}(k)
 =\one_{\{-N/2+1\le k\le0\}}.
\]
This sequence has two jumps, each of size one, so
\[
 \sum_{k\in\Z}
 \bigl|\Delta\widetilde\sigma^{(1)}(k)\bigr|=2.
\]

For $\sigma^{(2)}$, the jump at the origin and
the fundamental theorem of calculus give
\[
\sum_k|\Delta\widetilde\sigma^{(2)}(k)|
\le
|\chi(0)|+\int_{\mathbb R}|\chi'(s)|\,ds
\le 1+\|\chi'\|_{L^1}.
\]
For clarity, if $N\ge4$ put $K=N/2-1$ and
$f_w(x)=\chi(2\pi x/w)$. Since $w\le N$ and the cutoff
is supported in $(-1/8,1/8)$, we have $f_w(K)=0$.
The first difference of the zero extension includes the
jump at the origin and the upper endpoint:
\[
 \begin{aligned}
 \sum_{k\in\Z}|\Delta\widetilde\sigma^{(2)}(k)|
 &=|f_w(1)|+\sum_{k=1}^{K-1}|f_w(k+1)-f_w(k)|+|f_w(K)|\\
 &\le |f_w(0)|+\int_0^{\infty}|f_w'(x)|\dd x\\
 &\le |\chi(0)|+\|\chi'\|_{L^1(\R)}.
 \end{aligned}
\]
For $N=2$, there are no positive represented frequencies.

For $\sigma^{(3)}$,
the Leibniz rule yields
\[
\begin{aligned}
a_{w,v}'(\xi)
={}&\frac1v\phi'(\xi/v)
       (1-\chi(\xi/w))-\frac1w\phi(\xi/v)\chi'(\xi/w).
\end{aligned}
\]
To pass from the continuous derivative to the discrete
variation, use $a(0)=0$ and $K=N/2-1$:
\[
 \begin{aligned}
 \sum_{k\in\Z}|\Delta\widetilde\sigma^{(3)}(k)|
 &\le\sum_{k=0}^{K-1}
    \int_{2\pi k}^{2\pi(k+1)}|a_{w,v}'(\xi)|\dd\xi
    +\|a_{w,v}\|_\infty\\
 &\le\int_0^\infty|a_{w,v}'(\xi)|\dd\xi
    +\|\phi\|_\infty\\
 &\le\|\phi'\|_{L^1}
    +\|\phi\|_\infty\|\chi'\|_{L^1}
    +\|\phi\|_\infty.
 \end{aligned}
\]
This includes the possible terminal frequency jump;
for $N=2$ the third multiplier vanishes identically.
We obtain
\[
\sum_k|\Delta\widetilde\sigma^{(3)}(k)|
\le
\|\phi'\|_{L^1}
+\|\phi\|_\infty\|\chi'\|_{L^1}
+2\|\phi\|_\infty
\le C.
\]
Now fix any of these symbols $\sigma$ and set
\[
m_\tau(\kappa)
=\sigma(\kappa)e^{-\ii\tau\xi_\kappa/N}.
\]
The corresponding zero extension satisfies
$\widetilde m_\tau(k)=\widetilde\sigma(k)
\e^{-2\pi\ii\tau k/N}$. Its discrete Leibniz
identity is
\[
 \Delta\widetilde m_\tau(k)
=[\Delta\widetilde\sigma(k)]
   \e^{-2\pi\ii\tau(k+1)/N}
+\widetilde\sigma(k)
   \bigl(\e^{-2\pi\ii\tau(k+1)/N}
        -\e^{-2\pi\ii\tau k/N}\bigr).
\]
Since $\xi_\kappa=2\pi\kappa$ and
$|e^{\ii x}-e^{\ii y}|\le |x-y|$, we have
\[
\left|
e^{-\ii\tau\xi_{\kappa+1}/N}
-e^{-\ii\tau\xi_\kappa/N}
\right|
\le \frac{2\pi|\tau|}{N}.
\]
Consequently,
\[
\sum_\kappa
|\Delta\widetilde m_\tau(\kappa)|
\le
\sum_\kappa|\Delta\widetilde\sigma(\kappa)|+
\frac{2\pi|\tau|}{N}
\sum_\kappa|\widetilde\sigma(\kappa)|
\le C+2\pi\|\sigma\|_\infty|\tau|
\le C(1+|\tau|),
\]
Here one uses the fact that the zero extension has at most
$N-1$ nonzero indices, so we get
\[
 \sum_{k\in\Z}|\widetilde\sigma(k)|
 \le (N-1)\|\sigma\|_\infty
 \le N\|\sigma\|_\infty,
 \qquad
 \|m_\tau\|_\infty=\|\sigma\|_\infty.
\]
Thus both parts of the cyclic multiplier norm are controlled.
This proves
\[
\mathfrak M_N(m_\tau)\le C(1+|\tau|).
\]
Here $\mathfrak M_N$ is defined in
\eqref{eq:cyclic-marcinkiewicz-norm}.
Consequently, that lemma gives
\[
\left\|
\mathcal F_N^{-1}
\left(
\sigma(\kappa)e^{-\ii\tau\xi_\kappa/N}
\widehat d(\kappa)
\right)
\right\|_{\ell^q(\mathbb Z_N)}
\le
C_q(1+|\tau|)\|d\|_q.
\]
This is precisely
\eqref{eq:regular-coefficient-hypothesis}.
Thus Lemma~\ref{lem:regular-cauchy-remainder}
applies to all three regular contributions
in the nonresonant decomposition.
\end{remark}

\subsection{The exact frequency partition}
\label{subsec:exact-frequency-partition}

We split the coefficient vector into nonpositive and positive cyclic
frequencies.  For primes of scale \(w\), a smooth cutoff separates
small positive coefficient frequencies from dyadic frequency shells
of scale \(v\).  On each shell we compare the coefficient frequency
with the Euler-product frequency shifted by \(\log r\).  The two
ranges either remain separated or meet in the interval selected below.

For a positive-spectrum vector
\(
 e=\Pi_+e
\)
and a prime scale \(w=2^{h}\), we define
\begin{align}
 e_w^\chi
 &=
 \mathcal F_N^{-1}
 \left(
 \one_{\{\xi_\kappa>0\}}
 \chi(\xi_\kappa/w)\widehat e(\kappa)
 \right),
 \label{eq:positive-low-coefficient-piece}
 \\
 e_{w,v}
 &=
 \mathcal F_N^{-1}
 \left(
 \phi(\xi_\kappa/v)
 \bigl(1-\chi(\xi_\kappa/w)\bigr)
 \widehat e(\kappa)
 \right),
 \qquad
 v\in\mathcal V_N.
 \label{eq:outside-coefficient-piece}
\end{align}
Since
\[
 \sum_{v\in2^\Z}\phi(\xi/v)=1,
 \qquad \xi>0,
\]
one has
\begin{equation}
 e=e_w^\chi+\sum_{v\in\mathcal V_N}e_{w,v}.
 \label{eq:positive-coefficient-partition}
\end{equation}

For \(r\in\mathcal I_h\), define the low-frequency operator on an
arbitrary vector \(d\in\C^{\ZN}\), and the positive-low operator on a
positive-spectrum vector \(e=\Pi_+e\), by
\begin{align}
 (\mathcal T_y^{\mathrm{lo}}d)_r
 ={}&
 M_y^{-1/p}r^{-1/2}
 P_{[0,K_0]}^{r^-}
 \sum_{j<N}
 d_j\e^{\ii t_j\log r}f_{r^-}(t_j),
 \label{eq:low-source-column}
 \\
 (\mathcal L_y^{\mathrm{lo}}e)_r
 ={}&
 M_y^{-1/p}r^{-1/2}
 P_{[0,K_0]}^{r^-}
 \sum_{j<N}
 e_w^\chi(j)\e^{\ii t_j\log r}f_{r^-}(t_j),
 \label{eq:positive-low-column}
\end{align}
and
\begin{align}
 (\mathcal O_ye)_r
 =
 M_y^{-1/p}r^{-1/2}
 \sum_{v\in\mathcal V_N}
 P_{
 [0,K_0]\setminus[v/4-\log r,\,4v-\log r]
 }^{r^-}
\sum_{j<N}
 e_{w,v}(j)\e^{\ii t_j\log r}f_{r^-}(t_j).
 \label{eq:outside-column}
\end{align}
The spectral set in \eqref{eq:outside-column} is the complement of the closed interval \([v/4-\log r,\,4v-\log r]\) within \([0,K_0]\).

The resulting identity has five contributions: the high Euler-product
frequencies, the nonpositive coefficient frequencies, the small positive
frequencies, the positive shells outside the middle interval, and the
remaining middle intervals.  The last contribution is the resonant
term \(\mathcal R_y\Pi_+c\), where the shifted Euler-product and
coefficient frequencies can be comparable.

\begin{proposition}[Frequency decomposition]
\label{prop:exact-frequency-partition}
For every \(c\in\C^{\ZN}\),
\begin{align}
 \mathcal T_yc
 =
 \mathcal H_yc
 +
 \mathcal T_y^{\mathrm{lo}}(\Pi_-c)
 +
 \mathcal L_y^{\mathrm{lo}}(\Pi_+c)
+
 \mathcal O_y(\Pi_+c)
 +
 \mathcal R_y(\Pi_+c).
 \label{eq:exact-frequency-partition}
\end{align}
\end{proposition}

\begin{proof}
We fix \(r\in\mathcal I_h\), write \(w=2^h\), and set
\(e=\Pi_+c\). We prove the identity componentwise.

First, split
\[
P_{[0,\infty)}^{r^-}
=
P_{(K_0,\infty)}^{r^-}
+
P_{[0,K_0]}^{r^-},
\qquad
c=\Pi_-c+e.
\]
By the definitions, the high-frequency contribution is
\(\mathcal H_yc\), while the nonpositive coefficient
frequencies contribute
\(\mathcal T_y^{\mathrm{lo}}(\Pi_-c)\).

For the positive-frequency part, use
\eqref{eq:positive-coefficient-partition}:
\[
e=e_w^\chi+\sum_{v\in\mathcal V_N}e_{w,v}.
\]
At each represented positive frequency $\xi_\kappa$,
the coefficient identity is, more explicitly,
\[
 \widehat e(\kappa)
 =\left[\chi(\xi_\kappa/w)
  +\sum_{v\in\mathcal V_N}
    \phi(\xi_\kappa/v)
      \bigl(1-\chi(\xi_\kappa/w)\bigr)\right]
    \widehat e(\kappa).
\]
Every frequency meets only finitely many dyadic supports.

The first term gives \(\mathcal L_y^{\mathrm{lo}}e\).
For each dyadic shell \(v\), put
\[
I_{r,v}
=
[0,K_0]\cap[v/4-\log r,\,4v-\log r].
\]
The spectral partition
\[
P_{[0,K_0]}^{r^-}
=
P_{I_{r,v}}^{r^-}
+
P_{[0,K_0]\setminus I_{r,v}}^{r^-}
\]
separates the shell contribution into its
resonant and nonresonant parts.
Summing the complementary projections gives
\(\mathcal O_ye\).

It remains to identify the resonant part.
If \(e_{w,v}\ne0\), there is a positive coefficient
frequency \(\xi\) such that
\[
1-\chi(\xi/w)\ne0,
\qquad
\phi(\xi/v)\ne0.
\]
By the supports of \(\chi\) and \(\phi\),
\(\xi>w/16\) and \(\xi<2v\). Consequently,
\begin{equation}
w<32v.
\label{eq:active-shell-lower-bound}
\end{equation}
Moreover, if \(v\le w/8\), then on the retained
spectral range \(0\le\lambda\le K_0\),
\[
\lambda+\log r>\frac w2\ge4v,
\]
since \(r\in\mathcal I_h\). Hence \(I_{r,v}\)
is empty, and this shell has no resonant
contribution.
Every surviving resonant shell therefore satisfies
\(v>w/8\). Since \(v\) and \(w\) are dyadic,
\[
v=2^{h+\ell},
\qquad \ell\ge-2.
\]
On the support of \(\phi(\xi/v)\), we have
\[
\xi>\frac v2\ge\frac w8,
\]
so \(\chi(\xi/w)=0\). Thus \(e_{w,v}\)
coincides with the coefficient shell \(e_v\)
defined in \eqref{eq:coefficient-shell}.
Indeed, throughout this shell support,
\[
 \phi(\xi/v)\bigl(1-\chi(\xi/w)\bigr)
 =\phi(\xi/v),
 \qquad \xi>v/2\ge w/8.
\]
Hence the smooth cutoff disappears entirely from the
middle contribution, rather than merely being bounded.
Furthermore,
\[
I_{r,v}
=
[0,K_0]\cap
\left[
\frac{2^{h+\ell}}4-\log r,\,
4\cdot2^{h+\ell}-\log r
\right]
=J_r^{(\ell)}
\]
by \eqref{eq:resonant-row-interval}.
Therefore the sum of all resonant shell
contributions is exactly
\[
\sum_{\ell\ge-2}\mathcal R_y^{(\ell)}e
=
\mathcal R_ye
\]
by \eqref{eq:offset-resonant-row} and
\eqref{eq:resonant-column-sum}.
The sum is finite in each component: by convention the
coefficient shell is zero when $v\notin\mathcal V_N$.
The change of variables $v=2^{h+\ell}$ therefore produces
exactly the indicated $\ell\ge-2$ sum.

Combining these five contributions proves
\eqref{eq:exact-frequency-partition}.
\end{proof}

\subsection{Nonresonant term estimates}
\label{subsec:nonresonant-columns}

Here is the reason for the three simpler bounds.  Write
\(\Delta=\lambda+\log r-\xi_\kappa\), with \(\lambda\ge0\) and
\(r\) in the prime band of scale \(w\).  If \(\xi_\kappa\le0\),
then \(\Delta>w/2\).  For the small positive coefficient
frequencies, \(0<\xi_\kappa\le w/8\), so \(\Delta>3w/8\).

For a shell outside the middle interval, one has
\(v/2<\xi_\kappa<2v\) and either
\(\lambda+\log r<v/4\) or \(\lambda+\log r>4v\).
Thus \(\Delta<-v/4\) or \(\Delta>2v\), respectively.
These fixed-sign gaps justify the convergent Laplace integrals used
below.

\begin{proposition}[Nonresonant term estimates]
\label{prop:nonresonant-columns}
For every \(c\in\C^{\ZN}\),
\begin{equation}
 \norm{
 \mathcal T_y^{\mathrm{lo}}(\Pi_-c)
 }_{L^q(\ell^2)}
 \le
 C_p\norm c_q.
 \label{eq:negative-column-bound}
\end{equation}
For every positive-spectrum vector \(d=\Pi_+d\),
\begin{equation}
 \norm{
 \mathcal L_y^{\mathrm{lo}}d
 }_{L^q(\ell^2)}
 +
 \norm{
 \mathcal O_yd
 }_{L^q(\ell^2)}
 \le
 C_p\norm d_q.
 \label{eq:positive-nonresonant-bound}
\end{equation}
\end{proposition}

\begin{proof}
By \eqref{eq:frozen-live-conditional-identity}--
\eqref{eq:frozen-live-restoration}, it suffices to prove the corresponding
full-product estimates. 
The identities below are verified on Fourier characters and summed
using the absolute convergence of the Fourier expansion of \(F_w\);
the subsequent estimates ensure convergence of the Laplace integrals.

Fix \(w=2^h\).  For \(r\in\mathcal I_h\),
\(\mathscr D_r=D_{\mathrm{full}}+(\log r)I\) and \(\log r>w/2\).
By \eqref{eq:euler-source-flow}, the corresponding component is
\[
 X_r^{\mathrm{fr}}
 =
 M_y^{-1/p}r^{-1/2}\e^{\ii t_0\log r}
 \sum_{j<N}(\Pi_-c)_j\e^{\ii j\mathscr D_r/N}F_w.
\]
Since \(F_w\) has nonnegative spectrum and
\(\log r-\xi_\kappa>w/2\) for \(\xi_\kappa\le0\),
Lemma~\ref{lem:finite-cauchy-decomposition} gives the singular term
\begin{equation}
 \ii N(I-\e^{\ii \mathscr D_r})
 \int_0^\infty
 \e^{-t\log r}(\mathsf P_{Nt}^N*\Pi_-c)(0)
 \e^{-tD_{\mathrm{full}}}F_w\,\dd t.
 \label{eq:negative-sector-resolvent}
\end{equation}
Using \eqref{eq:transferred-damped-low-bound} and
Lemma~\ref{lem:frozen-low-source-restoration}, the prime-band norm is
bounded by
\begin{equation}
 CN\left(\frac wN\right)^{1/p}
 \int_0^\infty\e^{-wt/2}
 \abs{(\mathsf P_{Nt}^N*\Pi_-c)(0)}\,\dd t
 =C\widetilde B_{N/w}(\Pi_-c),
 \label{eq:negative-sector-trace-bound}
\end{equation}
where \(\rho=Nt\).  The regular term, with
\(\sigma(\kappa)=\one_{\{\xi_\kappa\le0\}}\), has norm at most
\(C_p(w/N)^{1/p}\norm c_q\).
Disjointness of the prime bands and \(q/2<1\) allow summation of these
bounds in \(\ell^q\).  Proposition~\ref{prop:boundary-traces},
\eqref{eq:positive-negative-projection-bounds}, and
\(\sum_{h=0}^H(2^h/N)^a\le C_p\) prove
\eqref{eq:negative-column-bound}.

Let \(d=\Pi_+d\).  On the support of
\(\chi(\xi_\kappa/w)\widehat d(\kappa)\),
\(0<\xi_\kappa\le w/8\), hence \(\log r-\xi_\kappa>3w/8\).
The singular term is therefore
\begin{equation}
 \ii N(I-\e^{\ii \mathscr D_r})
 \int_0^\infty
 \e^{-t\log r}(J_{t,w}*d)(0)
 \e^{-tD_{\mathrm{full}}}F_w\,\dd t.
 \label{eq:positive-low-singular-term}
\end{equation}
By \eqref{eq:positive-low-kernel-bound} and the same norm bound,
its prime-band norm is at most
\begin{align}
 CN\left(\frac wN\right)^{1/p}\frac wN
 \int_0^\infty(1+wt)^{M_0}\e^{-3wt/8}\,\dd t
 \sum_j\left(1+\frac wN\rho_j\right)^{-M_0}|d_j|
 \le C D_{N/w}(d).
 \label{eq:positive-low-trace-bound}
\end{align}
The regular term, with
\(\sigma(\kappa)=\one_{\{\xi_\kappa>0\}}\chi(\xi_\kappa/w)\),
has norm at most \(C_p(w/N)^{1/p}\norm d_q\).
The same argument, together with
Proposition~\ref{prop:boundary-traces}, gives
\begin{equation}
 \norm{\mathcal L_y^{\mathrm{lo}}d}_{L^q(\ell^2)}
 \le C_p\norm d_q.
 \label{eq:positive-low-column-bound}
\end{equation}

To estimate \(\mathcal O_yd\), define \((\mathcal JF)_r=\zeta_rF_r\).
For every bounded Borel function \(m\colon\R\to\C\), the following identity holds on finite vectors of trigonometric polynomials:
\begin{equation}
 \mathcal J^{-1}
 m(D_{\mathrm{full}})
 \mathcal J
 =
 m(\mathscr D),
 \qquad
 \mathscr D_r=D_{\mathrm{full}}+(\log r)I.
 \label{eq:shifted-row-conjugation}
\end{equation}
\noeqref{eq:shifted-row-conjugation}
Since \(\mathcal J\) is an isometry on \(L^q(\ell^2)\),
\eqref{eq:shifted-row-conjugation} transfers the interval bounds
of Proposition~\ref{prop:row-interval-projections} to \(\mathscr D\). Thus
\begin{align}
 \norm{\e^{-t\mathscr D}
 \one_{(4v,\infty)}(\mathscr D)F}_{L^q(\ell^2)}
 &\le C_q\e^{-4vt}\norm F_{L^q(\ell^2)},
 \label{eq:upper-outside-row-decay}\\
 \norm{\e^{t\mathscr D}
 \one_{(-\infty,v/4)}(\mathscr D)F}_{L^q(\ell^2)}
 &\le C_q\e^{vt/4}\norm F_{L^q(\ell^2)}.
 \label{eq:lower-outside-row-growth}
\end{align}\noeqref{eq:lower-outside-row-growth}\noeqref{eq:upper-outside-row-decay}
On the support of \(a_{w,v}\), \(v/2<\xi_\kappa<2v\). The singular
pieces have common prefactor \(\ii N(I-\e^{\ii \mathscr D_r})\) and
respective factors
\begin{align*}
 &\int_0^\infty
 (J^+_{t,w,v}*d)(0)\e^{-t\mathscr D_r}
 \one_{(4v,\infty)}(\mathscr D_r)F_w\,\dd t,\\
- &\int_0^\infty
 (J^-_{t,w,v}*d)(0)\e^{t\mathscr D_r}
 \one_{(-\infty,v/4)}(\mathscr D_r)F_w\,\dd t.
\end{align*}

The inequalities \eqref{eq:outside-plus-kernel-bound}--\eqref{eq:outside-minus-kernel-bound}
leave the factors
\((1+vt)\e^{-2vt}\) and \((1+vt)\e^{-vt/4}\) respectively,
each with integral \(O(v^{-1})\).
Thus the low-frequency estimate gives the bound
\begin{equation}
 CN\left(\frac wN\right)^{1/p}\frac vN\frac1v
 \sum_j\left(1+\frac vN\rho_j\right)^{-1}|d_j|
 =C\left(\frac wv\right)^{1/p}B_{N/v}(d).
 \label{eq:outside-singular-pair-bound}
\end{equation}

Writing \(w=2^h\) and \(v=2^k\),
\eqref{eq:active-shell-lower-bound} gives \(k\ge h-5\).
Since \(\sum_{n\ge-5}2^{-n/p}<\infty\),  Young's
inequality and Proposition~\ref{prop:boundary-traces} give
\begin{equation}
\sum_{h=0}^H
 \left(
 \sum_{\substack{k\ge h-5\\2^k\in\mathcal V_N}}
 2^{(h-k)/p}B_{N/2^k}(d)
 \right)^q
\le
C_p\sum_{\substack{k\in\Z\\2^k\in\mathcal V_N}}
 B_{N/2^k}(d)^q
 \le C_p\norm d_q^q.
 \label{eq:outside-singular-scale-sum}
\end{equation}
\noeqref{eq:outside-singular-scale-sum}

For the regular part, applying
Lemma~\ref{lem:regular-cauchy-remainder}
with \(\sigma(\kappa)=a_{w,v}(\xi_\kappa)\)
and the two componentwise spectral projections
in \eqref{eq:upper-outside-row-decay} and
\eqref{eq:lower-outside-row-growth}
gives \(C_p(w/N)^{1/p}\norm d_q\) per shell.
The range \(w/32<v\le4N\) contains
\(O(1+\log(N/w))\) dyadic shells.  Summing over disjoint prime bands using \(q/2<1\)
gives the bound
\begin{equation}
 C_p\norm d_q^q
 \sum_{h=0}^H
 \left(\frac{2^h}{N}\right)^a
 \left(1+\log\frac{N}{2^h}\right)^q
 \le C_p\norm d_q^q,
 \label{eq:outside-regular-scale-sum}
\end{equation}\noeqref{eq:outside-regular-scale-sum}
since \(N=2^H\) and
\(\sum_{n\ge0}2^{-na}(1+n)^q<\infty\).
Combining these estimates gives
\begin{equation}
 \norm{\mathcal O_yd}_{L^q(\ell^2)}\le C_p\norm d_q.
 \label{eq:outside-column-bound}
\end{equation}\noeqref{eq:outside-column-bound}

Together with \eqref{eq:positive-low-column-bound},
this completes the positive-frequency estimate.
By \eqref{eq:frozen-live-conditional-identity} and \eqref{eq:frozen-live-restoration},
the full-product bounds imply the smaller-prime estimates
\eqref{eq:negative-column-bound} and 
\eqref{eq:positive-nonresonant-bound}.
\end{proof}

\subsection{Reduction to the resonant term and the projected-vector estimate}
\label{subsec:resonant-reduction}

After the frequency partition, its high-frequency  term (Proposition~\ref{prop:high-source-tail}) and
nonresonant terms (Proposition~\ref{prop:nonresonant-columns}) have uniform bounds.  Hence the projected-vector
estimate reduces to the remaining middle intervals, which is the resonant term.    For bounded offsets \(\ell\),
the common-shift moment estimate and the cyclic square function control
the coefficient shells.  For large positive offsets, deformation of
the Euler product gives exponential decay in \(2^\ell\).

\begin{proposition}[Resonant-term estimate]
\label{prop:resonant-column}
For every \(d\in\C^{\ZN}\),
\begin{equation}
 \norm{
 \mathcal R_yd
 }_{L^q(\ell^2)}
 \le
 C_p\norm d_q.
 \label{eq:resonant-column-estimate}
\end{equation}
The constant $C_p$ depends only on $p$.
\end{proposition}

\begin{proof}
We estimate separately the bounded and large offsets.

Let
\[
 -2\le\ell\le3.
\]
Proposition~\ref{prop:frame-estimate} gives
\[
 \norm{
 \mathcal R_y^{(\ell)}d
 }_{L^q(\ell^2)}
 \le
 C_p
 \mathfrak P_y(z^{(\ell)})^{1/q}.
\]
By Proposition~\ref{prop:common-shift-occupation},
\[
 \mathfrak P_y(z^{(\ell)})^{1/q}
 \le
 C_p
 \norm{
 S_{\mathrm{cone}}z^{(\ell)}
 }_q.
\]
The coefficient shell
\(
 z_h^{(\ell)}=d_{2^{h+\ell}}
\)
has Fourier support satisfying
\[
 \abs\kappa
 \le
 C2^{h+\ell}.
\]
Thus Theorem~\ref{thm:cyclic-conical-square-function} applies with
\[
 C_*\le C2^\ell=O(1)
\]
throughout this finite offset range.  Consequently, by Proposition~\ref{prop:coefficient-square-function}
\begin{align}
 \norm{
 \mathcal R_y^{(\ell)}d
 }_{L^q(\ell^2)}
\le
 C_p\norm d_q.
 \label{eq:bounded-offset-resonant-bound}
\end{align}

Now let \(\ell\ge4\).  Proposition~\ref{prop:deformed-frame} gives
\[
 \norm{
 \mathcal R_y^{(\ell)}d
 }_{L^q(\ell^2)}
 \le
 C_p
 \e^{-\vartheta_02^\ell/8}
 \mathfrak P_y^{\vartheta_0}(z^{(\ell)})^{1/q},
 \qquad
 \vartheta_0=\frac1{32}.
\]
The Euler-product moment bound in Proposition~\ref{prop:common-shift-occupation} and the cyclic conical estimate in Theorem \ref{thm:cyclic-conical-square-function} give
\begin{align}
 \mathfrak P_y^{\vartheta_0}(z^{(\ell)})^{1/q}
 \le
 C_p
 \norm{
 S_{\mathrm{cone}}z^{(\ell)}
 }_q
\le
 C_p
 (1+C_*)^{\Gamma_q}
 \norm{
 \left(
 \sum_h\abs{d_{2^{h+\ell}}}^2
 \right)^{1/2}
 }_q.
 \label{eq:large-offset-conical-reduction}
\end{align}
Here
\[
 1+C_*
 \le
 C2^\ell
 =
 2^{\ell+C_0}
\]
for an absolute constant \(C_0\).  Proposition
\ref{prop:coefficient-square-function} therefore yields
\begin{equation}
 \norm{
 \mathcal R_y^{(\ell)}d
 }_{L^q(\ell^2)}
 \le
 C_p
 \e^{-\vartheta_02^\ell/8}
 2^{\Gamma_q(\ell+C_0)}
 \norm d_q.
 \label{eq:large-offset-resonant-bound}
\end{equation}
For fixed \(q\), the sequence in
\eqref{eq:large-offset-resonant-bound} is summable in \(\ell\).
Minkowski's inequality in \eqref{eq:resonant-column-sum},
together with \eqref{eq:bounded-offset-resonant-bound} and
\eqref{eq:large-offset-resonant-bound}, proves
\eqref{eq:resonant-column-estimate}.
\end{proof}

Combining the resonant-term bound with the
bounds for the other frequency pieces, we can now give the uniform
projected-vector estimate.

\begin{theorem}[Projected-vector estimate]
\label{thm:corrected-column}
For every fixed \(p>2\), there exists a constant \(C_p<\infty\),
depending only on \(p\), such that for every \(c\in\C^{\ZN}\),
\begin{equation}
 \norm{\mathcal T_yc}_{L^q(\ell^2(\{r\le y\}))}
 \le
 C_p\norm c_{\ell^q(\ZN)}.
 \label{eq:corrected-column-estimate}
\end{equation}
\end{theorem}

\begin{proof}
Apply the frequency decomposition \eqref{eq:exact-frequency-partition},
Propositions~\ref{prop:high-source-tail},
\ref{prop:nonresonant-columns}, and~\ref{prop:resonant-column},
and the projection bounds
\eqref{eq:positive-negative-projection-bounds}.
\end{proof}

\section{Proof of the main theorem}
\label{sec:main-proof}

Theorem~\ref{thm:corrected-column} bounds the square function of the
projected first-harmonic coefficients $\mathcal T_yc$. 
Next, we use this bound to control linear
combinations of critical-line evaluations. On the finite prime torus, these
evaluations are represented by normalized Euler products in a Hardy quotient $L^q(\Omega_y)/\mathcal N_p$.
We shall show that expanding the products, in increasing prime order, separates terms with direct
bounds from the first-harmonic sum. The latter is controlled by the
projected-vector estimate and a martingale inequality. Quotient duality then
gives the required sampling estimate.

Throughout this section, all Fourier coefficients and pairings use the
conventions in \eqref{eq:torus-fourier-pairing}.  
Recall the following finite-grid parameters defined in \eqref{eq:finite-grid-parameters},
\begin{align*}
&H \in\mathbb{N},\qquad
 N=2^H,\qquad
 y=\e^N,\qquad\\
 &t_0\in\mathbb{R},\qquad
 t_j=t_0+\frac{j}{N},
 \quad 0\le j<N,
\end{align*}
and write
\[
 M_y=\prod_{r\le y}(1-r^{-1})^{-1},
\]
which is defined in \eqref{eq:mertens-and-sb}.

On the finite prime torus \(\Omega_y\), we write the Fourier
coefficients and the pairing as
\[
 \widehat F(\nu)
 =\int_{\Omega_y}F(\zeta)\overline{\zeta^\nu}\dd m(\zeta),
 \qquad
 \ip{F}{G}
 =\int_{\Omega_y}F(\zeta)\overline{G(\zeta)}\dd m(\zeta).
\]
In particular, the pairing is linear in its first variable.

\subsection{Hardy quotient and Euler product}
\label{subsec:hardy-quotient-euler-source}

Let \(1<p<\infty\), \(q=p/(p-1)\), and \(d<\infty\).  Denote by 
\(H^p(\T^d)\) 
the \(L^p(\T^d)\)-closure of the analytic trigonometric polynomials, and define
\[
 \mathcal N_p
 =
 \left\{
 G\in L^q(\T^d):
 \ip{G}{F}=0
 \text{ for every }F\in H^p(\T^d)
 \right\}.
\]
Since the analytic characters \(\zeta^\nu\), \(\nu\in\Nzero^d\), span a
dense subspace of \(H^p(\T^d)\), one has
\begin{equation}\label{eq:hardy-annihilator-fourier}
 \mathcal N_p
 =
 \left\{
 G\in L^q(\T^d):
 \widehat G(\nu)=0
 \text{ for every }\nu\in\Nzero^d
 \right\}.
\end{equation}
In particular, if
\[
 \mu=(\mu_1,\ldots,\mu_d)\in\Z^d,
\]
and \(\mu_j<0\) for some \(j\), then
\[
 \zeta^\mu\in\mathcal N_p.
\]
Thus only Fourier coefficients with nonnegative exponents in every prime
variable affect pairings with analytic polynomials. We may therefore choose a
convenient representative of a quotient class rather than estimate the
original function in its \(L^q\)-norm.

By standard quotient-space and \(L^q\)-\(L^p\) duality, for every
\(F\in H^p(\T^d)\) the formula
\[
 \Lambda_F([G])\coloneqq\ip{G}{F}
\]
defines a continuous complex-linear functional on
\(L^q(\T^d)/\mathcal N_p\), and the map
\[
 H^p(\T^d)
 \longrightarrow
 \bigl(L^q(\T^d)/\mathcal N_p\bigr)^*,
 \qquad
 F\longmapsto\Lambda_F,
\]
is a conjugate-linear isometric bijection.  In particular,
\begin{equation}\label{eq:quotient-dual-isometry}
\norm{\Lambda_F}_{(L^q/\mathcal N_p)^*}
=
\norm F_{L^p(\T^d)}.
\end{equation}

For the remainder of this section, return to the fixed \(p>2\), so that
\(1<q<2\) and \(\alpha=1-2/p\in(0,1)\). 
For \(0<x\le1/2\), define
\begin{equation}\label{def:eta-function}
 \eta_x(z)
 =
 (1-x)^\alpha
 (1-\sqrt{x}z)^{-1}
 (1-\sqrt{x}\bar z)^{-\alpha},
 \qquad z\in\T.
\end{equation}
Using \eqref{eq:binomial-branch} and \(\bar z=z^{-1}\) on \(\T\), the
absolutely convergent expansions of the two factors give
\[
 \eta_x(z)
 =
 (1-x)^\alpha
 \sum_{n,m\ge0}
 \frac{(\alpha)_m}{m!}
 x^{(n+m)/2}z^{\,n-m}.
\]
Hence, for \(k\ge0\), collecting the terms with \(n-m=k\) yields
\begin{equation}
 \widehat{\eta_x}(k)
 =
 (1-x)^\alpha x^{k/2}
 \sum_{m=0}^{\infty}\frac{(\alpha)_m}{m!}x^m
 =
 x^{k/2}.
 \label{eq:one-prime-positive-coefficient}
\end{equation}
The extra conjugate factor and the normalization in \(\eta_x\) give exactly
these nonnegative Fourier coefficients, which are needed for critical-line
evaluation.
We use the coefficients with \(k\ge0\) to obtain the
critical-line evaluation identity in
Subsection~\ref{subsec:discrete-local-sampling}.
In particular, 
\(
 \widehat{\eta_x}(0)=1,
\)
\(
 \widehat{\eta_x}(1)=\sqrt{x},
\)
and
\begin{equation*}
\norm{\eta_x}_{A(\T)}
 \coloneqq\sum_{k\in\Z}\abs{\widehat{\eta_x}(k)}
<\infty,
\end{equation*}
so that \(\eta_x\in A(\T)\).
We take absolute values in the double series for \(\eta_x\)
and sum the geometric and binomial factors separately to obtain
\begin{align*}
 \norm{\eta_x}_{A(\T)}
 &\le(1-x)^\alpha
 \left(\sum_{n\ge0}x^{n/2}\right)
 \left(\sum_{m\ge0}\frac{(\alpha)_m}{m!}x^{m/2}\right)\\
 &=(1-x)^\alpha(1-\sqrt{x})^{-1-\alpha}<\infty.
\end{align*}
Recall that \(y=\e^N\) as in \eqref{eq:finite-grid-parameters}, and let \(r\le y\) be prime. The functions \(f_y(t)\) and \(f_{r^{-}}(t)\) are defined in \eqref{eq:euler-source} and \eqref{eq:old-prime-source}, respectively.  Then 
\begin{equation}
 f_y(t)
 =
 \prod_{\rho\le y}
 \eta_{1/\rho}
 \left(
 \e^{\ii t\log \rho}\zeta_\rho
 \right),
 \label{eq:euler-source-one-prime-product}
\end{equation}
where the product is over primes.
Since the product is finite and every factor belongs to \(A(\T)\),
\[
 f_y(t)\in A(\T^{\pi(y)}).
\]
The same holds for every \(f_{r^-}(t)\).

Finite linear combinations of these Euler products at the
sampling points still have absolutely summable Fourier
coefficients. Restricting to any subset of Fourier indices
can only decrease the sum of the absolute values of the
coefficients. Each resulting subseries therefore converges absolutely
in the Wiener norm on \(\Omega_y\) and hence in \(L^q(\Omega_y)\).
This convergence allows us to group Fourier terms for each
fixed \(y\). The \(L^q\) estimates below will give bounds
uniform in the prime cutoff \(y\).

For a prime \(r\), we define the sum of higher positive harmonics
\begin{equation}
 H_r^+(t)
=
 \sum_{k=2}^{\infty}
 r^{-k/2}
 \e^{\ii kt\log r}\zeta_r^k
=
 \frac{
 r^{-1}\e^{2\mathrm{i}t\log r}\zeta_r^2
 }{
 1-r^{-1/2}\e^{\ii t\log r}\zeta_r
 }.
 \label{eq:positive-higher-current-harmonics}
\end{equation}
It satisfies
\begin{equation}
 \abs{H_r^+(t)}
 \le
 \frac{r^{-1}}{1-r^{-1/2}}
 \le
 Cr^{-1}.
 \label{eq:positive-higher-current-bound}
\end{equation}
The higher positive harmonics thus begin at order \(r^{-1}\), whereas the
first positive harmonic has coefficient \(r^{-1/2}\) and needs a different
estimate.

In the harmonic decomposition
\eqref{eq:current-prime-harmonic-decomposition},
\(H_r^+(t)\) collects the powers \(\zeta_r^k\) with \(k\ge2\).
We treat the term linear in \(\zeta_r\) separately.
In the \(r\)-th summand of the ordered product expansion
at \(t=t_j\), the coefficient of \(\zeta_r\) is
\(r^{-1/2}\e^{\ii t_j\log r}f_{r^-}(t_j)\).
We sum the coefficients over \(j\) with weights \(M_y^{-1/p}c_j\)
and apply \(P_{[0,\infty)}^{r^-}\) in the variables
\((\zeta_\rho)_{\rho<r}\). The resulting coefficient of
\(\zeta_r\) is
\[
 (\mathcal T_yc)_r
 =
 r^{-1/2}P_{[0,\infty)}^{r^-}
 \left(
 M_y^{-1/p}\sum_{j<N}c_j\e^{\ii t_j\log r}f_{r^-}(t_j)
 \right).
\]

\subsection{An Euler-product representative}
\label{subsec:euler-quotient-representative}

Now we show that expanding the Euler product prime by prime in the Hardy quotient
isolates the first-harmonic sum, whose coefficients form
the projected vector.
We first put
\begin{equation}
 g_{y,t}=M_y^{-1/p}f_y(t),
 \label{eq:normalized-quotient-source}
\end{equation}
and, for \(c\in\C^{\ZN}\) and \(r\le y\), define
\begin{equation}
 A_r(c)
 =
 M_y^{-1/p}
 \sum_{j<N}
 c_j\e^{\ii t_j\log r}f_{r^-}(t_j).
 \label{eq:first-harmonic-old-prime-coefficient}
\end{equation}
Then the projected vector in \eqref{eq:corrected-column} is
\begin{equation}
 (\mathcal T_yc)_r
 =
 r^{-1/2}P_{[0,\infty)}^{r^-}A_r(c).
 \label{eq:corrected-column-A-r}
\end{equation}

We define
\begin{equation}
 H_0(c)
 =
 M_y^{-1/p}\sum_{j<N}c_j,
 \label{eq:quotient-root-term}
\end{equation}
and
\begin{equation}
 \mathcal Q_y(c)
 =
 M_y^{-1/p}
 \sum_{r\le y}
 \sum_{j<N}
 c_jf_{r^-}(t_j)H_r^+(t_j).
 \label{eq:quotient-higher-harmonic-remainder}
\end{equation}
The ordered expansion 
of $\sum_{j<N}c_jg_{y,t_j}$
has a constant term, a first-harmonic sum, and a higher-harmonic remainder. In the quotient, terms with a negative power of the new prime vanish. For the first harmonic, earlier-prime terms with
negative total logarithmic frequency also vanish, leaving the projected
coefficient in \eqref{eq:corrected-column-A-r}.

\begin{proposition}
[Euler-product representation]
\label{prop:quotient-representation}
For every \(c\in\C^{\ZN}\), on the finite prime torus \(\Omega_y=\T^{\pi(y)}\), one has
\begin{equation}
 \left[
 \sum_{j<N}c_jg_{y,t_j}
 \right]
 =
 \left[
 H_0(c)
 +
 \sum_{r\le y}\zeta_r(\mathcal T_yc)_r
 +
 \mathcal Q_y(c)
 \right]
 \quad\text{in }L^q(\Omega_y)/\mathcal N_p.
 \label{eq:quotient-representation}
\end{equation}
Moreover,
\begin{equation}
 \norm{H_0(c)}_q+\norm{\mathcal Q_y(c)}_q
 \le
 C_p\norm c_{\ell^q(\ZN)}.
 \label{eq:quotient-root-remainder-bound}
\end{equation}
\end{proposition}

The proposition states bounds for the constant term \(H_0(c)\)
and for \(\mathcal Q_y(c)\), which collects the powers
\(\zeta_r^k\) with \(k\ge2\) at each step indexed by \(r\).
To bound the full quotient class, we also need to estimate
\(\sum_{r\le y}\zeta_r(\mathcal T_yc)_r\) in \(L^q(\Omega_y)\).
We do this below by combining the \(L^q(\Omega_y;\ell^2)\)
bound for \(\mathcal T_yc\) with a martingale inequality.

\begin{proof}
For a finite ordered family of factors \(\eta_1,\ldots,\eta_m\),
we begin with the identity
\[
 \prod_{k=1}^{m}\eta_k
 =1+\sum_{k=1}^{m}
 \left(\prod_{i<k}\eta_i\right)(\eta_k-1).
\]
In the term indexed by \(k\), the product multiplying
\(\eta_k-1\) contains only the earlier factors. Thus, when
we order the factors by primes, this product depends only on
variables corresponding to smaller primes.

By \eqref{eq:euler-source-one-prime-product}, we have
\begin{equation}
 f_y(t)
 =
 1+
 \sum_{r\le y}
 f_{r^-}(t)
 \left(
 \eta_{1/r}
 \left(
 \e^{\ii t\log r}\zeta_r
 \right)-1
 \right).
 \label{eq:ordered-euler-telescoping}
\end{equation}
Equations~\eqref{eq:one-prime-positive-coefficient} and
\eqref{eq:positive-higher-current-harmonics} imply
\begin{equation}
 \eta_{1/r}
 \left(
 \e^{\ii t\log r}\zeta_r
 \right)-1
=
 \sum_{k<0}
 \widehat{\eta_{1/r}}(k)
 \e^{\ii kt\log r}\zeta_r^k
+
 r^{-1/2}\e^{\ii t\log r}\zeta_r
 +
 H_r^+(t).
\label{eq:current-prime-harmonic-decomposition}
\end{equation}
For each prime \(r\), we group the terms in the corresponding
summand of the ordered product expansion by the exponent \(k\)
of \(\zeta_r\). The table gives the contribution of each group
to the quotient representation.

\begin{center}
\small
\renewcommand{\arraystretch}{1.2}
\begin{tabularx}{\linewidth}{@{}c>{\raggedright\arraybackslash}X@{}}
\toprule
Exponent \(k\) & Role in the quotient representation \\
\midrule
\(k<0\)
& The factor \(f_{r^-}(t)\) is independent of \(\zeta_r\).
  Multiplication by \(f_{r^-}(t)\) preserves the negative
  exponent of \(\zeta_r\), so these terms belong to
  \(\mathcal N_p\). \\
\(k=0\)
& The constant Fourier coefficient of \(\eta_{1/r}\) is \(1\),
  so the constant term cancels in \(\eta_{1/r}-1\).
  The leading \(1\) outside the prime sum in the ordered
  product expansion gives \(H_0(c)\) after weighted summation
  over the sampling points. \\
\(k=1\)
& In the summand indexed by \(r\), the weighted sampling sum
  gives \(r^{-1/2}A_r(c)\) as the coefficient of \(\zeta_r\).
  To choose a quotient representative, we apply
  \(P_{[0,\infty)}^{r^-}\) to \(A_r(c)\), retaining the
  Fourier terms of \(A_r(c)\) with nonnegative total
  logarithmic frequency in \((\zeta_\rho)_{\rho<r}\).
  The resulting coefficient of \(\zeta_r\) is
  \((\mathcal T_yc)_r=r^{-1/2}P_{[0,\infty)}^{r^-}A_r(c)\). \\
\(k\ge2\)
& Multiplying \(H_r^+(t_j)\) by \(f_{r^-}(t_j)\), summing
  over \(j\) with weights \(M_y^{-1/p}c_j\), and then summing
  over \(r\) gives \(\mathcal Q_y(c)\). \\
\bottomrule
\end{tabularx}
\end{center}

We first compute the weighted sampling sum in \(L^q(\Omega_y)\).
In the summand for each prime \(r\), the formula below includes
the terms with negative powers of \(\zeta_r\) and the unprojected
first-harmonic coefficient \(r^{-1/2}A_r(c)\).
Multiplying \eqref{eq:ordered-euler-telescoping} by
\(M_y^{-1/p}c_j\), summing over \(j<N\), and using
\eqref{eq:current-prime-harmonic-decomposition}, we obtain
\begin{align}
 \sum_{j<N}c_jg_{y,t_j}
 &=
 H_0(c)
 +
 \mathcal E_y^-(c)
 +
 \sum_{r\le y}\zeta_r r^{-1/2}A_r(c)
 +
 \mathcal Q_y(c),
 \label{eq:euler-harmonic-assembly}
\end{align}
where
\[
 \mathcal E_y^-(c)
 =
 M_y^{-1/p}
 \sum_{r\le y}\sum_{j<N}
 c_j f_{r^-}(t_j)
 \sum_{k<0}
 \widehat{\eta_{1/r}}(k)
 \e^{\ii kt_j\log r}\zeta_r^k.
\]
Since \(f_{r^-}(t_j)\) is independent of \(\zeta_r\),
\eqref{eq:hardy-annihilator-fourier} gives
\begin{equation}
 \mathcal E_y^-(c)\in\mathcal N_p.
 \label{eq:negative-current-annihilator}
\end{equation}
We can see this directly from the Fourier support of
\(\mathcal E_y^-(c)\), which collects, at the step indexed
by \(r\), the terms with negative powers of \(\zeta_r\):
\[
 \supp\widehat{\mathcal E_y^-(c)}
 \subseteq
 \bigcup_{r\le y}
 \{\nu\in\Z^{\pi(y)}:\nu_r<0\}.
\]
Each character in this support has a negative exponent in
at least one prime variable and is therefore orthogonal to
every analytic monomial. The pairing vanishes on analytic
polynomials and hence on \(H^p(\Omega_y)\) by density.

Next, we write
\[
 A_r(c)
 =
 P_{[0,\infty)}^{r^-}A_r(c)
 +
 P_{(-\infty,0)}^{r^-}A_r(c).
\]
Let \(\mu\) be the exponent vector of a Fourier character in
\((\zeta_\rho)_{\rho<r}\). The projection \(P_{(-\infty,0)}^{r^-}\)
retains this character precisely when its total logarithmic
frequency is negative, that is, when
\[
 \sum_{\rho<r}\mu_\rho\log\rho<0.
\]
Since every \(\log\rho\) is positive, at least one coordinate
of \(\mu\) belonging to a prime smaller than \(r\) must be
negative for this inequality to hold.
If
\(
 \zeta^\mu
 =
 \prod_{\rho<r}\zeta_\rho^{\mu_\rho}
\)
occurs in \(P_{(-\infty,0)}^{r^-}A_r(c)\), then one has
\[
 \mu_\rho<0
\]
for some \(\rho<r\).
Consequently,
\begin{equation}
 \zeta_r r^{-1/2}
 P_{(-\infty,0)}^{r^-}A_r(c)
 \in\mathcal N_p.
 \label{eq:discarded-first-harmonic-annihilator}
\end{equation}
Multiplying by \(\zeta_r\) gives
\(\zeta_r\zeta^\mu=\zeta_r\prod_{\rho<r}\zeta_\rho^{\mu_\rho}\).
The exponent at \(r\) becomes \(1\), and the negative exponent
at a smaller prime \(\rho<r\) is preserved. Hence
\(\zeta_r\zeta^\mu\in\mathcal N_p\).
Thus, in \(L^q(\Omega_y)/\mathcal N_p\), we have
\begin{equation*}
 \left[
 \sum_{r\le y}\zeta_r r^{-1/2}A_r(c)
 \right]
=
 \left[
 \sum_{r\le y}
 \zeta_r r^{-1/2}
 P_{[0,\infty)}^{r^-}A_r(c)
 \right]
=
 \left[
 \sum_{r\le y}
 \zeta_r(\mathcal T_yc)_r
 \right],
\end{equation*}
where the last equality follows from
\eqref{eq:corrected-column-A-r}.  Together with
\eqref{eq:euler-harmonic-assembly} and
\eqref{eq:negative-current-annihilator}, this proves
\eqref{eq:quotient-representation}.

For the constant term, H\"older's inequality gives
\begin{equation*}
 \norm{H_0(c)}_q
=
 M_y^{-1/p}
 \abs{\sum_{j<N}c_j}
 \lesssim_p
 \norm c_q,
\end{equation*}
since \(y=\e^N\), and
\(
 M_y\asymp N.
\)
Applying H\"older's inequality to the sum of the \(N\)
coefficients \(c_j\) gives the factor \(N^{1/p}\).
Together with the normalization \(M_y^{-1/p}\), this gives
\[
 M_y^{-1/p}\abs{\sum_{j<N}c_j}
 \le M_y^{-1/p}N^{1/p}\norm c_q
 =\left(\frac{N}{M_y}\right)^{1/p}\norm c_q.
\]

Let
\[
 r_1<\cdots<r_m
\]
be the primes below \(y\), and we set
\[
 \mathscr F_0=\{\varnothing,\Omega_y\},
 \qquad
 \mathscr F_k
 =
 \sigma(\zeta_{r_1},\ldots,\zeta_{r_k}),
\]
and
\[
 \Delta_k
 =
 M_y^{-1/p}
 \sum_{j<N}
 c_jf_{r_k^-}(t_j)H_{r_k}^+(t_j).
\]
Then, one has
\[
 \mathcal Q_y(c)=\sum_{k=1}^{m}\Delta_k.
\]
We condition on \(\mathscr F_{k-1}\), which fixes the
variables indexed by primes smaller than \(r_k\) and hence
fixes \(f_{r_k^-}(t_j)\). The variable \(\zeta_{r_k}\) is
independent of \(\mathscr F_{k-1}\) and is uniform on \(\T\).
We can therefore average the absolutely convergent series
for \(H_{r_k}^+(t_j)\) term by term over \(\zeta_{r_k}\) to obtain
\[
 \E\!\left(H_{r_k}^+(t_j)\mid\mathscr F_{k-1}\right)
 =\sum_{\ell=2}^{\infty}r_k^{-\ell/2}
   \e^{\ii\ell t_j\log r_k}\E(\zeta_{r_k}^{\ell})
 =0.
\]
Every positive power of \(\zeta_{r_k}\) has mean zero.
Since \(f_{r_k^-}(t_j)\) is \(\mathscr F_{k-1}\)-measurable,
we obtain
\[
 \E(\Delta_k\mid\mathscr F_{k-1})
 =
 M_y^{-1/p}\sum_{j<N}c_jf_{r_k^-}(t_j)
 \E\!\left(H_{r_k}^+(t_j)\mid\mathscr F_{k-1}\right)
 =0.
\]
Thus \((\Delta_k)\) is a martingale-difference sequence.
By the Burkholder--Gundy square-function inequality
\cite[Chapter~2, pp.~33--79]{Long93} and \(q/2<1\), we have
\begin{equation}
 \norm{\mathcal Q_y(c)}_q^q
\le
 C_q
 \sum_{k=1}^{m}\norm{\Delta_k}_q^q.
 \label{eq:higher-harmonic-smoothness-sum}
\end{equation}
The martingale square-function estimate gives the first
inequality below. For the second, we use \(q/2<1\) to obtain
the pointwise bound
\(\bigl(\sum_{k=1}^{m}\abs{\Delta_k}^2\bigr)^{q/2}
\le\sum_{k=1}^{m}\abs{\Delta_k}^q\)
before taking expectations.
\begin{align*}
 \norm{\mathcal Q_y(c)}_q^q
 \le C_q\E\left[
   \left(\sum_{k=1}^{m}\abs{\Delta_k}^2\right)^{q/2}
 \right] \le C_q\sum_{k=1}^{m}\norm{\Delta_k}_q^q.
\end{align*}
By \eqref{eq:positive-higher-current-bound},
\eqref{eq:euler-source-normalization}, H\"older's inequality, and Mertens'
estimate,
\begin{equation*}
 \norm{\Delta_k}_q
 \lesssim_p
 r_k^{-1}(1+\log r_k)^{1/p}
 \norm c_q.
\end{equation*}
We first use the triangle inequality and the bound
\(\abs{H_{r_k}^+(t_j)}\le Cr_k^{-1}\).
We apply H\"older's inequality to obtain
\(\sum_{j<N}\abs{c_j}\le N^{1/p}\norm c_q\).
Together with the normalization \(M_y^{-1/p}\), this gives
the factor \(M_y^{-1/p}N^{1/p}=(N/M_y)^{1/p}\). We obtain
\begin{align*}
 \norm{\Delta_k}_q
 &\le CM_y^{-1/p}r_k^{-1}
    \sum_{j<N}\abs{c_j}\norm{f_{r_k^-}(t_j)}_q\\
 &\le C\left(\frac{N}{M_y}\right)^{1/p}
       r_k^{-1}M_{r_k^-}^{1/p}\norm c_q\\
 &\lesssim_p r_k^{-1}(1+\log r_k)^{1/p}\norm c_q.
\end{align*}
The second line uses
\(\norm{f_{r_k^-}(t_j)}_q=M_{r_k^-}^{1/p}\).
Mertens' estimates \(M_y\asymp N\) and
\(M_{r_k^-}\lesssim1+\log r_k\) then give the last line.
Since \(q>1\),
\[
 \sum_{r\ \mathrm{prime}}r^{-q}(1+\log r)^{q/p}
 \le\sum_{n=2}^{\infty}n^{-q}(1+\log n)^{q/p}<\infty.
\]
We then sum the \(q\)-th powers of the bounds for
\(\Delta_k\) to obtain
\[
 \norm{\mathcal Q_y(c)}_q^q
 \le C_p\norm c_q^q
       \sum_{r\le y}r^{-q}(1+\log r)^{q/p}
 \le C_p\norm c_q^q.
\]
Taking \(q\)-th roots gives the bound for
\(\norm{\mathcal Q_y(c)}_q\).
Therefore, we have
\[
 \norm{\mathcal Q_y(c)}_q
 \le
 C_p\norm c_q.
\]
Together with the bound for \(H_0(c)\), this proves
\eqref{eq:quotient-root-remainder-bound}.
\end{proof}

It remains to estimate the first harmonics in \eqref{eq:quotient-representation}. For each prime \(r\), the
projected coefficient depends only on smaller primes. Multiplication by the
independent, mean-zero variable \(\zeta_r\) therefore gives a martingale
difference. Its square function is the one controlled by
Theorem~\ref{thm:corrected-column}.

\begin{lemma}
[First-harmonic estimate]
\label{lem:first-harmonic-martingale}
For every \(c\in\C^{\ZN}\), one has
\begin{equation}
 \norm{
 \sum_{r\le y}
 \zeta_r(\mathcal T_yc)_r
 }_{L^q(\Omega_y)}
 \le
 C_p\norm c_{\ell^q(\ZN)}.
 \label{eq:first-harmonic-martingale-estimate}
\end{equation}
\end{lemma}

\begin{proof}
We use the ordered primes and filtration from the proof of
Proposition~\ref{prop:quotient-representation}.  First, we put
\[
 a_k=(\mathcal T_yc)_{r_k},
 \qquad
 d_k=\zeta_{r_k}a_k.
\]
We condition on \(\mathscr F_{k-1}\). This fixes the
variables indexed by primes smaller than \(r_k\), so the
coefficient \(a_k\) is fixed in the conditional average.
The remaining factor \(\zeta_{r_k}\) in \(d_k\) is independent
of \(\mathscr F_{k-1}\) and has mean zero. Hence
\[
 \E(\zeta_{r_k}\mid\mathscr F_{k-1})=0,
 \qquad
 \E(d_k\mid\mathscr F_{k-1})
 =a_k\E(\zeta_{r_k}\mid\mathscr F_{k-1})=0.
\]
Each \(d_k\) is \(\mathscr F_k\)-measurable and has conditional
mean zero given \(\mathscr F_{k-1}\). Thus \((d_k)\) is a
complex martingale-difference sequence for \((\mathscr F_k)\).
By \eqref{eq:corrected-column-A-r}, we have
\[
 d_k\in L^q(\mathscr F_k),
 \qquad
 \E(d_k\mid\mathscr F_{k-1})=0.
\]
Moreover, we have
\begin{equation}
 \sum_k\abs{d_k}^2
 =
 \sum_k\abs{a_k}^2
 =
 \sum_{r\le y}\abs{(\mathcal T_yc)_r}^2
 \quad\text{pointwise}.
 \label{eq:first-harmonic-square-function-identity}
\end{equation}
Since \(d_k=\zeta_{r_k}a_k\) and \(\abs{\zeta_{r_k}}=1\),
we have \(\abs{d_k}=\abs{a_k}\) pointwise.
The real and imaginary parts of \((d_k)\) are real
martingale-difference sequences. Their square functions satisfy
\[
 \left(\sum_k\abs{\Ree d_k}^2\right)^{1/2}
 \le\left(\sum_k\abs{d_k}^2\right)^{1/2},
 \qquad
 \left(\sum_k\abs{\Imm d_k}^2\right)^{1/2}
 \le\left(\sum_k\abs{d_k}^2\right)^{1/2}.
\]
We apply the martingale square-function estimate to the sums
of the real and imaginary parts and combine the bounds by
the triangle inequality. The pointwise identity then
identifies the square function of \((d_k)\) with the
\(\ell^2\)-norm of \(\mathcal T_yc\). We obtain
\begin{align*}
 \norm{\sum_kd_k}_{L^q(\Omega_y)}
 \le C_q
 \norm{\left(\sum_k\abs{d_k}^2\right)^{1/2}}_{L^q(\Omega_y)}=C_q\norm{\mathcal T_yc}_{L^q(\Omega_y;\ell^2)}\le C_p\norm c_{\ell^q(\ZN)}.
\end{align*}
Combining the Burkholder--Gundy square-function estimate, \eqref{eq:first-harmonic-square-function-identity} and Theorem~\ref{thm:corrected-column}, we finally obtain \eqref{eq:first-harmonic-martingale-estimate}.
\end{proof}

The above proposition bounds the constant and higher positive harmonics, while the
lemma bounds the first harmonics. Together they control the quotient class
of the combined Euler-product representative as below.

\begin{corollary}
[Quotient norm bound]
\label{cor:quotient-column-bound}
For every \(c\in\C^{\ZN}\), one has
\begin{equation}
 \norm{
 \left[
 \sum_{j<N}c_jg_{y,t_j}
 \right]
 }_{L^q(\Omega_y)/\mathcal N_p}
 \le
 C_p\norm c_{\ell^q(\ZN)}.
 \label{eq:quotient-column-bound}
\end{equation}
\end{corollary}

\begin{proof}
Combining Proposition~\ref{prop:quotient-representation} with Lemma~\ref{lem:first-harmonic-martingale}, we obtain \eqref{eq:quotient-column-bound}.
The quotient class of \(\sum_{j<N}c_jg_{y,t_j}\) has the
representative
\[
 H_0(c)+\sum_{r\le y}\zeta_r(\mathcal T_yc)_r+\mathcal Q_y(c).
\]
We apply the definition of the quotient norm and then the triangle
inequality to obtain
\begin{align*}
 \norm{\left[\sum_{j<N}c_jg_{y,t_j}\right]}_{L^q(\Omega_y)/\mathcal N_p} &\le
 \norm{H_0(c)+\sum_{r\le y}\zeta_r(\mathcal T_yc)_r+
       \mathcal Q_y(c)}_{L^q(\Omega_y)}\\
 &\le\norm{H_0(c)}_q
 +\norm{\sum_{r\le y}\zeta_r(\mathcal T_yc)_r}_q
 +\norm{\mathcal Q_y(c)}_q.
\end{align*}
The proposition bounds the \(L^q(\Omega_y)\)-norms of
\(H_0(c)\) and \(\mathcal Q_y(c)\) by
\(C_p\norm c_{\ell^q(\ZN)}\). The estimate for
\(\sum_{r\le y}\zeta_r(\mathcal T_yc)_r\) gives the same
bound for the middle term.
\end{proof}

\subsection{Discrete local sampling}
\label{subsec:discrete-local-sampling}

The quotient bound above controls all the evaluations on a grid at once. Here the
cutoff \(y\) restricts the prime divisors of \(P\), not the integers in its
Dirichlet expansion: arbitrarily high powers of those primes are allowed.
Thus no bound on the full Dirichlet frequencies is assumed.

We first prove the discrete sampling bound for every grid
starting point \(t_0\). In the next subsection, we integrate
this bound as \(t_0\) varies over an interval of length \(1/N\)
to obtain the local embedding inequality.

\begin{proposition}
[Discrete local sampling]
\label{prop:discrete-local-sampling}
Let \(P\) be a Dirichlet polynomial whose prime support is contained below
\(y=\e^N\). Then, for every \(t_0\in\R\), with
\[
 t_j=t_0+\frac{j}{N},
 \qquad 0\le j<N,
\]
one has
\begin{equation}
 \sum_{j=0}^{N-1}
 \abs{
 P\!\left(\frac12+\ii t_j\right)
 }^p
 \le
 C_pM_y\norm P_{\HDir^p}^{p},
 \label{eq:discrete-local-sampling}
\end{equation}
where \(C_p\) depends only on \(p\).
\end{proposition}

\begin{proof}
We let
\[
 \mathcal P=\Bohr P
 =
 \sum_n a_n\zeta^{\nu(n)}
 \in H^p(\Omega_y).
\]
For every \(n\) in the support of \(P\),
\[
 n=\prod_{r\le y}r^{\nu_r(n)}.
\]
Hence, by \eqref{eq:one-prime-positive-coefficient} and \eqref{eq:euler-source-one-prime-product}, we have
\begin{align*}
 \widehat{f_y(t)}(\nu(n))
=
 \prod_{r\le y}
 r^{-\nu_r(n)/2}
 \e^{\ii t\nu_r(n)\log r}
=
 n^{-1/2}\e^{\ii t\log n}.
\end{align*}
By orthogonality of the Fourier characters, the pairing with
the Bohr lift \(\mathcal P\) is
\[
 \ip{g_{y,t}}{\mathcal P}
 =M_y^{-1/p}\sum_n
   \widehat{f_y(t)}(\nu(n))\overline{a_n}.
\]
The coefficients \(a_n\) are conjugated because \(\mathcal P\)
is the second argument of the pairing.
Using \(g_{y,t}=M_y^{-1/p}f_y(t)\) and the pairing convention
\eqref{eq:torus-fourier-pairing}, it follows that
\begin{align*}
 \ip{g_{y,t}}{\mathcal P}
=
 M_y^{-1/p}
 \overline{
 P\!\left(\frac12+\ii t\right)
 }.
\end{align*}
Therefore, for every \(c\in\C^{\ZN}\), \eqref{eq:quotient-dual-isometry} and Corollary~\ref{cor:quotient-column-bound}
yield
\begin{align*}
M_y^{-1/p}
 \abs{
 \sum_{j<N}
 c_j
 \overline{
 P\!\left(\frac12+\ii t_j\right)
 }
 }
\le
 C_p
 \norm c_{\ell^q(\ZN)}
 \norm P_{\HDir^p}.
\end{align*}
Since \(\mathcal P\in H^p(\Omega_y)\), the pairing with
\(\mathcal P\) depends only on the quotient class of
\(\sum_{j<N}c_jg_{y,t_j}\).
We apply H\"older's inequality to each representative and
then take the infimum over the representatives. By definition,
\[
 \inf_{R\in\mathcal N_p}
 \norm{\sum_{j<N}c_jg_{y,t_j}+R}_{L^q(\Omega_y)}
 =
 \norm{\left[\sum_{j<N}c_jg_{y,t_j}\right]}
      _{L^q(\Omega_y)/\mathcal N_p}.
\]
Combining the resulting estimate with
Corollary~\ref{cor:quotient-column-bound}, we obtain
\begin{align*}
 M_y^{-1/p}
 \abs{\sum_{j<N}c_j
       \overline{P\!\left(\frac12+\ii t_j\right)}}
 &=\abs{\ip{\sum_{j<N}c_jg_{y,t_j}}{\mathcal P}}\\
 &\le
 \norm{\left[\sum_{j<N}c_jg_{y,t_j}\right]}
      _{L^q(\Omega_y)/\mathcal N_p}
 \norm{\mathcal P}_{L^p(\Omega_y)}\\
 &\le C_p\norm c_{\ell^q(\ZN)}\norm P_{\HDir^p}.
\end{align*}
We then take the supremum over \(\norm c_{\ell^q(\ZN)}\le1\).
By \(\ell^p\)-\(\ell^q\) duality, this supremum gives the
\(\ell^p\)-norm of the vector
\(\bigl(P(\frac12+\ii t_j)\bigr)_{0\le j<N}\), with the
normalization \(M_y^{-1/p}\) outside the supremum.
Thus
\begin{align*}
 M_y^{-1/p}
 \left(\sum_{j=0}^{N-1}
       \abs{P\!\left(\frac12+\ii t_j\right)}^p\right)^{1/p}
 =M_y^{-1/p}
\sup_{\norm c_{\ell^q(\ZN)}\le1}
   \abs{\sum_{j<N}c_j
\overline{P\!\left(\frac12+\ii t_j\right)}}
\le C_p\norm P_{\HDir^p}.
\end{align*}
Taking \(p\)-th powers changes the normalization on the
left to \(M_y^{-1}\). Multiplying both sides by \(M_y\)
gives the sampling bound with a constant depending only on \(p\).
Taking the supremum over \(\norm c_{\ell^q(\ZN)}\le1\) and using finite-dimensional complex \(\ell^p\)-duality prove \eqref{eq:discrete-local-sampling}.
\end{proof}

\subsection{Proof of the local embedding theorem}
\label{subsec:proof-local-embedding}

The discrete estimate above holds for every starting point of the grid. Next, we show that averaging over one grid spacing covers a unit interval exactly; Mertens' theorem then
cancels the grid normalization.

\begin{proof}[Proof of Theorem~\ref{thm:local-embedding}]
We fix a Dirichlet polynomial \(P\), choose \(H\) sufficiently large that
\[
 y=\e^{2^H}
\]
is an upper bound for all primes in the support of \(P\), and put
\(
 N=2^H.
\)
Once \(P\) is fixed, we may use any \(H\) large enough that
all primes in its support lie below \(y=\e^{2^H}\).
The sampling constant depends only on \(p\), so the same
bound applies for every such \(H\) and every starting point \(t_0\).

Fix \(\theta\in\R\).  For
\(
 0\le\tau\le\frac1N,
\)
we set
\[
 t_0=\theta+\tau.
\]
Substituting \(t_0=\theta+\tau\) into the grid definition gives
\[
 t_j=\theta+\tau+\frac jN,
 \qquad 0\le j<N.
\]
By Proposition~\ref{prop:discrete-local-sampling}, integrating \eqref{eq:discrete-local-sampling} in
\(\tau\in[0,1/N]\) gives
\begin{align}
 \int_0^{1/N}
 \sum_{j=0}^{N-1}
 \abs{
 P\!\left(
 \frac12+\ii\left(
 \theta+\tau+\frac jN
 \right)
 \right)
 }^p
 \dd\tau
\le
 C_p\frac{M_y}{N}\norm P_{\HDir^p}^{p}.
 \label{eq:grid-average-before-tiling}
\end{align}
We now consider the adjacent intervals
\[
 [\theta+j/N,\theta+(j+1)/N],
 \qquad 0\le j<N,
\]
which cover \([\theta,\theta+1]\) and have disjoint interiors.
In the \(j\)-th summand, we set \(t=\theta+\tau+j/N\).
As \(\tau\) ranges over \([0,1/N]\), the variable \(t\)
ranges over the \(j\)-th interval above. Integrating
term by term in the finite sum gives
\begin{align*}
 \int_0^{1/N}\sum_{j=0}^{N-1}
 \abs{P\!\left(\frac12+\ii\left(\theta+\tau+\frac jN\right)\right)}^p
 \,\dd\tau
 =\sum_{j=0}^{N-1}
 \int_{\theta+j/N}^{\theta+(j+1)/N}
 \abs{P\!\left(\frac12+\ii t\right)}^p\,\dd t.
\end{align*}
The change of variables \(t=\theta+\tau+j/N\) gives
\begin{equation} \label{eq:grid-average-tiling}
 \int_0^{1/N}
 \sum_{j=0}^{N-1}
 \abs{
 P\!\left(\frac12+\ii\left(\theta+\tau+\frac jN\right)\right)
 }^p\,\dd\tau
 =
 \int_\theta^{\theta+1}
 \abs{P\!\left(\frac12+\ii t\right)}^p\,\dd t.
\end{equation}
The right-hand side of the sampling bound is independent
of \(\tau\). Integrating its factor \(M_y\) over
\(0\le\tau\le1/N\) gives
\(M_y\int_0^{1/N}\dd\tau=M_y/N\). On the left, the integrals over the
adjacent intervals add up to the integral on
\([\theta,\theta+1]\). Hence
\[
 \int_\theta^{\theta+1}
 \abs{P\!\left(\frac12+\ii t\right)}^p\,\dd t
 \le C_p\frac{M_y}{N}\norm P_{\HDir^p}^{p}.
\]
By Mertens' estimate, \(M_y\asymp\log y=N\), so the ratio
\(M_y/N\) is bounded independently of \(H\).
The sampling constant is uniform in \(t_0=\theta+\tau\),
so the resulting bound holds uniformly in \(\theta\).
Since \(M_y\asymp N\) by Mertens' product theorem, this proves
Theorem~\ref{thm:local-embedding}.
\end{proof}

\section{Equivalent formulations and critical sampling}
\label{sec:consequences}

We first recall the earlier results relating local embedding to
half-plane Hardy spaces, Carleson measures, composition operators, and
a Carlson-type identity. We then prove that local embedding is equivalent
to a sampling inequality for Dirichlet polynomials with bounded frequencies,
and to its dual inequality for linear combinations of evaluations. Finally,
we compare boundary sampling with the interior point masses considered by
Olsen and Saksman.

The equivalences in the first subsection are established results. The
sampling characterization and the comparison at the sampling scale are
proved here.

\subsection{Classical equivalent formulations}
\label{subsec:classical-equivalence-map}

For \(\vartheta\in\R\), we write
$$
 \C_\vartheta=\{s\in\C:\Ree s>\vartheta\}.
$$
Let
\begin{equation}
\mathfrak C(z)
=
\frac12+\frac{1-z}{1+z},
\qquad z\in\mathbb D,
\label{eq:cayley-map-half-plane}
\end{equation}
so that \(\mathfrak C\) maps \(\mathbb D\) conformally onto
\(\C_{1/2}\).  The conformally invariant Hardy space
\(H_{\mathrm i}^p(\C_{1/2})\) consists of the holomorphic functions \(f\)
on \(\C_{1/2}\) such that \(f\circ\mathfrak C\in H^p(\mathbb D)\).  Its
boundary norm is
\begin{equation}
\norm f_{H_{\mathrm i}^p(\C_{1/2})}^{p}
=
\frac1\pi
\int_{\R}
\abs{f\!\left(\frac12+\ii t\right)}^p
\frac{\dd t}{1+t^2}.
\label{eq:conformal-half-plane-norm}
\end{equation}
By contrast, \(H^p(\C_{1/2})\) denotes the classical half-plane Hardy
space, defined using the unweighted \(L^p\)-boundary norm on
\(\Ree s=1/2\). On any fixed bounded interval, the weight in
\eqref{eq:conformal-half-plane-norm} is bounded above and below by
positive constants. 

A positive Borel measure \(\mu\) on \(\C_{1/2}\) is a Carleson measure
for a Hardy space \(X\) with exponent \(p\) if

$$
 \int_{\C_{1/2}}\abs{f(s)}^p\dd\mu(s)
 \lesssim
 \norm f_X^p
$$
for all \(f\in X\). Following Olsen and Saksman, we call such a measure
local if it has bounded support. The equivalence below concerns whether
measures that satisfy this inequality for classical half-plane Hardy
functions also satisfy it for Dirichlet series with the \(\HDir^p\)-norm.
The atomic testing condition uses measures of the form
\begin{equation}
\mu_S
=
\sum_n(2\sigma_n-1)\,\delta_{\sigma_n+\ii t_n},
\qquad
\sigma_n>\frac12.
\label{eq:olsen-saksman-atoms}
\end{equation}
The weight

$$
 2\sigma_n-1
 =
 2\left(\sigma_n-\frac12\right)
$$
is twice the distance of the atom from the critical boundary
\(\Ree s=1/2\). In particular, points nearer the boundary carry
smaller weights.

The Gordon--Hedenmalm class \(\mathcal G\) consists of symbols
\begin{equation}
\varphi(s)=c_0s+\varphi_0(s),
\qquad c_0\in\Nzero,
\label{eq:gordon-hedenmalm-symbol}
\end{equation}
where \(\varphi_0\) is a Dirichlet series converging uniformly in every
\(\C_\varepsilon\), \(\varepsilon>0\), and the following mapping
conditions hold: if \(c_0=0\), then
\(\varphi_0(\C_0)\subseteq\C_{1/2}\); if \(c_0\ge1\), then either
\(\varphi_0\equiv \ii\tau
\text{  for some }\tau\in\R\) or \(\varphi_0(\C_0)\subseteq\C_0\).  We write
\(\mathcal G_0\) for the characteristic-zero subclass and set
\(\mathcal C_\varphi f=f\circ\varphi\).  A particular characteristic-zero symbol used in the equivalence is
\begin{equation}
 \psi(s)
 =
 \mathfrak C(2^{-s})
 =
 \frac12+\frac{1-2^{-s}}{1+2^{-s}}.
 \label{eq:canonical-composition-symbol}
\end{equation}

To state the Carlson formulation explicitly, let \(r_j\) be the \(j\)-th
prime and put
\[
 \T_{1/2}^\infty
 =
 \{(r_j^{-1/2}\omega_j)_{j\ge1}:\omega\in\T^\infty\},
 \qquad
 \mathsf T_t(z_j)_{j\ge1}
 =
 (r_j^{-\ii t}z_j)_{j\ge1}.
\]
For \(F\in H^p(\mathbb D^\infty)\), we use the critical-boundary
representatives constructed by Saksman and Seip.  We write
\(\widetilde F_{1/2}\in L^p(\T^\infty)\) for the \(L^p\)-representative
of the restriction of \(F\) to the critical torus
\(\T_{1/2}^\infty\), and
\(\widetilde F(\mathsf T_t\tau)\) for the corresponding boundary value
along the prime orbit \(t\mapsto\mathsf T_t\tau\) issuing from
\(\tau\in\T_{1/2}^\infty\). 
The distinction here is between a time average along
one such orbit and the Haar average over $\T^\infty$;
the Carlson identity requires these averages to agree
for every starting point.

The following proposition records the precise exponent ranges of these
earlier equivalences.

\begin{proposition}[Equivalent classical formulations]
\label{prop:classical-equivalence-map}
Fix \(1\le p<\infty\).  The following assertions are equivalent.

\begin{enumerate}[label=\textup{(\roman*)}]
\item The local embedding property \(\mathrm{LEP}_p\) holds.

\item There is \(C_p<\infty\) such that
$$
 \norm P_{H_{\mathrm i}^p(\C_{1/2})}
 \le
 C_p\norm P_{\HDir^p}
$$
for every Dirichlet polynomial \(P\).

\item Every local Carleson measure for \(H^p(\C_{1/2})\) is a Carleson measure for \(\HDir^p\).

\item For every bounded set \(\Gamma\subset\C_{1/2}\) there exists
\(D_{p,\Gamma}<\infty\) such that every local Carleson measure
\(\mu_S\) for \(H^p(\C_{1/2})\) of the form
\eqref{eq:olsen-saksman-atoms}, with
\(\operatorname{supp}\mu_S\subset\Gamma\), is a Carleson measure for
\(\HDir^p\). More precisely, if \(C\) is a Carleson constant of
\(\mu_S\) for \(H^p(\C_{1/2})\), then \(D_{p,\Gamma}C\) is a
Carleson constant for \(\HDir^p\).

\item For every \(\varphi\in\mathcal G_0\), the composition operator
$$
 \mathcal C_\varphi\colon\HDir^p\longrightarrow\HDir^p
$$
is bounded.

\item The operator
$$
 \mathcal C_\psi\colon\HDir^p\longrightarrow\HDir^p,
$$
with \(\psi\) given by \eqref{eq:canonical-composition-symbol}, is
bounded.

\item A holomorphic self-map \(\varphi\) of \(\C_{1/2}\) induces a
bounded composition operator on \(\HDir^p\) if and only if it extends
holomorphically to \(\C_0\) as a symbol in \(\mathcal G\).
\end{enumerate}

If \(2\le p<\infty\), they are further equivalent to the following assertion:

\begin{enumerate}[label=\textup{(\roman*)},resume]
\item The \(p\)-Carlson identity
\begin{equation}
 \lim_{T\to\infty}
 \frac1T\int_0^T
 \abs{\widetilde F(\mathsf T_t\tau)}^p\dd t
 =
 \norm{\widetilde F_{1/2}}_{L^p(\T^\infty)}^p,
 \qquad
 \tau\in\T_{1/2}^\infty,
 \label{eq:p-carlson-identity}
\end{equation}
holds for every \(F\in H^p(\mathbb D^\infty)\) and every
\(\tau\in\T_{1/2}^\infty\).
\end{enumerate}
\end{proposition}

\begin{proof}
Bayart and Brevig proved
\(\textup{(i)}\Longleftrightarrow\textup{(ii)}\) in
\cite[Section~2.1]{BayartBrevig19}; the two formulations are related by
decomposing the weighted half-plane norm into unit intervals and using
vertical-translation invariance of the \(\HDir^p\)-norm.

The equivalence
\(\textup{(i)}\Longleftrightarrow\textup{(iii)}
\Longleftrightarrow\textup{(iv)}\)
is due to Olsen and Saksman
\cite[Theorem~4]{OlsenSaksman12}.  We use here the supportwise
quantitative form furnished by their proof. Namely, for every fixed bounded set
\(\Gamma\subset\C_{1/2}\), their closed-graph argument bounds
the \(\HDir^p\)-Carleson constant by a multiple of the
\(H^p(\C_{1/2})\)-Carleson constant.  Hence the comparison constant is uniform over measures
supported in \(\Gamma\), although it may depend on \(\Gamma\).
Conversely, in their proof that the atomic testing condition implies
the local embedding property, the relevant testing measures are all
supported in one fixed bounded set \(\Gamma\).  Thus this supportwise
form is sufficient and is equivalent to \(\mathrm{LEP}_p\).

For composition operators, Bayart and Brevig proved
\(\textup{(i)}\Longleftrightarrow\textup{(v)}
\Longleftrightarrow\textup{(vi)}\)
in \cite[Theorem~3]{BayartBrevig19}.
For the remaining parts of \textup{(vii)}, we use Bayart's
\(H^p\) composition-operator theorem
\cite{Bayart02}; see also the precise formulation in
\cite[Theorem~1.1]{BayartQueffelecSeip16}.
It gives the necessity that a bounded composition operator have a
Gordon--Hedenmalm symbol, and it gives boundedness for symbols of
positive characteristic.  The characteristic-zero sufficiency for
the present exponent \(p\) is supplied precisely by \textup{(v)}.
Thus \textup{(i)} implies \textup{(vii)}.  Conversely,
\textup{(vii)} applied to \(\psi\in\mathcal G_0\) gives
\textup{(vi)}, and hence \textup{(i)}.

Finally, for \(2\le p<\infty\), Saksman and Seip proved the equivalence
of \(\mathrm{LEP}_p\) and the \(p\)-Carlson identity
\eqref{eq:p-carlson-identity}; see
\cite[Section~3, equation~(17), and Section~4,
equation~(20)]{SaksmanSeip09}.
\end{proof}

Combining the sharp embedding range in
Corollary~\ref{cor:sharp-finite-range} with
Proposition~\ref{prop:classical-equivalence-map} gives the following
conclusion. 

\begin{corollary}[Sharp range in the classical formulations]
\label{cor:sharp-classical-equivalences}
 Assertions \textup{(i)}--\textup{(vii)} in Proposition~\ref{prop:classical-equivalence-map}
hold exactly for \(p\ge2\); assertion \textup{(viii)}
holds in its stated range.
\end{corollary}

\subsection{Discrete sampling and its dual formulation}
\label{subsec:critical-atomic-sampling}

We now give a sampling characterization of local embedding. Unlike the
prime cutoffs used in the proof of the main theorem, the integer cutoff
\(n<X\) below bounds the frequencies of the resulting function of \(t\) in \eqref{eq:finite-band-boundary-function}.

For an integer \(X\ge3\), we define
\begin{equation}
 \mathcal P_X
 =
 \left\{P(s)=\sum_{1\le n<X}a_n n^{-s}\right\},
 \qquad
 L_X=\log X.
 \label{eq:finite-band-dirichlet-class}
\end{equation}
For \(P\in\mathcal P_X\), put
\begin{equation}
 f_P(t)
 =
 P\!\left(\frac12+\ii t\right)
 =
 \sum_{n<X}a_n n^{-1/2}\e^{-\ii t\log n}.
 \label{eq:finite-band-boundary-function}
\end{equation}
With the Fourier convention
\[
 \widehat f(\xi)=\int_{\R}f(t)\e^{-\ii t\xi}\dd t,
 \qquad
 f(t)=\frac1{2\pi}\int_{\R}\widehat f(\xi)\e^{\ii t\xi}\dd\xi,
\]
the distributional Fourier support of \(f_P\) is contained in
\([-L_X,0]\). Thus \(L_X^{-1}=(\log X)^{-1}\) is the natural
sampling scale. The next lemma compares values at points separated
on this scale with integrals over unit intervals; the contribution of
distant intervals decreases rapidly.

\begin{lemma}[Sampling estimate for separated points]
\label{lem:reciprocal-bandwidth-sampling}
Fix \(0<c\le1\), \(M>1\), and \(1<u<\infty\).  Let \(L\ge1\) and
\[
 f(t)=\sum_{\nu=1}^m d_\nu\e^{-\ii \lambda_\nu t},
 \qquad \lambda_\nu\in[-L,L].
\]
Let \(I\subset\R\) be an interval of length one, and let
\(\mathcal T\subset I\) be finite and satisfy
\[
 \abs{\tau-\sigma}\ge\frac cL,
 \qquad \tau\ne\sigma.
\]
Then, we have
\begin{equation}
 \frac1L\sum_{\tau\in\mathcal T}\abs{f(\tau)}^u
 \le
 C_{c,M,u}
 \sum_{k\in\Z}(1+\abs k)^{-M}
 \int_{I+k}\abs{f(t)}^u\dd t,
 \label{eq:reciprocal-bandwidth-sampling}
\end{equation}
where the constant is independent of \(L\), \(I\), \(\mathcal T\), and
\(f\).
\end{lemma}

\begin{proof}
We choose \(\kappa\in\mathcal S(\R)\) such that
\(\widehat\kappa=1\) on \([-1,1]\), and put
\[
 K_L(t)=L\kappa(Lt).
\]
Then \(\widehat K_L(\xi)=\widehat\kappa(\xi/L)\), so the Fourier support
assumption gives the exact reproducing identity \(f=f*K_L\).  H\"older's
inequality yields
\begin{align*}
 \abs{f(\tau)}^u
 \le
\norm{K_L}_{L^1}^{u-1}
\int_{\R}\abs{f(t)}^u\abs{K_L(\tau-t)}\dd t=
\norm\kappa_{L^1}^{u-1}L
\int_{\R}\abs{f(t)}^u
\abs{\kappa(L(\tau-t))}\dd t.
\end{align*}
After summing over \(\tau\) and dividing by \(L\), it remains to bound
\[
 \sum_{\tau\in\mathcal T}
 \abs{\kappa(L(\tau-t))}.
\]
The set \(L\mathcal T\) is \(c\)-separated.  Since \(\kappa\) is
Schwartz, a decomposition into unit annuli about \(Lt\) gives
\[
 \sum_{\tau\in\mathcal T}
 \abs{\kappa(L(\tau-t))}
 \le
 C_{c,M}\bigl(1+L\dist(t,I)\bigr)^{-M}.
\]
For \(t\in I+k\), this is at most
\(C_{c,M}(1+\abs k)^{-M}\), after enlarging the constant for the finitely
many neighboring translates.  Substitution proves
\eqref{eq:reciprocal-bandwidth-sampling}.
\end{proof}

Each sample above is given weight \(1/L_X\), the scale of the prescribed
minimum separation. A unit interval contains at most \(O_c(L_X)\)
separated points, so the total weight stays bounded independently
of \(X\) and the sample set.

Fix \(1<p<\infty\), and let
\[
 p'=\frac{p}{p-1}.
\]
For \(s\in\C\), we write
\[
 \operatorname{ev}_s(P)=P(s),
 \qquad P\in\mathcal P_X.
\]

\begin{definition}[Critical atomic sampling]
\label{def:critical-atomic-sampling}
Fix \(0<c\le1\).  The property \(\mathrm{CAS}_{p,c}\) holds if there is
\(A_{p,c}<\infty\) such that, for every integer \(X\ge3\), every unit
interval \(I\), every finite set \(\mathcal T\subset I\) satisfying
\[
 \abs{\tau-\sigma}\ge\frac{c}{L_X},
 \qquad \tau\ne\sigma,
\]
and every \(P\in\mathcal P_X\),
\begin{equation}
 \left(
 \frac1{L_X}
 \sum_{\tau\in\mathcal T}
 \abs{P\!\left(\frac12+\ii\tau\right)}^p
 \right)^{1/p}
 \le
 A_{p,c}\norm P_{\HDir^p}.
 \label{eq:critical-atomic-sampling}
\end{equation}
We call such a pair \((X,\mathcal T)\) admissible.
Equivalently, the boundary atomic measures
\[
\nu_{X,\mathcal T}
=
\frac1{L_X}
\sum_{\tau\in\mathcal T}
\delta_{1/2+\ii\tau}
\]
satisfy uniform embedding estimates
on the corresponding spaces
$(\mathcal P_X,\norm{\cdot}_{\HDir^p})$.
\end{definition}

\begin{definition}[Dual sampling inequality]
\label{def:dual-critical-packet}
Fix \(0<c\le1\).  The property \(\mathrm{DCP}_{p,c}\) holds if there is
\(B_{p,c}<\infty\) such that every admissible pair
\((X,\mathcal T)\) and every complex vector
\((b_\tau)_{\tau\in\mathcal T}\) satisfy
\begin{equation}
 \sup_{\substack{P\in\mathcal P_X\\
                  \norm P_{\HDir^p}\le1}}
 \left|
 L_X^{-1/p}
 \sum_{\tau\in\mathcal T}
 b_\tau
 P\!\left(\frac12+\ii\tau\right)
 \right|
 \le
 B_{p,c}
 \left(
 \sum_{\tau\in\mathcal T}\abs{b_\tau}^{p'}
 \right)^{1/p'}.
 \label{eq:dual-critical-packet}
\end{equation}
\end{definition}

The dual inequality bounds all linear combinations of the same
point-evaluation functionals. The next theorem shows that either form
of the sampling estimate is equivalent to local embedding, not merely
a consequence of it.

\begin{theorem}[Critical sampling bridge]
\label{thm:critical-sampling-bridge}
Fix \(1<p<\infty\).  For every \(0<c\le1\), one has
\begin{equation}
 \mathrm{LEP}_p
 \quad\Longleftrightarrow\quad
 \mathrm{CAS}_{p,c}
 \quad\Longleftrightarrow\quad
 \mathrm{DCP}_{p,c}.
 \label{eq:critical-sampling-spine}
\end{equation}
More precisely, \(\mathrm{LEP}_p\) implies
\(\mathrm{CAS}_{p,c}\) for every \(c\in(0,1]\), while
\(\mathrm{CAS}_{p,c}\) for some \(c\in(0,1]\) already implies
\(\mathrm{LEP}_p\).
\end{theorem}

\begin{proof}
Assume first that \(\mathrm{LEP}_p\) holds.  Fix \(c\in(0,1]\) and
\(M>1\). We apply
Lemma~\ref{lem:reciprocal-bandwidth-sampling} to \(f_P\) from
\eqref{eq:finite-band-boundary-function}, with \(u=p\) and \(L=L_X\). Using the local embedding estimate on each
unit interval \(I+k\) gives
\[
 \frac1{L_X}
 \sum_{\tau\in\mathcal T}\abs{f_P(\tau)}^p
 \le
 C_{p,c,M}
 \sum_{k\in\Z}(1+\abs k)^{-M}
 \norm P_{\HDir^p}^p.
\]
Since \(M>1\), this proves \(\mathrm{CAS}_{p,c}\).

Conversely, suppose that \(\mathrm{CAS}_{p,c}\) holds for some
\(c\in(0,1]\).  Let \(I=[\theta,\theta+1]\), and partition \(I\) into
consecutive subintervals \(Q_j\) of length \(c/L_X\), except possibly
the last.  We choose \(\tau_j\in Q_j\) so that
\[
 \abs{f_P(\tau_j)}
 =
 \max_{t\in Q_j}\abs{f_P(t)}.
\]
Then we have
\begin{equation}
 \int_I\abs{f_P(t)}^p\dd t
 \le
 \frac{c}{L_X}\sum_j\abs{f_P(\tau_j)}^p.
 \label{eq:critical-sampling-cell-max}
\end{equation}
Splitting the selected points into the even and odd index classes gives
two admissible sets, since points two cells apart are separated by at
least \(c/L_X\).  Applying
\eqref{eq:critical-atomic-sampling} to both classes and using
\eqref{eq:critical-sampling-cell-max}, we obtain
\[
 \int_I\abs{f_P(t)}^p\dd t
 \le
 2cA_{p,c}^p\norm P_{\HDir^p}^p.
\]
The bound is independent of \(X\) and \(\theta\), and every Dirichlet
polynomial belongs to \(\mathcal P_X\) for all sufficiently large \(X\).
Hence \(\mathrm{LEP}_p\) holds.

For the dual formulation, we define
\[
 S_{X,\mathcal T}^{(p)}
 \colon
 (\mathcal P_X,\norm{\cdot}_{\HDir^p})
 \longrightarrow
 \ell^p(\mathcal T),
 \qquad
 S_{X,\mathcal T}^{(p)}P
 =
 \bigl(L_X^{-1/p}P(1/2+\ii\tau)\bigr)_{\tau\in\mathcal T}.
\]
Then \(\mathrm{CAS}_{p,c}\) is precisely the uniform boundedness of these
operators.  Under the complex-linear identification
\((\ell^p(\mathcal T))^*=\ell^{p'}(\mathcal T)\), one has
\[
 (S_{X,\mathcal T}^{(p)})^*b
 =
 L_X^{-1/p}
 \sum_{\tau\in\mathcal T}
 b_\tau\operatorname{ev}_{1/2+\ii\tau}.
\]
Thus \(\mathrm{DCP}_{p,c}\) is precisely the corresponding uniform
adjoint estimate.  Since
\(\norm{S_{X,\mathcal T}^{(p)}}=
\norm{(S_{X,\mathcal T}^{(p)})^*}\),
the two properties are equivalent.
\end{proof}

\begin{remark}[The critical sampling scale]
\label{R:why-CAS}
The integer cutoff $n<X$ restricts the
Fourier frequencies of the critical-line
function $f_P$ to $[-L_X,0]$,
where $L_X=\log X$.
This explains the sampling scale $1/L_X$.
Although the original local embedding
property imposes no such cutoff,
Theorem~\ref{thm:critical-sampling-bridge}
shows that the sampling estimates are
equivalent to it when their constants
are uniform in $X$.

The discrete formulation also provides
a way to view the estimates used in the
proof of the main theorem.
A single critical-line evaluation can be
represented using a normalized Euler product.
When several evaluations are considered
together, the corresponding Euler products
depend on the same prime coordinates.
The proof of the main theorem exploits
this shared dependence to obtain estimates
for multiple critical-line samples.

The sampling characterization itself is
not used in that proof, but it gives a
discrete interpretation of the local
embedding property.
\end{remark}

Combining Theorem~\ref{thm:critical-sampling-bridge} with
Corollary~\ref{cor:sharp-finite-range} gives the sharp range.

\begin{corollary}[Sharp critical-sampling range]
\label{cor:critical-sampling-sharp-range}
For each fixed \(0<c\le1\), the equivalent properties
\(\mathrm{CAS}_{p,c}\) and \(\mathrm{DCP}_{p,c}\) hold exactly for
\(2\le p<\infty\).
\end{corollary}

\subsection{Interior point masses at the sampling scale}
\label{subsec:critical-os-atoms}

We next compare the boundary sampling points with the interior atoms
of Olsen and Saksman by moving each point a distance \(\gamma/L_X\) into
the half-plane and assigning it weight \(2\gamma/L_X\). Here \(\gamma>0\) is
fixed. For an admissible pair \((X,\mathcal T)\), we define the measure
\begin{equation}
 \mu_{X,\mathcal T}^{(\gamma)}
 =
 \frac{2\gamma}{L_X}
 \sum_{\tau\in\mathcal T}
 \delta_{\frac12+\gamma/L_X+\ii\tau}.
 \label{eq:flat-critical-lift}
\end{equation}
Since
\[
 2\left(\frac12+\frac{\gamma}{L_X}\right)-1
 =
 \frac{2\gamma}{L_X},
\]
the measure \(\mu_{X,\mathcal T}^{(\gamma)}\) is exactly of the
Olsen--Saksman form \eqref{eq:olsen-saksman-atoms}. Its atoms all lie
at the same distance from the boundary. The following estimate shows
that their separation gives a uniform classical Carleson bound.

\begin{proposition}[Uniform Carleson bounds for interior point masses]
\label{prop:flat-critical-carleson}
Fix \(\gamma>0\) and \(0<c\le1\).  For every admissible pair
\((X,\mathcal T)\), the measure \(\mu_{X,\mathcal T}^{(\gamma)}\) satisfies
\[
 \mu_{X,\mathcal T}^{(\gamma)}(Q(J))
 \le
 2\left(1+\frac {\gamma}{c}\right)|J|
\]
for every half-plane Carleson square \(Q(J)\).  Consequently, for each
\(1\le p<\infty\), the measures \(\mu_{X,\mathcal T}^{(\gamma)}\) are
uniformly Carleson for \(H^p(\C_{1/2})\), independently of
\(X\) and \(\mathcal T\).
\end{proposition}

\begin{proof}
Let \(J\subset\R\) be an interval of length \(\ell\), and we write
\[
 Q(J)
 =
 \left\{
 s\in\C_{1/2}:
 \frac12<\Ree s<\frac12+\ell,\ 
 \mathrm{Im} s\in J
 \right\}.
\]
If \(\ell<\gamma/L_X\), then \(Q(J)\) does not reach the line carrying the
atoms of \(\mu_{X,\mathcal T}^{(\gamma)}\), and hence
\[
 \mu_{X,\mathcal T}^{(\gamma)}(Q(J))=0.
\]
If \(\ell\ge \gamma/L_X\), the \(c/L_X\)-separation of \(\mathcal T\) gives
\[
 \#(\mathcal T\cap J)
 \le
 1+\frac{\ell L_X}{c}.
\]
Therefore, one has
\[
 \mu_{X,\mathcal T}^{(\gamma)}(Q(J))
 \le
 \frac{2\gamma}{L_X}
 \left(1+\frac{\ell L_X}{c}\right)
 \le
 2\left(1+\frac {\gamma}{c}\right)\ell.
\]
\end{proof}

For \(p\ge2\), the classical Carleson bound can be transferred to
\(\HDir^p\) using Proposition~\ref{prop:classical-equivalence-map}.
The uniformity requires the atomic measures to be placed in a fixed
bounded region, as in the proof below.

\begin{corollary}[Sampling inequality for interior points]
\label{cor:flat-critical-sampling}
Fix \(2\le p<\infty\), \(\gamma>0\), and \(0<c\le1\).  There is
\(C_{p,\gamma,c}<\infty\) such that, for every admissible pair
\((X,\mathcal T)\) and every \(P\in\mathcal P_X\), one has
\[
 \left(
 \frac{2\gamma}{L_X}
 \sum_{\tau\in\mathcal T}
 \left|
 P\!\left(\frac12+\frac{\gamma}{L_X}+\ii\tau\right)
 \right|^p
 \right)^{1/p}
 \le
 C_{p,\gamma,c}\norm P_{\HDir^p}.
\]
\end{corollary}

\begin{proof}
By Proposition~\ref{prop:flat-critical-carleson}, the measures
\(\mu_{X,\mathcal T}^{(\gamma)}\) have uniformly bounded Carleson constants
for \(H^p(\C_{1/2})\).
Let \(I=[b,b+1]\) be a unit interval containing \(\mathcal T\), and
we translate the measure vertically by \(-\ii b\).  Both the classical
\(H^p(\C_{1/2})\)-Carleson norm and the \(\HDir^p\)-Carleson norm are
invariant under vertical translations.  Since \(X\ge3\), we have 
\(
 \frac{\gamma}{L_X}\le\frac{\gamma}{\log3},
\)
so every translated measure is supported in the fixed bounded set
\[
 \Gamma_\gamma
 =
 \left\{
 s\in\C_{1/2}:
 \frac12<\Ree s\le\frac12+\frac{\gamma}{\log3},
 \quad 0\le\mathrm{Im}\,s\le1
 \right\}.
\]
Proposition~\ref{prop:classical-equivalence-map}\textup{(iv)}, with
\(\Gamma=\Gamma_\gamma\), together with
Corollary~\ref{cor:sharp-finite-range}, therefore gives a uniform
\(\HDir^p\)-Carleson constant depending only on \(p,\gamma,c\).
Applying the resulting Carleson embedding to \(P\) gives the stated
estimate.
\end{proof}

\begin{remark}[Critical-scale correspondence]
\label{rem:critical-sampling-geometry}
For fixed \(\gamma>0\) and \(0<c\le1\), the boundary samples and
interior atoms are related by the common scale
\[
 \text{interior depth}
 \asymp
 \text{atomic mass}
 \asymp
 \text{critical sampling scale}
 \asymp
 (\log X)^{-1}.
\]
Indeed, these three quantities are respectively
\(\gamma/L_X\), \(2\gamma/L_X\), and \(c/L_X\).

For \(1<p<\infty\), the adjoint of the weighted sampling operator
associated with \(\mu_{X,\mathcal T}^{(\gamma)}\) is represented by the
following linear combination of point-evaluation functionals:
\[
 \left(\frac{2\gamma}{L_X}\right)^{1/p}
 \sum_{\tau\in\mathcal T}
 b_\tau\operatorname{ev}_{1/2+\gamma/L_X+\ii\tau},
\]
whereas \(\mathrm{DCP}_{p,c}\) uses
\( 
 L_X^{-1/p}
 \sum_{\tau\in\mathcal T}
 b_\tau\operatorname{ev}_{1/2+\ii\tau}.
\)
Thus the two linear combinations have the same normalization up to
the fixed factor \((2\gamma)^{1/p}\), but their evaluations occur on
different vertical lines: one in the interior and one on the critical
boundary.
\end{remark}

\newpage
\section*{Principal notation}

\begingroup
\small
\setlength{\LTleft}{0pt}
\setlength{\LTright}{0pt}
\setlength{\LTpre}{0.4em}
\setlength{\LTpost}{0pt}
\setlength{\tabcolsep}{4pt}
\renewcommand{\arraystretch}{1.03}

\begin{longtable}{@{}
  >{\raggedright\arraybackslash}p{0.25\textwidth}
  >{\raggedright\arraybackslash}p{0.58\textwidth}
  >{\raggedright\arraybackslash}p{0.13\textwidth}@{}}
\toprule
\textbf{Notation} & \textbf{Meaning} & \textbf{Location}\\
\midrule
\endfirsthead

\toprule
\textbf{Notation} & \textbf{Meaning} & \textbf{Location}\\
\midrule
\endhead

\endfoot
\bottomrule
\endlastfoot

\multicolumn{3}{@{}l}{\emph{Spaces and exponents}}\\*
\addlinespace[0.15em]

\(\Bohr P,\ \HDir^p\)
&
Bohr lift and Hardy space of Dirichlet series
&
Def.~\ref{def:bohr-lift}
\\

\(\mathrm{LEP}_p\)
&
The local embedding property
&
Def.~\ref{def:lep-p}
\\

\(q,s,\alpha,\delta,a,\beta\)
&
Exponent parameters associated with a fixed \(p>2\)
&
\S\ref{subsec:hardy-dirichlet-spaces}
\\

\addlinespace[0.22em]
\multicolumn{3}{@{}l}{\emph{Prime model and Euler products}}\\*
\addlinespace[0.15em]

\(\Omega_y,\ (\zeta_r)_{r\le y},\ U_u\)
&
Finite prime torus, coordinate variables, and prime flow
&
\S\ref{subsec:finite-prime-tori}
\\

\(D_{\mathrm{full}},\ D_{r^-},\
  \mathscr F_{r^-},\ \E_{r^-}\)
&
Prime-frequency generators, earlier-prime filtration, and conditional expectation
&
\S\ref{subsec:finite-prime-tori}
\\

\(P_E^{\mathrm{full}},\ P_E^{r^-}\)
&
Full and earlier-prime spectral projections
&
\S\ref{subsec:finite-prime-tori}
\\

\(M_b,\ S_b(t),\ f_b(t)\)
&
Mertens product, analytic Euler product, and normalized Euler-product source
&
\S\ref{subsec:euler-sources}
\\

\(f_{r^-}(t),\ f_h^\vartheta(t)\)
&
Normalized earlier-prime product and deformed common product (\(f_h^0=f_{\e^{2^h}}\))
&
\S\S~\ref{subsec:euler-sources},~\ref{subsec:cyclic-model}
\\

\addlinespace[0.22em]
\multicolumn{3}{@{}l}{\emph{Cyclic and multiscale notation}}\\*
\addlinespace[0.15em]

\(H,N,y,(t_j)_{j<N}\)
&
Cyclic parameters and sample grid; \(N=2^H\), \(y=\e^N\)
&
\S\ref{subsec:cyclic-model}
\\

\(\ZN,\ \Lambda_N,\ \widehat c\)
&
Cyclic group, frequency representatives, and discrete Fourier transform
&
\S\ref{subsec:cyclic-model}
\\

\(\theta_\kappa,\ \xi_\kappa,\ \rho_N\)
&
Angular frequency, physical frequency, and cyclic distance
&
\S\ref{subsec:cyclic-model}
\\

\(w_h,\ m_h,\ \mathcal D_h,\ Q_h(x)\)
&
Dyadic scale, block length, nonwrapping grid, and containing block
&
\S\ref{subsec:cyclic-model}
\\

\(\mathcal V_N,\ d_v\)
&
Positive dyadic scales and their coefficient-frequency components
&
\S\S~\ref{subsec:boundary-traces},~\ref{subsec:coefficient-square-functions}
\\

\multicolumn{3}{@{}l}{\emph{Gaussian fields, tree moments, and square functions}}\\*
\addlinespace[0.15em]

\(W_n,\ Y_n(\beta)\)
&
Critical-tree martingale and weighted leaf functional
&
\S\ref{subsec:critical-tree-functional};
\\

\(Z_w^\vartheta(t),\ Z_w(t),\ V_w\)
&
Deformed/undeformed Gaussian fields and the variance of $Z_w$
&
\S\S~\ref{subsec:gaussian-comparison-deformation},~\ref{subsec:localized-gaussian-occupation}
\\

\(S_{\mathrm{cone}}z,\ \mathfrak P_y^\vartheta(z)\)
&
Conical square function and Euler-product moment functional
&
\S\ref{sec:euler-occupation}
\\

\addlinespace[0.22em]
\multicolumn{3}{@{}l}{\emph{Sampling, frequency decomposition, and duality}}\\*
\addlinespace[0.15em]

\(S_hu,\ \mathcal I_h\)
&
Prime sampling operator and prime band
&
\S\ref{sec:frame-and-tails}
\\

\(\mathscr D_r\)
&
Full prime-frequency generator shifted by $\log r$
&
\S\ref{subsec:finite-cauchy-calculus}
\\

\addlinespace[0.15em]

\(\Pi_-,\ \Pi_+,\ \mathcal T_y c\)
&
Nonpositive and positive cyclic projections, and the projected vector
&
\S\ref{subsec:corrected-column-definition}
\\

\(\mathcal R_y,\ \mathcal R_y^{(\ell)}\)
&
Resonant term and its dyadic offset components
&
\S\S~\ref{subsec:resonant-row-arrays},~
       \ref{subsec:corrected-column-definition}
\\

\(\mathcal H_y,\ \mathcal T_y^{\mathrm{lo}},\
  \mathcal L_y^{\mathrm{lo}},\ \mathcal O_y\)
&
High-frequency term, low-frequency term,
and nonresonant components
&
\S\S~\ref{subsec:high-source-tail},~\ref{subsec:exact-frequency-partition}
\\

\(\mathcal N_p,\ [G],\ g_{y,t}\)
&
Hardy annihilator, quotient class, and normalized Euler-product representative
&
\S\S~\ref{subsec:hardy-quotient-euler-source},~
       \ref{subsec:euler-quotient-representative}
\\

\(A_r(c),\ \mathcal Q_y(c)\)
&
First-harmonic coefficient and higher-harmonic remainder
&
\S\ref{subsec:euler-quotient-representative}
\\

\(\mathcal P_X,\ L_X\)
&
Finite-band Dirichlet-polynomial class and logarithmic bandwidth
&
\S\ref{subsec:critical-atomic-sampling}
\\

\(\mathrm{CAS}_{p,c},\ \mathrm{DCP}_{p,c}\)
&
Critical atomic sampling and its dual inequality
&
\S\ref{subsec:critical-atomic-sampling}
\\

\(\mu_{X,\mathcal T}^{(\gamma)}\)
&
Interior atomic measure with depth $\gamma/L_X$ and mass $2\gamma/L_X$ per atom
&
\S\ref{subsec:critical-os-atoms}
\\
\end{longtable}

\section*{Acknowledgments}

\noindent 
\textit{Funding.}
X. F. was supported by the National Science and Technology
Council of Taiwan (Grant No.~114-2115-M-A49-003-MY3).
S. H. was supported by the National Natural Science Foundation
of China (Grant No.~12371133).
Q. Z. was supported by the National Natural Science Foundation
of China (Grant No.~12501162) and the Natural Science Foundation of Jiangsu
Province (Grant No.~BK20250832).

\bigskip

\noindent 
\textit{AI Statement.}
The authors used artificial-intelligence tools for language editing,
\LaTeX\ formatting, and assistance with local mathematical
reasoning.  The proof strategy and mathematical development are the
authors' own; all mathematical content was independently written and checked by the
authors, who take full responsibility.

\end{document}